\documentclass[11pt,reqno, oneside]{amsbook}

\usepackage[a4paper, total={6in, 9in}]{geometry}

\usepackage{graphicx}
\usepackage{amssymb}
\usepackage{epstopdf}
\usepackage{cases}
\usepackage{hyperref}
\usepackage[abbrev,nobysame]{amsrefs} 

\usepackage{amsmath,amsfonts,amsthm,mathrsfs,amssymb,cite}
\usepackage[usenames]{color}

\usepackage{empheq}
\newcommand*\widefbox[1]{\fbox{\hspace{2em}#1\hspace{2em}}}

\newtheorem{thm}{Theorem}[chapter]

\newtheorem{cor}[thm]{Corollary}
\newtheorem{lem}[thm]{Lemma}
\newtheorem{prop}[thm]{Proposition}
\theoremstyle{definition}
\newtheorem{defn}[thm]{Definition}

\newtheorem{exmp}[thm]{Example}

\theoremstyle{remark}
\newtheorem*{rem*}{Remark}
\newtheorem{rem}[thm]{Remark}
\numberwithin{equation}{chapter}

\newtheorem{ex}{Exercise}[chapter]

\usepackage{tikz}
\allowdisplaybreaks

\newcommand{\R}{\mathbb{R}} \newcommand{\mathR}{\mathbb{R}}
\newcommand{\Rn}{{\mathR^n}}

\newcommand{\N}{\mathbb{N}}
\newcommand{\scrD}{\mathscr{D}}
\newcommand{\calF}{\mathcal{F}} \newcommand{\calinF}{\mathcal{F}^{-1}}
\newcommand{\calR}{\mathcal{R}}

\newcommand{\scrS}{\mathscr{S}}
\newcommand{\scrE}{\mathscr{E}}

\newcommand{\supp}{\mathop{\rm supp}}
\newcommand{\ssupp}{\mathop{\rm sing\,supp}}

\newcommand{\Forall}{{~\forall\,}}
\newcommand{\Exists}{{~\exists\,}}
\newcommand{\st}{\textrm{~s.t.~}}

\DeclareMathOperator{\Smo}{Smo}
\DeclareMathOperator{\Char}{Char}
\DeclareMathOperator{\wf}{WF}

\newcommand{\df}{\mathrm{d}}
\newcommand{\dif}[1]{\,\mathrm{d}{#1}} 
\newcommand{\nrm}[2][]{ \| {#2} \|_{#1}}
\newcommand{\agl}[1][\cdot]{ \langle {#1} \rangle}
\newcommand{\Agl}[1][\cdot]{ \big\langle {#1} \big\rangle} 
\newcommand{\comment}[1]{} 		

\usepackage{tikz}

\newcommand{\sq}[1]{{\color{red}#1}}

\newcommand{\sgn}{\mathop{\rm sgn}}

\title{Lecture notes for pseudodifferential operators and microlocal analysis}
\author{Ma, Shiqi}
\address{Department of Mathematics and Statistics, University of Jyv\"askyl\"a}
\email{mashiqi01@gmail.com, shiqi.s.ma@jyu.fi}

\usepackage{imakeidx}
\makeindex

\begin{document}

\begin{abstract}
    This is a introductory course focusing some basic notions in pseudodifferential operators ($\Psi$DOs) and microlocal analysis. 
    We start this lecture notes with some notations and necessary preliminaries.
    Then the notion of symbols and $\Psi$DOs are introduced.
    In Chapter \ref{ch:OSInt-PM2021} we define the oscillatory integrals of different types.
    Chapter \ref{ch:SPL-PM2021} is devoted to the stationary phase lemmas.
    One of the features of the lecture is that the stationary phase lemmas are proved for not only compactly supported functions but also for more general functions with certain order of smoothness and certain order of growth at infinity.
    We build the results on the stationary phase lemmas.
    Chapters \ref{ch:SCP-PM2021}, \ref{ch:BddP-PM2021} and \ref{ch:SCla-PM2021} covers main results in $\Psi$DOs and the proofs are heavily built on the results in Chapter \ref{ch:SPL-PM2021}.
    Some aspects of the semi-classical analysis are similar to that of microlocal analysis.
    In Chapter \ref{ch:WF-PM2021} we finally introduce the notion of wavefront, and Chapter \ref{ch:ProSing-PM2021} focuses on the propagation of singularities of solution of partial differential equations.
    Important results are circulated by black boxes and some key steps are marked in red color.
    Exercises are provided at the end of each chapter.

    \vspace*{1em}
    
    \noindent Version: May, 2021.
\end{abstract}

\maketitle

\tableofcontents

\chapter{Preliminaries} \label{ch:pre-PM2021}

A good reference for is \cite[Chapters 1-5]{wong2014introduction}.

\section{Notations} \label{sec:nt-PM2021}

$\Rn$: the Euclidean space. For $x \in \Rn$, $|x| := \sqrt{x_1^2 + \cdots + x_n^2}$, and the inter produce $x \cdot y := \sum_{j = 1}^n x_j y_j$.
The notation $\agl[x] := (1+|x|^2)^{1/2}$ will be frequently used throughout the lecture.
For two quantities $\mathcal{A}$ and $\mathcal{B}$, we write $\mathcal{A}\lesssim \mathcal{B}$ to signify $\mathcal{A}\leq C \mathcal{B}$,
and write $\mathcal{A} \simeq \mathcal{B}$ to signify $C_1 \mathcal{B} \leq \mathcal{A} \leq C_2 \mathcal{B}$, for some generic positive constants $C$, $C_1$ and $C_2$.
It can be checked that $\agl[x] \simeq 1+|x|$.

\begin{lem} \label{lem:aX-PM2021}
	For any $s \in \R$ and any multi-index $\alpha$, there exists a constant $C$ independent of $x$ such that
	\[
	|\partial^\alpha (\agl[x]^s)| \leq C \agl[x]^{s-|\alpha|},
	\quad \text{when} \quad |x| \geq 1.
	\]
\end{lem}

The proof is left as an exercise.

$C^m(\Rn; \mathbb C)$ is the set of complex-valued functions that has continuous derivative up to order $m$.
$C_c^\infty(\Rn)$ is comprised of $C^\infty$ functions with compact support.

The Fourier and inverse Fourier transforms of $f$ are denoted as $\calF f$ (also $\hat f$) and $\calF^{-1} f$ (also $\check{f}$):
\begin{align*}
\hat f(\xi) = \calF f(\xi) & := (2\pi)^{-n/2} \int_{\Rn} e^{-ix\cdot \xi} f(x) \dif x, \\
\check{f}(\xi) = \calF^{-1} f(x) & := (2\pi)^{-n/2} \int_{\Rn} e^{ix\cdot \xi} f(\xi) \dif \xi.
\end{align*}

$\agl[f,g] := \int_{\Rn} f(x) g(x) \dif x$, $(f,g) := \int_{\Rn} f(x) \overline{g(x)} \dif x$, where $\overline{g(x)}$ is the complex conjugation of $g(x)$.

$\partial_j := \frac {\partial} {\partial x_j}$, $\boxed{D_j := \frac 1 i \partial_j}$, where $i$ is the imaginary unit.

\smallskip

Multi-index:
in $\Rn$, a multi-index is $\alpha = (\alpha_1, \dots, \alpha_n)$ where $\alpha_j$ are non-negative integers.
$D^\alpha := D_1^{\alpha_1} \cdots D_n^{\alpha_n}$, $\partial^\alpha := \partial_1^{\alpha_1} \cdots \partial_n^{\alpha_n}$, and $x^\alpha := x_1^{\alpha_1} \cdots x_n^{\alpha_n}$, and the length of $\alpha$ is $|\alpha| := \alpha_1 + \cdots + \alpha_n$.

\begin{lem} \label{lem:xa-PM2021}
	Assume $x \in \Rn$ and $\alpha$ is a multi-index.
	Then
	\[
	|x^\alpha| \leq |x|^{|\alpha|}.
	\]
\end{lem}

\begin{proof}
	We have
	\begin{align*}
	|x^\alpha|
	& = |x_1^{\alpha_1} \cdots x_n^{\alpha_n}|
	= |x_1^{\alpha_1}| \cdots |x_n^{\alpha_n}|
	= |x_1|^{\alpha_1} \cdots |x_n|^{\alpha_n} \\
	& \leq |x|^{\alpha_1} \cdots |x|^{\alpha_n}
	= |x|^{\alpha_1 + \cdots + \alpha_n}
	= |x|^{|\alpha|}.
	\end{align*}
\end{proof}

More on multi-index:
\begin{itemize}
	\item $\beta \leq \alpha$ means $\beta_j \leq \alpha_j$ for $j = 1, \dots,n$;
	
	\item the notion $\alpha - \beta$ is valid only when $\beta \leq \alpha$, and $\alpha - \beta := (\alpha_1 - \beta_1, \cdots, \alpha_n - \beta_n)$;
	
	\item $\alpha! := \alpha_1 ! \cdots \alpha_n!$;
	
	\item when $\beta \leq \alpha$, $\binom{\alpha}{\beta} := \frac{\alpha!}{\beta! (\alpha - \beta)!} = \binom{\alpha_1}{\beta_1} \cdots \binom{\alpha_n}{\beta_n}$, where $\binom{\alpha_j}{\beta_j} = \frac{\alpha_j!}{\beta_j! (\alpha_j - \beta_j)!}$;
\end{itemize}

\smallskip

A typical form of a linear differential operator is $\sum_{|\alpha| \leq m} a_\alpha(x) D^\alpha$.
If we denote a polynomial $p(x,\xi) := \sum_{|\alpha| \leq m} a_\alpha(x) \xi^\alpha$ where $\xi \in \Rn$, then
\[
\sum_{|\alpha| \leq m} a_\alpha(x) D^\alpha = p(x,D).
\]

\begin{lem} \label{lem:Dfg-PM2021}
	Assume $f, g \in C^\infty(\Rn)$ and $\alpha$ is a multi-index.
	Then
	\[
	D^\alpha(fg) = \sum_{\beta \leq \alpha} \binom{\alpha}{\beta} (D^{\alpha - \beta}f) (D^\beta g).
	\]
\end{lem}

The proof is left as an exercise.

\section{Schwartz Space and tempered distributions} \label{sec:st-PM2021}

\begin{defn}[Schwartz Space\index{Schwartz Space}] \label{defn:SchwartzSpace1-PM2021}
	Let $\varphi \in C^{\infty}(\Rn)$.
	For multi-indices $\alpha$ and $\beta$, we define the semi-norm $|\cdot|_{\alpha, \beta}$ of $\varphi$ as
	\begin{equation} \label{eq:SchNrm-PM2021}
	|\varphi|_{\alpha, \beta} := \sup_{x \in \Rn} |x^{\alpha} D^{\beta} \varphi(x)| < +\infty.
	\end{equation}
	We call $\varphi$ a \emph{Schwartz function\index{Schwartz function}} when $|\varphi|_{\alpha, \beta} < +\infty$ for any $\alpha$ and $\beta$.
	The set
	\[
	\{\varphi \in C^{\infty}(\Rn) \,;\, |\varphi|_{\alpha, \beta} < +\infty, ~\forall \alpha, \beta\}
	\]
	together with the topology induced by the set of semi-norms $|\cdot|_{\alpha, \beta}$ is call the \emph{Schwartz space}, denoted as $\boxed{\scrS(\Rn)}$.
\end{defn}

The topology $\mathcal{T}$ is induced by $\{ |\cdot|_{\alpha, \beta} \}$ is defined as follows.
Choose
\[
N(\alpha, \beta;\epsilon) := \{ \varphi \in \scrS(\Rn) \,;\, |\varphi|_{\alpha, \beta} < \epsilon \}
\]
to be open neighborhoods of point $0 \in \scrS(\Rn)$.
Choose
\[
\mathcal{N} :=\{ N(\alpha, \beta;\epsilon) \,;\, \alpha, \beta \text{~are multi-index}, \epsilon > 0 \}
\]
to be a open neighborhood basis of $0$, and $\varphi + \mathcal{N}$ the open neighborhood basis of $\varphi \in \scrS(\Rn)$. 
Then the topology $\mathcal{T}$ is generated by these open neighborhood basis, see \cite[\S 1.8]{JiangFuncBook} for more details.

\begin{defn}[Convergence in Schwartz space] \label{defn:SchwartzConvergence-PM2021}
	A sequence of functions $\{\varphi_j\}_j \subset \scrS(\Rn)$ is said to \emph{converge to zero in $\scrS(\Rn)$} if
	\begin{equation} \label{eq:SchwartzConvergence-PM2021}
	\Forall \alpha, \beta, ~|\varphi_j|_{\alpha,\beta} \to 0 \quad j \to +\infty,
	\end{equation}
	denoted as \emph{$\varphi_j \to 0$ in $\scrS(\Rn)$}.
\end{defn}

\begin{lem}
	We have
	\[
	\calF \scrS(\Rn) = \scrS(\Rn), \quad
	\partial^{\alpha} \scrS(\Rn) \subset \scrS(\Rn).
	\]
\end{lem}

The space $\scrS(\Rn)$ are often be used as test functions set.
There is also another commonly used test functions set: $C_c^\infty(\Rn)$.
In Fourier analysis the set $\scrS(\Rn)$ is more commonly used than $C_c^\infty(\Rn)$, and one of the reason is that $\scrS(\Rn)$ is closed for the Fourier transform $\calF$.
The uncertainty principle claims that the Fourier transform of any compactly supported function is impossible to be compactly supported,
namely, $\calF C_c^\infty(\Rn) \neq C_c^\infty(\Rn)$.

\begin{lem}
	$C_c^\infty(\Rn)$ is dense in $\scrS(\Rn)$.
\end{lem}

\begin{proof}
	Fix a function $\phi \in C_c^\infty(\Rn)$ satisfying $\phi \equiv 1$ when $|x| \leq 1$, and $\phi \equiv 0$ when $|x| \geq 2$, and $0 \leq \phi(x) \leq 1$.
	
	For any $\varphi \in \scrS(\Rn)$, denote $\varphi_{\epsilon}(x) := \varphi(x) \phi(\epsilon x)$, then $\{\varphi_\epsilon\}_{\epsilon > 0}$ is a sequence in $C_c^\infty(\Rn)$.
	Then for any multi-index $\alpha$ and $\beta$, we have
	\begin{align*}
	|\phi - \phi_\epsilon|_{\alpha,\beta}
	& = \sup_{x \in \Rn} |x^\alpha \partial^\beta 
	\big( \varphi(x) [1 - \phi(\epsilon x)] \big)|
	= \sup_{|x| \geq 1/\epsilon} |x^\alpha \partial^\beta 
	\big( \varphi(x) [1 - \phi(\epsilon x)] \big)| \\
	& \leq \sup_{1/\epsilon \leq |x| \leq 2/\epsilon} |x^\alpha \partial^\beta 
	\big( \varphi(x) [1 - \phi(\epsilon x)] \big)| + \sup_{|x| \geq 2/\epsilon} |x^\alpha \partial^\beta 
	\varphi(x)| \\
	& = \sup_{1/\epsilon \leq |x| \leq 2/\epsilon} |x^\alpha \partial^\beta
	\varphi(x)| |1 - \phi(\epsilon x)| + \mathcal O(\epsilon) + \sup_{|x| \geq 2/\epsilon} |x^\alpha \partial^\beta 
	\varphi(x)| \\
	& \leq \sup_{1/\epsilon \leq |x| \leq 2/\epsilon} |x^\alpha \partial^\beta
	\varphi(x)| + \mathcal O(\epsilon) + \sup_{|x| \geq 2/\epsilon} |x^\alpha \partial^\beta 
	\varphi(x)| \\
	& \leq 2\sup_{|x| \geq 1/\epsilon} |x^\alpha \partial^\beta
	\varphi(x)| + \mathcal O(\epsilon).
	\end{align*}
	Because $\sup_{\Rn} |x_j^2 x^\alpha \partial^\beta
	\varphi(x)| < +\infty$, we have that $|x_j^2 x^\alpha \partial^\beta
	\varphi(x)|$ is bounded in $\Rn$, so $|x|^2 |x^\alpha \partial^\beta
	\varphi(x)|$ is bounded in $\Rn$, thus $|x^\alpha \partial^\beta
	\varphi(x)| \leq C_{\alpha, \beta} \agl[x]^{-2}$ for certain constant $C_{\alpha, \beta}$.
	Therefore,
	\begin{equation*}
	|\phi - \phi_\epsilon|_{\alpha,\beta}
	\leq 2C_{\alpha, \beta} \sup_{|x| \geq 1/\epsilon} \agl[x]^{-2} + \mathcal O(\epsilon) \to 0, \quad \epsilon \to 0.
	\end{equation*}
	The proof is complete.	
\end{proof}

\begin{lem}
	Let $f \in \scrS(\Rn)$. Then $\Forall s \in \R$, we have
	$$(1+|x|^2)^s f(x) \in \scrS(\Rn).$$
\end{lem}

\begin{proof}
	Let $\alpha$ be a multi-index. Then
	$$D^\alpha \big[ (1+|x|^2)^s f(x) \big] = \sum_{\delta \leq \alpha} \binom{\alpha}{\delta} D^{\delta}\big( (1+|x|^2)^s\big) \cdot (D^{\alpha-\delta}f)(x).$$
	We should notice that $| D^{\delta}\big( (1+|x|^2)^s\big) |$ can be always controlled by $(1+|x|^2)^{t_{\delta}}$ for $t_{\delta} \in \R$ large enough:
	\[
	| D^{\delta}\big( (1+|x|^2)^s\big) | \leq (1+|x|^2)^{t_{\delta}}, \Forall x \in \Rn.
	\]
	So for any non-negative integer $k$ and multi-index $\alpha$, we have
	\begin{align*}
	|(1+|x|^2)^k \cdot D^{\alpha} \big[ (1+|x|^2)^s f(x) \big] | & \leq \sum_{\delta \leq \alpha} \binom{\alpha}{\delta} (1+|x|^2)^{t_{\delta}} \cdot |(D^{\alpha-\delta}f)(x)| \\
	& \leq \sum_{\delta \leq \alpha} \binom{\alpha}{\delta} | (1+|x|^2)^{t_{\delta}+k} (D^{\alpha-\delta}f)(x)| \\
	& \leq \sum_{\delta \leq \alpha} \binom{\alpha}{\delta} C_{k,\alpha,\delta} = C_{k,\alpha} < +\infty.
	\end{align*}
	We proved the conclusion.
\end{proof}

Schwartz functions are these who decay fast enough.
Now we introduce another type of functions which grow at infinity, but with a mild speed.
These functions are called tempered functions.

\begin{defn}[Tempered functions\index{tempered functions}] \label{defn:TeF-PM2021}
	Let $f$ be a measurable function defined on $\Rn$ such that
	$$\sup_{x \in \Rn} \big|(1+|x|)^{-m} f(x)\big| < +\infty$$
	for some positive integer $m$. Then we call $f$ a \emph{tempered function}. If $f$ is continuous, then we call it \emph{continuous tempered function}.
\end{defn}

\begin{lem} \label{lem:fSS-PM2021}
	Assume $f$ is a smooth function such that $\partial^\alpha f$ are tempered functions $\forall \alpha$, then we have	
	\(
	f \cdot \scrS(\Rn) \subset \scrS(\Rn).
	\)
\end{lem}

The proof is left as an exercise.

\begin{defn}[Tempered Distributions\index{tempered Distributions}] \label{defn:TeD-PM2021}
	A linear functional $T$ is called a \emph{tempered distribution} if for any sequence $\{\varphi_j\}_j$ of functions in $\scrS(\Rn)$ converging to zero in $\scrS(\Rn)$, we have
	\[
	T(\varphi_j) \to 0, \quad (j \to +\infty).
	\]
\end{defn}

It can be checked that the set of tempered distribution, denoted as $\boxed{\scrS'(\Rn)}$, is the dual of Schwartz Space $\scrS(\Rn)$.

Recall the semi-norm $|\cdot|_{\alpha, \beta}$ defined in \eqref{eq:SchNrm-PM2021}.
We define a new norm $|\cdot|_m$ as
\[
|\varphi|_m := \sum_{|\alpha|,\, |\beta| \leq m} |\varphi|_{\alpha, \beta},
\]
It can be seen that $|\varphi|_m \leq |\varphi|_{m+1}$.

\begin{lem} \label{lem:TDEq-PM2021}
	``\,$T \in \scrS'(\Rn)$'' is equivalent to the following statement:
	$$\text{there exists a constant } C \text{ such that } \Exists m \in \{0\} \cup \N^+ \st \boxed{ |T(\varphi)| \leq C|\varphi|_m }, \Forall \varphi \in \scrS(\Rn).$$
\end{lem}

\begin{proof}
	($\Leftarrow$) Assume $\Exists m \in \{0\} \cup \N^+ \st |T(\varphi)| \leq C|\varphi|_m, \Forall \varphi \in \scrS(\Rn)$. Then for every sequence $\{\varphi_k\}_k \subset \scrS(\Rn)$ satisfying $\varphi_k \to 0 ~(k \to +\infty)$, we have $|\varphi_k|_m \to 0 ~(k \to +\infty)$. So $|T(\varphi_k)| \leq C|\varphi_k|_m \to 0 ~(k \to +\infty)$. This means $T \in \scrS'(\Rn)$.
	
	($\Rightarrow$) Assume $T \in \scrS'(\Rn)$. Suppose that the claim is not true, then for every positive integer $M$, and every $m \in \{0\} \cup \N^+$, there exists $\varphi_{M, m} \in \scrS(\Rn)$ such that
	\[
	|T(\varphi_{M, m})| > M |\varphi_{M, m}|_m, \quad \text{so} \quad
	|T(\varphi_{M, m}/(M |\varphi_{M, m}|_m))| > 1.
	\]
	Let $\phi_{M,m} = \varphi_{M, m} / (M |\varphi_{M, m}|_m)$, then $|\phi_{M, m}|_m = \frac{1}{M}$ and $|T(\phi_{M,m})| \geq 1$. Further, we denote $\phi_M := \phi_{M,M}$, then $|\phi_M|_M = \frac{1}{M}$ and
	\begin{equation} \label{eq:TphM-PM2021}
	|T(\phi_M)| \geq 1, \Forall M \in N^+.
	\end{equation}
	Now for every $m \in \{0\} \cup \N^+$, when $j$ is large enough, we have
	$$|\phi_j|_m \leq |\phi_j|_j = \frac{1}{j} \to 0 ~( j \to +\infty).$$
	So according to the Definition \ref{defn:SchwartzConvergence-PM2021}, we have $\phi_j \to 0$ in $\scrS(\Rn)$, so according to the definition of tempered Distributions we shall have $|T(\phi_j)| \to 0$.
	But this is contradictory with \eqref{eq:TphM-PM2021}.
	
	The proof is complete.
\end{proof}

The notion of ``tempered function'' and ``tempered distribution'' are closely related.
Every tempered function $f$ defines a tempered distribution $T_f \in \scrS'(\Rn)$ by the following way:
$$T_f(\varphi) := \int_{\Omega} f(x) \varphi(x) \dif{x}, \Forall \varphi \in \scrS(\Rn).$$

At the first glance, the definition of tempered distribution is not a generalization of the definition of tempered function.
But the following theorem will characterize tempered distributions through tempered functions.

\begin{thm}[Schwartz representation Theorem] \label{thm:TeD-PM2021}
	Every $T \in \scrS'(\Rn)$ can be represented as a sum of certain order of derivative of continuous tempered functions in $\scrS(\Rn)$, i.e.~for every $T  \in \scrS'(\Rn)$, there exist a finite collection $T_{\alpha, \beta}$ of bounded continuous functions such that
	\[
	T = \sum_{|\alpha| + |\beta| \leq m} x^\alpha D^\beta T_{\alpha, \beta}
	\]
\end{thm}

See \S 1.4 in  \url{https://math.mit.edu/~rbm/iml/Chapter1.pdf}

\begin{thm} \label{thm:TFTD-PM2021}
	For $1 \leq p \leq +\infty$, there holds $\boxed{\scrS(\Rn) \subset L^p(\Rn) \subset \scrS'(\Rn).}$
\end{thm}

\section{Fourier transforms} \label{sec:FT-PM2021}

\begin{defn}[Fourier Transform\index{Fourier Transform} on $\scrS(\Rn)$] \label{defn:FourierTransform-1}
	Let $f \in \scrS(\Rn)$, then the \emph{Fourier transform} of $f$ is defined as
	$$\boxed{ (\calF f)(\xi) := (2\pi)^{-n/2} \int_{\Rn} e^{-i x \cdot \xi} f(x) \dif{x}, \Forall \xi \in \Rn, }$$
	where $x \cdot \xi = \sum_{i=1}^{n} x_i \xi_i$. We also denote the Fourier transform of $f$ as $\hat{f}$.
\end{defn}

\begin{defn}[Inverse Fourier Transform] \label{defn:FourierTransform-aul}
	Let $f \in \scrS(\Rn)$, then the \emph{inverse Fourier transform} of $f$ is defined as
	$$\boxed{ (\calinF f)(x) := (2\pi)^{-n/2} \int_{\Rn} e^{i x \cdot \xi} f(\xi) \dif{\xi}, \Forall \xi \in \Rn. }$$
	We also denote the inverse Fourier transform of $f$ as $\check{f}$.
\end{defn}

\begin{lem} \label{lem:FtransRelat-PM2021s}
	For every $f,g \in \scrS(\Rn)$, we have:
	
	\begin{enumerate}
	\item $\calF, \calinF \colon \scrS(\Rn) \to \scrS(\Rn)$ are linear bijection;
	
	\item $\langle \hat{f} , g \rangle = \langle f , \hat{g} \rangle$;
	
	\item $( f , g ) = ( \hat{f} , \hat{g} )$. (Parseval's Relation\index{Parseval's Relation});
	
	\item $\calF(f * g) = (2\pi)^{n/2} \hat{f} \cdot \hat{g}$;
	
	\item $\calF(f \cdot g) = (2\pi)^{-n/2} \hat{f} * \hat{g}$.
	\end{enumerate}
\end{lem}

Let's define an operator
$$\calR \colon f(x) \in \scrS(\Rn) \mapsto (\calR f)(x) = f(-x) \in \scrS(\Rn).$$
Then these four operators $\{ I,\calR,\calF,\calinF \}$ act very like $\{ 1,-1,i,-i \}$.
Denote a multiplication operation $X_j$ as $X_j \varphi(x) := x_j \varphi(x)$.
We have the following relations:
\begin{prop} \label{prop:RF-relations}
	$ $\newline
	\indent
	\begin{tabular}{ll}
		(1.a) $\calR\calF = \calF\calR = \calinF;$ & $\big(\, (-1) \cdot i = i \cdot (-1) = -i \,\big)$ \\
		(1.b) $\calR\calinF = \calinF\calR = \calF;$ & $\big(\, (-1) \cdot (-i) = (-i) \cdot (-1) = i \,\big)$ \\
		(1.c) $\calF\calF = \calinF\calinF = \calR;$ & $\big(\, i \cdot i = (-i) \cdot (-i) = -1 \,\big)$ \\
		(1.d) $\calR\calR = I.$ & $\big(\, (-1) \cdot (-1) = 1 \,\big)$ \\
		(2.a) $\calF D_j = X_j \calF, \quad \calF X_j = -D_j \calF$, \\
		(2.b) $\calF^2 D_j = -D_j \calF^2, \quad \calF^2 X_j = -X_j \calF^2$.
	\end{tabular}
\end{prop}

\begin{thm}[Plancherel Theorem\index{Plancherel Theorem}] \label{thm:PlancherelTheorem}
	$\calF$ and $\calinF$ defined on $\scrS(\Rn)$ can be extended uniquely to a unitary operator on $L^2(\Rn)$.
\end{thm}
\begin{proof}
	$C_c^{\infty}(\Rn) \subset \scrS(\Rn) \subset L^2(\Rn)$ and $\overline{C_c^{\infty}(\Rn)}^{\nrm[L^2(\Rn)]{\cdot}} = L^2(\Rn)$, So $\scrS(\Rn)$ is dense in $L^2(\Rn)$ with respect to the $L^2$ norm.
	
	For any $f \in L^2(\Rn)$, let $\{\varphi_n\}_n \subset \scrS(\Rn)$ such that $\nrm[L^2(\Rn)]{\varphi_n - f} \to 0 ~(n \to +\infty)$, then
	\[
	\nrm[L^2(\Rn)]{\calF \varphi_m - \calF \varphi_n}
	= \nrm[L^2(\Rn)]{\calF (\varphi_m - \varphi_n)}
	\]
	and by the Parseval's Relation we can continue
	\[
	\nrm[L^2(\Rn)]{\calF \varphi_m - \calF \varphi_n}
	= \nrm[L^2(\Rn)]{\varphi_m - \varphi_n} \to 0, \quad (m,n \to +\infty).
	\]
	Therefore $\{\calF \varphi_n\}_n$ is a Cauchy sequence in $L^2(\Rn)$ and has a limit. We denote the limit as $\calF f$ and assign it to $f$ as the Fourier transform of $f$.
\end{proof}

The Fourier transform and inverse Fourier transform can also be uniquely extended on $\scrS'(\Rn)$.
\begin{defn}[Fourier Transform on $\scrS'(\Rn)$] \label{defn:FourierTransform-2}
	Let $T \in \scrS'(\Rn)$, then the \emph{Fourier transform} and \emph{inverse Fourier transform} of $T$ are defined to be the linear functionals $\calF T$ on $\scrS(\Rn)$ given by
	\begin{equation*}
	\boxed{\begin{aligned}
	(\calF T)(\varphi) & := T(\hat{\varphi}), \Forall \varphi \in \scrS(\Rn), \\
	(\calinF T)(\varphi) & := T(\check{\varphi}), \Forall \varphi \in \scrS(\Rn).
	\end{aligned}}
	\end{equation*}
\end{defn}

\begin{thm} \label{thm:fouriertrans-linearbdd}
	For every $\varphi \in \scrS(\Rn), ~T \in \scrS'(\Rn)$, we have: \\
	\indent (1) $\calF, \calinF \colon \scrS'(\Rn) \to \scrS'(\Rn)$ are linear continuous bijection. \\
	\indent (2) $\calF(\varphi * T) = (2\pi)^{n/2} \hat{\varphi} \cdot \hat{T}$; \\
	\indent (3) $\calF(\varphi \cdot T) = (2\pi)^{-n/2} \hat{\varphi} * \hat{T}$.
\end{thm}
The Proposition \ref{prop:RF-relations} also holds on $\scrS'(\Rn)$. The operator $\calR$ for $T \in \scrS'(\Rn)$ is defined as:
$$(\calR T)(\varphi) := T(\calR \varphi), \Forall \varphi \in \scrS(\Rn).$$

The Fourier transform are both $(1,+\infty)$-type and $(2,2)$-type bounded. And we have $L^p(\Rn) \subset L^1(\Rn) + L^2(\Rn)$ when $1 < p < 2$.
Therefore we can define the Fourier transform $\calF$ on $L^1(\Rn) + L^2(\Rn)$ by \eqref{eq:FourierL1L2-PM2021} and then study the boundedness of Fourier transform on $L^p(\Rn)$ with $1 < p < 2$.
For details about these $L^1(\Rn) + L^2(\Rn)$ things, please Google ``Riesz--Thorin theorem''.
Therefore according to the Marcinkiewicz interpolation theorem (see \cite[Appendix B]{stein2016singular}), we have the following result.

\begin{equation} \label{eq:FourierL1L2-PM2021}
\begin{cases}
f = f_1 + f_2 \in L^p(\Rn), \\
f_1 \in L^1(\Rn), ~f_2 \in L^2(\Rn), \\
\calF_1 \colon \text{Fourier transform from } L^1(\Rn) \text{ to } L^1(\Rn), \\
\calF_2 \colon \text{Fourier transform from } L^2(\Rn) \text{ to } L^2(\Rn), \\
\calF(f) := \calF_1(f_1) + \calF_2(f_2).
\end{cases}
\end{equation}

\begin{thm}[Hausdorff-Young inequality] \label{thm:FourierBdd-PM2021}
	Define the Fourier transform on $L^1(\Rn) + L^2(\Rn)$ by (\ref{eq:FourierL1L2-PM2021}), then there exists a constant $C_p$ such that for all $f \in L^p(\Rn), ~(1 \leq p \leq 2)$, we have
	$$\boxed{ \nrm[L^{p'}(\Rn)]{\hat{f}} \leq C \nrm[L^{p}(\Rn)]{f}, }$$
	where $1 \leq p \leq 2$ and $1/p + 1/p' = 1$.
\end{thm}

\begin{proof}
	We know that $\calF$ and $\calinF$ are bounded from $L^1(\Rn)$ to $L^\infty(\Rn)$ and $L^2(\Rn)$ to $L^2(\Rn)$. So according to  Marcinkiewicz interpolation theorem, $\forall t \in [0,1]$, $\calF$ and $\calinF$ are bounded from $L^{t_1}(\Rn)$ to $L^{t_2}(\Rn)$, where
	\begin{equation*}
	\left\{\begin{aligned}
	t_1 & = \left( t \cdot \frac{1}{2} + (1-t) \cdot 1 \right)^{-1} = \frac{2}{2-t} \\
	t_2 & = \left( t \cdot \frac{1}{2} + (1-t) \cdot 0 \right)^{-1} = \frac{2}{t} \\
	\end{aligned}\right..
	\end{equation*}
	Let $p = t_1$ $p' = t_2$, then we proved the theorem.
\end{proof}

\section*{Exercise}

\begin{ex}
	Prove Lemma \ref{lem:aX-PM2021}.
\end{ex}

\begin{ex}
	Prove Lemma \ref{lem:Dfg-PM2021}.
\end{ex}

\begin{ex}
	Prove Lemma \ref{lem:fSS-PM2021}.
\end{ex}

\clearpage

\chapter{Pseudodifferential operators} \label{ch:SPDO-PM2021}

In this chapter we introduce the pseudodifferential operators, and in most of the place we abbreviate it as $\Psi$DOs.
First, we introduce symbols and its asymptotics.
Then the $\Psi$DOs are its kernels are defined.
Finally, we prove an important property of $\Psi$DOs--the pseudolocal property.
Other references are \cite[Chapter 6]{wong2014introduction}, \cite[\S 1 \& \S 3]{grigis94mic}.

\section{Symbols} \label{sec:SPDO-PM2021}

Recall the general form $p(x,D) = \sum_{|\alpha| \leq m} a_\alpha(x) D^\alpha$ of the linear differential operators mentioned is \S \ref{sec:nt-PM2021}.
For a test function $\varphi$, we have
\begin{align*}
p(x,D) \varphi(x)
& = \sum_{|\alpha| \leq m} a_\alpha(x) D^\alpha \varphi(x)
= \sum_{|\alpha| \leq m} a_\alpha(x) \calinF \{ \widehat{D^\alpha \varphi}\} (x) \\
& = \sum_{|\alpha| \leq m} a_\alpha(x) \calinF \{\xi^\alpha \hat \varphi\}(x) \\
& = \sum_{|\alpha| \leq m} a_\alpha(x) (2\pi)^{-n/2} \int_\Rn e^{ix \cdot \xi} \xi^\alpha \hat \varphi(\xi) \dif \xi \\
& = (2\pi)^{-n/2} \int_\Rn e^{ix \cdot \xi} \sum_{|\alpha| \leq m} a_\alpha(x) \xi^\alpha \hat \varphi(\xi) \dif \xi \\
& = (2\pi)^{-n/2} \int_\Rn e^{ix \cdot \xi} p(x,\xi) \hat \varphi(\xi) \dif \xi.
\end{align*}
This observation encourages us to define operators by functions $p(x,\xi)$.

\begin{defn}[Kohn-Nirenberg symbol\index{symbols}] \label{defn:symbol-PM2021}
	Let $m \in (-\infty,+\infty)$. Then we define $S^m$ to be the set of all functions $\sigma(x,\xi) \in C^{\infty}(\Rn \times \Rn; \mathbb C)$ such that for any two multi-indices $\alpha$ and $\beta$, there is a positive constant $C_{\alpha, \beta}$, independent of $(x,\xi)$, such that
	\[
	\boxed{| (D_{x}^{\alpha}D_{\xi}^{\beta}\sigma)(x,\xi)| \leq C_{\alpha, \beta} \agl[\xi]^{m-|\beta|}, \quad \forall x, \xi \in \Rn }
	\]
	holds.
	We call any function $\sigma$ in $S^m$ a \emph{symbol} of order $m$.
	We write $S^{-\infty} = \cap_{m \in \R} S^m$ and $S^{+\infty} = \cup_{m \in \R} S^m$.
\end{defn}

\begin{exmp}
	Here we give some examples of symbols.
	\begin{itemize}
	\item $\sum_{|\alpha| \leq m} a_\alpha(x) \xi^\alpha$ is a symbol of order $m$ when $a_\alpha \in \scrS(\Rn)$;
	
	\item $\scrS(\Rn) \subset S^{-\infty}$;
	
	\item Fix a bounded $\psi \in C^\infty(\Rn)$, then $\psi(x) \agl[\xi]^m$ is a symbol of order $m$;
	
	\item Fix a $\phi \in C_c^\infty(\Rn)$ with $\phi(0) = 1$, then $(1 - \phi(\xi))(1+|\xi|)^m$ is a symbol of order $m$.
	\end{itemize}
\end{exmp}

\begin{lem} \label{lem:ele-PM2021}
	Assume $\sigma_j \in S^{m_j}~(j = 1,2)$, then $\sigma_1 \sigma_2 \in S^{m_1 + m_2}$.
	$\partial^\alpha \sigma_1 \in S^{m_j - |\alpha|}$.
\end{lem}

The proof is left as an exercise.

One can also define a more general symbol which the effect of $x$ is taken into consideration, and the dimension of $x$ variable and $\xi$ variable can be different.

\begin{defn} \label{defn:symbolx-PM2021}
	Let $m \in (-\infty,+\infty)$ and $0 \leq \delta < \rho \leq 1$. 
	Then we define $S^m_{\rho,\delta}$ to be the set of all functions $\sigma(x,\xi) \in C^{\infty}(\R^{n_1} \times \R^{n_2} ; \mathbb C)$ such that for any two multi-indices $\alpha$ and $\beta$, there is a positive constant $C_{\alpha, \beta}$, depending on $\alpha$ and $\beta$ only, for which
	\[
	\boxed{ \big| (D_{x}^{\alpha}D_{\xi}^{\beta}\sigma)(x,\xi) \big| \leq C_{\alpha, \beta} \agl[\xi]^{m - \rho |\beta| + \delta |\alpha|}, \quad \forall x \in \R^{n_1}, \xi \in \R^{n_2} }
	\]
	holds.
	We also call any function $\sigma$ in $S^m_{\rho,\delta}$ a \emph{symbol}.
\end{defn}

The Kohn-Nirenberg symbol $S^m = S^m_{1,0}$.
In what follows, we only focus on $S^m$, and the situations for $S^m_{\rho,\delta}$ shall be followed in similar manners.

\smallskip

Now we introduce an important notion: the asymptotic expansion of symbols.

\begin{defn}[Asymptotics\index{asymptotics}] \label{defn:Asy-PM2021}
	Let symbol $a \in S^m$ and $a_j \in S^{m_j}$ $(j= 0,1,\cdots)$ where the orders $m_j$ satisfies
	\[
	m = m_0 > m_1 > \cdots > m_j > m_{j+1} \to -\infty, \quad j \to \infty.
	\]
	If
	\[
	a - \sum_{j = 0}^{N} a_j \in S^{m_{N+1}},
	\]
	holds for every integer $N$, we write
	\[
	a \sim \sum_j a_j \text{~in~} S^m,
	\]
	and we call $\{a_j\}$ an asymptotics of $a$.
	The $a_0$ is called the \emph{principal symbol} of $a$.
\end{defn}

We often write $a = b + S^m$ as a shorthand of $a = b + r$ for some $r \in S^m$.
Then we can summarize Definition \ref{defn:Asy-PM2021} as follows,
\begin{equation*}
\boxed{a \sim \sum_j a_j \ \text{in} \ S^{m_{N+1}}
	\quad \Leftrightarrow \quad
	a = \sum_{j = 0}^{N} a_j + S^{m_{N+1}}.}
\end{equation*}

Now let's randomly pick up some $m$, $m_j$ that satisfy the requirement in Definition \ref{defn:Asy-PM2021}, and randomly pick up $a_j \in S^{m^j}$.
A natural question is to ask, does there exist $a \in S^m$ such that $a \sim \sum_j a_j \text{~in~} S^m$?
The answer is yes.

\begin{thm} \label{thm:Asy-PM2021}
	For any $m$ and $m_j$ satisfying
	\[
	m = m_0 > m_1 > \cdots > m_j > m_{j+1} \to -\infty, \quad j \to \infty,
	\]
	and for any $a_j \in S^{m_j}$, there exists a symbol (not unique) $a \in S^m$ such that $a \sim \sum_j a_j \text{~in~} S^m$.
\end{thm}

When $x \to +\infty$, $\frac 1 {1-1/\agl[x]} = 1+1/\agl[x]+1/\agl[x]^2 + \mathcal O(1/\agl[x]^3)$.
Arbitrarily pick up $\alpha_j$, is there a function $f(x)$ such that in $[1,+\infty)$,
\begin{equation} \label{eq:fxA-PM2021}
f(x) = \sum_{0 \leq j \leq N} \alpha_j/\agl[x]^j + \mathcal O(1/\agl[x]^{N+1}), \quad x \to +\infty,
\end{equation}
holds for all $N \in \mathbb N$?
The answer is no and an example is $\alpha_j := j!$ (the convergence radius goes to infinity as $N$ grows).
The problem is that $\alpha_j/\agl[x=1]^j$ will be too big when $j \to +\infty$.
However, we can fix this problem by cutoff, so that there exist a function $f$ (not unique!) such that \eqref{eq:fxA-PM2021} holds on intervals $[A_j,+\infty)$ where the $A_j$ is in accordance with $a_j$, $A_1 < A_2 < \cdots$.
The key step is to choose a cutoff function $\chi_j$ to cutoff term $\alpha_j/\agl[x]^j$ such that
\begin{equation} \label{eq:c2j-PM2021}
\chi_j(x) \alpha_j/\agl[x]^j \leq 1/2^j, \text{~or~} 1/3^j \text{~etc}.
\end{equation}
The following function satisfies the requirement:
\begin{equation} \label{eq:chij-PM2021}
\left\{\begin{aligned}
& \chi_j \in C^\infty(\R), \\
& \chi_j \equiv 0, |x| \leq 2 (\alpha_j)^{1/j}, \\
& \chi_j \equiv 1, |x| \geq 1 + 2 (\alpha_j)^{1/j}, \\
& 0 \leq \chi_j \leq 1, \text{otherwise}.
\end{aligned}\right.
\end{equation}
The requirement \eqref{eq:chij-PM2021} can also be realized by fix some $\chi$ satisfying
\begin{equation} \label{eq:chi-PM2021}
\left\{\begin{aligned}
& \chi \in C^\infty(\R), \ 0 \leq \chi(x) \leq 1 \\
& \chi \equiv 0 \text{~when~} |x| \leq 1, \ \chi_j \equiv 1 \text{~when~} |x| \geq 2.
\end{aligned}\right.
\end{equation}
and then set $\chi_j(x) := \chi(\epsilon_j x)$, where the $\epsilon_j$ shall be chosen according to \eqref{eq:chij-PM2021}.

\begin{proof}[Sketch of the proof of Theorem \ref{thm:Asy-PM2021}]
	Choose suitable coefficients $\epsilon_j$ and define
	\[
	a(x,\xi) := \sum_{j \geq 0} \chi(\epsilon_j \xi) a_j(x,\xi), \quad x,\xi \in \Rn.
	\]
	It can be checked that $\chi(\epsilon_j \xi) \in S^0$.
	For any fixed $(x_0, \xi_0)$, there is only finitely many terms in $\sum_{j \geq 0} \chi(\epsilon_j \xi_0) a_j(x_0,\xi_0)$ which are non-zero, so $a(x,\xi) \in C^{\infty}(\Rn \times \Rn; \mathbb C)$.
	Moreover, we need to show first $a$ is a symbol, and second $a$ is an asymptotics of $a_j$.
	
	First, we show that $a \in S^m$.
	It can be checked that for any multi-index $\beta$,
	\[
	|D_{\xi}^{\beta} (\chi(\epsilon_j \xi))| \leq C_\beta \agl[\xi]^{-\beta}, \quad \forall \xi \in \Rn,
	\]
	where the constant $C_\beta$ is independent of $\epsilon_j$.
	we notice that every term $\chi(\epsilon_j \xi) a_j(x,\xi)$ is in $S^{0 + m_j} = S^{m_j}$, so
	\begin{align*}
	& \ D_{x}^{\alpha} D_{\xi}^{\beta} (\chi(\epsilon_j \xi) a_j(x,\xi))
	= {\color{red}\chi(2\epsilon_j \xi)} D_{x}^{\alpha} D_{\xi}^{\beta} (\chi(\epsilon_j \xi) a_j(x,\xi)) \\
	= & \ \chi(2\epsilon_j \xi) \sum_{\beta' \leq \beta} \binom{\beta}{\beta'} D_{\xi}^{\beta'} (\chi(\epsilon_j \xi)) \cdot (D_{x}^{\alpha} D_{\xi}^{\beta - \beta'} a_j)(x,\xi) \\
	\leq & \ \chi(2\epsilon_j \xi) \sum_{\beta' \leq \beta} \binom{\beta}{\beta'} C_{\beta'} \agl[\xi]^{-|\beta'|} \cdot C_{\alpha,\beta,\beta'} \agl[\xi]^{m_j - |\beta| + |\beta'|} \\
	= & \ \chi(2\epsilon_j \xi) \sum_{\beta' \leq \beta} \binom{\beta}{\beta'} C_{\beta'} C_{\alpha,\beta,\beta'} \agl[\xi]^{m_j - |\beta|}
	= \chi(2\epsilon_j \xi) C_{\alpha,\beta} \agl[\xi]^{m_j - |\beta|} \\
	= & \ C_{\alpha,\beta} {\color{red}\agl[\xi]^{m_j - m}} \chi(2\epsilon_j \xi) \cdot \agl[\xi]^{m - |\beta|} \\
	\leq & \ C_{\alpha,\beta} {\color{red}(2\epsilon_j)^{m - m_j}} \chi(2\epsilon_j \xi) \cdot \agl[\xi]^{m - |\beta|} \\
	\leq & \ C_{\alpha,\beta} (2\epsilon_j)^{m - m_j} \cdot \agl[\xi]^{m - |\beta|},
	\end{align*}
	where the change ``$\agl[\xi]^{m_j - m} \to (2\epsilon_j)^{m - m_j}$'' is due to the presence of $\chi(2\epsilon_j \xi)$.
	Hence,
	\begin{equation*}
	|D_{x}^{\alpha} D_{\xi}^{\beta} a(x,\xi)|
	\leq \agl[\xi]^{m - |\beta|} \cdot \sum_{j \geq 0} C_{\alpha,\beta} (2\epsilon_j)^{m - m_j}.
	\end{equation*}
	We choose $\epsilon_j$ to decrease fast enough such that $\sum_{j \geq 0} C_{\alpha,\beta} (2\epsilon_j)^{m - m_j}$ is finite for every $\alpha, \beta$ (see \cite[Theorem 6.10]{wong2014introduction} for details).
	We proved $a \in S^m$.
	
	Second, to show $a \sim \sum_j a_j$ in $S^m$, we see
	\begin{align*}
	a - \sum_{0 \leq j \leq N} a_j
	& = \sum_{0 \leq j \leq N}[\chi(\epsilon_j \xi) - 1] a_j(x,\xi) + \sum_{j \geq N+1} \chi(\epsilon_j \xi) a_j(x,\xi) \\
	& \in S^{-\infty} + S^{m_{N+1}} = S^{m_{N+1}}.
	\end{align*}
	The proof is complete.
\end{proof}

\section{Pseudodifferential operators} \label{sec:PDOs-PM2021}

\subsection{Some basics about the \texorpdfstring{$\Psi$DOs}{PsiDOs}} \label{subsec:PsDO-PM2021}

Based on the notion of symbols, we introduce the pseudodifferential operators.

\begin{defn}[Pseudodifferential operator\index{pseudodifferential operators}, $\Psi$DO] \label{defn:PseudoDO-PM2021}
	Let $\sigma$ be a symbol. Then the \emph{pseudo-differential operator} $T_{\sigma}$, defined on $\scrS(\Rn)$ and associated with $\sigma$, is defined as
	\begin{align*}
	(T_{\sigma}\varphi)(x)
	& := (2\pi)^{-n/2} \int_{\Rn} e^{ix \cdot \xi} \sigma(x,\xi) \hat{\varphi}(\xi) \dif{\xi} \\
	& \ = \boxed{(2\pi)^{-n} \iint_{\Rn \times \Rn} e^{i(x-y) \cdot \xi} \sigma(x,\xi) \varphi(y) \dif y \dif \xi}, \quad \forall \varphi \in \scrS(\Rn).
	\end{align*}
	We denote the set of $\Psi$DOs of order $m$ as $\Psi^m$
	We write $\Psi^{-\infty} = \cap_{m \in \R} \Psi^m$ and $\Psi^{+\infty} = \cup_{m \in \R} \Psi^m$.
\end{defn}

\begin{exmp}
	Here we give some examples of $\Psi$DOs:
	\begin{itemize}
	\item $- \Delta \in \Psi^2$, with symbol $|\xi|^2$;
	
	\item $\sum_{|\alpha| \leq m} a_\alpha(x) D^\alpha \in \Psi^m$, with symbol $\sum_{|\alpha| \leq m} a_\alpha(x) \xi^\alpha$;
	
	\item $(I - \Delta)^{m/2} \in \Psi^m$, which is defined by the symbol $\agl[\xi]^m = (I + |\xi|^2)^{m/2}$;
	
	\item The DtN map of the Calder\'on problem is a $\Psi$DO living on the boundary, see \cite{MR1029119}.
	\end{itemize}
\end{exmp}

It is an interesting question to ask for the symbol when given a certain $\Psi$DO.

\begin{exmp}
	Some simple $\Psi$DOs whose symbol are also simple:
	\begin{itemize}
	\item $D \mapsto \xi$;
	
	\item $-\Delta = D \cdot D$, so $-\Delta \mapsto |\xi|^2$.
	\end{itemize}
\end{exmp}

Similar to Lemma \ref{lem:ele-PM2021}, we have the following claim, whose proof will be provided in Theorem \ref{thm:com-PM2021}.

\begin{lem} \label{lem:elT-PM2021}
	Assume $\sigma_j \in S^{m_j}~(j = 1,2)$, then $T_{\sigma_1} \circ T_{\sigma_2} \in \Psi^{m_1 + m_2}$.
\end{lem}

We show that the map $\sigma \mapsto T_\sigma$ is a bijection.

\begin{lem} \label{lem:sBi-PM2021}
	map $\sigma \mapsto T_\sigma \in \mathcal L(\scrS(\Rn), \scrS(\Rn))$ is a bijection.
\end{lem}

\begin{proof}
	The $\Psi$DO $T_\sigma$ is defined by $\sigma$, so the surjectivity is obvious.
	The injectivity amounts to prove $T_\sigma = T_\tau \Rightarrow \sigma = \tau$.
	
	Let's assume $\sigma$ and $\tau$ are two symbols and $T_\sigma = T_\tau$, then
	\[
	\int e^{ix \cdot \xi} [\sigma(x, \xi) - \tau(x, \xi)] \hat \varphi(\xi) \dif \xi = 0
	\]
	holds for any $x \in \Rn$ and any $\varphi \in \scrS(\Rn)$.
	Replace $\varphi$ by its inverse Fourier transform, and fix $x$ to some $x_0$, we can see
	\[
	\int e^{ix_0 \cdot \xi} [\sigma(x_0, \xi) - \tau(x_0, \xi)] \varphi(\xi) \dif \xi = 0.
	\]
	The arbitrary of $\varphi$ gives
	\[
	e^{ix_0 \cdot \xi} [\sigma(x_0, \xi) - \tau(x_0, \xi)] = 0, \quad \forall x_0 \in \Rn.
	\]
	And the arbitrary of $x_0$ gives $\sigma(x_0, \xi) = \tau(x_0, \xi)$ for $\forall x_0 \in \Rn$.
	The injectivity is proved.
	We arrive at the conclusion.
\end{proof}

\begin{rem} \label{rem:FPw-PM2021}
For people who the first time encounter the form
\[
T_{\sigma}\varphi(x) \simeq \int_{\Rn} e^{ix \cdot \xi} \sigma(x,\xi) \hat{\varphi}(\xi) \dif{\xi},
\]
one may think that $T_{\sigma}\varphi(x)$ is just the inverse Fourier transform of $\sigma(x,\xi) \hat{\varphi}(\xi)$,
and consequently, $\sigma(x,\xi) \hat{\varphi}(\xi)$ can be recovered by taking the Fourier transform of $T_{\sigma}\varphi(x)$.
Unfortunately this is not true.
The function $\sigma(x,\xi) \hat{\varphi}(\xi)$ depends not only $\xi$ but also $x$, so that is not a Fourier transform anymore.
When a symbol $a$ is independent of $x$, we have
\[
\calF \{ T_{\sigma}\varphi\}(\xi) = a(\xi) \hat{\varphi}(\xi).
\]
But when $a$ depends on $x$, the Fourier transform of $(T_{\sigma}\varphi)(x)$ is generally \emph{NOT} $a(x,\xi)\hat{\varphi}(\xi)$.
In generally, we cannot use the expression above to get the symbol $a$:
\[
\boxed{ a(x,\xi) \hat{\varphi}(\xi) \neq \calF \big\{ (T_{\sigma}\varphi)(\cdot) \big\}(\xi). }
\]
\end{rem}

Similar to generalizing the Fourier transform from functions to distributions, the notion of $\Psi$DOs can also extend to $\scrS'(\Rn)$ by using duality arguments.
Formally speaking, we have the following computation,
\begin{align}
(T_{\sigma}u,\varphi)
& = \int (T_{\sigma}u) \overline \varphi \dif x
= \int \big[ (2\pi)^{-n/2} \int e^{ix \cdot \xi} \sigma(x,\xi) \hat u(\xi) \dif{\xi} \big] \overline{\varphi(x)} \dif x \nonumber \\
& = \int \big[ (2\pi)^{-n} \int e^{i(x - y) \cdot \xi} \sigma(x,\xi) u(y) \dif y \dif{\xi} \big] \overline{\varphi(x)} \dif x \nonumber \\
& = \int u(y) \big[ \overline{(2\pi)^{-n}  \int e^{i(y - x) \cdot \xi} {\color{red} \overline \sigma(x,\xi)} \varphi(x) \dif x \dif \xi} \big] \dif y \nonumber \\
& \sim \int u(y) \big[ \overline{(2\pi)^{-n} \int e^{i(y - x) \cdot \xi} {\color{red} \sum_\alpha \frac {(x-y)^\alpha} {\alpha!} \partial_y^\alpha \overline \sigma(y,\xi)} \varphi(x) \dif x \dif \xi} \big] \dif y \quad \text{(Taylor's)} \nonumber \\
& \sim \int u(y) \big[ \overline{(2\pi)^{-n} \int {\color{red}(-D)_\xi^\alpha} (e^{i(y - x) \cdot \xi}) {\color{red} \sum_\alpha \frac {1} {\alpha!} \partial_y^\alpha \overline \sigma(y,\xi)} \varphi(x) \dif x \dif \xi} \big] \dif y \nonumber \\
& \sim \int u(y) \big[ \overline{(2\pi)^{-n} \int e^{i(y - x) \cdot \xi} {\color{red} \sum_\alpha \frac {1} {\alpha!} \partial_y^\alpha D_\xi^\alpha \overline \sigma(y,\xi)} \varphi(x) \dif x \dif \xi} \big] \dif y \nonumber \\
& = \int u(y) \big[ \overline{(2\pi)^{-n/2} \int e^{iy \cdot \xi} {\color{red} \sum_\alpha \frac {1} {\alpha!} \partial_y^\alpha D_\xi^\alpha \overline \sigma(y,\xi)} \hat \varphi(\xi) \dif \xi} \big] \dif y \nonumber \\
& = \int u(y) \overline{T_{\sigma^*} \varphi(y)} \dif y, \qquad \sigma^*(y,\xi) := \sum_\alpha \frac {1} {\alpha!} \partial_y^\alpha D_\xi^\alpha \overline \sigma(y,\xi) \nonumber \\
& = (u, T_{\sigma^*} \varphi). \label{eq:Tuph-PM2021}
\end{align}
The computation \eqref{eq:Tuph-PM2021} implies the existence of the adjoint of $T_\sigma$ (denoted as $T_\sigma^*$), and we leave the rigorous proof of the existence of $T_\sigma^*$ to \S \ref{sec:AdT-PM2021}.
Now, by assuming the existence of $T_\sigma^*$, we extend the domain of $T_\sigma$ from $\scrS(\Rn)$ to $\scrS'(\Rn)$ as follows.

\begin{defn}[Pseudodifferential operators in $\scrS'$] \label{defn:PseudoDOOnDistribution-PM2021}
	Let $\sigma$ be a symbol. For every $u \in \scrS'(\Rn)$, we can define the \emph{pseudo-differential operator} $T_{\sigma}$ acting on $u$ as
	\[
	\boxed{ (T_{\sigma}u,\varphi) := (u, T_{\sigma}^* \varphi), \Forall \varphi \in \scrS(\Rn). }
	\]
	where $T_\sigma^*$ is the adjoint of $T_\sigma$ and the bracket $(\cdot,\cdot)$ signifies the pair of distributions with test functions.
\end{defn}

\begin{lem} \label{lem:TSS-PM2021}
	Let $\sigma$ be a symbol, and denote its corresponding {\rm $\Psi$DO} as $T_\sigma$.
	Then $T_\sigma(\scrS(\Rn)) \subset \scrS(\Rn)$.
	And also, $T_\sigma(\scrS'(\Rn)) \subset \scrS'(\Rn)$.
\end{lem}

\begin{proof}[Sketch of the proof]
	Let $\varphi \in \scrS(\Rn)$, we need to show $x^\alpha D^\beta (T_\sigma \varphi)$ are bounded in $\Rn$.
	We can show
	\[
	x^\alpha D^\beta (T_\sigma \varphi)(x)
	\simeq \sum_{\alpha', \beta'} \int e^{ix\cdot \xi} \tilde \sigma_{\alpha', \beta'}(x,\xi) \calF\{x^{\alpha'} D^{\beta'} \varphi\}(\xi) \dif \xi
	\]
	where $\tilde \sigma_{\alpha', \beta'}$ are also symbols of certain orders, say, $m$.
	Because $\varphi \in \scrS(\Rn)$, we know that the Fourier transform $\calF\{x^{\alpha'} D^{\beta'} \varphi\} \in \scrS(\Rn)$, so
	\[
	|\calF\{x^{\alpha'} D^{\beta'} \varphi\}(\xi)| \lesssim \agl[\xi]^{-M}
	\]
	for any positive integer $M$.
	Hence,
	\[
	|x^\alpha D^\beta (T_\sigma \varphi)(x)|
	\lesssim \sum_{\alpha', \beta'} \int \agl[\xi]^{m} \agl[\xi]^{-M} \dif \xi < +\infty
	\]
	when we take $M$ to be large enough.
	
	For the second conclusion, from $(T_{\sigma}u,\varphi) := (u, T_{\sigma}^* \varphi), \Forall \varphi \in \scrS(\Rn)$ we have
	\[
	|(T_{\sigma}u,\varphi)| \leq \nrm{u} \nrm{T_{\sigma}^* \varphi}
	\lesssim \nrm{u} |\varphi|_m
	\]
	where $m$ is the order of $\sigma$.
	Then by Lemma \ref{lem:TDEq-PM2021} we can conclude $T_{\sigma}u \in \scrS'(\Rn)$.
\end{proof}

\smallskip

In conclusion, there holds
\begin{equation}
T_\sigma \colon
\left\{\begin{aligned}
\scrS & \to \scrS, \\
\scrS' & \to \scrS',
\end{aligned}\right.
\end{equation}
where $\scrS$ is a shorthand for $\scrS(\Rn)$.
Space $\scrS$ represents functions which are extremely smooth (good), while $\scrS'$ represents ``functions'' which are extremely rough (bad).
To quantize the goodness and the badness, we introduce the potential spaces.

\subsection{Sobolev spaces} \label{subsec:Hsp-PM2021}

\begin{defn}[Sobolev spaces\index{Sobolev spaces}] \label{defn:Hsp-PM2021}
	We denote
	\[
	\boxed{H^{s,p}(\Rn) := \{f \in \scrS'(\Rn) \,;\, (I - \Delta)^{s/2} f \in L^p(\Rn) \}},
	\]
	and define the norm $\nrm[H^{s,p}]{f} := \nrm[L^p(\Rn)]{(I - \Delta)^{s/2} f}$.
	Write $H^s(\Rn) := H^{s,2}(\Rn)$.
\end{defn}

\begin{lem} \label{lem:Hsp-PM2021}
	The normed vector space $(H^{s,p}(\Rn), \nrm[H^{s,p}]{\cdot})$ is a Banach space, and $(H^s(\Rn), \nrm[H^s]{\cdot})$ is a Hilbert space.
\end{lem}

\begin{thm} \label{thm:Tsm-PM2021}
	Let $\sigma \in S^m$ and denote its corresponding {\rm $\Psi$DO} as $T_\sigma$.
	Then the mapping $T_\sigma \colon H^s(\Rn) \to H^{s-m}(\Rn)$ is bounded.
\end{thm}

The proof of Theorem \ref{thm:Tsm-PM2021} is based on the $L^2$ boundedness of $\Psi$DOs of order 0.
Formally speaking,
\begin{align*}
\nrm[H^{s-m}]{T_\sigma f}
& = \nrm[L^2]{(I - \Delta)^{(s-m)/2} \circ T_\sigma f}
\lesssim \nrm[L^2]{(I - \Delta)^{s/2} f}
= \nrm[H^s]{f}.
\end{align*}

\begin{thm} \label{thm:SNI-PM2021}
	For a fixed constant $s \in \R$, $\Forall r,t \colon r \leq s \leq t, \Forall C > 0, \Forall \varphi \in \scrS(\Rn)$, we have:
	\begin{equation} \label{eq:SNI-PM2021}
	\boxed{ \nrm[H^s]{\varphi}^2 \leq \frac{1}{C^{t-s}} \nrm[H^t]{\varphi}^2 + C^{s-r} \nrm[H^r]{\varphi}^2. }
	\end{equation}
\end{thm}

\begin{rem} \label{rem:SNI-PM2021}
	We know that when $t > s$, $\nrm[H^s]{\varphi}$ can be controlled by $\nrm[H^t]{\varphi}$:
	$$\nrm[H^s]{\varphi} \leq 1 \cdot \nrm[H^t]{\varphi}.$$
	The key point of Theorem \ref{thm:SNI-PM2021} is that $\nrm[H^s]{\varphi}$ can even be ``controlled'' by $c \cdot \nrm[H^t]{\varphi}$ with $0 < c < 1$.
	But we need to pay for this: being dominated only by $c \cdot \nrm[H^t]{\varphi}$ is not enough. Due to the fact that $c$ is less than 1, certain ``byproduct'' should cost to compensate the advantage, and this so-called ``byproduct'' is $\nrm[H^r]{\varphi}$.
\end{rem}

\begin{proof}[Proof of Theorem \ref{thm:SNI-PM2021}]
	When $0 < C \leq 1$, it is trivial. When $C \geq 1$, we have:
	\begin{align*}
	\nrm[H^s]{\varphi}^2 &  = \nrm[2]{\calinF \sigma_{-s} \calF \varphi}^2 = \nrm[2]{\sigma_{-s} \hat{\varphi}}^2 = \int_{\Rn} \agl[\xi]^{2s} |\hat{\varphi}(\xi)|^2 \dif{\xi} \\
	& = \int_{ \{ \agl[\xi] \leq \sqrt{C} \} } \agl[\xi]^{2s} |\hat{\varphi}(\xi)|^2 \dif{\xi} + \int_{ \{ \agl[\xi] > \sqrt{C} \} } \agl[\xi]^{2s} |\hat{\varphi}(\xi)|^2 \dif{\xi} \\
	& = \int_{ \{ \agl[\xi] \leq \sqrt{C} \} } \agl[\xi]^{2s-2r} \cdot \agl[\xi]^{2r} |\hat{\varphi}(\xi)|^2 \dif{\xi} + \int_{ \{ \agl[\xi]^{-1} < \frac{1}{\sqrt{C}} \} } (\agl[\xi]^{-1})^{2t-2s} \cdot \agl[\xi]^{2t} |\hat{\varphi}(\xi)|^2 \dif{\xi} \\
	& \leq \int_{ \{ \agl[\xi] \leq \sqrt{C} \} } \sqrt{C}^{2s-2r} \cdot \agl[\xi]^{2r} |\hat{\varphi}(\xi)|^2 \dif{\xi} + \int_{ \{ \agl[\xi]^{-1} < \frac{1}{\sqrt{C}} \} } (\frac{1}{\sqrt{C}})^{2t-2s} \cdot \agl[\xi]^{2t} |\hat{\varphi}(\xi)|^2 \dif{\xi} \\
	& \leq \sqrt{C}^{2s-2r} \int_{\Rn} \agl[\xi]^{2r} |\hat{\varphi}(\xi)|^2 \dif{\xi} + (\frac{1}{\sqrt{C}})^{2t-2s} \int_{\Rn} \agl[\xi]^{2t} |\hat{\varphi}(\xi)|^2 \dif{\xi} \\
	& = C^{s-r} \nrm[H^r]{\varphi}^2 + C^{s-t} \nrm[H^t]{\varphi}^2.
	\end{align*}
	This completes the proof.
\end{proof}

\begin{rem}
	In the proof of Theorem \ref{thm:SNI-PM2021}, when $\varphi$ is compactly supported and $s = 0$, and if we replace $\agl[\xi]$ by $|\xi|$ and choose $r = 0$ and $C$ to be small enough and use the fact that
	\[
	\int_{ \{ |\xi| \leq \sqrt{C} \} } |\hat{\varphi}(\xi)|^2 \dif{\xi}
	\lesssim C^{n/2} \sup_{|\xi| \leq \sqrt{C}} |\hat{\varphi}(\xi)|^2
	\lesssim C^{n/2} \nrm[L^1]{\varphi}^2
	\leq C^{n/2} \sqrt{|\supp \varphi|} \nrm[L^2]{\varphi},
	\]
	we can prove the Poincare's inequality
	\(
	\nrm[L^2]{\varphi} \lesssim \nrm[L^2]{(-\Delta)^{t/2} \varphi}
	\)
	for the set of functions with uniformly compact support.
\end{rem}

Noticing that $\agl[\xi] \geq 1$, we can further extend Theorem \ref{thm:SNI-PM2021} to more generalized situation.
\begin{thm} \label{thm:SNIE-PM2021}
	For a fixed constant $s \in \R$, $\Forall t \geq s$ and $\Forall r \in \R^1$, $\Forall \epsilon > 0, \Forall \varphi \in \scrS(\Rn)$, there exists a constant $C_{r,s,t,\epsilon}$ such that:
	\begin{equation} \label{eq:SNIE-PM2021}
	\boxed{ \nrm[H^s]{\varphi}^2 \leq \epsilon \nrm[H^t]{\varphi}^2 + C_{r,s,t,\epsilon} \nrm[H^r]{\varphi}^2. }
	\end{equation}
\end{thm}

As mentioned in Remark \ref{rem:SNI-PM2021}, Theorem \ref{thm:SNIE-PM2021} expresses the same information, in addition that the ``byproduct'' can be $\nrm[H^r]{\varphi}$ with any $r \in \R^1$. 
\begin{proof}[Proof of Theorem \ref{thm:SNIE-PM2021}]
	We pick up some constant $C > 1$ first, and then we decide its value later.
	\begin{align*}
	\nrm[H^s]{\varphi}^2 &  = \nrm[2]{\calinF \sigma_{-s} \calF \varphi}^2 = \nrm[2]{\sigma_{-s} \hat{\varphi}}^2 = \int_{\Rn} \agl[\xi]^{2s} |\hat{\varphi}(\xi)|^2 \dif{\xi} \\
	& = \int_{ \{ 1 \leq \agl[\xi] \leq \sqrt{C} \} } \agl[\xi]^{2s} |\hat{\varphi}(\xi)|^2 \dif{\xi} + \int_{ \{ \agl[\xi] > \sqrt{C} \} } \agl[\xi]^{2s} |\hat{\varphi}(\xi)|^2 \dif{\xi} \\
	& = \int_{ \{ 1 \leq \agl[\xi] \leq \sqrt{C} \} } \agl[\xi]^{2s-2r} \cdot \agl[\xi]^{2r} |\hat{\varphi}(\xi)|^2 \dif{\xi} + \int_{ \{ \agl[\xi]^{-1} < \frac{1}{\sqrt{C}} \} } (\agl[\xi]^{-1})^{2t-2s} \cdot \agl[\xi]^{2t} |\hat{\varphi}(\xi)|^2 \dif{\xi} \\
	& \leq \max\{1, \sqrt{C}^{2s-2r} \} \cdot \int_{ \{ 1 \leq \agl[\xi] \leq \sqrt{C} \} } \agl[\xi]^{2r} |\hat{\varphi}(\xi)|^2 \dif{\xi} \\
	& \quad + \int_{ \{ \agl[\xi]^{-1} < \frac{1}{\sqrt{C}} \} } (\frac{1}{\sqrt{C}})^{2t-2s} \cdot \agl[\xi]^{2t} |\hat{\varphi}(\xi)|^2 \dif{\xi} \\
	& \leq \max\{1, \sqrt{C}^{2s-2r} \} \cdot \int_{\Rn} \agl[\xi]^{2r} |\hat{\varphi}(\xi)|^2 \dif{\xi} + (\frac{1}{\sqrt{C}})^{2t-2s} \int_{\Rn} \agl[\xi]^{2t} |\hat{\varphi}(\xi)|^2 \dif{\xi} \\
	& = \max\{1, \sqrt{C}^{2s-2r} \} \cdot \nrm[H^r]{\varphi}^2 + C^{s-t} \nrm[H^t]{\varphi}^2.
	\end{align*}
	Now let $C = \epsilon^{t-s}$ and let $C_{r,s,t,\epsilon} = \max\{1, \sqrt{C}^{2s-2r} \}$, then we completes the proof.
\end{proof}

\subsection{Other phases}

Beside the phase $(x-y) \cdot \xi$ in the expression
\[
(T_{\sigma}\varphi)(x)
= (2\pi)^{-n} \int e^{i(x-y) \cdot \xi} \sigma(x,y,\xi) \varphi(y) \dif y \dif \xi
\]
in Definition \ref{defn:PseudoDO-PM2021}, it is possible to use more general functions as the phase functions and the corresponding operators are still $\Psi$DOs, i.e.,
\[
P \varphi(x)
= (2\pi)^{-n} \int e^{i \phi(x,y, \xi)} \sigma(x,y,\xi) \varphi(y) \dif y \dif \xi
\]
will still be a $\Psi$DO is the $\phi$ satisfies certain conditions.
See \cite[\S 3.2]{so17fo} for details.

\section{Kernels} \label{sec:ker-PM2021}

The expression in Definition \ref{defn:PseudoDO-PM2021} can also be represented as
\begin{equation} \label{eq:Tker-PM2021}
	\boxed{(T_{\sigma}\varphi)(x)
	= \int_{\Rn} K(x,y) \varphi(y) \dif y,}
\end{equation}
where $K(x,y)$ is called the \emph{\underline{kernel}}\index{kernel} of $T_\sigma$,
\begin{equation*}
	K(x,y) := (2\pi)^{-n} \int e^{i(x-y) \cdot \xi} \sigma(x,\xi) \dif \xi,
\end{equation*}
and the integration shall be understood as an oscillatory integral (see Definition \ref{defn:kerP-PM2021}).

Differential operators such as $P = \sum_{j=1}^n x_j \partial_j$ maps $\scrS(\Rn) \to \scrS(\Rn)$, $C_c^\infty(\Rn) \to C_c^\infty(\Rn)$, and $\mathcal{E}(\Rn) \to \mathcal{E}(\Rn)$, and so by duality argument, we know differential operators $P$ maps $\scrS'(\Rn) \to \scrS'(\Rn)$, $\mathcal{D}'(\Rn) \to \mathcal{D}'(\Rn)$, and $\mathcal{E}'(\Rn) \to \mathcal{E}'(\Rn)$:
\begin{equation*}
P \colon
\left\{\begin{aligned}
\mathcal{E}'(\Rn) & \to \mathcal{E}'(\Rn) \\
\mathcal{D}'(\Rn) & \to \mathcal{D}'(\Rn)
\end{aligned}\right.
\end{equation*}
But for pseudo-differential operators $T_\sigma$, generally speaking, we only have $T_\sigma \colon \mathcal{E}'(\Rn) \to \mathcal{D}'(\Rn)$.
\begin{equation*}
T_\sigma \colon
\left\{\begin{aligned}
\mathcal{E}'(\Rn) & \to \mathcal{D}'(\Rn) \\
\mathcal{E}'(\Rn) & \not\to \mathcal{E}'(\Rn) \\
\mathcal{D}'(\Rn) & \not\to \mathcal{D}'(\Rn).
\end{aligned}\right.
\end{equation*}
A $\Psi$DO which maps $\mathcal{E}'$ to $\mathcal{E}'$ is called \emph{\underline{properly supported}}\index{properly supported}.
In fact any $\Psi$DO can be divided into a properly supported part and a $C^\infty$-smooth part.

\begin{lem} \label{lem:kedi-PM2021}
	Assume $m \in \R$ and $\sigma \in S^m$ is a symbol, and $K(x,y)$ is the kernel of $T_\sigma$.
	Then for any $\epsilon > 0$, there exists two symbols $\sigma_1 \in S^m$ and $\sigma_2 \in S^{-\infty}$ such that $\sigma = \sigma_1 + \sigma_2$, $T_{\sigma_1}$ is properly supported, $T_{\sigma_2}$ is smooth,
	and their kernels $K_1, K_2$ has the following properties:
	\begin{align*}
	\supp K_1 & \subset \{ (x,y) \in \R^{2n} \,;\, |x-y| \leq \epsilon \}, \quad \text{(properly supported)} \\
	\supp K_2 & \subset \{ (x,y) \in \R^{2n} \,;\, |x-y| \geq \epsilon/2 \}. \quad \text{(smooth)}
	\end{align*}
\end{lem}

\begin{proof}[Idea of the proof]
	The proof needs Theorem \ref{thm:ReV-PM2021}.
	
	Fix a cutoff function $\chi \in C_c^\infty(\R)$ such that $\chi(t) = 1$ when $|t| \leq \epsilon/2$ and $\chi(t) = 0$ when $|t| \geq \epsilon$.
	We have
	\begin{equation*}
	T_\sigma \varphi(x)
	= (2\pi)^{-n} \int e^{i(x-y) \cdot \xi} \sigma(x,\xi) \varphi(y) \dif y \dif \xi \\
	= T' \varphi(x) + T'' \varphi(x),
	\end{equation*}
	where
	\begin{equation} \label{eq:2Ka-PM2021}
	\left\{\begin{aligned}
	T' \varphi(x) & = (2\pi)^{-n} \int e^{i(x-y) \cdot \xi} \chi(|x-y|^2) \sigma(x,\xi) \varphi(y) \dif y \dif \xi, \\
	T'' \varphi(x) & = (2\pi)^{-n} \int e^{i(x-y) \cdot \xi} (1-\chi(|x-y|^2)) \sigma(x,\xi) \varphi(y) \dif y \dif \xi.
	\end{aligned}\right.
	\end{equation}
	By Theorem \ref{thm:ReV-PM2021} we see that there exist $\sigma_1$, $\sigma_2 \in S^m$ such that $T' = T_{\sigma_1}$ and $T'' = T_{\sigma_2}$, so $T_\sigma = T_{\sigma_1} + T_{\sigma_2} = T_{\sigma_1 + \sigma_2}$.
	By Lemma \ref{lem:sBi-PM2021} we know $\sigma = \sigma_1 + \sigma_2$.
	
	The fact $\sigma_2 \in S^{-\infty}$ can be seen when using the asymptotics in Theorem \ref{thm:ReV-PM2021}, namely,
	\[
	\sigma_2(x,\xi)
	= \sum_{|\alpha| \leq N} \frac 1 {\alpha!} D_y^\alpha \partial_\eta^\alpha \big( (1-\chi(|x-y|^2)) \sigma(x,\eta) \big) |_{(y, \eta) = (x, \xi)} + S^{m-N-1} = S^{m-N-1}
	\]
	holds for $\forall N \in \mathbb N$, so $\sigma_2 \in S^{-\infty}$.

	From \eqref{eq:2Ka-PM2021} we can see
	\begin{equation*}
	\left\{\begin{aligned}
	K_1(x,y) & = (2\pi)^{-n} \int e^{i(x-y) \cdot \xi} \chi(|x-y|^2) \sigma(x,\xi) \dif \xi, \\
	K_2(x,y) & = (2\pi)^{-n} \int e^{i(x-y) \cdot \xi} (1-\chi(|x-y|^2)) \sigma(x,\xi) \dif \xi.
	\end{aligned}\right.
	\end{equation*}
	which implies $T_1$ is properly supported.
	And the requirements for the $\supp K_1$ and $\supp K_1$ can be seen from the expression above.
	The proof is complete.
\end{proof}

\section{Pseudolocal property} \label{sec:pslo-PM2021}

We talk about singular support and pseudolocal property.

\begin{defn}[Singular support\index{singular support}] \label{defn:ssupp-PM2021}
	For a distribution $u \in \mathscr D'$, we define its \emph{singular support} to be the complement of the set $\bigcap \{ O \subset \Rn \,;\, O \text{~is open and~} A \subset O\}$ where
	\[
	A = \{x \in \Rn \,;\, u \text{~is~} C^\infty \text{~at~} x \}.
	\]
	We denote the singular support of a distribution $u \in \mathscr D'$ as $\ssupp u$.
\end{defn}

It is obvious that $\ssupp u$ is as closed set and
\[
\ssupp u \subset \supp u.
\]
We know a differential operator doesn't increase the support of a distribution, but this is not true for a $\Psi$DO.
More specifically, if a distribution $u$ is supported in $\Omega$, then $T u$ might not be supported in a domain $\Omega$ anymore.
Instead, $\Psi$DOs have another property, called pseudolocal property, which means $\Psi$DOs don't increase the singular support of a distribution.

\begin{thm}[Pseudolocal property\index{pseudolocal property}] \label{thm:pslo-PM2021}
	Assume $T$ is a {\rm $\Psi$DO}, then
	\[
	\boxed{\ssupp (Tu) \subset \ssupp u.}
	\]
\end{thm}

\begin{proof}
	Assume $x_0 \notin \ssupp u$.
	Because $\ssupp u$ is closed, we can find $\epsilon > 0$ such that $u$ is $C^\infty$ in $B(x_0,\epsilon)$.
	According to Lemma \ref{lem:kedi-PM2021}, we can divided $T$ into $T_1$ and $T_2$ such that $T_2$ is $C^\infty$-smooth and the kernel $K_1$ of $T_1$ satisfies
	\begin{equation*}
	\supp K_1 \subset \{ (x,y) \in \R^{2n} \,;\, |x-y| \leq \epsilon/4 \}.
	\end{equation*}
	Hence, for $\forall x \in B(x_0,\epsilon/4)$, we have $K(x,y) = 0$ when $|y - x_0| \geq \epsilon/2$.
	
	Fix a function $\chi \in C_c^\infty$ such that $\chi(y) = 1$ when $|y-x_0| \leq \epsilon/2$ and $\supp \chi = B(x_0,\epsilon)$, then
	\[
	\forall x \in B(x_0,\epsilon/4), \quad
	T_1 u(x)
	= \int K(x,y) u(y) \dif y
	= \int K(x,y) \chi(y) u(y) \dif y
	= T_1 (\chi u)(x).
	\]
	Note that $\chi u \in C_c^\infty \subset \scrS$ due to the fact that $u$ is $C^\infty$ in $\supp \chi$, so $T_1 (\chi u) \in \scrS$.
	Because $T_1 u = T_1 (\chi u) \in \scrS$ on $B(x_0,\epsilon/4)$, we conclude that $T_1 u \in C^\infty(B(x_0,\epsilon/4))$.
	Also, $T_2 u \in C^\infty$ because $T_2$ is $C^\infty$-smooth.
	In total, $T u$ is $C^\infty$-smooth in a small neighborhood of $x_0$, so $x_0 \notin \ssupp (Tu)$.
	
	We obtain $(\ssupp u)^c \subset (\ssupp (Tu))^c$.
	The proof is complete.
\end{proof}

\section*{Exercise}

\begin{ex}
	Prove $\scrS(\Rn) \subset S^{-\infty}$, namely, $\forall \varphi \in \scrS(\Rn)$, $\varphi(\xi) \in S^{-\infty}$.
\end{ex}

\begin{ex}
	Prove Lemma \ref{lem:ele-PM2021}.
\end{ex}

\begin{ex}
	Prove Lemma \ref{lem:TSS-PM2021}.
	See \cite[Prop.~6.7]{wong2014introduction} for reference.
\end{ex}

\chapter{Oscillatory integrals} \label{ch:OSInt-PM2021}

In \S \ref{sec:SPDO-PM2021} we encountered the notion of kernel of a $\Psi$DO,
\begin{equation*}
K(x,y) := (2\pi)^{-n} \int e^{i(x-y) \cdot \xi} \sigma(x,\xi) \dif \xi.
\end{equation*}
which might not be integral in the Lebesgue sense (e.g.~when $\sigma(x,\xi) = 1$).
However, if we look back to the original definition of a $\Psi$DO,
\[
(T_{\sigma}\varphi)(x)
= (2\pi)^{-n/2} \int_{\Rn} e^{ix \cdot \xi} \sigma(x,\xi) \hat{\varphi}(\xi) \dif{\xi},
\]
the integral above is always well-defined in the Lebesgue sense, because $\hat \varphi$ is rapidly decaying.
Specifically, for any $m \in \R$ and any $\sigma \in S^m$, we have
\[
|(T_{\sigma}\varphi)(x)|
\lesssim \int |\sigma(x,\xi) \hat{\varphi}(\xi)| \dif{\xi}
\lesssim \int_{\Rn} \agl[\xi]^m \agl[\xi]^{-m-n-1} \dif \xi
< +\infty.
\]
The problems emerges when we expand the Fourier transform $\hat \varphi$ (by a variable $y$) and {\bf exchange} the integration order of $y$ and $\xi$:
\begin{equation} \label{eq:T2-PM2021}
\begin{aligned}
(T_{\sigma}\varphi)(x)
& = (2\pi)^{-n} \int \big( \int e^{i(x-y) \cdot \xi} \sigma(x,\xi) \varphi(y) \dif y \big) {\color{red}\dif \xi}, \\
\int_{\Rn} K(x,y) \varphi(y) \dif y
& = (2\pi)^{-n} \int \big( \int e^{i(x-y) \cdot \xi} \sigma(x,\xi) {\color{red}\dif \xi} \big) \varphi(y) \dif y.
\end{aligned}
\end{equation}
According to Fubini's theorem, this exchange is valid only when all of the integrals involved are absolutely integrable.
From time to time we will encounter integrals of the form \eqref{eq:T2-PM2021}, but also more general than that.
A rigorous framework is appealing for making the integrals of these type always well-defined.

\section{Oscillatory integrals - Type I} \label{sec:OI1-PM2021}

Generally speaking, for any $u \in \scrS(\Rn)$ and $\sigma \in S^m(\R_x^n \times R_\xi^N)$, one is interested in the following integral
\begin{equation} \label{eq:osci-PM2021}
	I(u) := \int e^{i \varphi(x,\xi)} \sigma(x,\xi) u(x) \dif x \dif \xi
\end{equation}
where $\varphi$ is a phase function defined as follows.

\begin{defn}[Phase function\index{phase function}] \label{defn:phaf-PM2021}
	Function $\varphi$ is called a \emph{phase function (of order $\mu$)} if it satisfies
	\begin{enumerate}
		\item $\varphi \in C^\infty(\R_x^n \times (\R_\xi^N \backslash \{0\}); \R)$ is real-valued;
		
		\item $\varphi$ is homogeneous w.r.t.~$\xi$ of order $\mu > 0$, i.e.~$\varphi(x,t\xi) = t^\mu \varphi(x,\xi)$;
		
		\item $|\nabla_{(x,\xi)} \varphi(x,\xi)| \neq 0$ for $\forall (x,\xi) \in \R_x^n \times (\R_\xi^N \backslash \{0\})$.
	\end{enumerate}
\end{defn}

There are different ways to define the notion of phase functions, see \cite[\S 7.8]{horm2003IIV}, and we don't pursue diversity here.
Note that $n$ might not equal $N$, and most of the results in \S \ref{sec:SPDO-PM2021} holds also for the case $n \neq N$.
Here we consider phases of order $\mu$, instead of just order 1, because in \S \ref{ch:SCP-PM2021} and \S \ref{ch:SCla-PM2021} we do encounter phases of order 2.
The condition $\mu > 0$ is indispensable.

In contrast with \eqref{eq:T2-PM2021}, in \eqref{eq:osci-PM2021} it is not sure that integrating first w.r.t~$x$ (or $\xi$) can guarantee it's integrable.
Instead, we study
\begin{equation} \label{eq:Ieci-PM2021}
\lim_{\epsilon \to 0^+} I_\epsilon(u)
:= \lim_{\epsilon \to 0^+} \int e^{i \varphi(x,\xi)} \sigma(x,\xi) {\color{red}\chi(\epsilon \xi)} u(x) \dif x \dif \xi,
\end{equation}
where $\chi$ is a function in $C_c^\infty(\Rn)$ with $\chi(0) = 1$.
We show that the limit \eqref{eq:Ieci-PM2021} exists and its value is independent of the choice of $\chi$.

\begin{thm} \label{thm:Ieci-PM2021}
	Assume $m \in \R$, $\sigma \in S^m(\R_x^n \times R_\xi^N)$ and $\varphi$ is a phase function of order $\mu$.
	Fix a function $\chi \in C_c^\infty(\Rn)$ with $\chi(0) = 1$.
	Assume either $u \in C_c^\infty(\Rn)$, or $u \in \scrS(\Rn)$ and $\partial_x^\alpha \varphi(x,\xi)$ is tempered w.r.t.~$x$ for any $\alpha$.
	Then the limit \eqref{eq:Ieci-PM2021} exists and its value is independent of the choice of $\chi$, and it equals to
	\[
	\int e^{i \varphi(x,\xi)} L^T \big( \sigma(x,\xi) u(x) \big) \dif x \dif \xi,
	\]
	when integer $T > (m+N)/\mu$ where $L$ is given in Lemma \ref{lem:ope-PM2021}.
\end{thm}

The proof Theorem \ref{thm:Ieci-PM2021}, we fist do some preparation.

\begin{lem} \label{lem:chbdd-PM2021}
	Assume $\chi \in C_c^\infty(\Rn)$ and let $\epsilon \in \R$.
	There exists a constant $C$ independent of $\epsilon$ such that
	\[
	\boxed{|\partial_\xi^{\alpha} \big( \chi(\epsilon \xi) \big)| \leq C \agl[\xi]^{-|\alpha|}.}
	\]
\end{lem}

\begin{proof}
	Because $\chi \in C_c^\infty(\Rn)$, there exists a fixed constant $C$ such that $\chi(\epsilon \xi) \equiv 0$ when $|\epsilon \agl[\xi]| \geq C$.
	When $|\epsilon \agl[\xi]| \geq C$, $\chi(\epsilon \xi) = 0$ so $\partial_\xi^{\alpha} \big( \chi(\epsilon \xi) \big) = 0$;
	when $|\epsilon \agl[\xi]| \leq C$, we have
	\[
	|\partial_\xi^{\alpha} \big( \chi(\epsilon \xi) \big)|
	= |\epsilon|^{|\alpha|} |\partial_\xi^{\alpha} \chi(\epsilon \xi)|
	\leq (C\agl[\xi])^{-|\alpha|} \sup_\Rn |\partial_\xi^{\alpha} \chi|
	\simeq C\agl[\xi]^{-|\alpha|}.
	\]
	We arrive at the conclusion.
\end{proof}

\begin{lem} \label{lem:ope-PM2021}
	Assume $\varphi$ is a phase function $\varphi$ of order $\mu$, and $\partial_x^\alpha \varphi(x,\xi)$ is tempered w.r.t.~$x$ for any $\alpha$.
	Then there exists an first order linear differential operator
	\[
	L = a_j(x,\xi) \partial_{x_j} + b_j(x,\xi) \partial_{\xi_j} + c(x,\xi)
	\]
	such that ${}^t L(e^{i \varphi(x,\xi)}) = e^{i \varphi(x,\xi)}$, and for any fixed $x_0$, $a_j(x_0,\cdot) \in S^{-\mu}$, $b_j(x_0,\cdot) \in S^{1-\mu}$, $c(x_0,\cdot) \in S^{-\mu}$,
	and $a_j, b_j, c$ are tempered functions of $x$ variable.
\end{lem}

Here $\agl[{}^t L f,g] := \agl[f,Lg]$, where the integral is w.r.t.~$(x,\xi)$, and $f, g \in C_0^\infty$.
${}^t L$ is call the \emph{transpose} of $L$, e.g.~${}^t (\nabla_\xi) = -\nabla_\xi$.

\begin{proof}[Proof of Lemma \ref{lem:ope-PM2021}]
	We write $\nabla_x \varphi = \varphi_x$ and $\nabla_\xi \varphi = \varphi_\xi$ for short.
	Fix a $\chi \in C_c^\infty(\R^N)$ with $\chi \equiv 1$ in a neighborhood of 0.
	Construct
	\[
	M := (1 - \chi(\xi)) \frac {\varphi_x \cdot D_x + \agl[\xi]^2 \varphi_\xi \cdot D_\xi} {|\varphi_x|^2 + \agl[\xi]^2 |\varphi_\xi|^2} + \chi(\xi).
	\]
	We mention several facts about $M$:
	\begin{itemize}
	\item First, $M$ is well-defined.
	Note that the denominator $|\varphi_x|^2 + \agl[\xi] |\varphi_\xi|^2 \geq |\varphi_x|^2 + |\varphi_\xi|^2 \neq 0$ when $(x,\xi) \in \R_x^n \times (\R_\xi^N \backslash \{0\})$, and the point $\xi = 0$ has been cutoff by $1-\chi$, so $M$ is always well-defined;
	
	\item Second, away from $\xi = 0$,
	$\varphi_x(x_0, \cdot) \in S^\mu$, $\varphi_\xi(x_0, \cdot) \in S^{\mu-1}$, the denominator$(x_0, \cdot) \in S^{2\mu}$;
	
	
	\item Third,
	\(
	M e^{i \varphi(x,\xi)}
	= (1 - \chi) e^{i \varphi(x,\xi)} + \chi e^{i \varphi(x,\xi)}
	= e^{i \varphi(x,\xi)}.
	\)
	\end{itemize}
	
	The transpose of $M$ is the desired operator.
	Indeed, it can be checked that, when $x$ is fixed,
	\begin{equation} \label{eq:tM-PM2021}
	{}^t M = (1-\chi) S^{-\mu} \partial_{x} + (1-\chi) S^{1-\mu} \partial_{\xi} + (1-\chi) S^{-\mu} + \chi(\xi).
	\end{equation}
	The proof is complete.
\end{proof}

\begin{rem} \label{rem:tM-PM2021}
	Here $S^{-\mu} \partial_{x}$ is a shorthand of $a \in C^\infty$ such that $a_j \partial_{x_j}$ for some $a(x_0,\cdot) \in S^{-\mu}$.
	Readers should note that \eqref{eq:tM-PM2021} is somehow misleading because the coefficients might not be bounded w.r.t.~$x$, e.g.~${}^t M = x_1 \partial_{x_1}$.
	However, they must be tempered, and these tempered growth will be balanced by the rapid decay of $u$,
	The notations in \eqref{eq:tM-PM2021} wouldn't hurt.
\end{rem}

\begin{proof}[Proof of Theorem \ref{thm:Ieci-PM2021}]
	Choose $L$ according to Lemma \ref{lem:ope-PM2021}, then
	\begin{align}
	I_\epsilon(u)
	& = \int ({}^t L)^T (e^{i \varphi(x,\xi)}) \, \sigma(x,\xi) \chi(\epsilon \xi) u(x) \dif x \dif \xi \nonumber \\
	& = \int e^{i \varphi(x,\xi)} L^T \big( \sigma(x,\xi) \chi(\epsilon \xi) u(x) \big) \dif x \dif \xi. \label{eq:Ieu-PM2021}
	\end{align}
	Readers should note that the transpose of $L$ is realized by the {\color{red}classical integration by parts} (nothing fancy here), and it is the presence of $\chi(\epsilon \xi)$ that cancels the boundary terms and makes the integration by parts applicable.
	
	The conditions ``$a_j \in S^{-\mu}$, $b_j \in S^{1-\mu}$, $c \in S^{-\mu}$'' in Lemma \ref{lem:ope-PM2021} give us
	\begin{align*}
	& \ L^T \big( \sigma(x,\xi) \chi(\epsilon \xi) u(x) \big) \\
	= & \ \big( (a_j(x,\xi) \partial_{x_j} + b_j(x,\xi) \partial_{\xi_j} + c(x,\xi) \big)^T \big( \sigma(x,\xi) \chi(\epsilon \xi) u(x) \big) \\
	= & \ \sum_{|\alpha+\beta| \leq T} S^{-\mu|\alpha| + (1-\mu)|\beta| - \mu(T - |\alpha| - |\beta|)} \partial_x^\alpha \partial_\xi^\beta \big( \sigma(x,\xi) \chi(\epsilon \xi) u(x) \big) \\
	= & \ \sum_{|\alpha+\beta| \leq T} S^{|\beta| - \mu T} \partial_x^\alpha \partial_\xi^\beta \big( \sigma(x,\xi) \chi(\epsilon \xi) u(x) \big) \\
	= & \ \sum_{|\alpha+\beta| \leq T} \sum_{\beta' + \beta'' = \beta} C_{\beta', \beta''} S^{|\beta| - \mu T} \partial_x^\alpha \big(  \partial_\xi^{\beta'} \sigma(x,\xi)  \partial_\xi^{\beta''} [\chi(\epsilon \xi)] u(x) \big) \\
	= & \ \sum_{|\alpha+\beta| \leq T} \sum_{\beta' + \beta'' = \beta} C S^{|\beta| - \mu T} \partial_x^\alpha \big(  \partial_\xi^{\beta'} \sigma(x,\xi)  \partial_\xi^{\beta''} [\chi(\epsilon \xi)] u(x) \big) \\
	= & \ \sum_{|\alpha+\beta| \leq T} \sum_{\beta' + \beta'' = \beta} \sum_{\alpha' + \alpha'' = \alpha} C S^{|\beta| - \mu T} \partial_x^{\alpha'} \partial_\xi^{\beta'} \sigma(x,\xi) {\color{red}\partial_\xi^{\beta''} [\chi(\epsilon \xi)]} \partial_x^{\alpha''} u(x).
	\end{align*}
	The term $\partial_\xi^{\beta''} [\chi(\epsilon \xi)]$ is the only term that depends on $\epsilon$.
	Hence, by Lemma \ref{lem:chbdd-PM2021} we can have
	\begin{align*}
	& \ |L^T \big( \sigma(x,\xi) \chi(\epsilon \xi) u(x) \big)| \\
	\leq & \ \sum_{|\alpha+\beta| \leq T} \sum_{\beta' + \beta'' = \beta} \sum_{\alpha' + \alpha'' = \alpha} C \agl[\xi]^{|\beta| - \mu T} C \agl[\xi]^{m-|\beta'|} C \agl[\xi]^{-|\beta''|} \scrS(\R_x^n) \\
	\leq & \ C \agl[\xi]^{m - \mu T} \scrS(\R_x^n),
	\end{align*}
	where the constant $C$ is independent of $\epsilon$.
	Then $T$ is chosen to be larger than $(m+N)/\mu$, the integrand in \eqref{eq:Ieu-PM2021} is bounded by a absolutely integral function.
	Therefore, according to LDCT, the limit $\lim_{\epsilon \to 0^+} I_\epsilon(u)$ exists.
	Readers may think where we used the condition $\mu > 0$.
	
	We also show that the limit is independent of $\chi$.
	Fix a $T > (m+N)/\mu$, then by \eqref{eq:Ieu-PM2021} we have
	\begin{align}
	\lim_{\epsilon \to 0^+} I_\epsilon(u)
	& = \lim_{\epsilon \to 0^+} \int e^{i \varphi(x,\xi)} L^T \big( \sigma(x,\xi) \chi(\epsilon \xi) u(x) \big) \dif x \dif \xi \nonumber \\
	& = \int e^{i \varphi(x,\xi)} \lim_{\epsilon \to 0^+} L^T \big( \sigma(x,\xi) \chi(\epsilon \xi) u(x) \big) \dif x \dif \xi \quad \text{(thanks to LDCT)} \nonumber \\
	& = \int e^{i \varphi(x,\xi)} L^T \big( \sigma(x,\xi) u(x) \big) \dif x \dif \xi, \label{eq:Iuin-PM2021}
	\end{align}
	which implies $\lim_{\epsilon \to 0^+} I_\epsilon(u)$ is independent of $\chi$.
	The proof is complete.
\end{proof}

Readers may think about if the framework can be generalized to symbols in $S_{\rho, \delta}^m$.

Now let's summarize the definition of oscillatory integrals.

\begin{defn}[Oscillatory integral\index{oscillatory integral I}] \label{defn:OsD-PM2021}
	For any $m \in \R$, any $\sigma \in S^m(\R_x^n \times R_\xi^N)$, and any phase function $\varphi$ of order $\mu$,
	and either $u \in C_c^\infty(\Rn)$, or $u \in \scrS(\Rn)$ and $\partial_x^\alpha \varphi(x,\xi)$ is tempered w.r.t.~$x$ for any $\alpha$, the integral
	\begin{equation*} 
	I(u) = \int e^{i \varphi(x,\xi)} \sigma(x,\xi) u(x) \dif x \dif \xi
	\end{equation*}
	is defined as
	\begin{equation} \label{eq:Iec2-PM2021}
	\boxed{I(u)
	:= \lim_{\epsilon \to 0^+} \int e^{i \varphi(x,\xi)} \sigma(x,\xi) \chi(\epsilon \xi) u(x) \dif x \dif \xi,}
	\end{equation}
	where the result is independent of $\chi$, as long as $\chi \in C_c^\infty$ and $\chi(0) = 1$.
	The limit equals
	\begin{equation*} \label{eq:OsD-PM2021}
	\boxed{I(u)
	= \int e^{i \varphi(x,\xi)} L^T \big( \sigma(x,\xi) u(x) \big) \dif x \dif \xi}
	\end{equation*}
	when integer $T > (m+N)/\mu$, where $L$ is given in Lemma \ref{lem:ope-PM2021}.
\end{defn}

In many cases we will meet oscillatory integrals involving parameters.

\begin{lem} \label{lem:Osxy-PM2021}
	For any $\sigma \in S^m(\R_x^{n_1} \times \R_y^{n_2} \times R_\xi^N)$, and any phase function $\varphi$ of order $\mu$, 
	and for either $u \in C_c^\infty(\R_x^{n_1} \times \R_y^{n_2})$, or $u \in \scrS(\R_x^{n_1} \times \R_y^{n_2})$ and $\partial_{(x,y)}^\alpha \varphi(x,\xi)$ is tempered w.r.t.~$(x,y)$ for any $\alpha$, the integral
	\begin{equation} \label{eq:oscy-PM2021}
	I(u)(y) := \int e^{i \varphi(x,y,\xi)} \sigma(x,y,\xi) u(x,y) \dif x \dif \xi
	\end{equation}
	is a well-defined oscillatory integral, 
	and $I \colon \scrS(\R^{n_1}) \to \scrS(\R^{n_2})$ bounded.
	Moreover, we have
	\begin{align*}
	\frac{\partial}{\partial y}\big( I(u)(y) \big)
	& = \int \frac{\partial}{\partial y}\big( e^{i \varphi(x,y,\xi)} \sigma(x,y,\xi) u(x,y) \big) \dif x \dif \xi, \\
	\int I(u)(y) \dif y
	& = \int e^{i \varphi(x,y,\xi)} \sigma(x,y,\xi) u(x,y) \dif x \dif y \dif \xi.
	\end{align*}
\end{lem}

We omit the proof.
The take-home message of Lemma \ref{lem:Osxy-PM2021} is that oscillatory integrals can have parameters, and there are much freedom to put operations w.r.t.~$y$ inside the integration $I(u)(y)$.

Now we go back to $\Psi$DOs and its kernel.
We have intuitively claimed that the kernel of $T_\sigma$ is of the form
\[
K(x,y) = \int e^{i (x-y) \cdot \xi} \sigma(x,\xi) \dif \xi.
\]
Note that this object has variables $(x,y)$, so a proper candidate of test functions should be $w(x,y) \in \scrS$.
We choose $w(x,y) = u(y) v(x)$ where $u,v \in \scrS$, then formally we should have
\[
\agl[K, v \otimes u]
\simeq \int e^{i (x-y) \cdot \xi} \sigma(x,\xi) u(y) v(x) \dif x \dif y \dif \xi
\simeq \agl[T_\sigma u, v].
\]
The integral above is exactly an example of Lemma \ref{lem:Osxy-PM2021}, so it is a well-defined oscillatory integral.
Now the kernel of a $\Psi$DO can be defined.

\begin{defn}[Kernel\index{kernel}] \label{defn:kerP-PM2021}
	Assume $m \in \R$ and $\sigma \in S^m$, and $T_\sigma$ is the corresponding $\Psi$DO.
	The \emph{kernel} of $T_\sigma$ is defined as a map:
	\[
	K_\sigma \colon w \in \scrS(\R^{2n}) \ \mapsto \ \boxed{\agl[K_\sigma,w] := (2\pi)^{-n} \int e^{i (x-y) \cdot \xi} \sigma(x,\xi) w(x,y) \dif x \dif y \dif \xi} \in \mathbb C.
	\]
	When $x \neq y$, we write $K_\sigma(x,y)$ as
	\begin{equation} \label{eq:kerP-PM2021}
	\boxed{K_\sigma(x,y) = (2\pi)^{-n} \int e^{i (x-y) \cdot \xi} \sigma(x,\xi) \dif \xi.}
	\end{equation}
\end{defn}

The well-definedness of Definition \ref{defn:kerP-PM2021} is guaranteed by Lemma \ref{lem:Osxy-PM2021}.

\begin{rem} \label{rem:kerP-PM2021}
	From \eqref{eq:kerP-PM2021}, we know that when $m < -N$, $K_\sigma$ is a well-defined bounded function for any $(x,y)$ because 
	\[
	|K_\sigma|
	\lesssim \int |\sigma(x,\xi)| \dif \xi
	\lesssim \int \agl[\xi]^m \dif \xi \leq C < +\infty.
	\]
	This implies when the order of $\sigma$ is small enough, there should hold some type of boundedness for $T_\sigma$, and we will cover this in \S \ref{ch:BddP-PM2021}.
	However when $m \geq -N$, only when $x \neq y$ the kernel $K_\sigma$ can be expressed as \eqref{eq:kerP-PM2021}.
\end{rem}

\begin{lem} \label{lem:Kuv-PM2021}
	Under the assumption of Definition \ref{defn:kerP-PM2021}, we have
	\[
	\agl[K_\sigma,u(y) v(x)]
	= \agl[T_\sigma u, v].
	\]
\end{lem}

We omit the proof.

\begin{lem} \label{lem:keSmo-PM2021}
	Under the assumption of Definition \ref{defn:kerP-PM2021}, when $x \neq y$, $K_\sigma$ is $C^\infty$ smooth
	Moreover, for $T$ large enough, we have
	\[
	\boxed{|K_\sigma(x,y)| \leq C_T |x-y|^{-T}, \quad |x-y| \geq 1.}
	\]
\end{lem}

\begin{proof}
	For any fixed point $(x,y)$ with $x \neq y$, we show that $K_\sigma$ is $C^\infty$ at this point.
	Fix a function $\chi \in C^\infty$ satisfying $\chi \equiv 0$ in a small neighborhood $U$ of the diagonal $\{x=y\}$ and $\chi \equiv 1$ in the interior of the complement of $U$.
	We can shrink $U$ such that for any $x \neq y$, $\chi(x,y) = 1$.
	For any $w \in \scrS$, we have
	\[
	\agl[\partial_x^\alpha (\chi K_\sigma),w]
	\simeq \int e^{i (x-y) \cdot \xi} \xi^\alpha \sigma(x,\xi) w(x,y) \dif x \dif y \dif \xi.
	\]
	Apply the operator $L$ to $e^{i (x-y) \cdot \xi}$ and integrate by parts, we obtain
	\[
	\agl[\partial_x^\alpha \chi K_\sigma, w]
	= \agl[(2\pi)^{-n} \int e^{i (x-y) \cdot \xi} \big( \frac {-(x-y) \cdot D_\xi} {|x-y|^2} \big)^T \big( \xi^\alpha \sigma(x,\xi) \big) \dif \xi, w \chi],
	\]
	which implies
	\[
	\partial_x^\alpha (\chi K_\sigma)(x,y) \simeq \int e^{i (x-y) \cdot \xi} \big( \frac {-(x-y) \cdot D_\xi} {|x-y|^2} \big)^T \big( \xi^\alpha \sigma(x,\xi) \big) \dif \xi.
	\]
	It can be checked that when $T$ is large enough, the integral above will be absolutely integrable, and is of the order $|x-y|^{-T}$ for $T$ large enough.
	Therefore, $K_\sigma \in C^\infty(\R^{2n} \backslash \{x=y\})$, and $K$ satisfies the desired decay.	
\end{proof}

By Lemma \ref{lem:keSmo-PM2021}, we see that $K(x,y)$ behave nicely when off the diagonal, thanks to the notion of oscillatory integrals.
However, on the diagonal, $K(x,y)$ might still be ill-defined.
See Example \ref{exmp:ke1d-PM2021}.

\begin{exmp} \label{exmp:ke1d-PM2021}
	The kernel corresponding to the identity operator (symbol $= 1$) is the distribution $\delta(x-y)$.
	This is because
	\begin{align*}
	\agl[K,v \otimes u]
	& = \agl[T_1 u, v]
	= \agl[u, v]
	= \int u(x) v(x) \dif x \\
	& = \int \delta(x-y) u(y) v(x) \dif x \dif y \\
	& = \agl[\delta(x-y),v \otimes u].
	\end{align*}
	which gives $K(x,y) = \delta(x-y)$.
\end{exmp}

\begin{exmp} \label{exmp:ke2d-PM2021}
	The kernel corresponding to $D_1$ (symbol $= \xi_1$) is $D_1 \delta(x-y)$.
	This is because
	\begin{align*}
	\agl[K,v \otimes u]
	& = \agl[D_1 u, v]
	= \int D_1 u(x) v(x) \dif x
	= \int \delta(x-y) D_1 u(y) v(x) \dif x \dif y \\
	& = -\int D_{y_1} \big( \delta(x-y) \big) u(y) v(x) \dif x \dif y
	= \int D_1 \delta(x-y) u(y) v(x) \dif x \dif y \\
	& = \agl[D_1 \delta(x-y),v \otimes u].
	\end{align*}
	which gives $K(x,y) = D_1 \delta(x-y)$.
	Besides, readers may also have tried another way to compute the kernel and get a zero result: when $x \neq y$
	\begin{align*}
	K(x,y)
	& = (2\pi)^{-n} \int e^{i (x-y) \cdot \xi} \xi_1 \dif \xi
	= (2\pi)^{-n} \int \big( \frac {(x-y) \cdot D_{\xi}} {|x-y|^2} \big)^2 \big( e^{i (x-y) \cdot \xi} \big) \xi_1 \dif \xi \\
	& = (2\pi)^{-n} \int e^{i (x-y) \cdot \xi} \big( \frac {-(x-y) \cdot D_{\xi}} {|x-y|^2} \big)^2 (\xi_1) \dif \xi
	= 0.
	\end{align*}
	This result is technically correct (because $D_1 \delta(x-y) = 0$ when $x \neq y$), but is not complete: it cannot speak about the behavior of $K$ on the diagonal.
	This example told us, none of the methods is the best one to get most accurate expression for a kernel, sometimes we need to do complicated and delicate computations.
\end{exmp}

\section{Oscillatory integrals - Type II} \label{sec:OI2-PM2021}

We know the Fourier transform of a constant function is the $\delta$ function and so the inverse Fourier transform of the $\delta$ function should be the constant, namely,
\begin{equation} \label{eq:xyx1-PM2021}
	\int e^{i(x-y) \cdot \xi} \dif x \dif \xi \simeq 1.
\end{equation}
This integral can be regarded as the $I(u)$ defined in Lemma \ref{lem:Osxy-PM2021} where the symbol and the $u$ are both constant 1.
However, this is not covered by Lemma \ref{lem:Osxy-PM2021} because $1 \notin \scrS$.
The map $I$ in Lemma \ref{lem:Osxy-PM2021} is defined on $\scrS$.
Now by using duality arguments we shall generalize it from $\scrS$ to $S^{+\infty}$.

Let $\sigma \in S^m$ and $u \in \scrS$, so $T_\sigma u \in \scrS$.
Let $f \in S^{+\infty}$, then $f$ is a smooth tempered function, so $\agl[T_\sigma u, f]$ is meaningful and
\[
\agl[T_\sigma u, f]
= \lim_{\epsilon \to 0^+} \agl[T_\sigma u, \chi(\epsilon \cdot) f]
\]
holds for any $\chi \in C_c^\infty$.
Expanding the integral, we have
\begin{align*}
\agl[T_\sigma u, f]
& = \lim_{\epsilon \to 0^+} \int e^{i(x-y) \cdot \xi} \sigma(x,\xi) u(y) f(x) \chi(\epsilon x) \dif x \dif y \dif \xi \\
& = \lim_{\epsilon \to 0^+} \int e^{i(x-y) \cdot \xi} \sigma(x,\xi) {\color{red}\chi (\epsilon \xi)} u(y) f(x) \chi(\epsilon x) \dif x \dif y \dif \xi \\
& = \lim_{\epsilon \to 0^+} \int u(y) \big( \int e^{i(x-y) \cdot \xi} \sigma(x,\xi) f(x) {\color{red}\chi(\epsilon x) \chi (\epsilon \xi)} \dif x \dif \xi \big) \dif y.
\end{align*}
This inspired us to define \eqref{eq:xyx1-PM2021} as
\[
\lim_{\epsilon \to 0^+} \int e^{i(x-y) \cdot \xi} \chi(\epsilon x) \chi (\epsilon \xi) \dif x \dif \xi.
\]
More generally, we can generalize Definition \ref{defn:OsD-PM2021} as follows.

\begin{defn}[Oscillatory integral\index{oscillatory integral II}] \label{defn:OsE-PM2021}
	For any $m \in \R$ and $\rho > 0$, any $\sigma \in S_{\rho}^m(R^N)$, and any phase function $\varphi \in C^\infty(\R^N \backslash \{0\})$ of order $\mu$ (real-valued, $\varphi(\theta) = t^\mu \varphi(\theta)$, and $\nabla_\theta \varphi(\theta) \neq 0$ when $\theta \neq 0$) satisfying
	\begin{equation} \label{eq:rhm1-PM2021}
	\boxed{\rho + \mu > 1,}
	\end{equation}
	the integral
	\begin{equation*} 
	I= \int e^{i \varphi(\theta)} \sigma(\theta) \dif \theta
	\end{equation*}
	is defined as
	\begin{equation} \label{eq:Iec-PM2021}
	\boxed{I
	:= \lim_{\epsilon \to 0^+} \int e^{i \varphi(\theta)} \sigma(\theta) \chi(\epsilon \theta) \dif \theta,}
	\end{equation}
	where the result is independent of $\chi$, as long as $\chi \in C_c^\infty$ and $\chi(0) = 1$.
	The limit equals
	\begin{equation} \label{eq:OsE-PM2021}
	\boxed{I
	= \int e^{i \varphi(\theta)} L^T \big( \sigma(\theta) \big) \dif \theta}
	\end{equation}
	when $T > (m+N)/(\rho + \mu - 1)$, where $L$ is given in Lemma \ref{lem:ope2-PM2021} below.
\end{defn}

\begin{rem} \label{rem:OsE-PM2021}
	The space $S_{\rho}^m(R^N)$ is defined as
	\[
	S_{\rho}^m(R^N) := \{ \varphi \in C^\infty(\R^N) \,;\, |\partial^\alpha \varphi(\theta)| \lesssim \agl[\theta]^{m - \rho|\alpha|} \}.
	\]
	See \cite[\S I.8.1]{alinhac2007pseudo} for more details.
	Also, the formula \eqref{eq:kerP-PM2021} in the definition of the kernel is meaningful now.
\end{rem}

\begin{lem} \label{lem:ope2-PM2021}
	Under the condition in Definition \ref{defn:OsE-PM2021}, there exists an first order linear differential operator
	\[
	L = b_j(\theta) \partial_{\theta_j} + c(\theta)
	\]
	such that ${}^t L(e^{i \varphi(\theta)}) = e^{i \varphi(\theta)}$, $0 \notin \supp b_j$, and $b_j \in S^{1-\mu}$, $c \in S^{-\mu}$, and in a small neighborhood of $\theta = 0$ there holds $b \equiv 0$ and $c \equiv 1$.
\end{lem}

\begin{proof}
	We write $\nabla_\theta \varphi = \varphi_\theta$ for short.
	Fix a $\chi \in C_c^\infty(\Rn)$ with $\chi \equiv 1$ in a neighborhood of 0.
	Construct
	\[
	M := (1 - \chi(\theta)) \frac {\varphi_\theta \cdot D_\theta} {|\varphi_\theta|^2} + \chi(\theta).
	\]
	We mention several facts about $M$:
	\begin{itemize}
		\item First, $M$ is well-defined.
		Note that the denominator $|\varphi_\theta|^2 \neq 0$ when $\theta \in \R^N \backslash \{0\}$, and the point $\theta = 0$ has been cutoff by $1-\chi$, so $M$ is always well-defined;
		
		\item Second, away from $\theta = 0$, $\varphi_\theta \in S^{\mu-1}$;
		
		\item Third,
		\(
		M e^{i \varphi(\theta)}
		= (1 - \chi) e^{i \varphi(\theta)} + \chi e^{i \varphi(\theta)}
		= e^{i \varphi(\theta)}.
		\)
	\end{itemize}
	
	The transpose of $M$ is the desired operator.
	Indeed, it can be checked that, when $x$ is fixed,
	\begin{equation*} 
	{}^t M = (1-\chi) S^{1-\mu} \partial_{\theta} + (1-\chi) S^{-\mu} + \chi(\theta).
	\end{equation*}
	The proof is complete.
\end{proof}

Readers may compare Lemma \ref{lem:ope2-PM2021} with Lemma \ref{lem:ope-PM2021}.

\begin{lem} \label{lem:Lrho-PM2021}
	Assume $L$ is chosen as in Lemma \ref{lem:ope2-PM2021}, and $\sigma \in S_{\rho}^m(R^N)$, then
	\[
	|L^T \big( \sigma(\theta) \big)| \lesssim \agl[\theta]^{m - (\rho + \mu - 1)T}.
	\]
\end{lem}

The proof is left as an exercise.
Combining Lemmas \ref{lem:ope2-PM2021} and \ref{lem:Lrho-PM2021}, we obtain a result similar to Theorem \ref{thm:Ieci-PM2021}.

\begin{thm} \label{thm:Iec2-PM2021}
	Assume $m \in \R$ and $\rho > 0$, $\sigma \in S_{\rho}^m(R^N)$ and $\varphi \in C^\infty(\R^N \backslash \{0\})$ is a phase function of order $\mu$.
	Fix a function $\chi \in C_c^\infty(\Rn)$ with $\chi(0) = 1$.
	Then the limit \eqref{eq:Iec-PM2021} exists and its value is independent of the choice of $\chi$, and it equals to
	\[
	\int e^{i \varphi(\theta)} L^T \big( \sigma(\theta) \big) \dif \theta,
	\]
	when $T > (m+N)/(\rho + \mu - 1)$ where $T$ is given in Lemma \ref{lem:ope2-PM2021}.
\end{thm}

The proof is similar to that of Theorem \ref{thm:Ieci-PM2021}.
The generalized definition of the oscillatory integral can handle more cases.
One of the examples is as follows.

\begin{lem} \label{lem:Tli-PM2021}
	The following equality holds in oscillatory sense,
	\begin{equation*}
	\int_{\R^{2n}} e^{\pm ix\cdot \xi} \dif x \dif \xi = (2\pi)^{n}.
	\end{equation*}
\end{lem}

\begin{proof}
	We shall regard $(x,\xi)$ as the $\theta$ in Definition \ref{defn:OsE-PM2021}, then this phase function $x \cdot \xi$ is of order 2.
	One can also check that $\nabla_{(x,\xi)} (x\cdot \xi) \neq 0$ when $(x,\xi) \neq 0$.
	We choose the cutoff function as $\chi(\epsilon x) \chi(\epsilon \xi)$ where $\chi \in C_c^\infty$ with $\chi(0) = 1$, then
	\begin{align*}
	\int e^{ix\cdot \xi} \dif x \dif \xi
	& {\color{red}:=} \lim_{\epsilon \to 0^+} \int e^{ix\cdot \xi} \chi(\epsilon x) \chi(\epsilon \xi) \dif x \dif \xi
	= \lim_{\epsilon \to 0^+} \int (\int e^{ix\cdot \xi} \chi(x) \dif x) \cdot \chi(\epsilon^2 \xi) \dif \xi \\
	& = (2\pi)^{n/2} \lim_{\epsilon \to 0^+} \int \widehat\chi(-\xi) \chi(\epsilon^2 \xi) \dif \xi
	= (2\pi)^{n/2} \int \widehat\chi(-\xi) \dif \xi \qquad \text{(LDCT)} \\
	& = (2\pi)^{n} \chi(0) = (2\pi)^{n}.
	\end{align*}
	The case for $e^{-ix\cdot \xi}$ is similar.
	Note that all of the integrals above are usual integral besides the first one on the LHS. 
\end{proof}

The following result will be useful.

\begin{lem} \label{lem:Tlic-PM2021}
	The following equality holds in oscillatory sense,
	\begin{equation*}
	\int_{\R^{2n}} e^{\pm ix\cdot \xi} x^\alpha \xi^\beta \dif x \dif \xi = (\pm i)^{|\alpha|} (2\pi)^{n} \alpha! \delta^{\alpha \beta}.
	\end{equation*}
\end{lem}

\begin{proof}
	We have
	\begin{align*}
	& \int e^{\pm ix\cdot \xi} x^\alpha \xi^\beta \dif x \dif \xi
	= \int (\pm D_\xi)^\alpha \big( e^{\pm ix\cdot \xi} \big) \xi^\beta \dif x \dif \xi \\
	= & \int e^{\pm ix\cdot \xi} (\mp D_\xi)^\alpha \big( \xi^\beta \big) \dif x \dif \xi
	= (\pm i)^{|\alpha|} \int e^{\pm ix\cdot \xi} \partial_\xi^\alpha (\xi^\beta) \dif x \dif \xi.
	\end{align*}
	It can be check that $\partial_\xi^\alpha (\xi^\beta) = \beta!/(\beta - \alpha)!\, \xi^{\beta - \alpha}$ when $\alpha \leq \beta$, and $= 0$ otherwise.
	Hence, when $\alpha \leq \beta$ we have
	\begin{equation} \label{eq:rxi-PM2021}
	\int e^{\pm ix\cdot \xi} x^\alpha \xi^\beta \dif x \dif \xi
	= (\pm i)^{|\alpha|} \beta!/(\beta - \alpha)! \int e^{\pm ix\cdot \xi} \xi^{\beta - \alpha} \dif x \dif \xi.
	\end{equation}
	When $\alpha \neq \beta$, we can continue
	\begin{equation*}
	\int e^{\pm ix\cdot \xi} \xi^{\beta - \alpha} \dif x \dif \xi
	\simeq \int D_x^{\beta - \alpha} \big( e^{\pm ix\cdot \xi} \big) \dif x \dif \xi
	= 0
	\simeq \int e^{\pm ix\cdot \xi} D_x^{\beta - \alpha} \big( 1 \big) \dif x \dif \xi
	= 0.
	\end{equation*}
	Therefore, when $\alpha \neq \beta$, we have
	\begin{equation} \label{eq:rx0-PM2021}
	\int_{\R^{2n}} e^{\pm ix\cdot \xi} x^\alpha \xi^\beta \dif x \dif \xi = 0.
	\end{equation}
	
	When $\alpha = \beta$, by Lemma \ref{lem:Tli-PM2021} we can continue \eqref{eq:rxi-PM2021} as follows,
	\[
	\int_{\R^{2n}} e^{\pm ix\cdot \xi} x^\alpha \xi^\beta \dif x \dif \xi
	= (\pm i)^{|\alpha|} \alpha! \int e^{\pm ix\cdot \xi} \dif x \dif \xi
	= (\pm i)^{|\alpha|} \alpha! (2\pi)^n.
	\]
	The proof is complete.
\end{proof}

\begin{rem} 
	The space $S_{\rho}^m(R^N)$ is defined as
	\[
	S_{\rho}^m(R^N) := \{ \varphi \in C^\infty(\R^N) \,;\, |\partial^\alpha \varphi(\theta)| \lesssim \agl[\theta]^{m - \rho|\alpha|} \}.
	\]
	See \cite[\S I.8.1]{alinhac2007pseudo} for more details.
	Also, the formula \eqref{eq:kerP-PM2021} in the definition of the kernel is meaningful now.
\end{rem}

We can generalize Lemma \ref{lem:Tli-PM2021}
\begin{lem} \label{lem:Tlf-PM2021}
	Assume $m \in \R$ and $f \in S_{\rho}^m(\Rn)$ (see Remark \ref{rem:OsE-PM2021}) with $\rho + 2 > 1$.
	The following equality holds in oscillatory sense,
	\begin{equation*}
	\int_{\R^{2n}} e^{\pm ix\cdot \xi} f(x) \dif x \dif \xi = (2\pi)^{n} f(0).
	\end{equation*}
\end{lem}

\begin{rem} \label{rem:Tlf-PM2021}
	Lemma \ref{lem:Tlf-PM2021} indicates the the ``inverse Fourier transform'' is indeed the inverse of ``Fourier transform''.
\end{rem}

\begin{proof}
	We only show the case $+ix \cdot \xi$.
	The condition on $f$ guarantees the integral is well-defined, see Definition \ref{defn:OsE-PM2021}.
	By Taylor's expansion we have
	\begin{equation*}
	f(x) = f(0) + \sum_{j = 1}^n x_j g_j(x),
	\quad \text{where} \quad g_j(x) := \int_0^1 \partial_{x_j} f(tx), \dif t.
	\end{equation*}
	so
	\begin{align*}
	\int e^{ix\cdot \xi} f(x) \dif x \dif \xi
	& = \int e^{ix\cdot \xi} [f(0) + \sum_{j = 1}^n x_j g_j(x)] \dif x \dif \xi \\
	& = f(0) \int e^{ix\cdot \xi} \dif x \dif \xi + \sum_{j = 1}^n \int D_{\xi_j} (e^{ix\cdot \xi}) g_j(x) \dif x \dif \xi \\
	& = (2\pi)^{n} f(0) - \sum_{j = 1}^n \int e^{ix\cdot \xi} D_{\xi_j} \big( g_j(x) \big) \dif x \dif \xi \\
	& = (2\pi)^{n} f(0).
	\end{align*}
	The proof is complete.
\end{proof}

\section*{Exercise}

\begin{ex}
	Check that the ${}^t M$ given in \eqref{eq:tM-PM2021} satisfies \eqref{eq:tM-PM2021}. Hint: utilize the second fact about $M$ to facilitate the derivation.
\end{ex}

\begin{ex}
	Prove Lemma \ref{lem:Lrho-PM2021}.
\end{ex}

\clearpage

\chapter{Stationary phase lemmas} \label{ch:SPL-PM2021}

The stationary phase lemmas is a useful tools for computing certain asymptotics.
Some useful references are \cite[\S 5]{dim1999spe}, \cite[\S 19.3]{eskin2011lectures}, \cite[\S 7.7]{horm2003IIV}, \cite[\S 3]{zw2012semi}.

From time to time we will encounter oscillatory integrals of the form
\begin{equation} \label{eq:Ivx-PM2021}
I(\lambda) = \int e^{i\lambda\varphi(x)} a(x) \dif x
\end{equation}
where $\varphi$ is a phase function of some order and $a$ is a symbol ($|\partial^\beta a(\xi)| \lesssim \agl[\xi]^{m - |\beta|}$).
In \S \ref{ch:OSInt-PM2021} we have introduced schemes to make $I(\lambda)$ well-defined.
Now we focus on the asymptotics of $I(\lambda)$ w.r.t.~$\lambda \to +\infty$ when $\varphi$ satisfies certain conditions.

When $\varphi$ is linear, i.e.~$\varphi(x) = p \cdot x$ for certain fixed vector $p \neq 0$, there is no critical point of $\varphi$ ($|\nabla \varphi| = |p| \neq 0$).
In this case we call $\varphi$ \emph{non-stationary}.
For the non-stationary case, the asymptotics of $I$ is straightforward:
\begin{align*}
I(\lambda) 
& = \int e^{i\lambda p \cdot x} a(x) \dif x
= \int \big( \frac {p \cdot D_x} {\lambda |p|^2} \big)^N (e^{i\lambda p \cdot x}) \, a(x) \dif x \\
& = \lambda^{-N} \int e^{i\lambda p \cdot x} \big( \frac {-p \cdot D_x} {|p|^2} \big)^N (a(x)) \dif x
\lesssim |p|^{-N} \lambda^{-N} \int \sum_{|\beta| = N} C_\beta \partial^\beta a(x) \dif x \\
& \lesssim |p|^{-N} \lambda^{-N} \int \agl[x]^{m - N} \dif x
\lesssim \lambda^{-N},
\end{align*}
provided $N$ is large enough.
This means that $\int e^{i\lambda p \cdot x} a(x) \dif x$ is of rapid decay w.r.t.~$\lambda$.

The gradient of $\varphi$ has been put in the denominator in the derivation above, so the method will not be applicable when the phase function contains critical points.
In this case we call the phase \emph{stationary}.
In this chapter we devote ourselves into the stationary case.

\section{A simple case}

To help readers understand the method of stationary phase, we start with a simple case where the phase is stationary.
To that end, we need to do some preparations.

\subsection{Preliminaries}

We need the Taylor's expansion\index{Taylor's expansion}.
Suppose $f \in C^{N+1}(\Rn; \mathbb C)$, then we have that
\begin{align}
f(x)
& = \sum_{|\delta| \leq N} \frac{1}{\delta!} \big( \partial^{\delta}f \big)(x_0) \cdot (x - x_0)^{\delta} \nonumber \\
& \quad + (N+1) \sum_{|\delta| = N+1} \frac{(x - x_0)^{\delta}}{\delta!} \int_0^1 (1-t)^N \big( \partial^{\delta}f \big)(x_0 + t(x-x_0)) \dif{t}. \label{eq:TaylorInt-PM2021}
\end{align}
The proof of \eqref{eq:TaylorInt-PM2021} can be found in most of the calculus textbook and we omit it here.

Secondly, for a measurable function $u$ in $\Rn$, as long as $\partial^\alpha u \in L^1(\Rn)$ for $|\alpha| \geq n+1$, then $\hat u$ exists and there exists a constant $C$ such that
\begin{equation} \label{eq:Fuu-PM2021}
\nrm[L^1(\Rn)]{\hat u} \leq C \sum_{|\alpha| \leq n + 1} \nrm[L^1(\Rn)]{\partial^\alpha u}.
\end{equation}

\begin{proof}
	We have
	\begin{align*}
	\int |\hat u(\xi)| \dif \xi
	& = \int \agl[\xi]^{-n-1} |\agl[\xi]^{n+1} \hat u(\xi)| \dif \xi
	\leq C \sup_{\Rn} |\agl[\xi]^{n+1} \hat u(\xi)| \\
	& \leq C \sum_{|\alpha| \leq n + 1} C_\alpha \sup_{\Rn} |\xi^\alpha \hat u(\xi)|
	\leq \sum_{|\alpha| \leq n + 1} C_\alpha \sup_{\Rn} |\calF \{ \partial^\alpha u \} (\xi)| \\
	& \leq \sum_{|\alpha| \leq n + 1} C_\alpha \nrm[L^1(\Rn)]{\partial^\alpha u}.
	\end{align*}
	We arrive at the conclusion.
\end{proof}

We also need the following transformation.
For a fixed non-degenerate, symmetric, real-valued square matrix $Q$, we have
\begin{equation} \label{eq:FrGauiQ-PM2021}
\calF \{ e^{\pm i\,\agl[Q \cdot, \cdot]/2} \}(\xi) = \frac {e^{\pm i \frac \pi 4 {\rm sgn}\, Q}} {|\det Q|^{1/2}} e^{\mp i \agl[Q^{-1} \xi,\xi]/2}.
\end{equation}
Here the non-degeneracy condition of $Q$ means $\det Q \neq 0$.
\begin{proof}
	We have
	\begin{align*}
	\calF \{ e^{\pm i\,|\cdot|^2/2} \}(\xi)
	& = (2\pi)^{-n/2} \int_{\Rn} e^{-ix\cdot \xi} e^{\pm i\,|x|^2/2} \dif x
	= (2\pi)^{-n/2} e^{\mp i|\xi|^2/2} \int_{\Rn} e^{\pm i|x \mp \xi|^2/2} \dif x \\
	& = \pi^{-n/2} e^{\mp i|\xi|^2/2} \int_{\Rn} e^{\pm i|x|^2} \dif x.
	\end{align*}
	By standard Cauchy's integral theorem we can have
	\[
	\int_{-\infty}^{+\infty} e^{\pm ix^2} \dif x = \sqrt{\pi} e^{\pm i \pi/4},
	\]
	so we can continue,
	\begin{equation} \label{eq:fixxi-PM2021}
	\calF \{ e^{\pm i\,|\cdot|^2/2} \}(\xi)
	= \pi^{-n/2} e^{\mp i|\xi|^2/2} (\sqrt \pi e^{\pm i \frac \pi 4})^n
	= e^{\pm i \frac \pi 4 n} e^{\mp i|\xi|^2/2}.
	\end{equation}
	We left the computation from \eqref{eq:fixxi-PM2021} to \eqref{eq:FrGauiQ-PM2021} as an exercise.
\end{proof}

\subsection{A simple case}

We study the quadratic case in $\R^1$.

\begin{lem} \label{lem:sst-PM2021}
	Assume $a \in C_c^\infty(\R)$ with $a(0) \neq 0$.
	Fix an  arbitrary integer $N \in \mathbb N$.
	Then for the integral $I(\lambda)$:
	\[
	I(\lambda) = \int_{\R} e^{i\lambda x^2/2} a(x) \dif x,
	\]
	there holds
	\begin{equation} \label{eq:Isi-PM2021}
	I(\lambda)
	= \left( \frac{2\pi}{\lambda} \right)^{1/2} e^{i\frac{\pi}{4}} \sum_{0 \leq j \leq N} \frac {\lambda^{-j}} {j!} \left( \frac i 2 \right)^j a^{(2j)}(0) + \mathcal O(\lambda^{-\frac 1 2 -N-1} \sum_{j\leq 2N+4} \sup |a^{(j)}|),
	\end{equation}
	where $a^{(j)}$ signifies $\frac{\df^j a}{\df x^j}$.
\end{lem}

\begin{proof}
	By the Plancherel theorem (which claims $(f,g) = (\hat f, \hat g)$) we have	
	\begin{align*}
	I(\lambda)
	& = \int \overline{ e^{-i\lambda x^2/2} } a(x) \dif x
	= \int \overline{ (\lambda)^{-1/2} e^{-\frac {i\pi} 4} e^{\frac i {2\lambda} \xi^2} } \cdot \hat a(\xi) \dif \xi \nonumber \\
	& = \lambda^{-1/2} e^{\frac {i\pi} 4} \int e^{-\frac {i \xi^2} {2} h} \hat a(\xi) \dif \xi, \qquad h := \lambda^{-1}.
	\end{align*}
	Using \eqref{eq:TaylorInt-PM2021} to expand $I(\lambda)$ w.r.t.~$h$ at $h = 0$, we obtain
	\begin{align}
	\int e^{-\frac {i \xi^2} {2} h} \hat a(\xi) \dif \xi
	& = \sum_{j = 0}^N \int h^j (-\frac {i \xi^2} {2})^j/j! \, \hat a(\xi) \dif \xi \nonumber \\
	& \quad + |\frac{h^{N+1}}{\color{red}N!} {\color{red}\int_0^1} \int {\color{red}(1-t)^{N}} (-\frac {i \xi^2} {2})^{N+1} {\color{red}e^{-\frac {i \xi^2} {2} th}} \hat a(\xi) \dif \xi {\color{red}\dif t}| \nonumber \\
	& = \sum_{j = 0}^N h^j \frac {(\frac {i} {2})^j} {j!} \int \calF\{a^{(2j)}\}(\xi) \dif \xi + \mathcal O(h^{N+1} \int |\xi^{2N+2} \hat a(\xi)| \dif \xi) \nonumber \\
	& = \sum_{j = 0}^N \frac {(2\pi)^{1/2} (\frac {i} {2})^j h^j} {j!} a^{(2j)}(0) + \mathcal O(h^{N+1} \nrm[L^1]{\calF\{a^{(2N+2)}\}}) \label{eq:hjFa-PM2021} \\
	& = \sum_{j = 0}^N \frac {(2\pi)^{1/2} (\frac {i} {2})^j h^j} {j!} a^{(2j)}(0) + \mathcal O(h^{N+1} \sum_{j\leq 2N+4} \nrm[L^1(\Rn)]{a^{(j)}}).\nonumber
	\end{align}
	Note that we used \eqref{eq:Fuu-PM2021}.	Combining the computations and changing $h$ back to $\lambda^{-1}$, we arrive at 
	\begin{equation*}
	I(\lambda)
	= \left( \frac{2\pi}{\lambda} \right)^{1/2} e^{\frac {i\pi} 4} \sum_{j = 0}^N \frac{\lambda^{-j}} {j!} \left( \frac {i} {2} \right)^j a^{(2j)}(0) + \mathcal O(\lambda^{-\frac 1 2 -N-1} \sum_{j\leq 2N+4} \sup |a^{(j)}|).
	\end{equation*}
	The proof is complete.
\end{proof}

From this short proof, we see extract several main steps:
\begin{enumerate}
	\item use Plancherel theorem to turn $\lambda$ into $\lambda^{-1}$ in the exponent;
	
	\item expand the integral w.r.t.~$h := \lambda^{-1}$ at $h = 0$ with integral remainder;
	
	\item estimate the remainder with \eqref{eq:Fuu-PM2021}.
\end{enumerate}

\section{Lemma Statements}

\index{stationary phase lemma}
\begin{thm} \label{thm:2-PM2021}
	Let $n \in \N^+$ be the dimension.
	We consider the oscillatory integral $I(\lambda)$:
	\[
	I(\lambda) = \int_{\Rn} e^{i\lambda \agl[Q(x - x_0), x - x_0]/2} a(x;\lambda) \dif{x},
	\]
	where $\agl[Q x, y]$ signifies $(Qx)^T y$ as matrix multiplication.
	Fix two arbitrary integers $M$, $N \in \mathbb N$, and we assume
	\begin{itemize}
		
		\item $Q$ is a non-degenerate, symmetric, real-valued matrix;
		
		\item for each $\lambda$, $a(\cdot;\lambda) \in C^{n+2N+3}(\Rn;\mathbb C)$;
		
		\item for each $\lambda$, $a(\cdot;\lambda) \in C^{M}(\Rn;\mathbb C)$, and $\forall \alpha : |\alpha| \leq M$, there exists $\lambda$-dependent constants $C_{M,\alpha}(\lambda) > 0$ such that $\forall x \in \Rn$ there holds
		\[
		|\partial_x^\alpha a(x - x_0; \lambda)| < C_{M,\alpha}(\lambda) \agl[x]^{2M-n-1-|\alpha|}.
		\]
		
	\end{itemize}
	Then the integral $I(\lambda)$ is well-defined in the oscillatory integral sense, and as $\lambda \to +\infty$ we have
	\begin{subequations} \label{eq:IQ-PM2021}
		\begin{empheq}[box=\widefbox]{align}
		I(\lambda)
		& = \left( \frac{2\pi}{\lambda} \right)^{n/2} \frac{e^{i\frac{\pi}{4}\sgn Q}} {|\det Q|^{1/2}} \sum_{0 \leq j \leq N} \frac {\lambda^{-j}} {j!} \left( \frac {\agl[Q^{-1}D, D]} {2i} \right)^j \big( a(x; \lambda) \big) |_{x = x_0} \nonumber \\
		& \qquad + \mathcal O \big( \lambda^{-\frac{n}{2}-N-1} \times \sum_{|\alpha| \leq n+2N+3} \sup_{B(x_0,1)} |\partial^\alpha a(\cdot; \lambda)| \big) \tag{\ref{eq:IQ-PM2021}} \\
		& \qquad + \mathcal O \big( \lambda^{-M} \times \sum_{|\alpha| \leq M} \sup_{\Rn} \frac {|\partial^\alpha a(x - x_0; \lambda)|} {\agl[x]^{2M-n-1-|\alpha|}} \big), \nonumber
		\end{empheq}
	\end{subequations}
	where $B(x_0,1)$ stands for the ball centered at $x_0$ with radius 1, and $\sgn Q$ stands for the difference between the number of positive eigenvalues and the number of negative eigenvalues of the matrix $Q$.
\end{thm}

\begin{rem} \label{rem:1-PM2021}
	In contrast to many other versions of the stationary phase lemma, here we don't require $a$ to be compactly supported.
	Instead, some other boundedness conditions are required, which makes the oscillatory integral well-defined.
\end{rem}

\begin{rem} \label{rem:2-PM2021}
	Eq.~\eqref{eq:IQ-PM2021} is not an asymptotics w.r.t.~$j$, but is rather w.r.t.~$\lambda$. 
	To get enough terms w.r.t.~$j$, one could choose $M$ be large enough first, and then check if $a$ satisfies the requirements.
\end{rem}

\begin{rem} \label{rem:3-PM2021}
	The integers $M$ and $N$ in Theorem \ref{thm:2-PM2021} shall be chosen properly to serve for your own purposes.
	For example, if one cares more about the decaying behavior w.r.t.~$\lambda$, then the $M$ can be set to $\lceil n/2 \rceil +N+1$.
	However, if one is dealing with these functions $a$ which doesn't have good decaying behavior at the infinity, then one could set $M$ to be large enough such that $\agl[x]^{2M-n-1-|\alpha|}$ can dominate $\partial^\alpha a$, with the cost that we should demand more smoothness of $a(x)$.
\end{rem}

\begin{rem} \label{rem:4-PM2021}
	The unit ball $B(x_0,1)$ involved in the term $\sup_{B(0,1)} |\partial^\alpha a|$ can be changed to other bounded domain containing $x_0$.
	But one should be careful that when the domain used has a very small radius, the underlying coefficients of the $\mathcal O(\cdot)$ term will be very large accordingly.
\end{rem}

\begin{rem} \label{rem:5-PM2021}
	The function $a$ is allowed to be dependent on $\lambda$, hence the expression \eqref{eq:IQ-PM2021} is an asymptotic expansion only when $\partial_x^\alpha a(x;\lambda)$ doesn't increase significantly when $\lambda \to +\infty$.
\end{rem}


If chosen $M = n + 2N + 3$, Theorem \ref{thm:2-PM2021} will be simplified as follows.

\begin{prop} \label{prop:2-PM2021}
	Let $n \in \N^+$ be the dimension.
	We consider the oscillatory integral $I(\lambda)$:
	\[
	I(\lambda) = \int_{\Rn} e^{i\lambda \agl[Q(x - x_0), x - x_0]/2} a(x;\lambda) \dif{x},
	\]
	where $\agl[Q x, y]$ signifies $(Qx)^T y$ as matrix multiplication.
	Fix an integer $N \in \mathbb N$, and we assume
	\begin{itemize}
		
		\item $Q$ is a non-degenerate, symmetric, real-valued matrix;
		
		\item for each $\lambda$, $a(\cdot;\lambda) \in C^{n + 2N + 3}(\Rn;\mathbb C)$, and $\forall \alpha : |\alpha| \leq n + 2N + 3$, there exists $\lambda$-dependent constants $C_{N,n,\alpha}(\lambda) > 0$ such that $\forall x \in \Rn$ there holds
		\begin{equation} \label{eq:as2-PM2021}
		|\partial_x^\alpha a(x - x_0; \lambda)| < C_{N,n,\alpha}(\lambda) \agl[x]^{2N+2}.
		\end{equation}
	\end{itemize}
	Then the integral $I(\lambda)$ is well-defined in the oscillatory integral sense, and as $\lambda \to +\infty$ we have
	\begin{subequations} \label{eq:IQs-PM2021}
		\begin{empheq}[box=\widefbox]{align}
		I(\lambda)
		& = \left( \frac{2\pi}{\lambda} \right)^{n/2} \frac{e^{i\frac{\pi}{4}\sgn Q}} {|\det Q|^{1/2}} \sum_{0 \leq j \leq N} \frac {\lambda^{-j}} {j!} \left( \frac {\agl[Q^{-1}D, D]} {2i} \right)^j \big( a(x; \lambda) \big) |_{x = x_0} \nonumber \\
		& \qquad + \mathcal O \big( \lambda^{-\frac{n}{2}-N-1} \sum_{|\alpha| \leq n+2N+3} \sup_{x \in \Rn} \frac {|\partial^\alpha a(x - x_0; \lambda)|} {\agl[x]^{n+4N+5-|\alpha|}} \big). \tag{\ref{eq:IQs-PM2021}}
		\end{empheq}
	\end{subequations}
\end{prop}

Proposition \ref{prop:2-PM2021} can be extend to a more general case where the phase function is not quadratic.

\begin{thm}[Stationary phase lemma \cite{eskin2011lectures}] \label{thm:1-PM2021}
	We consider the oscillatory integral $I(\lambda)$:
	\[
	I(\lambda) = \int_{\Rn} e^{i\lambda \varphi(x)} a(x) \dif{x}.
	\]
	For an arbitrary integer $N \in \mathbb N$, assume
	\begin{itemize}
		\item $a \in C^{n+2N+3}(\Rn;\mathbb C)$ with $\sum_{|\alpha| \leq n + 2N+3} \sup_{\Rn} |\partial^\alpha a| < +\infty$;
		
		\item $\varphi \in C^{n+2N+6}(\Rn;\R)$ with $\sum_{|\alpha| \leq n+2N+6} \sup_{\Rn} |\partial^\alpha \varphi| < +\infty$;
		
		\item $x_0$ is the only critical point of $\varphi(x)$ on $\supp a(x)$, i.e., $\varphi(x_0) = \nabla \varphi(x_0) = 0$, $\varphi_x(x) \neq 0$ for $x \neq x_0$;
		
		\item the Hessian $\varphi_{xx}(x_0) := [\frac{\partial^2\varphi}{\partial x_j \partial x_k}(x_0)]_{j,k=1}^n$ satisfies $\det \varphi_{xx}(x_0) \neq 0$.
	\end{itemize}
	Then the integral $I(\lambda)$ is well-defined in the oscillatory integral sense, and as $\lambda \to +\infty$ we have
	\begin{subequations} \label{eq:I-PM2021}
		\begin{empheq}[box=\widefbox]{align}
		I(\lambda)
		& = \left( \frac{2\pi}{\lambda} \right)^{n/2} \frac{e^{i\lambda\varphi(x_0) + i\frac{\pi}{4}\sgn \varphi_{xx}(x_0)}}{|\det \varphi_{xx}(x_0)|^{1/2}} \Big( a(x_0) + \sum_{j=1}^{N} a_j(x_0)\lambda^{-j} \Big) \nonumber \\
		& \quad + \mathcal O \big( \lambda^{-\frac{n}{2}-N-1} \times \sum_{|\alpha| \leq n+2N+3} \sup_{\Rn} |\partial^\alpha a| \times \sum_{|\alpha| \leq n+2N+6} \sup_{\Rn} |\partial^\alpha \varphi| \big), \tag{\ref{eq:I-PM2021}} 
		\end{empheq}
	\end{subequations}
	for some functions $a_j, ~1 \leq j \leq N$.
\end{thm}

\begin{rem}
	Proposition \ref{prop:2-PM2021} is a special case of Theorem \ref{thm:1-PM2021} where
	\[
	\varphi(x) = \agl[Q(x - x_0), x - x_0]/2,
	\]
	which guarantees $\varphi_{xx}(x_0) = Q$.
	Theorem \ref{thm:1-PM2021} is not a generalization of Theorem \ref{thm:2-PM2021} because unlike the quadratic phase function in Theorem \ref{thm:2-PM2021}, the phase function $\varphi$ in Theorem \ref{thm:1-PM2021} is not assumed to possess the property that $|\nabla \varphi(x)| \simeq |x|$ as $|x|$ is large.
	However if $\varphi$ is a homeomorphism of $\Rn$, it is possible to generalize Theorem \ref{thm:2-PM2021}.
\end{rem}

In one-dimensional case, explicit expressions for these $a_j(x_0)$ are given in \cite[(3.4.11)]{zw2012semi}, and the details are given in \cite[Second proof of Theorem 3.11]{zw2012semi}.
For explicit expressions for these $a_j(x_0)$ in higher dimension, readers may refer to \cite[Theorem 7.7.5]{horm2003IIV} for details.
In \cite[Lemma 19.3]{eskin2011lectures} there is also another routine to prove the stationary phase lemma.
\cite[Chapter 5]{dim1999spe} by Dimassi and Sj\"ostrand is also a good reference.
See also \cite[\S 2.3 \& \S 6.4]{won89asy}.

\section{Proofs of the results} \label{ch:pof-PM2021}

We first prove the quadratic case.

\begin{proof}[Proof of Theorem \ref{thm:2-PM2021}]
	
	We omit notationally the dependence of $a$ on $\lambda$ until related clarifications are needed.
	Without loss of generality we assume $x_0 = 0$, and $a(0) = 1$.
	For readers' convenient we rewrite the expression of $I$ here: $I(\lambda) = \int_{\Rn} e^{i\lambda \agl[Qx, x]/2} a(x) \dif{x}$.
	
	
	{\bf Step 1:} cutoff singularity of the phase function.
	According to the assumption on $Q$, we know there exists a decomposition $Q = P \Lambda P^T$ where $P$ is an orthogonal matrix and $\Lambda := (\alpha_j)_{j=1,\cdots,n}$ is a diagonal matrix.
	Make the change of variable $y = P^T x$, we can have
	\begin{align}
	I
	& = \int_\Rn a(Py) e^{i\lambda \sum_{j=1}^{n} \alpha_j y_j^2/2} \dif y \nonumber \\
	& = \int_\Rn e^{i\lambda \sum_{j=1}^{n} \alpha_j y_j^2/2} (1-\chi(y)) f(y) \dif y + \int_\Rn e^{i\lambda \sum_{j=1}^{n} \alpha_j y_j^2/2} \chi(y) f(y) \dif y \quad (\det P = 1) \nonumber \\
	& =: J_1 + J_2, \label{eq:IQJ12-PM2021}
	\end{align}
	where $f(y) := a(Py)$ and $\chi \in C_c^\infty(\Rn)$ is a cutoff function satisfying $0 \leq \chi \leq 1$ and $\chi \equiv 1$ in a neighborhood of the origin.
	We will see:
	\newline \hspace*{\fill}
	\underline{$J_1$ is rapidly decaying} and \underline{$J_2$ gives the desired asymptotics}.
	\hspace*{\fill}
	
	{\bf Step 2:} $J_1$ is rapidly decaying.
	Noting that neighborhoods of the origin is not included in the support of the integrand in $J_1$, we can estimate $J_1$ by using integration by parts (in the oscillatory integral sense).
	For any integer $M \in \N$ we have
	\begin{align}
	|J_1|
	& = |\int_\Rn \big( \frac {\sum_j \alpha_j^{-1} y_j \partial_j} {i\lambda |y|^2} \big)^M (e^{i\lambda \sum_{j=1}^{n}\alpha_j y_j^2/2} ) \big[ (1-\chi(y)) f(y) \big] \dif y| \nonumber \\
	& \lesssim \lambda^{-M} \int_\Rn |\big( \sum_j \partial_j \circ (y_j |y|^{-2}) \big)^M \big[ (1-\chi(y)) f(y) \big] | \dif y \nonumber \\
	& \lesssim \lambda^{-M} \int_{\supp (1-\chi)} \sum_{|\alpha| \leq M} C_{M; \alpha} |y|^{|\alpha|-2M} |\partial^\alpha ((1-\chi) f(y))| \dif y \label{eq:J1Ind-PM2021} \\
	& \lesssim \lambda^{-M} \big( C \sum_{|\alpha| \leq M} \sup_{\{0 < \chi < 1\}} |\partial^\alpha f| + \sum_{|\alpha| \leq M} C_{M; \alpha} \int_{\{\chi = 0\}} |y|^{|\alpha|-2M} |\partial^\alpha f(y)| \dif y \big) \label{eq:J1Ine-PM2021} \\
	& \lesssim \lambda^{-M} \sum_{|\alpha| \leq M} \big( \sup_{\{0 < \chi < 1\}} |\partial^\alpha a| + \int_{\{\chi = 0\}} |y|^{-n-1} \agl[y]^{|\alpha|-2M+n+1} |\partial^\alpha a(y)| \dif y \big) \nonumber \\
	& \lesssim \lambda^{-M} \sum_{|\alpha| \leq M} \big( \sup_{\{0 < \chi < 1\}} |\partial^\alpha a| + \sup_{\Rn} \frac {|\partial^\alpha a(y)|} {\agl[y]^{2M-n-1-|\alpha|} } \big) \nonumber \\
	& \lesssim \lambda^{-M} \sum_{|\alpha| \leq M} \sup_{\Rn} \frac {|\partial^\alpha a(y)|} {\agl[y]^{2M-n-1-|\alpha|}}. \label{eq:J1f-PM2021}
	\end{align}
	The inequality \eqref{eq:J1Ind-PM2021} is due to the fact that
	\[
	\big( \sum_{1 \leq j \leq n} \partial_j \circ (y_j |y|^{-2}) \big)^M \varphi = \sum_{|\alpha| \leq M} C_{M; \alpha} |y|^{|\alpha|-2M} \partial^\alpha \varphi
	\]
	which can be derived by induction and we omit the details.
	Inequality \eqref{eq:J1Ine-PM2021} is due to the fact that $\partial^\alpha ((1-\chi) f) = \partial^\alpha f$ in $\{\chi = 0\}$.

	{\bf Step 3:} $J_2$ and Plancherel theorem.
	We turn to $J_2$.
	Keep in mind that $f(y) = a(Py)$ and $\chi f$ is compactly support and in $C_c^{n+2N+3}(\Rn)$.
	Here we analyze $I$ by borrowing idea from \cite[First proof of Theorem 3.11]{zw2012semi}.
	First we use Plancherel theorem (which states $(f,g) = (\hat f, \hat g)$),
	\begin{align}
	J_2
	& = \int_\Rn \overline{ e^{-i\lambda \sum_{j=1}^{n} \alpha_j y_j^2/2} } \chi f(y) \dif y
	= \int_\Rn \overline{ \calF \{ e^{-i\lambda \sum_{j=1}^{n} \alpha_j (\cdot)^2/2} \}(\xi) } \cdot \widehat{\chi f}(\xi) \dif \xi \nonumber \\
	& = \int_\Rn \overline{ (\lambda)^{-n/2} \frac {e^{-\frac {i\pi} 4 {\rm sgn}\, Q}} {|\det Q|^{1/2}} e^{\frac i {2\lambda} \alpha_j^{-1} \xi_j^2} } \cdot\widehat{\chi f}(\xi) \dif \xi \qquad (\text{by~} \eqref{eq:FrGauiQ-PM2021})\nonumber \\
	& =: \big( \frac {2\pi} \lambda \big)^{n/2} \frac {e^{\frac {i\pi} 4 {\rm sgn}\, Q}} {|\det Q|^{1/2}} J(1/\lambda, 1/\lambda, \chi f), \label{eq:I1J-PM2021}
	\end{align}
	where we ignored the summation notation over $j$ and the function $J$ is defined by
	\begin{equation} \label{eq:Jh-PM2021}
	J(h_1, h_2, \chi f) := (2\pi)^{-n/2} \int_\Rn e^{\frac {\xi_j^2 h_1} {i2\alpha_j}} \cdot \widehat{\chi f}(\xi; 1/h_2) \dif \xi.
	\end{equation}
	Note that in \eqref{eq:Jh-PM2021} we put emphasize on the dependence of $f$ on $h_2$ (i.e.~dependence of $a$ on $\lambda$).
	The smoothness of $J$ w.r.t.~$h_1$ is guaranteed by the $L^1$ of derivatives of $f$, namely, we have the following claim whose justification will be clear in \eqref{eq:JhR-PM2021},
	\[
	\forall m \in \mathbb N,\, \max_{|\alpha| \leq n + 2m + 1} \nrm[L^1(\Rn)]{\partial^\alpha f} < +\infty \ \Rightarrow \ J(\cdot, h_2, f) \in C^m(\R).
	\]

	{\bf Step 4:} Taylor's expansion.
	We abbreviate $\partial_{h_1} J$ as $\partial_1 J$.
	Expand $J$ w.r.t.~$h_1$,
	\begin{align*}
	\partial_1^k J(0, h_2, \chi f)
	& = (2\pi)^{-n/2} \int_\Rn \partial_1^{k} (e^{\frac {\xi_j^2 h_1} {i2\alpha_j}}) |_{h_1 = 0} \cdot \widehat{\chi f}(\xi) \dif \xi \\
	& = (2\pi)^{-n/2} \int_\Rn (\sum_j \frac {\xi_j^2} {i2\alpha_j})^k \cdot \widehat{\chi f}(\xi) \dif \xi \\
	& = (2\pi)^{-n/2} \int_\Rn \widehat{T^k \chi f}(\xi) \dif \xi
	= T^k f(0)
	= T^k \big( a(Py) \big) |_{y = 0} \\
	& = (\frac i 2 P^{lj} P^{kj} \alpha_j^{-1} \partial_{kl})^k a(0)
	= A^k a(0; 1/h_2),
	\end{align*}
	where $T = \frac i 2 \sum_j \frac {\partial_j^2} {\alpha_j}$ and \underline{$A = \frac i 2 (Q^{-1})^{jl} \partial_{jl} = \frac 1 {2i} \agl[Q^{-1}D, D]$} (recall that $D = \frac 1 i \nabla$ is vertical).
	We expand $J$ using Taylor series (i.e.~\eqref{eq:TaylorInt-PM2021}),
	\begin{align}
	J(h, h_2, \chi f)
	& = \sum_{k\leq N} \frac{h^k}{k!} \partial_h^{k}J(0, h_2, \chi f) + \frac{h^{N+1}} {N!} \int_0^1 (1-t)^{N} \cdot \partial_1^{N+1}J(th, h_2, \chi f) \dif{t} \nonumber \\
	& = \sum_{0 \leq k \leq N} \frac{(hA)^k}{k!} a(0; h_2) + \frac{h^{N+1}} {N!} \int_0^1 (1-t)^{N} \cdot \partial_1^{N+1}J(th, h_2, \chi f) \dif{t}. \label{eq:JhExp-PM2021}
	\end{align}

	{\bf Step 5:} The remainder term.
	By invoking \eqref{eq:Jh-PM2021}, the remainder term in \eqref{eq:JhExp-PM2021} can be estimated as
	\begin{align*}
	& \ |\frac{h^{N+1}}{N!} \int_0^1 (1-t)^{N} \cdot \partial_1^{N+1}J(th, h_2, \chi f) \dif{t}| \\
	\leq & C_N h^{N+1} \int_\Rn |(\frac {-i} {4\alpha_j} \xi_j^2)^{N+1} \cdot \widehat{\chi f}(\xi; 1/h_2)| \dif \xi \\
	\leq & \ C_N h^{N+1} \sum_{|\beta| \leq 2N+2} C_\beta \nrm[L^1(\Rn)]{(\partial_x^\beta (\chi f(\cdot; 1/h_2)))^\wedge}.
	\end{align*}
	By using \eqref{eq:Jh-PM2021} and \eqref{eq:Fuu-PM2021}, we can continue
	\begin{align}
	& \ |\frac{h^{N+1}}{N!} \int_0^1 (1-t)^{N} \cdot \partial_1^{N+1}J(th, h_2, \chi f) \dif{t}| \nonumber \\
	\leq & \ C_N h^{N+1} \sum_{\substack{|\alpha| \leq n+1 \\ |\beta| \leq 2N + 2}} \nrm[L^1(\Rn)]{\partial_x^{\alpha + \beta} (\chi f(\cdot; 1/h_2))} \\
	\leq & \ C_N h^{N+1} \sum_{|\alpha| \leq n+ 2N + 3} \nrm[L^1(\Rn)]{\partial_x^{\alpha} (\chi f(\cdot; 1/h_2))} \nonumber \\
	\leq & \ C_N h^{N+1} \sum_{|\alpha| \leq n+2N+3} \sup_{\supp \chi} |\partial^\alpha a(\cdot; 1/h_2)|. \label{eq:JhR-PM2021}
	\end{align}
	Letting $h = h_2 = 1/\lambda$ and combining \eqref{eq:I1J-PM2021}, \eqref{eq:Jh-PM2021}, \eqref{eq:JhExp-PM2021} and \eqref{eq:JhR-PM2021}, we obtain
	\begin{align}
	J_2
	& = \left( \frac {2\pi} \lambda \right)^{n/2} \frac {e^{\frac {i\pi} 4 {\rm sgn}\, Q}} {|\det Q|^{1/2}} \sum_{j \leq N} \frac {\lambda^{-j}} {j!} \left( \frac {\agl[Q^{-1}D, D]} {2i} \right)^j a(0; \lambda) \nonumber \\
	& \quad + C_N \lambda^{-\frac n 2-N-1} \sum_{|\alpha| \leq n+2N+3} \sup_{\supp \chi} |\partial^\alpha a(\cdot; \lambda)|. \label{eq:J2t-PM2021}
	\end{align}

	Combining \eqref{eq:J2t-PM2021} with \eqref{eq:IQJ12-PM2021}, \eqref{eq:J1f-PM2021}, we have
	\begin{align*}
	I
	& = \left( \frac {2\pi} \lambda \right)^{n/2} \frac {e^{\frac {i\pi} 4 {\rm sgn}\, Q}} {|\det Q|^{1/2}} \sum_{j \leq N} \frac {\lambda^{-j}} {j!} \left( \frac {\agl[Q^{-1}D, D]} {2i} \right)^j a(0; \lambda) \\
	& \qquad + \mathcal O \big( \lambda^{-\frac n 2-N-1} \sum_{|\alpha| \leq n+2N+3} \sup_{\supp \chi} |\partial^\alpha a(\cdot; \lambda)| \big) + \mathcal O \big( \lambda^{-M} \sum_{|\alpha| \leq M} \sup_{\Rn} \frac {|\partial^\alpha a(y; \lambda)|} {\agl[y]^{2M-n-1-|\alpha|}} \big),
	\end{align*}
	which is \eqref{eq:IQ-PM2021}.
	The proof is complete.
\end{proof}

\begin{proof}[Proof of Proposition \ref{prop:2-PM2021}]
	The statement of Proposition \ref{prop:2-PM2021} is almost the same as Theorem \eqref{thm:2-PM2021}, except that $M$ is set to be $n + 2N + 3$.
	Hence, we set $M := n + 2N + 3$, then we have $-M \leq -n/2 - N - 1$, 
	so
	\begin{align*}
	\lambda^{-M} \sum_{|\alpha| \leq M} \sup_{\Rn} \frac {|\partial^\alpha a(y; \lambda)|} {\agl[y]^{2M-n-1-|\alpha|}}
	& = \lambda^{-M} \sum_{|\alpha| \leq n + 2N + 3} \sup_{\Rn} \frac {|\partial^\alpha a(y; \lambda)|} {\agl[y]^{2(n + 2N + 3)-n-1-|\alpha|}} \\
	& \leq \lambda^{-n/2 - N - 1} \sum_{|\alpha| \leq n + 2N + 3} \sup_{\Rn} \frac {|\partial^\alpha a(y; \lambda)|} {\agl[y]^{n + 4N + 5-|\alpha|}}.
	\end{align*}
	
	Also, for the first remainder term in \eqref{eq:IQ-PM2021} we have
	\begin{align*}
	\sum_{|\alpha| \leq n+2N+3} \sup_{B(x_0,1)} |\partial^\alpha a(\cdot; \lambda)|
	& \leq C_{x_0,n,N} \sum_{|\alpha| \leq n+2N+3} \sup_{y \in B(x_0,1)} \frac {|\partial^\alpha a(y; \lambda)|} {\agl[y]^{n + 4N + 5-|\alpha|}} \\
	& \leq C_{x_0,n,N} \sum_{|\alpha| \leq n+2N+3} \sup_{\Rn} \frac {|\partial^\alpha a(y; \lambda)|} {\agl[y]^{n + 4N + 5-|\alpha|}}.
	\end{align*}
	Combining these with \eqref{eq:IQ-PM2021}, we arrive at \eqref{eq:IQs-PM2021}.
	The proof is complete.
\end{proof}

Based on Theorem \ref{thm:2-PM2021}, now we prove the more general case.

\begin{proof}[Proof of Theorem \ref{thm:1-PM2021}]
	Without loss of generality we assume $x_0 = 0$, $\varphi(0) = 0$ and $a(0) = 1$.
	Hence by Taylor's expansion \eqref{eq:TaylorInt-PM2021} we have 
	\[
	\varphi(x) 
	= \sum_{j,k \leq n} x_j x_k \int_0^1 (1-t)\, \partial_{jk} \varphi(tx) \dif t
	= x^T \cdot\int_0^1 (1-t) \varphi_{xx}(tx) \dif t \cdot x.
	\]
	Note that $|\varphi_{xx}(0)| \neq 0$ and $|\varphi_{xx}(x)|$ is continuous on $x$ (\underline{$\varphi \in C^2$}), thus there exists a positive constant \boxed{r} such that $|\varphi_{xx}(x)| > |\varphi_{xx}(0)|/2 > 0$ for all $x \in B(0,r)$.
	Fix a cutoff function $\chi \in C_c^\infty(\Rn)$ such that $\supp \chi \subset B(0,r)$ and $\chi \equiv 1$ in $B(0,r/2)$.
	Hence:
	\begin{itemize}
	\item  on $B(0,r)$, matrix $\varphi_{xx}$ is non-degenerate;
	
	\item on $\supp a \backslash B(0,r)$, $|\nabla \varphi(x)|$ is uniformly bounded away from $0$.
	\end{itemize} 

	{\bf Step 1:} cutoff singularity of the phase function.
	We divide $I$ into two parts
	\begin{equation} \label{eq:ISplit-PM2021}
	I(\lambda)
	= \int_{\Rn} (1 - \chi(x)) a(x) e^{i\lambda \varphi(x)} \dif x + \int_{\Rn} \chi(x) a(x) e^{i\lambda \varphi(x)} \dif x
	:= I_1 + I_2,
	\end{equation}
	and we will show that $I_1$ is rapidly decreasing w.r.t.~$\lambda$ while $I_2$ can be analyzed by using Theorem \ref{thm:2-PM2021}.

	{\bf Step 2:} $I_1$ is rapidly decaying.
	For $I_1$, denote
	\(
	L = \sum_{j=1}^n \frac {\varphi_{x_j}} {|\nabla \varphi|^2} \partial x_j,
	\)
	where $\varphi_{x_j}$ is short for $\partial_{x_j} \varphi$. 
	Then $\frac 1 {i\lambda} L e^{i\lambda \varphi} = e^{i\lambda\varphi} \text{ and } {}^t L f = \sum_{j=1}^n \partial_{x_j} \big( \frac {\varphi_{x_j} f} {|\nabla \varphi|^2} \big)$.
	For any $N \in \mathbb{N}^+$, $I_1$ can be easily estimated as follows (which requires 
	$a \in C^{n+N+1}(\Rn)$ and $\varphi \in C^{n+N+2}(\Rn)$)
	\begin{align}
	I_1
	& = \int_\Rn (1-\chi) a \cdot ((i\lambda)^{-n-N-1} L^{N+1} e^{i\lambda \varphi(x)}) \dif{x} \nonumber \\
	& = (i\lambda)^{-n-N-1} \int_\Rn ({}^t L)^{n+N+1} ((1-\chi)a) \cdot e^{i\lambda \varphi(x)} \dif{x} \nonumber\\
	& = \mathcal{O} (\lambda^{-n-N-1} \sum_{|\alpha| \leq n+N+1} \nrm[L^1(\Rn)]{\partial^\alpha a}), \quad \lambda \to \infty \label{eq:I2-PM2021}.
	\end{align}
	As mentioned before, due to the presence of $1-\chi$, the denominator $|\nabla \varphi|^2$ in $L$ keeps a positive distance away from 0, guaranteeing that $({}^t L)^N ((1-\chi)a)$ is bounded and compactly supported.

	{\bf Step 3:} Turn $I_2$ into quadratic phase form (e.g.~``$J_2$'').
	Now we turn to $I_2$.
	Because $\varphi \in C^2(\Rn)$, $\varphi_{xx}(x)$ is symmetric and thus there exist orthogonal matrix $P(x)$ and diagonal matrix $\Lambda(x) = (\alpha_j(x))_{j=1,\cdots,n}$ such that 
	\begin{equation*} 
	2 \int_0^1 (1-t) \varphi_{xx}(tx) \dif t = P(x) \Lambda(x) P^T(x).
	\end{equation*}
	Especially we have
	\( 
	P(0) \Lambda(0) P^T(0) = \varphi_{xx}(0).
	\)
	Denote $\alpha_j = \alpha_j(0)$ and $n \times n$ diagonal matrix $\Lambda := (\alpha_j)_{j=1,\cdots,n}$ for short. Thus
	\begin{equation*} 
	\Lambda(x) = (\sqrt{\frac {\alpha_j(x)} {\alpha_j}})_{j=1,\cdots,n} \cdot \Lambda \cdot (\sqrt{\frac {\alpha_j(x)} {\alpha_j}})_{j=1,\cdots,n}.
	\end{equation*}
	Note that we can choose the support of $\chi$ to be small enough such that, on $\supp \chi$, $\alpha_j(x)$ doesn't change sign, so $\alpha_j(x) / \alpha_j$ will always be positive on $\supp \chi$.
	This grants the use of the square root operation.
	
	Make the change of variable:
	\begin{equation} \label{eq:xToy-PM2021}
	y = \Phi(x) := \left( \sqrt{\frac {\alpha_j(x)} {\alpha_j}} \right)_{j=1,\cdots,n} \cdot P^T(x) \cdot x.
	\end{equation}
	Note that
	\begin{equation} \label{eq:sm1-PM2021}
	\varphi \in C^{n+2N+6} \Rightarrow \Phi \in C^{n+2N+4}.
	\end{equation}
	We have
	\begin{align*}
	\varphi(x)
	& = \frac 1 2 x^T \cdot \big[ 2 \int_0^1 (1-t) \varphi_{xx}(tx) \dif t \big] \cdot x
	= \frac 1 2 x^T \cdot \big[ P(x) \Lambda(x) P^T(x) \big] \cdot x \\
	& = \frac 1 2 [P^T(x) \cdot x]^T \cdot (\sqrt{\frac {\alpha_j(x)} {\alpha_j}})_{j=1,\cdots,n} \cdot \Lambda \cdot (\sqrt{\frac {\alpha_j(x)} {\alpha_j}})_{j=1,\cdots,n} \cdot [P^T(x) \cdot x] \\
	& = \frac 1 2 \big[ \big( \sqrt{\frac {\alpha_j(x)} {\alpha_j}} \big)_{j=1,\cdots,n} \cdot P^T(x) \cdot x \big]^T \cdot \Lambda \cdot \big[ \big( \sqrt{\frac {\alpha_j(x)} {\alpha_j}} \big)_{j=1,\cdots,n} \cdot P^T(x) \cdot x \big] \\
	& = \frac 1 2 \agl[\Lambda y, y].
	\end{align*}
	
	We have $\Phi(0) = 0$. It is easy to check that $\frac {\partial \Phi} {\partial x} (0) = P^T(0).$
	From \eqref{eq:xToy-PM2021} it is clear that there exists inverse of $\Phi$, i.e.~$\phi = \Phi^{-1}$
	Note that $x = \phi(\Phi(x))$ and
	\begin{equation} \label{eq:sm2-PM2021}
	\Phi \in C^{n+2N+4} \Rightarrow \phi \in C^{n+2N+4}.
	\end{equation}
	We have
	\begin{align*}
	I_2
	& = \int_\Rn \chi(\phi(y)) a(\phi(y)) \cdot e^{i\lambda \agl[\Lambda y, y]/2} \dif{\phi(y)} \nonumber\\
	& = \int_\Rn \chi(\phi(y)) a(\phi(y)) |\det \nabla_y \phi (y)| \cdot e^{i\lambda \agl[\Lambda y, y]/2} \dif{y} \nonumber\\
	& = \int_\Rn f(y) e^{i\lambda \agl[\Lambda y, y]/2} \dif{y} 
	\end{align*}
	where
	\begin{equation*} 
	f(y) = \chi(\phi(y)) \cdot a(\phi(y)) \cdot |\det \nabla_y \phi (y)|,
	\end{equation*}
	Note that
	\begin{equation} \label{eq:sm3-PM2021}
	\phi \in C^{n+2N+4},\, a \in C^{n+2N+2} \Rightarrow f \in C^{n+2N+2}.
	\end{equation}
	Now we can conclude from \eqref{eq:sm3-PM2021}, \eqref{eq:sm3-PM2021} and \eqref{eq:sm3-PM2021} that
	\begin{equation} \label{eq:sm4-PM2021}
	\varphi \in C^{n+2N+6},\, a \in C^{n+2N+3} \Rightarrow f \in C^{n+2N+3}.
	\end{equation}
	
	{\bf Step 4:} Apply Theorem \ref{thm:2-PM2021}.
	By using Theorem \ref{thm:2-PM2021}, we can obtain
	\begin{align}
	I_1(\lambda)
	& = \left( \frac{2\pi}{\lambda} \right)^{n/2} \frac{e^{i\frac{\pi}{4}\sgn \Lambda}} {|\det \Lambda|^{1/2}} \sum_{0 \leq j \leq N} \frac {\lambda^{-j}} {j!} \left( \frac {\agl[\Lambda^{-1}D, D]} {2i} \right)^j f(0) \nonumber \\
	& \qquad + \mathcal O(\lambda^{-\frac{n}{2}-N-1} \times \sum_{|\alpha| \leq n+2N+3} \sup_{\Rn} |\partial^\alpha f|) \nonumber \\
	& = \left( \frac{2\pi}{\lambda} \right)^{n/2} \frac{e^{i\frac{\pi}{4}\sgn \Lambda}} {|\det \Lambda|^{1/2}} \sum_{0 \leq j \leq N} \frac {\lambda^{-j}} {j!} \left( \frac {\agl[\Lambda^{-1}D, D]} {2i} \right)^j f(0) \nonumber \\
	& \qquad + \mathcal O(\lambda^{-\frac{n}{2}-N-1} \times \sum_{|\alpha| \leq n+2N+3} \sup_{\Rn} |\partial^\alpha a| \times \sum_{|\alpha| \leq n+2N+6} \sup_{\Rn} |\partial^\alpha \varphi|). \label{eq:I1Tem-PM2021}
	\end{align}
	It can be checked that $\sgn \Lambda = \sgn \varphi_{xx}(0)$ and $\det \Lambda = \det \varphi_{xx}(0)$.

	{\bf Step 5:} The leading term.
	We are now almost arrive at \eqref{eq:I-PM2021} except for the explicit computation of the leading term in \eqref{eq:I-PM2021} and \eqref{eq:I1Tem-PM2021}.
	From equality $x = \phi(\Phi(x))$ we know $I = \nabla_y \phi(\Phi(x)) \cdot \nabla_x \Phi(x)$.
	Formula \eqref{eq:xToy-PM2021} implies $\Phi(0) = 0$ and $\nabla_x \Phi(0) = P^T(0)$,
	hence $\det \nabla_y \phi(0) = \det \nabla_y \phi(\Phi(0)) = \big( \det \nabla_x \Phi(0) \big)^{-1} = \big( \det P^T(0) \big)^{-1} = 1$.
	Therefore,
	\begin{equation} \label{eq:fToa-PM2021}
	f(0) = \chi(\phi(0)) \cdot a(\phi(0)) \cdot |\det \nabla_y \phi (0)|
	= \chi(0) \cdot a(0) = a(0).
	\end{equation}
	Combining \eqref{eq:ISplit-PM2021}, \eqref{eq:I2-PM2021}, \eqref{eq:I1Tem-PM2021} and \eqref{eq:fToa-PM2021}, we arrive at the conclusion.
\end{proof}

\section*{Exercise}

\begin{ex}
	Use \eqref{eq:fixxi-PM2021} to derive \eqref{eq:FrGauiQ-PM2021}.
\end{ex}

\begin{ex}
	Show details about how to derive \eqref{eq:hjFa-PM2021} from $\int e^{-\frac {i \xi^2} {2} h} \hat a(\xi) \dif \xi$.
\end{ex}

\begin{ex}
	(optional) In \eqref{eq:Jh-PM2021}, if we instead set
	\[
	J(h, \chi f) := (2\pi)^{-n/2} \int_\Rn e^{\frac {\xi_j^2 h} {i2\alpha_j}} \cdot \widehat{\chi f}(\xi; 1/h) \dif \xi,
	\]
	and later on expand $J$ w.r.t.~$h$ at $h=0$,
	will the computations following \eqref{eq:Jh-PM2021} still give the desired result?
	Explain the reason briefly.
\end{ex}

\begin{ex} \label{ex:IYe1-PM2021}
	Assume $a \in C_c^\infty(\R^{2n})$ and denote a Lebesgue integral
	\[
	I(y,\eta; \lambda) := (2\pi)^{-n} \int_{\R^{2n}} e^{i\lambda x \cdot \xi} a(x+y, \xi + \eta) \dif x \dif \xi.
	\]
	\begin{enumerate}
		\item fix $y$ and $\eta$, and use Proposition \ref{prop:2-PM2021} to find the asymptotic expansion of $I$ w.r.t.~$\lambda$ as $\lambda \to +\infty$;
		
		\item write down the first $1+n$ terms (the leading term $+$ the first order terms) of the asymptotic expansion.
	\end{enumerate}
	Hint: $x \cdot \xi = \frac 1 2 \agl[Q(x,\xi), (x,\xi)]$ with $Q = \begin{pmatrix}
	0 & I_{n \times n} \\ I_{n \times n} & 0
	\end{pmatrix}$,
	where $(x,\xi)$ is treated as a vertical vector.
\end{ex}

\begin{ex} \label{ex:IYe2-PM2021}
	Assume symbol $a \in S^m(\Rn \times \Rn)$ and denote an oscillatory integral
	\[
	I(y,\eta) := (2\pi)^{-n} \int_{\R^{2n}} e^{ix \cdot \xi} a(x+y, \xi + \eta) \dif x \dif \xi.
	\]
	\begin{enumerate}
		\item is $I$ well-defined? If it is, should the cutoff function $\chi$ (cf.~\eqref{eq:Iec-PM2021}) be chosen to cutoff $\xi$ alone using $\chi(\epsilon \xi)$, or cutoff $x$ alone using $\chi(\epsilon x)$, or cutoff both $x$ and $\xi$ together using $\chi(\epsilon x,\epsilon \xi)$?
		
		
		\item use Proposition \ref{prop:2-PM2021} to find the asymptotic expansion of $I$ w.r.t.~$\agl[\eta]$ as $|\eta| \to +\infty$;
		
		\item write down the first $1+n$ terms (the leading term $+$ the first order terms) of the asymptotic expansion.
		
		\item compare with the result in Exercise \ref{ex:IYe1-PM2021}, and revise Remark \ref{rem:5-PM2021}.
	\end{enumerate}
	Hint: Perform the change of variable $\xi \to \agl[\eta] \xi$.
\end{ex}

\chapter{Symbolic calculus of \texorpdfstring{$\Psi$DOs}{PsiDOs}} \label{ch:SCP-PM2021}

In this chapter we show certain symbolic calculus of $\Psi$DOs.
We need some preparations.

\begin{lem} \label{lem:abm-PM2021}
	Assume $a$, $b \in \R$ such that $|a| \geq 1$ and $|b| \geq 1$, then for every $m \in \R$ there exists a constant independent of $a$, $b$ such that
	\[
	\boxed{\agl[ab]^m \leq C_m \agl[a]^m \agl[b]^m, \quad |a| \geq 1, \ |b| \geq 1.}
	\]
\end{lem}

\begin{proof}
	When $m \geq 0$, we have
	\[
	\agl[ab]^m
	\simeq (1+|ab|)^m
	\leq (1+|a|)^m (1+|b|)^m
	\leq \agl[a]^m \agl[b]^m.
	\]
	When $m < 0$, because $|a|$, $|b| \geq 1$, we have
	\begin{align*}
	\agl[ab]^m
	& \simeq \frac 1 {(1+|ab|)^{|m|}}
	< \frac 1 {|ab|^{|m|}}
	= \agl[a]^m \agl[b]^m (\frac {\agl[a]} {|a|} \frac {\agl[b]} {|b|})^{|m|} \\
	& \lesssim \agl[a]^m \agl[b]^m.
	\end{align*}
	We proved the result.
\end{proof}

\begin{lem}[Peetre's inequality] \label{lem:Peetre-PM2021}
For $\forall a, b \in \Rn \text{~and~} \forall m \in \R$, there exists a constant $C_m$ independent of $a$ and $b$ such that
\begin{equation*}
	\boxed{\agl[a \pm b]^m \leq C_m \agl[a]^m \agl[b]^{|m|}.}
\end{equation*}
\end{lem}

\begin{proof}
	For any $a$, $b \in \Rn$, we have
	\begin{equation*}
	1 + |a-b| \leq 1 + |a| + |b| \leq (1 + |a|) \cdot (1 + |b|).
	\end{equation*}
	Note that $\agl[a] \simeq 1 + |a|$, so we can conclude Lemma \ref{lem:Peetre-PM2021} for the case where $m \geq 0$.
	
	When $m < 0$, we use the fact:
	\begin{gather*}
	1 + |a| \leq 1 + |a-b| + |b| \leq (1 + |a-b|) \cdot (1 + |b|) \\
	\Rightarrow \ (1 + |a-b|) \geq (1 + |a|) \cdot (1 + |b|)^{-1}.
	\end{gather*}
	Now assume $m < 0$, we have
	\begin{equation*}
	(1 + |a-b|)^m \leq (1 + |a|)^m \cdot (1 + |b|)^{-m} = (1 + |a|)^m (1 + |b|)^{|m|}.
	\end{equation*}
	The proof is complete.
\end{proof}

\section{Composition of \texorpdfstring{$\Psi$DOs}{PsiDOs}}
\label{sec:Com-PM2021}

Assume $a \in S^{m_1}$ and $b \in S^{m_2}$.
For notational convenience we denote $T = T_a \circ T_b$, thus for any $\varphi \in \scrS$, we have
\begin{align}
T \varphi
& = (2\pi)^{-n} \int e^{i(x-y) \cdot \xi} a(x,\xi) T_b \varphi(y) \dif y \dif \xi \nonumber \\
& = (2\pi)^{-n} \int e^{i(x-z) \cdot \eta} \big( (2\pi)^{-n} \int e^{i(x-y) \cdot (\xi - \eta)} a(x,\xi) b(y,\eta) \dif y \dif \xi \big) \varphi(z) \dif z \dif \eta \nonumber \\
& = (2\pi)^{-n} \int e^{i(x-z) \cdot \eta} \big( (2\pi)^{-n} \int e^{-iy \cdot \xi} a(x,\eta+\xi) b(x+y,\eta) \dif y \dif \xi \big) \varphi(z) \dif z \dif \eta \nonumber \\
& = (2\pi)^{-n} \int e^{i(x-z) \cdot \eta} c(x,\eta) \varphi(z) \dif z \dif \eta, \label{eq:cob-PM2021}
\end{align}
where $c$ is defined as the oscillatory integral
\begin{equation} \label{eq:cet-PM2021}
c(x,\eta) := (2\pi)^{-n} \int e^{-iy \cdot \xi} a(x,\eta+\xi) b(x+y,\eta) \dif y \dif \xi.
\end{equation}
If we could show $c \in S^m$ for certain $m$, then it implies the composition of $\Psi$DOs is still a $\Psi$DO.
We use the stationary phase lemma under oscillatory integrals developed in \S \ref{ch:SPL-PM2021} to show this expectation.

To show $c \in S^m$, the task boils down to show the asymptotics of $c$ and its derivatives w.r.t.~$|\eta|$, thus we set $\lambda := \agl[\eta]$, so
\[
c(x,\eta) = (2\pi)^{-n} \lambda^n \int e^{-i\lambda y \cdot \xi} a(x,\lambda(\tilde \eta + \xi)) b(x+y,\eta) \dif y \dif \xi, \quad \text{where} \quad \tilde \eta := \eta/\agl[\eta].
\]
To make better correspondence with the notations in \S \ref{ch:SPL-PM2021}, we set
\[
c_{x,\eta}(y, \xi) := a(x,\lambda(\tilde \eta + \xi)) b(x+y,\eta),
\]
thus
\begin{equation} \label{eq:cc-PM2021}
c(x,\eta) = (2\pi)^{-n} \lambda^n \int_{\R^{2n}} e^{i\lambda \agl[Q(y, \xi), (y, \xi)]/2} c_{x,\eta}(y, \xi) \dif {(y,\xi)},
\end{equation}
where $(y, \xi)$ is treated as a $2n$-dim vertical vector and
\[
Q =
\begin{pmatrix}
0 & -I \\
-I & 0
\end{pmatrix}
\quad (\Rightarrow \ Q^{-1} = Q, \ \sgn Q = 0, \text{~and~} \det Q = \pm 1).
\]
In $c_{x,\eta}(y, \xi)$, we regard $(x,\eta)$ as irrelevant parameters make the following correspondence:
\begin{center}
	\begin{tabular}{|c|c|c|c|c|c|}
	\hline
	& function & variable & fixed point & in total & dimension \\ \hline
	In Prop.~\ref{prop:2-PM2021} & $a$ & $x$ & $x_0$ & $a(x - x_0)$ & $n$ \\ \hline
	at here & $c_{x,\eta}$ & $(y, \xi)$ & $(y_0, \xi_0) = 0 $ & $c_{x,\eta}(y, \xi)$ & $2n$ \\ \hline
	\end{tabular}
\end{center}
To use Proposition \ref{prop:2-PM2021}, the only thing left to check is \eqref{eq:as2-PM2021}, namely, to check
\begin{equation} \label{eq:cas2-PM2021}
\forall \alpha, \beta : |\alpha| + |\beta| \leq 2n + 2N + 3, \quad |\partial_y^\alpha \partial_\eta^\beta \big( c_{x,\eta}(y, \xi) \big)|
\lesssim C_{N,n,\alpha, \beta}(\lambda) \agl[(y,\xi)]^{2N+2}.
\end{equation}

For $|\xi| \geq 2$, we have
\begin{align*}
|\partial_y^\alpha \partial_\xi^\beta \big( c_{x,\eta}(y, \xi) \big)|
& = |\partial_y^\alpha \partial_\xi^\beta \big[ a(x,\lambda(\tilde \eta + \xi)) b(x+y,\eta) \big]| \\
& \leq C_{\alpha, \beta} \lambda^{|\beta|} |\partial_\xi^{\beta} a(x,\lambda(\tilde \eta + \xi))| \cdot |\partial_x^\alpha b(x+y,\eta)| \\
& \leq C_{\alpha, \beta} \lambda^{|\beta|} \agl[\lambda(\tilde \eta + \xi)]^{m_1 - |\beta|} \agl[\eta]^{m_2}.
\end{align*}
Because $|\tilde \eta| < 1$, when $|\xi| \geq 2$ we can have $|\tilde \eta + \xi| \geq 1$.
Recall that $\lambda \geq 1$.
Hence when $|\xi| \geq 2$, we can use Lemma \ref{lem:abm-PM2021} to continue the computation as follows,
\begin{align}
|\partial_y^\alpha \partial_\xi^\beta \big( c_{x,\eta}(y, \xi) \big)|
& \leq C_{\alpha, \beta} \lambda^{|\beta|} \agl[\lambda]^{m_1 - |\beta|} \agl[\tilde \eta + \xi]^{m_1 - |\beta|} \agl[\eta]^{m_2} \nonumber \\
& \leq C_{\alpha, \beta} \lambda^{m_1} \agl[\tilde \eta + \xi]^{m_1 - |\beta|} \lambda^{m_2} \quad (\lambda = \agl[\eta] \Rightarrow \lambda \simeq \agl[\lambda]) \nonumber \\
& \leq C_{\alpha, \beta} \lambda^{m_1 + m_2} \agl[\xi]^{m_1 - |\beta|} \agl[\tilde \eta]^{|m_1 - |\beta||} \quad (\text{Lemma \ref{lem:Peetre-PM2021}}) \nonumber \\
& \leq C_{\alpha, \beta} \lambda^{m_1 + m_2} \agl[\xi]^{m_1 - |\beta|}. \label{eq:cxe-PM2021}
\end{align}
We emphasize that \eqref{eq:cxe-PM2021} holds when $|\xi| \geq 2$, and the constant $C_{\alpha, \beta}$ is uniform for $x$, $y$, $\eta$.
Then, due to the continuity, \eqref{eq:cxe-PM2021} actually holds for all $\xi$.
Hence, the condition \eqref{eq:cas2-PM2021} is satisfied when $2N + 2 > m_1$, with $C_{N,n,\alpha, \beta}(\lambda) = C_{\alpha, \beta} \lambda^{m_1 + m_2}$, so we can use Proposition \ref{prop:2-PM2021} directly on \eqref{eq:cc-PM2021} to obtain
\begin{align}
c(x,\eta)
& = (2\pi)^{-n} \lambda^n \int_{\R^{2n}} e^{i\lambda \agl[Q(y, \xi), (y, \xi)]/2} c_{x,\eta}(y, \xi) \dif {(y,\xi)} \nonumber \\
& = (2\pi)^{-n} \lambda^n \times \left( \frac{2\pi}{\lambda} \right)^{n} \sum_{0 \leq j \leq N} \frac {\lambda^{-j}} {j!} \left( \frac {\agl[Q^{-1} D_{(y,\xi)}, D_{(y,\xi)}]} {2i} \right)^j c_{x,\eta}(0, 0) \nonumber \\
& \quad + \lambda^n \times \mathcal O \big( \lambda^{-n-N-1} \sum_{|\alpha| + |\beta| \leq 2n+2N+3} \sup_{(y,\xi) \in \R^{2n}} \frac {|\partial_y^\alpha \partial_\xi^\beta \big( c_{x,\eta}(y,\xi) \big)|} {\agl[(y,\xi)]^{2n+4N+5 - |\alpha| - |\beta|}} \big) \nonumber \\
& = \sum_{0 \leq j \leq N} \frac {\lambda^{-j}} {j!} (D_y \cdot \nabla_\xi)^j c_{x,\eta}(0, 0) + \mathcal O \big( \lambda^{-N-1+m_1 + m_2} \big) \nonumber \\
& = \sum_{|\alpha| \leq N} \frac {\lambda^{-|\alpha|}} {\alpha!} D_y^\alpha \partial_\xi^\alpha \big( c_{x,\eta}(y,\xi) \big) |_{(y,\xi) = (0, 0)} + \mathcal O \big( \lambda^{-N-1+m_1 + m_2} \big) \label{eq:cc3-PM2021} \\
& = \sum_{|\alpha| \leq N} \frac {\lambda^{-|\alpha|}} {\alpha!} \lambda^{|\alpha|} \partial_\eta^\alpha a(x,\lambda \tilde \eta) D_x^\alpha b(x,\eta) + \mathcal O \big( \lambda^{-N-1+m_1 + m_2} \big) \nonumber \\
& = \sum_{|\alpha| \leq N} \frac {1} {\alpha!} \partial_\eta^\alpha a(x,\eta) D_x^\alpha b(x,\eta) + \mathcal O \big( \lambda^{-N-1+m_1 + m_2} \big). \label{eq:cc2-PM2021}
\end{align}
In \eqref{eq:cc3-PM2021} we used
\begin{equation} \label{eq:npExp-PM2021}
\boxed{(D_y \cdot \nabla_\xi)^j = (D_{y_1} \partial_{\xi_1} + \cdots + D_{y_n} \partial_{\xi_n})^j = \sum_{|\alpha| = j} \frac {j!} {\alpha!} D_y^\alpha \partial_\xi^\alpha.}
\end{equation}

By letting $N$ to be large enough, \eqref{eq:cc2-PM2021} implies the following inequality
\begin{equation} \label{eq:cc4-PM2021}
|\partial_x^\alpha \partial_\eta^\beta c(x,\eta)| \lesssim \agl[\eta]^{m_1 + m_2 - |\beta|}
\end{equation}
holds when $|\alpha| = |\beta| = 0$.
To show the case when $\alpha$ and/or $\beta$ are nonzero, we compute
\begin{align}
\partial_x^\alpha \partial_\eta^\beta c(x,\eta)
& = (2\pi)^{-n} \lambda^n \partial_x^\alpha \partial_\eta^\beta \int e^{-i\lambda y \cdot \xi} a(x,\lambda(\tilde \eta + \xi)) b(x+y,\eta) \dif y \dif \xi \nonumber \\
& \simeq \lambda^n \sum_{\alpha, \beta}  \partial_\eta^\beta \int e^{-i\lambda y \cdot \xi} \partial_x^{\alpha'} \partial_\eta^{\beta'} a(x,\lambda(\tilde \eta + \xi)) \partial_x^{\alpha''} \partial_\eta^{\beta''} b(x+y,\eta) \dif y \dif \xi. \label{eq:cc5-PM2021}
\end{align}
Note that $\lambda \tilde \eta = \eta$.
Then we repeat the long computation (with the help of Proposition \ref{prop:2-PM2021}) as in \eqref{eq:cc2-PM2021}, and this can gives \eqref{eq:cc4-PM2021} for all nonzero $\alpha$ and $\beta$.
The rigorous computation is left as a exercise.
Therefore, $c \in S^{m_1 + m_2}$.

By letting $N$ to be large enough, \eqref{eq:cc2-PM2021} implies
\[
c(x,\eta) \sim \sum_\alpha \frac {1} {\alpha!} \partial_\eta^\alpha a(x,\eta) D_x^\alpha b(x,\eta),
\]
We proved the following result:

\begin{thm}\index{composition of $\Psi$DOs} \label{thm:com-PM2021}
	Assume $m_1$, $m_2 \in \R$, $a \in S^{m_1}$ and $b \in S^{m_2}$.
	Then $T_a \circ T_b \in \Psi^{m_1 + m_2}$.
	Denote the symbol of $T_a \circ T_b$ as \boxed{a \# b}, then $a \# b \in S^{m_1 + m_2}$ and
	\[
	\boxed{a \# b(x,\xi) \sim \sum_\alpha \frac {1} {\alpha!} \partial_\eta^\alpha \big( a(x,\eta) \big) \big|_{\eta = \xi} D_y^\alpha \big( b(y,\xi) \big) \big|_{y = x}.}
	\]
\end{thm}


\begin{rem} \label{rem:com-PM2021}
	We deliberately write $\partial_\eta^\alpha \big( a(x,\eta) \big) \big|_{\eta = \xi}$ instead of $\partial_\xi^\alpha a(x,\xi)$, to avoid possible computation mistakes.
	The same for $b$.
\end{rem}

\begin{rem} \label{rem:com2-PM2021}
	When symbol $a$ is of the form $a(x,\xi) = \sum_{|\alpha| \leq m_1} d_\alpha(x) \xi^\alpha$ where $d_\alpha \in C^\infty$ are all bounded, or when symbol $b(x,\xi)$ is independent of $x$-variable,
	the asymptotics in Theorem \eqref{thm:com-PM2021} stops in finite term and the asymptotic is ``exact'': we can replace `$\sim$' by `$=$'.
	This can be seen from the expression \eqref{eq:cet-PM2021} of $c(x,\eta)$.
	See also Exercise \ref{ex:com2-PM2021} and \cite[Remark 2.6.9]{mart02Anin}.
\end{rem}

From Theorem \ref{thm:com-PM2021} we know, if $a \in S^{m_1}$ and $b \in S^{m_2}$, then 
\begin{equation} \label{eq:ab12-PM2021}
	a \# b
	= ab + S^{m_1 + m_2 - 1}
	= ab + \frac 1 i \nabla_\xi a \cdot \nabla_x b + S^{m_1 + m_2 - 2}
	= ab + \frac 1 i \{a, b\} + S^{m_1 + m_2 - 2}.
\end{equation}

\section{Reduction of variables}
\label{sec:ReV-PM2021}

As we have seen in \eqref{eq:Tuph-PM2021} that
\[
(T_{\sigma}u,\varphi)
= (u, (2\pi)^{-n}  \int e^{i(y - x) \cdot \xi} \overline \sigma(x,\xi) \varphi(x) \dif x \dif \xi).
\]
In practice we may encounter $\Psi$DOs of the form
\[
\int e^{i(x-y) \cdot \xi} a(x,y,\xi) \varphi(y) \dif y \dif \xi
\]
where the symbol $a$ depends not only on $x$ but also on $y$, e.g.~in \S \ref{sec:AdT-PM2021} we shall see $\Psi$DOs possessing this type of symbols.
We have the following result.

\begin{thm} \label{thm:ReV-PM2021}
	Assume $a \in S^m(\R_x^n \times \R_y^n \times \R_\xi^n)$, then there exists symbol $a' \in S^m(\R_x^n \times \R_\xi^n)$ such that
	\begin{equation} \label{eq:ReV-PM2021}
	T_{a'} \varphi(x) = (2\pi)^{-n} \int e^{i(x-y) \cdot \xi} a(x,y,\xi) \varphi(y) \dif y \dif \xi, \quad \forall \varphi \in \scrS(\Rn),
	\end{equation}
	and this $T_{a'}$ takes the following as its kernel:
	\[
	K(x,y) := (2\pi)^{-n} \int e^{i(x-y) \cdot \xi} a(x,y,\xi) \dif \xi.
	\]
	Moreover, $a'$ has the asymptotics
	\[
	\boxed{a'(x,\xi) \sim \sum_\alpha \frac 1 {\alpha!} D_y^\alpha \partial_\eta^\alpha \big( a(x,y,\eta) \big) |_{(y, \eta) = (x, \xi)}.}
	\]
\end{thm}

If \eqref{eq:ReV-PM2021} holds, we will have
\[
\int e^{i(x-y) \cdot \xi} a'(x,\xi) \varphi(y) \dif y \dif \xi
= \int e^{i(x-y) \cdot \xi} a(x,y,\xi) \varphi(y) \dif y \dif \xi
\]
and so we can expect
\[
\int e^{i(x-y) \cdot \xi} a'(x,\xi) \dif \xi
= \int e^{i(x-y) \cdot \xi} a(x,y,\xi) \dif \xi
\]
to hold in the oscillatory integral sense.
By changing $y$ to $y+x$, we see the LHS is a Fourier transform,
\[
\calF_\xi \{ a'(x,\xi) \} (y)
= (2\pi)^{-n/2} \int e^{-iy \cdot \xi} a(x,y+x,\xi) \dif \xi,
\]
so
\begin{align*}
a'(x,\eta)
& = (2\pi)^{-n} \int e^{iy \cdot \eta} \dif y \cdot \int e^{-iy \cdot \xi} a(x,y+x,\xi) \dif \xi \\
& = (2\pi)^{-n} \int e^{-iy \cdot (\xi - \eta)} a(x,y+x,\xi) \dif y \dif \xi \\
& = (2\pi)^{-n} \int e^{-iy \cdot \xi} a(x,y+x,\xi + \eta) \dif y \dif \xi \\
& = (2\pi)^{-n} \lambda^n \int e^{-i\lambda y \cdot \xi} a(x,y+x, \lambda(\xi + \tilde \eta)) \dif y \dif \xi,
\end{align*}
where again $\lambda := \agl[\eta]$ and $\tilde \eta := \eta/\agl[\eta]$.
The rigorous proof we go by first set $a'$ as in this way, and then prove $a'$ is a symbol of order $m$.

\begin{proof}[Proof of Theorem \ref{thm:ReV-PM2021}]
We set
\[
a'(x,\eta)
= (2\pi)^{-n} \lambda^n \int e^{-i\lambda y \cdot \xi} a(x,y+x, \lambda(\xi + \tilde \eta)) \dif y \dif \xi,
\]
where $\lambda := \agl[\eta]$ and $\tilde \eta := \eta/\agl[\eta]$.
Following the arguments preceding this proof, we can show that $a'$ satisfies \eqref{eq:ReV-PM2021}.
It's left to show $a'$ satisfies the asymptotics, which will automatically show $a' \in S^m$.

To show $a'$ satisfies the asymptotics, we use the stationary phase lemma in a similar manner as in \S \ref{sec:Com-PM2021}.
We set
\[
a_{x,\eta}(y, \xi) := a(x,x+y, \lambda(\xi + \tilde \eta)),
\]
thus
\begin{equation} \label{eq:aa-PM2021}
a'(x,\eta) = (2\pi)^{-n} \lambda^n \int_{\R^{2n}} e^{i\lambda \agl[Q(y, \xi), (y, \xi)]/2} a_{x,\eta}(y, \xi) \dif {(y,\xi)},
\end{equation}
where $y$ and $\xi$ is treated as horizontal vector and
\[
Q =
\begin{pmatrix}
0 & -I \\
-I & 0
\end{pmatrix}
\quad (\Rightarrow \ Q^{-1} = Q, \ \sgn Q = 0, \text{~and~} \det Q = \pm 1).
\]
For $|\xi| \geq 2$, we have
\begin{align*}
|\partial_y^\alpha \partial_\xi^\beta \big( a_{x,\eta}(y, \xi) \big)|
& = |\partial_y^\alpha \partial_\xi^\beta \big[ a(x,x+y, \lambda(\xi + \tilde \eta)) \big]| \\
& \leq C_{\alpha, \beta} \lambda^{|\beta|} |(\partial_y^\alpha \partial_\xi^{\beta} a)(x, x+y, \lambda(\xi + \tilde \eta))| \\
& \leq C_{\alpha, \beta} \lambda^{|\beta|} \agl[\lambda(\xi + \tilde \eta)]^{m-|\beta|} \\
& \leq C_{\alpha, \beta} \lambda^{|\beta|} \agl[\lambda]^{m-|\beta|} \agl[\xi + \tilde \eta]^{m-|\beta|} \quad (\text{Lemma \ref{lem:abm-PM2021}}) \\
& \lesssim C_{\alpha, \beta} \lambda^{m} \agl[\xi]^{m-|\beta|} \agl[\tilde \eta]^{|m-|\beta||} \qquad (\lambda \simeq \agl[\lambda], \ \text{Lemma \ref{lem:Peetre-PM2021}}) \\
& \leq C_{\alpha, \beta} \lambda^{m} \agl[\xi]^{m-|\beta|}.
\end{align*}
Hence, the condition \eqref{eq:as2-PM2021} is satisfied when $2N + 2 > m$, with $C_{N,n,\alpha, \beta}(\lambda) = C_{\alpha, \beta} \lambda^{m}$, so we can use Proposition \ref{prop:2-PM2021} directly on \eqref{eq:aa-PM2021} to obtain
\begin{align}
a'(x,\eta)
& = (2\pi)^{-n} \lambda^n \int_{\R^{2n}} e^{i\lambda \agl[Q(y, \xi), (y, \xi)]/2} c_{x,\eta}(y, \xi) \dif {(y,\xi)} \nonumber \\
& = (2\pi)^{-n} \lambda^n \times \left( \frac{2\pi}{\lambda} \right)^{n} \sum_{0 \leq j \leq N} \frac {\lambda^{-j}} {j!} \left( \frac {\agl[Q^{-1} D_{(y,\xi)}, D_{(y,\xi)}]} {2i} \right)^j c_{x,\eta}(0, 0) \nonumber \\
& \quad + \lambda^n \times \mathcal O \big( \lambda^{-n-N-1} \sum_{|\alpha| + |\beta| \leq 2n+2N+3} \sup_{(y,\xi) \in \R^{2n}} \frac {|\partial_y^\alpha \partial_\xi^\beta \big( c_{x,\eta}(y,\xi) \big)|} {\agl[(y,\xi)]^{2n+4N+5 - |\alpha| - |\beta|}} \big) \nonumber \\
& = \sum_{0 \leq j \leq N} \frac {\lambda^{-j}} {j!} (D_y \cdot \nabla_\xi)^j c_{x,\eta}(0, 0) + \mathcal O \big( \lambda^{-N-1+m} \big) \nonumber \\
& = \sum_{|\alpha| \leq N} \frac {\lambda^{-|\alpha|}} {\alpha!} D_y^\alpha \partial_\xi^\alpha \big( c_{x,\eta}(y,\xi) \big) |_{(y,\xi) = (0, 0)} + \mathcal O \big( \lambda^{-N-1+m} \big) \nonumber \\
& = \sum_{|\alpha| \leq N} \frac {\lambda^{-|\alpha|}} {\alpha!} D_y^\alpha \partial_\xi^\alpha \big( a(x,y+x, \lambda \xi + \eta) \big) |_{(y,\xi) = (0, 0)} + \mathcal O \big( \lambda^{-N-1+m} \big) \nonumber \\
& = \sum_{|\alpha| \leq N} \frac {1} {\alpha!} D_y^\alpha \partial_\xi^\alpha \big( a(x,y+x, \xi + \eta) \big) |_{(y,\xi) = (0, 0)} + \mathcal O \big( \lambda^{-N-1+m} \big) \nonumber \\
& = \sum_{|\alpha| \leq N} \frac {1} {\alpha!} D_y^\alpha \partial_\xi^\alpha \big( a(x,y, \xi) \big) |_{(y,\xi) = (x, \eta)} + \mathcal O \big( \lambda^{-N-1+m} \big).
\label{eq:aa2-PM2021}
\end{align}

Due to the same logic as in \eqref{eq:cc4-PM2021}-\eqref{eq:cc5-PM2021}, we can let $N$ to be large enough, and by doing so, \eqref{eq:aa2-PM2021} can implies $c \in S^m$ and
\[
c(x,\eta) \sim \sum_\alpha \frac 1 {\alpha!} D_y^\alpha \partial_\xi^\alpha \big( a(x,y,\xi) \big) |_{(y, \xi) = (x, \eta)},
\]
The proof is complete.
\end{proof}

Theorem \ref{thm:ReV-PM2021} completes the proof of Lemma \ref{lem:kedi-PM2021}.

\section{The Adjoint and transpose}
\label{sec:AdT-PM2021}

We define the adjoint and transpose of the $\Psi$DO $T_a$ acting on Schwartz functions as follows,
\begin{equation} \label{eq:AdTdef-PM2021}
	\begin{aligned}
	\text{adjoint\index{adjoint}~} T_a^*: & \quad (T_a^* u, v) := (u, T_a v), \\
	\text{transpose\index{transpose}~} {}^t T_a: & \quad \agl[{}^t T_a u, v] := \agl[u, T_a v],
	\end{aligned}
\end{equation}
where $u,v \in \scrS$.

\begin{thm} \label{thm:AdT-PM2021}
	Assume $a(x,\xi) \in S^m$.
	The $T_a^*$ and ${}^t T_a$ defined in \eqref{eq:AdTdef-PM2021} exist uniquely, and both are {\rm $\Psi$DOs}.
	There exist symbols $a^*$ and ${}^t a$ of the same order as $a$ such that $T_a^* = T_{a^*}$ and ${}^t T_a = T_{{}^t a}$.
	Moreover, we have the asymptotics
	\begin{align*}
	& \boxed{a^*(x,\xi) \sim \sum_\alpha \frac 1 {\alpha!} D_x^\alpha \partial_\xi^\alpha \overline a(x,\xi),} \\
	& \boxed{ {}^t a(x,\xi) \sim \sum_\alpha \frac {(-1)^{|\alpha|}} {\alpha!} D_x^\alpha \partial_\xi^\alpha a(x,-\xi).}
	\end{align*}
\end{thm}

\begin{rem} \label{rem:AdT-PM2021}
	The computation \eqref{eq:Tuph-PM2021} gives an very efficient intuitive way to compute the asymptotics of $a^*$.
\end{rem}

\begin{proof}
	Here we only show the proof for $a^*$, and that of ${}^t a$ is left as an exercise.
	
	{\bf Step 1.} Existence.
	As explained at the beginning of \S \ref{sec:ReV-PM2021}, for $u,v \in \scrS$ we have
	\[
	(u, T_{a} v)
	= ((2\pi)^{-n} \int e^{i(y-x) \cdot \xi} \overline a(x,\xi) u(x) \dif x \dif \xi, v)
	\]
	so if we define a mapping $T$ as
	\[
	T u(y) := (2\pi)^{-n} \int e^{i(y-x) \cdot \xi} \overline a(x,\xi) u(x) \dif x \dif \xi,
	\]
	then $(T u, v) := (u, T_a v)$.
	Also, this $T$ is of the form \eqref{eq:ReV-PM2021}, so by Theorem \ref{thm:ReV-PM2021} we know $T$ is a $\Psi$DO.

	{\bf Step 2.} Uniqueness.
	Assume there are two adjoint of $T$, and we denote them as $T_1$ and $T_2$, respectively.
	Then for any $u,v \in \scrS$ we can conclude
	\[
	(T_1 u, v) = (u, T_a v) = (T_2 u, v)
	\ \Rightarrow \
	((T_1 - T_2) u, v) = 0.
	\]
	Hence, $(T_1 - T_2) u = 0$ for any $u \in \scrS$ and so $T_1 = T_2$.

	{\bf Step 3.} Asymptotics.
	Theorem \ref{thm:ReV-PM2021} suggests that the symbol of $T$, denoted as $a^*$, satisfies the asymptotics:
	\[
	a^*(x,\xi)
	\sim \sum_\alpha \frac 1 {\alpha!} D_y^\alpha \partial_\eta^\alpha \big( \overline a(y,\eta) \big) |_{(y, \eta) = (x,\xi)}
	= \sum_\alpha \frac 1 {\alpha!} D_x^\alpha \partial_\xi^\alpha \overline a(x,\xi).
	\]
	The proof is complete.
\end{proof}

\section*{Exercise}

\begin{ex}
	Use stationary phase lemmas to complete the estimate in $\partial_x^\alpha \partial_\eta^\beta c(x,\eta)$ in \eqref{eq:cc5-PM2021}.
	Hint: mimic the computations in \eqref{eq:cc2-PM2021}.
\end{ex}

\begin{ex}
	Assume $a \in S^{m_1}$ and $b \in S^{m_2}$.
	Utilize Theorem \ref{thm:com-PM2021} to show that $[T_a, T_b] \in \Psi^{m_1 + m_2 - 1}$, where $[T_a, T_b] := T_a T_b - T_b T_a$ is called the \emph{commutator} \index{commutator} of $T_a$ and $T_b$, and $T_a T_b$ is a shorthand of the composition $T_a \circ T_b$.
\end{ex}

\begin{ex} \label{ex:com2-PM2021}
	Prove the statement in Remark \ref{rem:com2-PM2021}.
	In Theorem \ref{thm:com-PM2021}, assume $a(x,\xi) = \sum_{|\alpha| \leq m_1} d_\alpha(x) \xi^\alpha$ where $d_\alpha \in C^\infty$ are all bounded, or assume $b = b(\xi)$, then show that
	\[
	c(x,\eta) = \sum_{|\alpha| \leq N} \frac {1} {\alpha!} \partial_\eta^\alpha a(x,\eta) D_x^\alpha b(x,\eta)
	\]
	for some finite integer $N$.
	Hint: substitute the expressions of $a$ or $b$ into \eqref{eq:cet-PM2021} and use Lemma \ref{lem:Tlf-PM2021}.
\end{ex}

\begin{ex}
	Mimic the proof for $a^*$ in Theorem \ref{thm:AdT-PM2021} to prove the result for ${}^t a$.
\end{ex}

\begin{ex}
	Let $T_1$, $T_2$ be two {\rm $\Psi$DOs}.
	Show that $(T_1^*)^* = T_1$ and $(T_1 T_2)^* = T_2^* T_1^*$.
	Here ``$T^*$'' stands for taking the adjoint of $T$.
\end{ex}

\chapter{Parametrix and Boundedness of \texorpdfstring{$\Psi$DOs}{PsiDOs}} \label{ch:BddP-PM2021}

In this chapter we investigate the parametrix and boundedness of $\Psi$DOs, both of which heavily utilize the symbolic calculus.
The notion of parametrix can be understood as the approximate inverse, or the inverse module $C^\infty$ an operator.
For a homogeneous polynomial $T(\xi) := \sum_{|\alpha| = m} a_\alpha \xi^\alpha$, its corresponding operator $T := T(D)$ is a $\Psi$DO.

To find the inverse, a typical idea is to design $S(\xi) := 1/T(\xi)$ and let $S := S(D)$.
Inaccurately this seems to give us $ST = I$ where $I$ is the identity operator, which is (inaccurately) because by Theorem \ref{thm:com-PM2021} (and Remark \ref{rem:com2-PM2021}) we have
\[
\text{symbol of~} ST
= \sum_\alpha \frac {1} {\alpha!} \partial_\xi^\alpha \big( T(\xi) \big) D_x^\alpha \big( S(\xi) \big)
=  T(\xi) S(\xi) = 1.
\]
Unfortunately, this is wrong, because $1/T(\xi)$ has singularities when $T(\xi) = 0$.
And due to this reason, $S$ may not be a {\rm $\Psi$DO} so Theorem \ref{thm:com-PM2021} is not applicable here.

However, the $S$ can be saved if we cutoff the singularity.
Specifically, fix a $\chi \in C_c^\infty$ with $\chi(0) = 1$ and we re-design $S$ as $S(\xi) := (1 - \chi(\xi))/T(\xi)$ and once again let $S := S(D)$.
It is straightforward that this new $S(\xi)$ is a symbol and so $S$ is a $\Psi$DO.
Again, by Theorem \ref{thm:com-PM2021} (and Remark \ref{rem:com2-PM2021}) we have
\begin{align*}
\text{symbol of~} ST
& = \sum_\alpha \frac {1} {\alpha!} \partial_\xi^\alpha \big( T(\xi) \big) D_x^\alpha \big( S(\xi) \big)
=  T(\xi) S(\xi) \\
& = T(\xi) (1 - \chi(\xi))/T(\xi) \\
& = 1 - \chi(\xi).
\end{align*}
It is also true that the symbol of $TS = 1 - \chi(\xi)$.
Note that $\chi(D) \in \Psi^{-\infty}$, so we conclude
\[
ST = I + \Psi^{-\infty}, \quad TS = I + \Psi^{-\infty}.
\]
This inspires us to introduce the notion of parametrix.

\section{Parametrix} \label{sec:para-PM2021}

In what follows we use $I$ to signify the identity operator unless otherwise stated.

\begin{defn}[Parametrix\index{parametrix}] \label{defn:para-PM2021}
	Assume $m \in \R$ and $T \in \Psi^m$. If there exists a {\rm $\Psi$DO} $S$ such that $ST - I \in \Psi^{-\infty}$, we call $S$ a \emph{left parametrix} of $T$.
	If $TS - I \in \Psi^{-\infty}$, we call $S$ a \emph{right parametrix} of $T$.
	We call $S$ a \emph{parametrix} of $T$ if it is both a left and a right parametrix.
\end{defn}

The notion of left and right parametrix is somewhat redundant.

\begin{lem} \label{lem:lre-PM2021}
	Assume both $S$ and $T$ both {\rm $\Psi$DOs}.
	If $S$ is a left (right) parametrix of $T$, and $T$ has a right (left) parametrix, then $S$ is also a right (left) parametrix of $T$.
\end{lem}

\begin{proof}
	We only prove the left-case.
	There exists $S'$ such that $TS' = I + \Psi^{-\infty}$.
	From $ST = I + \Psi^{-\infty}$ we have $(ST)S' = S' + \Psi^{-\infty} = S(TS')$, so $S' + \Psi^{-\infty} = S(I + \Psi^{-\infty})$, which gives $S = S' + \Psi^{-\infty}$.
	Therefore,
	\[
	TS = T(S' + \Psi^{-\infty})
	= TS' + \Psi^{-\infty}
	= I + \Psi^{-\infty} + \Psi^{-\infty}
	= I + \Psi^{-\infty},
	\]
	which implies $S$ is a right parametrix of $T$.
\end{proof}

The parametrix of a $\Psi$DO is not always exists.
And in contrast with the notion of inverse of an operator, when parametrices exist, they are not unique.

\begin{lem} \label{lem:PNUi-PM2021}
	Assume $S$ is a parametrix of $T$, and $R \in \Psi^{-\infty}$, then $S+R$ is also a parametrix of $T$.
\end{lem}

The proof is left as an exercise.
One of the condition that guarantees the existence of parametrix is the ellipticity.

\begin{defn}[Ellipticity\index{ellipticity}] \label{defn:elp-PM2021}
	Assume $m \in \R$ and $a \in S^m$. We call $a$ and also its corresponding {\rm $\Psi$DO} $T_a$ \emph{elliptic} when there exist fixed positive constants $C$ and $R$ such that
	\[
	\boxed{|a(x,\xi)| \geq C\agl[\xi]^m, \quad \text{when~} x \in \Rn, \ |\xi| \geq R.}
	\]
\end{defn}

There is an equivalent definition for the ellipticity of a symbol.

\begin{lem} \label{lem:aEv-PM2021}
	Assume $m \in \R$ and $a \in S^{m}$.
	The ellipticity condition for $a$ is equivalent to the fact that there exist two positive constants $C$ and $D$ such that
	\begin{equation} \label{eq:aEv-PM2021}
		\boxed{|a(x,\xi)| \geq C \agl[\xi]^{m} - D \agl[\xi]^{m-1}, \Forall x,\xi \in \Rn.}
	\end{equation}
\end{lem}

\begin{proof}
	Assume $a \in S^m$ is elliptic, then there are constants $C$, $R > 0$ such that
	\[
	|a(x,\xi)| / \agl[\xi]^m
	\geq C, \quad \forall |\xi| \geq R,
	\]
	so for any positive constant $D$ we have
	\begin{equation} \label{eq:absm-PM2021}
		|a(x,\xi)| / \agl[\xi]^m
		\geq C - D \agl[\xi]^{-1},
	\end{equation}
	for $\forall |\xi| \geq R$.
	If we set $D := C \agl[R]$, then
	\[
	\forall |\xi| \leq R, \quad
	C \agl[\xi] \leq D \ \Rightarrow \ C - D \agl[\xi]^{-1} \leq 0,
	\]
	so \eqref{eq:absm-PM2021} holds for both $|\xi| \geq R$ and $|\xi| \leq R$.
	This gives \eqref{eq:aEv-PM2021}.
	
	On the other hand, from \eqref{eq:aEv-PM2021} it is easy to see $a$ is elliptic.
\end{proof}

We will show that
\[
\boxed{\text{Ellipticity} \quad \Leftrightarrow \quad \exists \, \text{parametrix}.}
\]

First, we show the ellipticity condition gives the existence of parametrices.

\begin{thm}[Ellipticity $\Rightarrow$ parametrix] \label{thm:para1-PM2021}
	Assume $m \in \R$ and $a \in S^m$ and $a$ is elliptic, then $T_a$ has a parametrix.
\end{thm}

\begin{proof}
	Here we use the notation $\sigma(T)$ to represent the symbol of a $\Psi$DO $T$, the well-definedness of the mapping $\sigma$ is guaranteed by Lemma \ref{lem:sBi-PM2021}.
	We denote $T_a$ as $A$ for simplicity.
	Fix a cutoff function $\chi \in C_c^\infty(\Rn)$ such that $\chi(\xi) = 1$ when $|\xi| \leq R$ and $\chi(\xi) = 0$ when $|\xi| \geq R+1$, where the $R$ is given in Definition \ref{defn:elp-PM2021}.
	
	{\bf Step 1.}~Define $b_0(x,\xi) := (1-\chi(\xi))/a(x,\xi)$ and $B_0 := T_{b_0}$, then $b_0$ is well-defined because the denominator is nonzero in the support of $1-\chi$.
	Also, it can be checked that $b_0$ is a symbol of order $-m$ (see Exercise \ref{ex:1as-PM2021}).
	Then according to Theorem \ref{thm:com-PM2021}, we have
	\[
	\sigma(A B_0)
	= a (1-\chi)/a - r_1
	= 1 - \chi - r_1, \quad \text{for some} \quad r_1 \in S^{-1}.
	\]
	
	{\bf Step 2.}~Define $b_1(x,\xi) := (1-\chi(\xi))/a(x,\xi) \cdot r_1(x,\xi) \in S^{-m-1}$ and $B_1 := T_{b_1}$.
	Again, according to Theorem \ref{thm:com-PM2021}, we have
	\begin{align*}
	\sigma(A (B_0 + B_1))
	& = \sigma(A B_0) + \sigma(A B_1)
	= 1 - \chi - r_1 + a(1-\chi)/a r_1 - r_2 \\
	& = 1 - (1+r_1) \chi - r_2 , \quad \text{for some} \quad r_2 \in S^{-2}.
	\end{align*}

	{\bf Step 3.}~Define recursively $b_j(x,\xi) := (1-\chi(\xi))/a(x,\xi) \cdot r_j(x,\xi) \in S^{-m-j}$ and $B_j := T_{b_j}$.
	According to Theorem \ref{thm:com-PM2021}, we have
	\begin{align*}
	\sigma(A (B_0 + \cdots + B_j))
	& = \sigma(A (B_0 + \cdots + B_{j-1})) + \sigma(A B_j) \\
	& = [1 - (1+r_1 + \cdots + r_{j-1}) \chi - r_j] + a(1-\chi)/a r_j - r_{j+1} \\
	& = 1 - (1+r_1 + \cdots + r_j) \chi - r_{j+1}, \quad \text{for some} \quad r_{j+1} \in S^{-j-1}.
	\end{align*}

	{\bf Step 4.}~According to Theorem \ref{thm:Asy-PM2021}, there exists $b \in S^{-m}$ such that
	\(
	b \sim \sum_j b_j.
	\)
	Denote $B = T_b$, so for any $N \in \N$ there holds $B = B_0 + \cdots B_N + \Psi^{-m-N-1}$.
	Hence we can compute the symbol of $AB$ as follows,
	\begin{align}
	\sigma(AB)
	& = \sigma(A (B_0 + \cdots B_N + \Psi^{-m-N-1})) \nonumber \\
	& = \sigma(A (B_0 + \cdots B_N)) + \sigma(A \Psi^{-m-N-1}) \nonumber \\
	& = 1 - (1+r_1 + \cdots + r_N) \chi - r_{N+1} + S^{-N-1}
	= 1 + S^{-N-1}, \label{eq:sAB-PM2021}
	\end{align}
	where the last equal sign is due to $\chi \in S^{-\infty}$ and $r_{N+1} \in S^{-N-1}$.
	Due to the arbitrariness of $N$, \eqref{eq:sAB-PM2021} implies that
	\[
	AB - I \in \Psi^{-\infty},
	\]
	so $B$ is right parametrix of $A$.
	By repeating steps 1-4 we can also show $A$ has a right parametrix, so by Lemma \ref{lem:lre-PM2021} we conclude that $B$ is a parametrix of $A$.
\end{proof}

Second, we show the existence of parametrices gives the ellipticity.

\begin{thm}[Parametrix $\Rightarrow$ ellipticity] \label{thm:para2-PM2021}
	Assume $m \in \R$ and $a \in S^m$ and $T_a$ has either a right parametrix or a left parametrix, then $T_a$ is elliptic.
\end{thm}

\begin{proof}
	Assume $T_b$ is the right parametrix, then $b$ is necessarily a symbol of order $-m$, so
	\[
	\sigma(T_a T_b) = ab + S^{-1}, \quad \text{and} \quad
	\sigma(T_a T_b) = \sigma(I + \Psi^{-\infty}) = 1 + S^{-\infty},
	\]
	thus
	\[
	ab = 1 + S^{-1} + S^{-\infty} = 1 + S^{-1}.
	\]
	Therefore, when $|\xi|$ is large enough
	\[
	\forall (x,\xi) \in \R^{2n}, \quad |a(x,\xi) b(x,\xi) - 1| \leq C\agl[\xi]^{-1}.
	\]
	Therefore, $\agl[\xi] \geq C/2$ is large enough, we can conclude
	\[
	|a(x,\xi) b(x,\xi)| \geq 1/2 \quad \Rightarrow \quad
	|a(x,\xi)| \geq 1/(2|b(x,\xi)|).
	\]
	This gives
	\[
	|a(x,\xi)| \geq \agl[\xi]^{m}/2 \quad \text{when} \quad \agl[\xi] \geq C/2,
	\]
	so $a$ is elliptic.
	
	The proof for the left-case is similar.
\end{proof}

From Theorems \ref{thm:para1-PM2021} \& \ref{thm:para2-PM2021}, we see that the condition ``$T$ has a right (left) parametrix'' in Lemma \ref{lem:lre-PM2021} can be lifted.

\begin{prop} \label{prop:Blr-PM2021}
	Assume both $S$ and $T$ are {\rm $\Psi$DOs}.
	If $S$ is a left (right) parametrix of $T$, then $S$ is also a right (left) parametrix of $T$.
\end{prop}

\begin{proof}
	If $S$ is a left (right) parametrix of $T$, then by Theorem \ref{thm:para2-PM2021} we know that $T$ is elliptic, so by Theorem \ref{thm:para1-PM2021} we know $T$ has a right (left) parametrix.
	Then Lemma \ref{lem:lre-PM2021} tells us $S$ is a right (left) parametrix of $T$.
\end{proof}

We recall that when $T$ is a $\Psi$DO, $T$ doesn't increase the singular support of a distribution (see Theorem \ref{thm:pslo-PM2021}).
Now if we know $T$ is also elliptic, then $T$ doesn't decrease the singular support.

\begin{lem} \label{lem:Blr-PM2021}
	Assume $T$ is an elliptic {\rm $\Psi$DO} and $u \in \mathcal E'$, then
	\[
	\boxed{\ssupp (Tu) = \ssupp u.}
	\]
\end{lem}

Readers may compare Lemma \ref{lem:Blr-PM2021} with Theorem \ref{thm:pslo-PM2021}.

\begin{proof}
	Denote $Tu = f$, then Theorem \ref{thm:pslo-PM2021} implies
	\[
	\ssupp (Tu) \subset \ssupp u.
	\]
	Theorem \ref{thm:para1-PM2021} implies $T$ possesses parametrices.
	Let $S$ be a parametrix of $T$.
	Then we have $Sf = STu = (I + \Psi^{-\infty}) u = u + C^\infty(\Rn)$, so
	\[
	\ssupp u = \ssupp(Sf) \subset \ssupp f
	= \ssupp (Tu).
	\]
	The proof is done.
\end{proof}

We will revisit the notion of parametrix and ellipticity in \S \ref{sec:mlpara-PM2021}.

\section{The \texorpdfstring{$L^2$}{L2} boundedness} \label{sec:L2B-PM2021}

\begin{lem}[Schur estimate\index{Schur estimate}] \label{lem:SchEst-PM2021}
	Assume $K \in L_{loc}^1(\R^{2n})$ and for $\varphi \in L_{loc}^1(\Rn)$ we denote $T \varphi(x) := \int_\Rn K(x,y) \varphi(y) \dif y$.
	Also, denote
	\begin{equation} \label{eq:LRK-PM2021}
	L := \sup_{x \in \Rn} \int_\Rn |K(x,y)| \dif y, \quad
	R := \sup_{y \in \Rn} \int_\Rn |K(x,y)| \dif x.
	\end{equation}
	When $L, R < +\infty$, for $\forall p \in [1, +\infty]$ and $\varphi \in L^p(\Rn)$ we have
	\[
	\nrm[L^p]{T\varphi} \leq L^{1-1/p} R^{1/p} \nrm[L^p]{\varphi}.
	\]
\end{lem}

\begin{proof}
	When $p = +\infty$ is trivial, we have
	\begin{align*}
	\nrm[L^\infty]{T \varphi}
	& = {\mathop {\rm ess\,sup}}_x |\int K(x,y) \varphi(y) \dif y|
	\leq {\mathop {\rm ess\,sup}}_x \int |K(x,y)| \dif y \cdot {\mathop {\rm ess\,sup}}_y |\varphi(y)| \\
	& = L \nrm[L^\infty]{\varphi}.
	\end{align*}
	
	When $p = 1$, we have
	\begin{align*}
	\nrm[L^1]{T \varphi}
	& = \nrm[L^1]{\int K(x,y) \varphi(y) \dif y} 
	\leq \int \nrm[L^1]{K(\cdot,y)} |\varphi(y)| \dif y \\
	& \leq R \int |\varphi(y)| \dif y
	= R \nrm[L^1]{\varphi}.
	\end{align*}
	
	Now we assume $1 < p < +\infty$.
	Let $p'= p/(p-1)$, so $1 = 1/p + 1/p'$.
	We have
	\begin{align*}
	|T \varphi(x)|
	& \leq \int |K(x,y) \varphi(y)| \dif y
	= \int |K(x,y)|^{1/p'} |K(x,y)|^{1/p} |\varphi(y)| \dif y \\
	& \leq \big( \int |K(x,y)| \dif y \big)^{1/p'} \big( \int |K(x,y)| |\varphi(y)|^p \dif y \big)^{1/p} \qquad \text{(by H\"older's ineq.)} \\
	& \leq L^{1/p'} \big( \int |K(x,y)| |\varphi(y)|^p \dif y \big)^{1/p}.
	\end{align*}
	Hence,
	\begin{align*}
	\nrm[L^p]{T\varphi}
	& \leq L^{1/p'} \big( \iint |K(x,y)| |\varphi(y)|^p \dif y \dif x \big)^{1/p}
	\leq L^{1/p'} \big( R \int |\varphi(y)|^p \dif y \big)^{1/p} \\
	& \leq L^{1/p'} R^{1/p} \nrm[L^p]{\varphi}.
	\end{align*}
	The proof is complete.
\end{proof}

As already mentioned in Remark \ref{rem:kerP-PM2021}, when the order $m$ is small enough, $T_\sigma$ possesses certain types of boundedness.

\begin{lem} \label{lem:Tbd-PM2021}
	In $\Rn$, we assume $m < -n$ and $\sigma \in S^m(\R_x^n \times \R_\xi^n)$, then the {\rm $\Psi$DO} $T_\sigma \colon L^p(\Rn) \to L^p(\Rn)$ is bounded.
\end{lem}

\begin{proof}
	Denote the kernel of $T_\sigma$ as $K$, so
	\[
	K(x,y) = (2\pi)^{-n} \int e^{i(x-y) \cdot \xi} \sigma(x,\xi) \dif \xi.
	\]
	Because $\sigma \in S^m$ with $m < -n$, we know that integral above is absolutely integrable.
	This means that $K$ is a well-defined function in $\R^{2n}$, especially, $K$ is well-defined on the diagonal $\{(x,x) \,;\, x \in \Rn \}$.
	However, we remind the readers that the condition ``$m < -n$'' doesn't guarantee that $K$ is also $C^\infty$ on the diagonal
	(recall that Lemma \ref{lem:keSmo-PM2021} tells us $K$ is $C^\infty$ off diagonal).
	The value of $K$ on $\R^{2n}$ is uniformly bounded, because
	\[
	|K(x,y)|
	\leq (2\pi)^{-n} \int |\sigma(x,\xi)| \dif \xi
	\leq (2\pi)^{-n} \int C \agl[\xi]^{m} \dif \xi
	\leq C.
	\]
	
	Because $K$ is well-defined and uniformly bounded on $\R^{2n}$, we can define the corresponding $L$ and $R$ of it as in \eqref{eq:LRK-PM2021}, and we can also enhance the estimate in Lemma \ref{lem:keSmo-PM2021} as follows,
	\[
	|K(x,y)| \leq C \agl[x-y]^{-n-1}, \quad \forall x, y \in \Rn,
	\]
	which implies both $L$ and $R$ are finite.
	Because $T_\sigma \varphi(x) = \int_\Rn K(x,y) \varphi(y) \dif y$, we can use Lemma \ref{lem:SchEst-PM2021} to conclude
	\(
	\nrm[L^p]{T_\sigma \varphi} \lesssim \nrm[L^p]{\varphi}.
	\)
	The proof is complete.
\end{proof}

\begin{thm}[$L^2$ boundedness] \label{thm:T2B-PM2021}
	Assume symbol $a \in S^0$, then $T_a \colon L^2(\Rn) \to L^2(\Rn)$ is bounded.
\end{thm}

\begin{proof}
	Recall the definition for ``$a \# b$'' in Theorem \ref{thm:com-PM2021}.
	To prove the result, it amounts to find a suitable positive constant $M$ such that for $\forall \varphi \in \scrS$,
	\[
	\nrm[L^2]{T_a \varphi} \leq M \nrm[L^2]{\varphi} \quad \Leftrightarrow \quad
	((M - T_a^* T_a) \varphi, \varphi) \geq 0.
	\]
	Our strategy is: we try to find such a $M$ so that $M - T_a^* T_a$ can be represented as $B^* B$ for some $B$ so that
	\[
	((M - T_a^* T_a) \varphi, \varphi)
	= (B^* B \varphi, \varphi)
	= (B\varphi, B\varphi)
	\geq 0.
	\]
	
	{\bf Step 1.} Symbolic calculus.
	Because $a \in S^0$, we know $|a(x,\xi)| \leq C$ uniformly for some $C$.
	Let
	\begin{equation}
	M = M_1 + M_2, \quad \text{where} \quad M_1 := 2\sup_{\R^{2n}} |a(x,\xi)|^2 + 1,
	\end{equation}
	and $M_2$ shall be determined later, and define
	\[
	b(x,\xi) := \sqrt{M_1 - |a(x,\xi)|^2}.
	\]
	It can be checked that $b \in S^0$.
	We use $\sigma(T)$ to signify the symbol of $T$.
	Then by Theorems \ref{thm:com-PM2021} \& \ref{thm:AdT-PM2021} we have
	\[
	\sigma(T_b^* T_b)
	= |b|^2 + S^{-1}
	= M_1 - |a(x,\xi)|^2 + S^{-1},
	\]
	and also
	\[
	\sigma(M_1 - T_a^* T_a)
	= M_1 - \sigma(T_a^* T_a)
	= M_1 - (|a(x,\xi)|^2 + S^{-1}).
	\]
	Hence
	\[
	\sigma(M - T_a^* T_a)
	= M_2 + \sigma(M_1 - T_a^* T_a)
	= M_2 + \sigma(T_b^* T_b) + S^{-1},
	\]
	which implies
	\[
	M - T_a^* T_a = M_2 + T_b^* T_b - R, \quad \text{for some $\Psi$DO~} R \in \Psi^{-1}.
	\]
	Therefore it is equivalent to prove
	\[
	((M_2 + T_b^* T_b - R) \varphi, \varphi) \geq 0,
	\]
	so we only need to prove
	\begin{equation} \label{eq:RM2-PM2021}
	(R\varphi, \varphi) \leq M_2\nrm[L^2]{\varphi}^2.
	\end{equation}
	
	{\bf Step 2.}
	To prove \eqref{eq:RM2-PM2021}, we can do the following derivations:
	\begin{align*}
	(R\varphi, \varphi) \leq M_2 \nrm[L^2]{\varphi}^2
	& \ \Leftarrow \ |(R\varphi, \varphi)| \leq M_2 \nrm[L^2]{\varphi}^2 \\
	& \ \Leftarrow \ \nrm[L^2]{R\varphi} \nrm[L^2]{\varphi} \leq M_2 \nrm[L^2]{\varphi}^2 \\
	& \ \Leftarrow \ \underline{\nrm[L^2]{R\varphi} \leq M_2 \nrm[L^2]{\varphi}} \\
	& \ \Leftarrow \ (R^* R \varphi, \varphi) \leq M_2^2 \nrm[L^2]{\varphi}^2 \\
	& \ \Leftarrow \ |(R^* R \varphi, \varphi)| \leq M_2^2 \nrm[L^2]{\varphi}^2 \\
	& \ \Leftarrow \ \underline{\nrm[L^2]{R^* R \varphi} \leq M_2^2 \nrm[L^2]{\varphi}}.
	\end{align*}
	We observe that
	\begin{equation} \label{eq:RRtR-PM2021}
	\nrm[L^2]{R^* R \varphi} \leq M_2^2 \nrm[L^2]{\varphi} \ \Rightarrow \ \nrm[L^2]{R \varphi} \leq M_2 \nrm[L^2]{\varphi}.
	\end{equation}

	{\bf Step 3.}
	Using \eqref{eq:RRtR-PM2021} iteratively, we can obtain $(R^* R)^* R^* R$, $((R^* R)^* R^* R)^* (R^* R)^* R^* R$, etc, and each time the order of the corresponding $\Psi$DO decreases by at least 1.
	We will end up with a $\Psi$DO of order less than $-n$ in finite time.
	And by Lemma \ref{lem:Tbd-PM2021}, that operator is $L^2$-bounded.
	Then we use \eqref{eq:RRtR-PM2021} to bring the boundedness back to $R$, so we arrive at
	\begin{equation*}
	\nrm[L^2]{R \varphi} \leq M_2 \nrm[L^2]{\varphi}, \quad \forall \varphi \in \scrS.
	\end{equation*}
	This gives \eqref{eq:RM2-PM2021}.
	The proof is complete.
\end{proof}

As a corollary of Theorem \ref{thm:T2B-PM2021}, we have the following $H^m$ boundedness for any $T \in \Psi^m$.

\begin{cor}[$H^m$ boundedness] \label{cor:HmB-PM2021}
	Assume $T \in \Psi^m$, then for any $s \in \R$, the mapping $T \colon H^{s+m}(\Rn) \to H^{s}(\Rn)$ is bounded.
\end{cor}

\begin{proof}
	Denote $J := (I - \Delta)^{1/2}$.
	Because $T \in \Psi^m$, we have $J^s T J^{-s-m} \in \Psi^0$.
	Hence for any $\varphi \in \scrS(\Rn)$ we have
	\begin{align*}
	\nrm[H^s]{T \varphi}
	& = \nrm[L^2]{J^s T \varphi}
	= \nrm[L^2]{J^s T J^{-s-m} J^{s+m} \varphi} \\
	& \leq C \nrm[L^2]{J^{s+m} \varphi}
	= C \nrm[H^{s+m}]{\varphi}.
	\end{align*}
	By a density argument we can extend the result to any $\varphi \in H^{s+m}$.
	The proof is done.
\end{proof}

Theorem \ref{thm:T2B-PM2021} can be generalized to a more general case.
The $L^2$-boundedness results are given in \cite{hor71con, Cal71Bdd, cava72ac}.
Then A.~Calder\'on and R.~Vaillancourt generalized their own result \cite{Cal71Bdd} in \cite{cava72ac}.
We comment that \cite{hw87th} gives an elementary proof of the results in \cite{cava72ac}.
Here we restate the main results in \cite{Cal71Bdd, cava72ac} as follows.
Recall the symbol space $S^m_{\rho,\delta}$ defined in Definition \ref{defn:symbolx-PM2021}.

\begin{thm}[Calder\'on-Vaillancourt Theorem\cite{Cal71Bdd}] \label{thm:CaVai1-PM2021}
	Assume $a \in S_{0,0}^0(\R_x^n \times \R_\xi^n; \mathbb C)$, then the {\rm $\Psi$DO} $T_a$ is bounded in $L^2(\Rn)$, and there exist constants $C$, $N$ such that
	\begin{equation} \label{eq:CaVai1-PM2021}
	\forall \varphi \in L^2(\Rn), \quad \nrm[L^2]{T_a \varphi} \leq C \max_{|\alpha + \beta| \leq N} \nrm[L^\infty]{\partial_x^\alpha \partial_\xi^\beta a} \nrm[L^2]{\varphi}.
	\end{equation}
\end{thm}

\begin{thm}[Generalized Calder\'on-Vaillancourt Theorem \cite{cava72ac}] \label{thm:CaVai-PM2021}
	Let $a \in C^\infty(\R_x^n \times \R_y^n \times \R_\xi^n; \mathbb C)$, and $0 \leq \rho \leq \delta_j < 1~(j=1,2)$ and $M/n \geq \frac 1 2 (\delta_1 + \delta_2) - \rho$.
	If there exists a constant $\mathcal C$ such that $\forall (x,\xi) \in \Rn \times \Rn$,
	\[
	|\partial_{x}^{\alpha_1} \partial_\xi^\beta a(x,y,\xi)| \leq \mathcal C \agl[\xi]^{-M + \delta_1|\alpha_1| - \rho |\beta|}, \quad
	|\partial_{y}^{\alpha_2} \partial_\xi^\beta a(x,y,\xi)| \leq \mathcal C \agl[\xi]^{-M + \delta_2|\alpha_2| - \rho |\beta|}
	\]
	holds for all $0 \leq |\beta| \leq 2\lceil n/2 \rceil + 2$ and $0 \leq |\alpha_j| \leq 2m_j$ $(j=1,2)$ with $m_j$ being the least integer satisfying $m_j (1 - \delta_j) \geq 5n/4$,
	then the linear operator $T_a$ defined as
	\[
	T_a \varphi(x) := (2\pi)^{-n} \iint_{\Rn \times \Rn} e^{i(x-y) \cdot \xi} a(x,y,\xi) \varphi(y) \dif y \dif \xi
	\]
	is bounded from $L^2(\Rn)$ to $L^2(\Rn)$ and $\nrm[L^2 \to L^2]{T_a} \leq C_{\delta_1, \delta_2, n} \mathcal C$ for some constant $C_{\delta_1, \delta_2, n}$.
\end{thm}

For simplicity, we summarize a easy-to-use $L^2$-boundedness result as follows,
\[
\boxed{a \in S_{0,0}^0 \ \Rightarrow \ \nrm[L^2 \to L^2]{T_a} < \infty.}
\]

For the $L^p$-boundedness ($1 < p < +\infty$) result, readers may refer to \cite{Coi78au} (in French) and \cite{hwa94Lp}.

\section{G\aa rding's inequalities} \label{sec:Gar-PM2021}

We use notation $\Re f$ to signify the real-valued part of any object $f$.
Recall the Sobolev spaces $H^{s,p}(\Rn)$ defined in Definition \ref{defn:Hsp-PM2021}, and the corresponding Sobolev norms $\nrm[H^{m,p}]{\cdot}$ and $\nrm[H^m]{\cdot}$.
We denote $J^m := (I - \Delta)^{m/2}$ and $J := J^1$, namely, $J^m$ takes $\agl[\xi]^m$ as its symbol.
It can be checked that $J^m \in \Psi^m$, $J^{m_1} J^{m_2} = J^{m_1 + m_2}$, $J$ is self-adjoint, and $J^0$ is the identity operator.

\subsection{G\aa rding's Inequality} \label{subsec:GIn-PM2021}

\begin{defn}[Strongly elliptic\index{strong ellipticity}] \label{defn:GarIne-PM2021}
	Let $m \in \R$.
	A symbol $a$ is said to be \emph{strongly elliptic} of order $2m$, if $a \in S^{2m}$ and if there exist fixed positive constants $C$, $R$ such that
	\begin{equation*} 
	\boxed{ \Re a(x,\xi) \geq C \agl[\xi]^{2m}, \Forall |\xi| \geq R, }
	\end{equation*}
	holds.
\end{defn}

Similar to Lemma \ref{lem:aEv-PM2021}, there is an equivalent definition for the strong ellipticity of a symbol.

\begin{lem} \label{lem:GEv-PM2021}
	Assume $m \in \R$ and $a \in S^{2m}$.
	The strong ellipticity condition for $a$ is equivalent to the fact that there exist two positive constants $C$ and $D$ such that
	\begin{equation} \label{eq:GEv-PM2021}
	\boxed{\Re a(x,\xi) \geq C \agl[\xi]^{2m} - D \agl[\xi]^{2m-1}, \Forall x,\xi \in \Rn.}
	\end{equation}
\end{lem}

\begin{proof}
	Assume $a \in S^{2m}$ is strongly elliptic, then there are constants $C$, $R > 0$ such that
	\[
	\Re a(x,\xi) / \agl[\xi]^{2m}
	\geq C, \quad \forall |\xi| \geq R,
	\]
	so for any positive constant $D$ we have
	\begin{equation} \label{eq:rsm-PM2021}
	\Re a(x,\xi) / \agl[\xi]^{2m}
	\geq C - D \agl[\xi]^{-1},
	\end{equation}
	for $\forall |\xi| \geq R$.
	Also, because $a$ is a symbol of order $2m$, for some $M > 0$ we have,
	\[
	|\Re a(x,\xi)| / \agl[\xi]^{2m}
	\leq |a(x,\xi)| / \agl[\xi]^{2m}
	\leq M \quad \Rightarrow \quad
	\Re a(x,\xi) / \agl[\xi]^{2m}
	\geq -M, \quad \forall \xi \in \Rn.
	\]
	We set $D$ to be large enough such that
	\[
	-M \geq \sup_{|\xi| \leq R} (C - D \agl[\xi]^{-1}), \quad \text{e.g.} \quad D := (C + M) \agl[R],
	\]
	then \eqref{eq:rsm-PM2021} holds for both $|\xi| \geq R$ and $|\xi| \leq R$.
	This gives \eqref{eq:GEv-PM2021}.
	
	On the other hand, from \eqref{eq:GEv-PM2021} it is easy to see $a$ is strongly elliptic.
\end{proof}

We are ready for the G\aa rding's Inequality.

\begin{thm}[G\aa rding's inequality\index{G\aa rding's inequality}] \label{thm:GarIne-PM2021}
	Assume $m \in \R$ and the symbol $a \in S^{2m}$ is strongly elliptic.
	Then we can find a positive constant $C$ and a positive constant $C_s$ for every reals numbers $s \geq \frac{1}{2}$ such that
	\begin{equation} \label{eq:GarIne-PM2021}
	\boxed{ \Re (T_{a}\varphi,\varphi) \geq C \nrm[H^m]{\varphi}^2 - C_s \nrm[H^{m-s}]{\varphi}^2, \Forall \varphi \in \scrS(\Rn). }
	\end{equation}
\end{thm}

\begin{rem}
	When $a(x,\xi) = \agl[\xi]^{2m}$, then $a$ is strongly elliptic and
	\[
	\Re (T_{a}\varphi,\varphi)
	= \Re (J^{2m}\varphi,\varphi)
	= \Re (J^m \varphi, J^m \varphi)
	= \nrm[H^m]{\varphi}^2,
	\]
	which implies \eqref{eq:GarIne-PM2021}. 
	Theorem \ref{thm:GarIne-PM2021} implies that even if a symbol is not of the form $\agl[\xi]^{2m}$ but is only strongly elliptic of order $2m$, then $T_a$ still possesses some positiveness.
\end{rem}

\begin{proof}[Proof of Theorem \ref{thm:GarIne-PM2021}]
	Let's denote the symbol of $T_a^*$ as $a^*$, then it can be checked that
	\[
	\Re (T_{a}\varphi,\varphi)
	= (T_{\frac 1 2(a + a^*)} \varphi, \varphi).
	\]
	
	{\bf Step 1.} When $m = 0$.
	Because $a$ is strongly elliptic and $m = 0$, by Lemma \ref{lem:GEv-PM2021} and Theorem \ref{thm:AdT-PM2021} we have
	\[
	\frac 1 2(a + a^*)
	= \Re a + r
	\geq C - D \agl[\xi]^{-1} + r
	= C - r,
	\]
	where $r$ is a generic symbol in $S^{-1}$.
	This makes it legal to define a symbol\footnote{See Exercise \ref{ex:bs0-PM2021}.} $b \in S^0$ as follows,
	\begin{equation} \label{eq:bar-PM2021}
	b(x,\xi) := \big( \frac 1 2(a + a^*) - \frac 1 2 C + r \big)^{1/2}
	\end{equation}
	Then by Theorems \ref{thm:com-PM2021} \& \ref{thm:AdT-PM2021} we have (symbolic calculus)
	\[
	b^* \# b = (\overline b + S^{-1}) b + S^{-1}
	= b^2 + S^{-1}
	= \frac 1 2(a + a^*) - \frac 1 2 C + r + S^{-1},
	\]
	so
	\[
	\frac 1 2(a + a^*)= b^* \# b + \frac 1 2 C + r,
	\]
	where $r$ is a generic symbol in $S^{-1}$.
	Therefore,
	\begin{align}
	\Re (T_{a}\varphi,\varphi)
	& = (T_{\frac 1 2(a + a^*)} \varphi, \varphi)
	= (T_b^* T_b \varphi, \varphi) + \frac C 2 \nrm[L^2]{\varphi}^2 + (R \varphi, \varphi) \nonumber \\
	& = \nrm[L^2]{T_b \varphi}^2 + \frac C 2 \nrm[L^2]{\varphi}^2 + (R \varphi, \varphi)
	\geq \frac C 2 \nrm[L^2]{\varphi}^2 + (R \varphi, \varphi), \label{eq:RTa2-PM2021}
	\end{align}
	for some $R \in \Psi^{-1}$.
	We have
	\begin{align}
	|(R \varphi, \varphi)|
	& = |(R J^{1/2} J^{-1/2} \varphi, J^{1/2} J^{-1/2} \varphi)|
	= |(J^{1/2} R J^{1/2} (J^{-1/2} \varphi), J^{-1/2} \varphi)| \nonumber \\
	& \leq \nrm[L^2]{\underline{J^{1/2} R J^{1/2}} (J^{-1/2} \varphi)} \nrm[H^{-1/2}]{\varphi}
	\leq C' \nrm[L^2]{J^{-1/2} \varphi} \nrm[H^{-1/2}]{\varphi} \label{eq:JRBdd-PM2021} \\
	& = C' \nrm[H^{-1/2}]{\varphi}^2. \label{eq:Rvv-PM2021}
	\end{align}
	Note that in \eqref{eq:JRBdd-PM2021} we used the facts $J^{1/2} R J^{1/2} \in \Psi^0$ and operators in $\Psi^0$ are $L^2$-bounded.
	Combining \eqref{eq:Rvv-PM2021} with \eqref{eq:RTa2-PM2021} we arrive at
	\begin{equation} \label{eq:Tav-PM2021}
	\Re (T_{a}\varphi,\varphi)
	\geq \frac C 2 \nrm[L^2]{\varphi}^2 - C' \nrm[H^{-1/2}]{\varphi}^2.
	\end{equation}

	{\bf Step 2.} When $m \neq 0$.
	Let $T_{a'} = J^{-m} T_a J^{-m}$ for certain $a' \in S^0$.
	Then it can checked that there exist $C,D > 0$ so that
	\begin{equation} \label{eq:apa-PM2021}
	\Re a'(x,\xi) \geq C - D\agl[\xi]^{-1}, \quad \forall x, \xi \in \Rn.
	\end{equation}
	Hence according to Lemma \ref{lem:GEv-PM2021}, $a'$ is strongly elliptic, so by using the result in {\bf Step 1} we can have
	\begin{align}
	\Re (T_a \varphi,\varphi)
	& = \Re (J^m T_{a'} J^m \varphi, \varphi)
	= \Re (T_{a'} J^m \varphi,J^m \varphi) \nonumber \\
	& \geq \frac C 2 \nrm[L^2]{J^m \varphi}^2 - C' \nrm[H^{-1/2}]{J^m \varphi}^2 \nonumber \\
	& = \frac C 2 \nrm[H^m]{\varphi}^2 - C' \nrm[H^{m-1/2}]{\varphi}^2. \label{eq:Tav2-PM2021}
	\end{align}
	
	{\bf Step 3.}
	From Theorem \ref{thm:SNI-PM2021} we have	
	\[
	\nrm[H^{m-1/2}]{\varphi}^2 \leq \frac{1}{D^{1/2}} \nrm[H^m]{\varphi}^2 + D^{s-1/2} \nrm[H^{m-s}]{\varphi}^2,
	\]
	for any $D > 0$.
	Set $D$ to be small enough and substitute the inequality above into \eqref{eq:Tav-PM2021} and \eqref{eq:Tav2-PM2021}, we arrive at the conclusion.
\end{proof}

G\aa rding's Inequality is used for giving the existence and uniqueness of the following type equation:
\[
(T_a + \lambda I)u = f.
\]
Let $m \geq 1/2$ and $s = m$, and assume $a \in S^{2m}$ is strongly elliptic symbol, then
\[
C \nrm[H^m]{\varphi}^2 - \lambda_0 \nrm[L^2]{\varphi}^2
\leq \Re (T_{a}\varphi,\varphi), \quad \forall \varphi \in \scrS(\Rn)
\]
for some constant $\lambda_0 > 0$, then for all $\lambda > \lambda_0$, we can conclude
\[
C \nrm[L^2]{\varphi}^2
\leq C \nrm[H^m]{\varphi}^2
\leq \Re ((T_a + \lambda) \varphi,\varphi)
= \Re (\varphi,(T_a^* + \lambda) \varphi)
\leq \nrm[L^2]{\varphi} \nrm[L^2]{(T_a^* + \lambda) \varphi},
\]
which leads to a coercive condition:
\[
C \nrm[L^2]{\varphi}
\leq \nrm[L^2]{(T_a^* + \lambda) \varphi}.
\]
Combining this with the Lax-Milgram theorem we can conclude that:

\begin{cor} \label{cor:ufU-PM2021}
	Assume $m \geq 1/2$ and $a \in S^{2m}$ is strongly elliptic.
	There exists a constant $\lambda_0$ such that when any $\lambda > \lambda_0$,
	for any $f \in L^2(\Rn)$ there exists a unique weak solution $u \in L^2(\Rn)$ satisfying the equation
	\[
	(T_a + \lambda) u = f.
	\]
\end{cor}

\subsection{Sharp G\aa rding's Inequality} \label{subsec:GInS-PM2021}

In the proof of Theorem \ref{thm:GarIne-PM2021} later on, we see that having a strictly positive lower bound for $\Re a$ is critical, and the method in that proof will fail if the lower bound reduces to zero.
However, when $\Re a \geq 0$, one can still obtain some lower bound of $\Re (T_{a}\varphi,\varphi)$ and that result is called sharp G\aa ding's inequality.

\begin{thm}[Sharp G\aa rding's Inequality\index{sharp G\aa rding's Inequality}] \label{thm:GInS-PM2021}
	For a symbol $a \in S^{2m}$ satisfying 
	\begin{equation} \label{eq:GInS-PM2021}
	\boxed{ \Re a(x,\xi) \geq 0, \Forall |\xi| \geq R,}
	\end{equation}
	we can find a positive constant $C$ such that
	\begin{equation*}
	\boxed{ \Re (T_{a}\varphi,\varphi) \geq - C \nrm[H^{m-1/2}]{\varphi}^2, \Forall \varphi \in \scrS(\Rn). }
	\end{equation*}
\end{thm}

The prove \eqref{thm:GInS-PM2021}, we introduce the \emph{\underline{wave packet transform}}.
The wave packet transform $W \colon L^2(\Rn) \to L^2(\Rn \times \Rn)$ is defined as (see \cite[Theorem 4.2.3]{chen2006pseudodifferential})
\begin{equation} \label{eq:wpt-jpj}
W u(z,\xi) := c_n \agl[\xi]^{n/4} (e^{i(\cdot) \cdot \xi - \agl[\xi]|\cdot|^2} *_z u),
\end{equation}
and its conjugate in terms of the $L^2$-inner product is given by,
\begin{equation} \label{eq:wptC-jpj}
W^* F(x) = c_n \int \agl[\xi]^{n/4} (e^{i(\cdot) \cdot \xi - \agl[\xi]|\cdot|^2} *_x F(\cdot,\xi)) \dif \xi,
\end{equation}
where the constant $c_n = 2^{-n/4} \pi^{-3n/4}$ and $(f *_x g)$ signifies $\int f(x-y) g(y) \dif y$.

\begin{lem}
	The wave packet transform $W$ defined in \eqref{eq:wpt-jpj} is a bounded linear operator.
\end{lem}

\begin{proof}
	The linearity is obvious.
	
	To show the boundedness, we compute $\nrm[L^2(\Rn \times \Rn)]{Wu}^2$,
	\begin{align*}
	\nrm[L^2(\Rn \times \Rn)]{Wu}^2
	& \simeq \iint | \agl[\xi]^{n/4} (e^{i(\cdot) \cdot \xi - \agl[\xi]|\cdot|^2} *_z u) |^2 \dif z \dif \xi \\
	& = \int \agl[\xi]^{n/2} \int |\calF \{ (e^{i(\cdot) \cdot \xi - \agl[\xi]|\cdot|^2}\}(\eta)|^2 \cdot |\hat u(\eta)|^2 \dif \eta \dif \xi \quad (\text{Plancherel theorem}) \\
	& = \int \big( \int \agl[\xi]^{n/2} |\calF \{ (e^{i(\cdot) \cdot \xi - \agl[\xi]|\cdot|^2}\}(\eta)|^2 \dif \xi \big) \cdot |\hat u(\eta)|^2 \dif \eta \\
	& \simeq \int \big( \int \agl[\xi]^{n/2} |\agl[\xi]^{-n/2} e^{- \agl[\xi]^{-1} |\eta - \xi|^2/4}|^2 \dif \xi \big) \cdot |\hat u(\eta)|^2 \dif \eta \\
	& = \int \big( \int \agl[\xi]^{-n/2} e^{- \agl[\xi]^{-1} |\eta - \xi|^2/2} \dif \xi \big) \cdot |\hat u(\eta)|^2 \dif \eta \\
	& \leq \int \big( \int e^{-|\eta - \xi|^2/2} \dif \xi \big) \cdot |\hat u(\eta)|^2 \dif \eta
	\lesssim \int |\hat u(\eta)|^2 \dif \eta \\
	& = \nrm[L^2(\Rn)]{u}^2.
	\end{align*}
	The proof is complete.
\end{proof}

\begin{proof}[Proof of Theorem \ref{thm:GInS-PM2021}]
	
	similar to Proof of Theorem \ref{thm:GarIne-PM2021}, the general cases w.r.t.~$m$ stem from the special case where $m = 1/2$.
	Let's assume $m = 1/2$ for the time being and try to show $\Re (T_{a}\varphi,\varphi) \gtrsim -\nrm[-1/2,2]{\varphi}^2$.

	The condition \eqref{eq:GInS-PM2021} can be replaced by ``$\Re a(x,\xi) \geq 0,\, \forall \xi \in \Rn$'', 
	and this is because we can fix some $\chi \in C_c^\infty(\Rn;\Rn)$ (thus $\chi \in S^{-\infty}$) satisfying $\chi(\xi) \geq \sup_{(x,\xi)} \Re a(x,\xi)$ when $\{ |\xi| \leq R \}$, 
	and then we can obtain $\Re a(x,\xi) + \chi(\xi) \geq 0,\, \forall \xi \in \Rn$.
	Note that $\Re (\chi \varphi,\varphi) \gtrsim -\nrm[L^2]{\varphi}^2$ because $\chi \in S^{-\infty}$.
	Therefore, from now on we assume $\Re a(x,\xi) \geq 0,\, \forall \xi \in \Rn$.

	Denote as $b(x,\xi)$ the symbol of $W^* \Re a W$.
	The operator $T_b := W^* \Re a W$ is defined by $T_b \varphi(x) = W^* (\Re a \cdot W\varphi)(x)$ for $\varphi \in \scrS(\Rn)$.
	We have
	\begin{align}
	T_b \varphi(x)
	& \simeq \int \agl[\xi]^{n/4} \int e^{i(x - y) \cdot \xi - \agl[\xi]|x - y|^2} \Re a(y,\xi) W \varphi(y,\xi) \dif y \dif \xi \nonumber \\
	& \simeq \int \agl[\xi]^{n/4} \int e^{i(x - y) \cdot \xi - \agl[\xi]|x - y|^2} \Re a(y,\xi) \agl[\xi]^{n/4} \int e^{i(y-z) \cdot \xi - \agl[\xi]|y-z|^2} \varphi(z) \dif z \dif y \dif \xi \nonumber \\
	& \simeq \iint e^{i(x - z) \cdot \xi} (\int e^{- \agl[\xi](|y-z|^2 + |x-y|^2)} \agl[\xi]^{n/2} \Re a(y,\xi) \dif y) \varphi(z) \dif z \dif \xi \nonumber \\
	& = \iint e^{i(x - z) \cdot \xi} \tilde a(x,z,\xi) \varphi(z) \dif z \dif \xi, \label{eq:Tbta-jpj}
	\end{align}
	where $\tilde a(x,z,\xi) := \agl[\xi]^{n/2} \int e^{- \agl[\xi](|y-z|^2 + |x-y|^2)} \Re a(y,\xi) \dif y$.
	It can be checked that $\tilde a \in S_{1,1/2}^1$.
	By \cite[Theorem 2.4.1]{chen2006pseudodifferential}, we have the asymptotic expansion
	\begin{equation} \label{eq:taexp-jpj}
	b(x,\xi) = \tilde a(x,x,\xi) + \sum_{j = 1}^n e_j(x,\xi) + r, \quad r \in S_{1,1/2}^0,
	\end{equation}
	where $e_j(x,\xi) = \partial_{z_j} D_{\xi_j} \tilde a(x,z,\xi) |_{z = x}$.
	Note that $e_j$ belongs to $S_{1,1/2}^{1/2}$, not $S_{1,1/2}^0$, and this is why we expand $\tilde a$ to the second order.
	The symbols $e_j$ is purely imaginary because $\tilde a$ is real.

	For $\tilde a(x,x,\xi)$, we have
	\begin{align}
	\tilde a(x,x,\xi)
	& = \agl[\xi]^{n/2} \Re \int e^{-2\agl[\xi] |x-y|^2} a(y,\xi) \dif y \nonumber \\
	& = \agl[\xi]^{n/2} \Re \int e^{-2\agl[\xi] |x-y|^2} \big[ a(x,\xi) + \sum_{j} (y-x)^{(j)} \partial_{x_j} a(x,\xi) \nonumber \\
	& \quad + \sum_{|\alpha| = 2} (y-x)^{\alpha} \partial_{x}^\alpha a(\rho x + (1-\rho) y,\xi)/2 \big] \dif y \nonumber \\
	& \simeq \Re a(x,\xi) + \agl[\xi]^{n/2} \Re\partial_{x_j} a(x,\xi) \sum_{j} \int e^{-2\agl[\xi] |x-y|^2} (y-x)^{(j)} \dif y \nonumber \\
	& \quad + \agl[\xi]^{n/2} \sum_{|\alpha| = 2} \Re \int e^{-2\agl[\xi] |x-y|^2} (y-x)^{\alpha} \cdot \partial_{x}^\alpha a(\rho x + (1-\rho) y,\xi)/2 \dif y \nonumber \\
	& = \Re a(x,\xi) + \agl[\xi]^{n/2} \sum_{|\alpha| = 2} \Re \int e^{-2\agl[\xi] |x-y|^2} (y-x)^{\alpha} \cdot \partial_{x}^\alpha a(\rho x + (1-\rho) y,\xi)/2 \dif y \nonumber \\
	& = \Re a(x,\xi) + r', \quad r' \in S^0. \label{eq:taexp1-jpj}
	\end{align}
	The last equal sign in \eqref{eq:taexp1-jpj} is due to the following computation,
	\begin{align*}
	& \ |\agl[\xi]^{n/2} \int e^{-2\agl[\xi] |x-y|^2} (y-x)^{\alpha} \cdot \partial_{x}^\alpha a(\rho x + (1-\rho) y,\xi)/2 \dif y| \\
	\lesssim & \ \agl[\xi]^{n/2} \int e^{-2\agl[\xi] |x-y|^2} |(y-x)^{\alpha}| \cdot \agl[\xi] \dif y \\
	= & \ \int e^{-2|\agl[\xi]^{1/2} y|^2} |(\agl[\xi]^{1/2} y)^{\alpha}| \dif (\agl[\xi]^{1/2} y) \\
	= & \ \int e^{-2|y|^2} |y^{\alpha}| \dif y \leq C
	\end{align*}
	for some positive constant $C$.

	Combining \eqref{eq:taexp-jpj} and \eqref{eq:taexp1-jpj}, we obtain
	\begin{equation} \label{eq:Res1-jpj}
	b(x,\xi) = \Re a(x,\xi) + \sum_{j = 1}^n e_j(x,\xi) + r, \quad e_j \in S_{1,1/2}^{1/2}, \quad r \in S_{1,1/2}^0,
	\end{equation}
	and thus
	\begin{equation} \label{eq:sig-jpj}
	(a + a^*)/2 = b(x,\xi) - \sum_{j = 1}^n e_j(x,\xi) - r, \quad r \in S_{1,1/2}^0.
	\end{equation}
	The $r$ in \eqref{eq:Res1-jpj} and \eqref{eq:sig-jpj} are different from each other and are also different from the $r$ in \eqref{eq:taexp-jpj}.
	Now we have
	\begin{align}
	\Re (T_a \varphi, \varphi)
	& = [(T_a \varphi, \varphi) + \overline{(T_a \varphi, \varphi)}]/2
	= [(T_a \varphi, \varphi) + (\varphi, T_a \varphi)]/2
	= (T_{(a + a^*)/2} \varphi, \varphi) \nonumber \\
	& = (T_b \varphi, \varphi) - (T_{e_j} \varphi, \varphi) - (T_r \varphi, \varphi) \nonumber \\
	& = (\Re a \cdot W \varphi, W\varphi) - (T_{e_j} \varphi, \varphi) - (T_r \varphi, \varphi) \nonumber \\
	& \geq - (T_{e_j} \varphi, \varphi) - (T_r \varphi, \varphi)  \quad (\text{because~} \Re a \geq 0) \nonumber \\
	& = - \Re (T_{e_j} \varphi, \varphi) - \Re (T_r \varphi, \varphi)
	= - (T_{(e_j + e_j^*)/2} \varphi, \varphi) - \Re (T_r \varphi, \varphi) \nonumber \\
	& \geq - (T_{(e_j + e_j^*)/2} \varphi, \varphi) - C \nrm[L^2]{\varphi}^2,
	\label{eq:ReT1-jpj}
	\end{align}
	for some positive constant $C$.
	The $L^2$-boundedness of operators whose symbol come from $S_{1,1/2}^{0}$ can be proved in a similar manner as in the proof of that of $S^0$, cf.~\cite[Theorem 4.1.1]{chen2006pseudodifferential} and \cite[Theorem 5.1]{alinhac2007pseudo}.

	Recall that $e_j \in S_{1,1/2}^{1/2}$ and $e_j$ is purely imaginary, thus the principal symbol of $T_{(e_j + e_j^*)/2}$ equals to zero and hence $(e_j + e_j^*)/2 \in S_{0,1/2}^{1/2}$.
	Therefore,
	\begin{equation} \label{eq:ej-jpj}
	(T_{(e_j + e_j^*)/2} \varphi, \varphi)
	\leq C \nrm[L^2]{\varphi}^2
	\end{equation}
	for some positive constant $C$.
	
	Combining \eqref{eq:ReT1-jpj} and \eqref{eq:ej-jpj}, we arrive at the conclusion for the case $m = 1/2$.
	Based on the result regarding $m = 1/2$, the proof of general cases become trivial.
\end{proof}

\section*{Exercise}

\begin{ex}
	Prove Lemma \ref{lem:PNUi-PM2021}.
\end{ex}

\begin{ex} \label{ex:1as-PM2021}
	Assume $m \in \R$ and $a \in S^m$ and $a$ is elliptic.
	Fix a cutoff function $\chi \in C^\infty(\Rn)$ such that $\chi(\xi) = R$ when $|\xi| \leq 1$ and $\chi(\xi) = 0$ when $|\xi| \geq R+1$, where the $R$ comes from the definition of the ellipticity of $a$.
	Define $r_0(x,\xi) := (1-\chi(\xi))/a(x,\xi)$.
	Prove that $r_0 \in S^{-m}$.
\end{ex}

\begin{ex} \label{ex:bs0-PM2021}
	Prove the $b(x,\xi)$ defined in \eqref{eq:bar-PM2021} is indeed a symbol and is of order $0$.
	Hint: use \cite[Lemma 2.1.1]{alinhac2007pseudo} or \cite[Lemma 17.2]{wong2014introduction}.
\end{ex}

\begin{ex}
	Prove \eqref{eq:apa-PM2021} is true.
\end{ex}

\begin{ex}
	Assume the symbols $a$ and $b$ are elliptic.
	Show that $T_a T_b$ and $T_a^*$ are also elliptic. 
	Hint: utilize Lemma \ref{lem:aEv-PM2021}.
\end{ex}

\chapter{Semi-classical \texorpdfstring{$\Psi$DOs}{PsiDOs} and its symbolic calculus} \label{ch:SCla-PM2021}

Semiclassical analysis shares lots of features with $\Psi$DO theory, while also keeping some of its own specialties.
One of the application of semiclassical analysis is Carleman estimates.

\section{Semi-classical \texorpdfstring{$\Psi$DOs}{PsiDOs}} \label{sec:ScP-PM2021}

\subsection{Symbol classes} \label{subsec:Sycl-PM2021}

\begin{defn}[Order function] \label{defn:orfu-PM2021}
	A measurable function $m \colon \R^{2n} \to \R^+$ is call an \emph{order function} if there exist constants $C > 0$ and $N \in \mathbb N$ such that
	\[
	\boxed{m(z_1 - z_2) \leq C \agl[z_1]^N m(z_2), \quad \forall z_1, z_2 \in \R^{2n}.}
	\]
	The integer $N$ is called the order of $m$.
\end{defn}

For any $a$, $b \in \R$, $m(x,\xi) = \agl[x]^a \agl[\xi]^b$ are an order functions with $N = 2 \max\{|a|, |b|\}$.
If $m_1$, $m_2$ are order functions, so does $m_1 m_2$.

\begin{defn}[Semiclassical symbol class] \label{defn:Sycl-PM2021}
	Let $h \in (0,1)$, $\delta \in [0, \frac 1 2]$ and $m$ be an order function with order $N$.
	For $a(\cdot ; h) \in C^\infty(\R^{2n})$, we say $a \in S_\delta(m)$ with order $N$ if
	\[
	\boxed{|\partial_z^\alpha a(z;h)| \leq C_\alpha h^{-\delta|\alpha|} m(z), \quad \forall z \in \R^{2n}.}
	\]
	Define a family of seminorms
	\[
	|a(\cdot;h)|_{S_\delta(m), \alpha} = |a|_{\alpha} := \sup_{z \in \R^{2n}} \frac {|\partial_z^\alpha a(z;h)|} {h^{-\delta|\alpha|} m(z)},
	\]
	and so the semiclassical symbol class $S_\delta(m)$ is given by
	\[
	S_\delta(m) := \{a(z;h) \in C^\infty \,;\, \forall \text{~multi-index~} \alpha,\, |a|_{\alpha} < +\infty \}.
	\]
	We abbreviate $S_\delta(1)$ as $S_\delta$ and $S_0(1)$ as $S$.
\end{defn}

Note that in contrast to the Kohn-Nirenberg symbol (cf.~Definitions \ref{defn:symbol-PM2021} \& \ref{defn:symbolx-PM2021}), the semiclassical symbol doesn't gain decay w.r.t.~its arguments after being differentiated.

We write $a(\cdot;h) = \mathcal O_{S_\delta(m)}(f(h))$ if for every multi-index $\alpha$, there exist $h_0$ and $a$ such that $|a|_\alpha \leq C_\alpha f(h)$ holds for all $h \in (0,h_0)$, namely,
\begin{equation*}
\boxed{a(\cdot;h) = \mathcal O_{S_\delta(m)}(f(h))
	\ \Leftrightarrow \
	|\partial^\alpha a(\cdot;h)| \lesssim f(h) h^{-\delta|\alpha|} m.}
\end{equation*}
It can be checked that $\forall a \in S_\delta(m)$, we have $f(h) a = O_{S_\delta(m)}(f(h))$.
For $S_\delta(m)$ and $a_j \in S_\delta(m)$, we write $a \sim \sum_j h^j a_j$ if $a - \sum_{j = 0}^{N} h^j a_j = \mathcal O_{S_\delta(m)}(h^{N+1})$.

\begin{lem} \label{lem:Sycp-PM2021}
	Assume $0 \leq \delta \leq \frac 1 2$ and $a \in S_\delta(m)$ and $a_1 \in S_\delta(m_1)$, $a_2 \in S_\delta(m_2)$.
	Then $hD_j a \in S_\delta(m)$ and $a_1 a_2 \in S_\delta(m_1 m_2)$.
\end{lem}

\begin{proof}
	We can compute
	\begin{align*}
	|\partial^\alpha (hD_j a)|
	& = h |\partial^{\alpha + e_j} a|
	\leq C_\alpha h h^{-\delta(|\alpha + e_j|)} m
	= C_\alpha h h^{-\delta|\alpha| - \delta} m
	= C_\alpha h^{1-\delta} h^{-\delta|\alpha|} m \\
	& \leq h_0 C_\alpha h^{-\delta|\alpha|} m.
	\end{align*}
	Hence $hD_j a \in S_\delta(m)$.
	We omit the rest of the proof.
\end{proof}

\begin{defn}[Asymptotics] \label{defn:Asy2-PM2021}
	For symbol $a$, $a_j \in S_\delta(m)~(j= 0,1,\cdots)$, we write $a \sim \sum_j h^{(1-2\delta)j} a_j$ in $S_\delta(m)$ if $a - \sum_{j = 0}^{N} h^{(1-2\delta)j} a_j = \mathcal O_{S_\delta(m)}(h^{(1-2\delta)(N+1)})$ holds for every $n \in \mathbb N$, namely,
	\begin{equation*}
	\boxed{a \sim \sum_j h^{(1-2\delta)j} a_j \ \text{in} \ S_\delta(m)
		\quad \Leftrightarrow \quad
		a = \sum_{j = 0}^{N} h^{(1-2\delta)j} a_j + h^{(1-2\delta)(N+1)} S_\delta(m).}
	\end{equation*}
	Here $h^{(1-2\delta)(N+1)} S_\delta(m)$ means $h^{(1-2\delta)(N+1)} r$ for some $r \in S_\delta(m)$.
	The $a_0$ is called the \emph{principal symbol} of $a$.
\end{defn}

The asymptotics is more about $h$ than the Kohn-Nirenberg symbol which is more about $\xi$.
To avoid confusion, we would like to comment in advance that even though the definition of asymptotics semiclassical symbol is in the form $a \sim a_0 + h^{1-2\delta} a_1 + h^{(1-2\delta)2} a_2 + \cdots$, but later we may see $a$ be expressed as $b_0 + h b_1 + h^2 b_2 + \cdots$, e.g.~in \eqref{eq:SeES-PM2021}.
The difference is that it is $h^{2\delta j} b_j$ rather than $b_j$ itself that is in $S_{\delta}(m)$.

\begin{thm}
	For $\forall a_j \in S_\delta(m)~(j = 0,1, \cdots)$, there always exists $a \in S_\delta(m)$ such that $a \sim \sum_j h^{(1-2\delta)j} a_j$.
\end{thm}

\begin{proof}
	We choose a cutoff function $\chi \in C_c^\infty(\R)$ satisfying $\chi \equiv 1$ in $(-1,1)$, $0 \leq \chi \leq 1$, $\chi$ is decreasing in the interval $(1,2)$ and $\supp \chi \subset (-2,2)$.
	Note that we define $\chi$ on the whole real axis but will only use its definition on the positive real axis.
	
	{\bf Step 1.}
	Define
	\[
	a := \sum_{j \geq 0} \chi(\lambda_j h) h^{(1-2\delta)j} a_j
	\]
	for some $\lambda_j$ which shall be determined.
	Our scheme is to choose $\lambda_j > 0$ properly (grows fast enough) such that $a$ will be well-defined at each point and satisfies Definition \ref{defn:Asy2-PM2021}.
	From the construction of $\chi$ it can be checked that
	\begin{equation} \label{eq:chTri1-PM2021}
	\forall h \geq 0,\, \forall k \geq 0, \quad \chi(h)h^{(1-2\delta)k} \leq 2^{(1-2\delta)k}.
	\end{equation}
	Hence,
	\begin{align*}
	|a|
	& = |\sum_{j \geq 0} \chi(\lambda_j h) h^{(1-2\delta)j} a_j|
	= |\sum_{j \geq 0} \chi(\lambda_j h) (\lambda_j h)^{(1-2\delta)j} (\lambda_j h)^{-(1-2\delta)j} h^{(1-2\delta)j} a_j| \\
	& \leq \sum_{j \geq 0} (2/\lambda_j)^{(1-2\delta)j} |a_j|,
	\end{align*}
	and similarly,
	\begin{equation*}
	|\partial^\alpha a|
	\leq \sum_{j \geq 0} (2/\lambda_j)^{(1-2\delta)j} |\partial^\alpha a_j|
	\leq \sum_{j \geq 0} C_{j,\alpha} (2/\lambda_j)^{(1-2\delta)j} \cdot h^{-\delta|\alpha|} m.
	\end{equation*}

	{\bf Step 2.}
	For a specific $\alpha$, we only need to choose $\{\lambda_{j,\alpha}\}_{j \geq 0}$ grow fast enough such that
	\[
	\sum_{j \geq 0} C_{j,\alpha} (2/\lambda_{j,\alpha})^{(1-2\delta)j}
	\]
	is finite, and one example is $\lambda_{j,\alpha} = 3 C_{j,\alpha}^{1/[(1-2\delta)j]}$.
	Then using diagonal arguments we could choose a suitable set $\{\lambda_j\}$ from $\{\lambda_{j,\alpha}\}_{j \geq 0}$.
	However, we want $\{\lambda_j\}$ to grow even more faster for our later use; particularly, to guarantee \eqref{eq:lajN-PM2021} is finite.
	To that end, for each fixed multi-index $\alpha$ and non-negative integer $M$, we first choose $\{\lambda_{j,\alpha,M}\}_{j \geq 0}$ to grpw fast enough w.r.t.~$j$ such that
	\begin{equation}
	\left\{\begin{aligned}
	& \sum_{j \geq 0} C_{j+M,\alpha} (2/\lambda_{j,\alpha,M})^{(1-2\delta)j} < +\infty, \\
	& \lambda_{j,\alpha,M'} \geq \lambda_{j,\alpha,M} \text{~when~} M' \geq M, \\
	& \lambda_{j,\alpha',M} \geq \lambda_{j,\alpha,M} \text{~when~} \alpha' \geq \alpha,
	\end{aligned}\right.
	\end{equation}
	then we choose
	\[
	\lambda_j := \lambda_{j,(j,j,\cdots,j),j}
	\]
	where $(j,j,\cdots,j)$ stands for the multi-index of which the value of every component is $j$.
	By doing so, we are guaranteed that the sum $\sum_{j \geq 0} C_{j+M,\alpha} (2/\lambda_{j})^{(1-2\delta)j}$ is finite for every $\alpha$ and $M$.
	Back to the estimate of $|\partial^\alpha a|$, we are guaranteed that  $a$ is well-defined and $a \in S_\delta(m)$.
	It remains to show $a \sim \sum_j h^{(1-2\delta)j} a_j$ in $S_\delta(m)$.

	{\bf Step 3.}
	To analyze $a - \sum_{j = 0}^{N} h^j a_j$, we use another trick similar to \eqref{eq:chTri1-PM2021},
	\begin{equation} \label{eq:chTri2-PM2021}
	\forall h \geq 0,\, \forall k \geq 0, \quad |\chi(h) - 1| h^{-k} \leq 1.
	\end{equation}
	The verification of \eqref{eq:chTri2-PM2021} is left as an exercise.
	By 
	\eqref{eq:chTri1-PM2021} and 
	\eqref{eq:chTri2-PM2021}, for $h \in (0,h_0)$ where $h_0 < 1$ we have
	\begin{align}
	& |\partial^\alpha(a - \sum_{j = 0}^{N} h^{(1-2\delta)j} a_j)| \nonumber \\
	\leq & \sum_{j = 0}^{N} |\chi(\lambda_j h) - 1| \,h^{(1-2\delta)j}\, |\partial^\alpha a_j| + \sum_{j \geq N+1} \chi(\lambda_j h) h^{(1-2\delta)j}\, |\partial^\alpha a_j| \nonumber \\
	\leq & \sum_{j = 0}^{N} |\chi(\lambda_j h) - 1| (\lambda_j h)^{-(1-2\delta)(N+1-j)} \cdot (\lambda_j h)^{(1-2\delta)(N+1-j)} h^{(1-2\delta)j} |\partial^\alpha a_j| \nonumber \\
	& + \sum_{j \geq 0} \chi(\lambda_{j} h) (\frac {\lambda_{j} h} {2})^{(1-2\delta)j} \cdot C_{j+N+1,\alpha} (\frac {2} {\lambda_{j} h})^{(1-2\delta)j} h^{(1-2\delta) (j+N+1)} h^{-\delta|\alpha|} m \label{eq:laj2-PM2021} \\
	\leq & \big[\sum_{j = 0}^{N} \lambda_j^{(1-2\delta)(N+1-j)} C_{j,\alpha} + \sum_{j \geq 0}  C_{j+N+1,\alpha} (\frac {2} {\lambda_{j}})^{(1-2\delta)j} \big] h^{(1-2\delta) (N+1)}  h^{-\delta|\alpha|} m \label{eq:lajN-PM2021} \\
	\leq & \tilde C_{N,\alpha,h_0} h^{(1-2\delta) (N+1)}  h^{-\delta|\alpha|} m. \nonumber
	\end{align}
	Here in \eqref{eq:laj2-PM2021} we used $\chi(\lambda_j h) \leq \chi(\lambda_{j-1} h)$.
	Hence,
	\[
	a - \sum_{j = 0}^{N} h^{(1-2\delta)j} a_j = \mathcal O_{S_\delta(m)}(h^{(1-2\delta)(N+1)}).
	\]
	The proof is complete.
\end{proof}

\subsection{Semiclassical pseudodifferential operators} \label{subsec:SPsiDO-PM2021}

Just as Kohn-Nirenberg symbols, every semiclassical symbol produces an operator, and is semiclassical situation, these operators are also described as quantizations of the corresponding symbols.

\begin{defn}[Quantization] \label{defn:QDe-PM2021}
	We quantize the symbol $a(x,\xi)$ by means of \eqref{eq:QDe3-PM2021} for $\forall t \in [0,1]$.
	And we also denote \emph{Standard quantization} and \emph{Weyl quantization} as in \eqref{eq:QDe2-PM2021}-\eqref{eq:QDe1-PM2021},
	\begin{alignat}{2}
	\text{(Standard quant.:)} & \quad & a(x,hD) u
	& := (2\pi h)^{-n} \int e^{i(x-y) \cdot \xi/h} a(x, \xi) u(y) \dif y \dif \xi, \label{eq:QDe2-PM2021} \\
	\text{(Weyl quant.:)} & \quad & a^w(x,hD) u
	& := (2\pi h)^{-n} \int e^{i(x-y) \cdot \xi/h} a(\frac {x+y} 2, \xi) u(y) \dif y \dif \xi, \label{eq:QDe1-PM2021} \\
	\text{(General quant.:)} & \quad & a_t(x,hD) u
	& := (2\pi h)^{-n} \int e^{i(x-y) \cdot \xi/h} a(tx + (1-t)y, \xi) u(y) \dif y \dif \xi. \label{eq:QDe3-PM2021}
	\end{alignat}
	These operators defined above are called \emph{semiclassical pseudodifferential operators} (abbreviated as {\rm S$\Psi$DOs}).
	We denote the set of \underline{\rm S$\Psi$DO} with symbols coming from $S_\delta(m)$ as ${\rm Op}_h(S_\delta(m))$.
\end{defn}

According to Definition \ref{defn:QDe-PM2021} we know that $a^w(x,hD) = a_{\frac 1 2}(x,hD)$ and $a(x,hD) = a_{1}(x,hD)$.
It is trivial to see
\[
a(x,\xi) = f(x) \xi_j \quad \Leftrightarrow \quad a(x,hD) = f(x) hD_j.
\]

We introduce the $h$-dependent Fourier transform.

\begin{defn} \label{defn:SFT-PM2021}
	The \emph{semiclassical Fourier transform} $\calF_h$ and its inverse $\calF_h^{-1}$ are defined as
	\begin{align}
	\calF_h u(\xi)
	& := (2\pi h)^{-n/2} \int_{\Rn} e^{-ix \cdot \xi/h} u(x) \dif x, \label{eq:SFT-PM2021} \\
	\calF_h^{-1} u(x)
	& := (2\pi h)^{-n/2} \int_{\Rn} e^{ix \cdot \xi/h} u(\xi) \dif \xi. \label{eq:SFTi-PM2021}
	\end{align}
\end{defn}

It can be checked that
\begin{equation} \label{eq:QFa-PM2021}
\boxed{a(x,hD) u := \calF_h^{-1} \{ a(x,\cdot) \calF_h u(\cdot) \}.}
\end{equation}
Formula \eqref{eq:QFa-PM2021} is one of the reason why the semiclassical Fourier transform shall defined as in Definition \ref{defn:SFT-PM2021}.

\begin{lem} \label{lem:QSdel-PM2021}
	Assume $\delta \in \R$ and $a \in S_\delta(m)$.
	Then for $\forall t\in [0,1]$, we have that the operator $a_t$ satisfies $a_t(x, hD) \colon \scrS(\Rn) \to \scrS(\Rn)$ and $a_t(x, hD) \colon \scrS'(\Rn) \to \scrS'(\Rn)$, and the mappings are bounded with norm depending on $\delta$ and $h$, but uniformly on $t$.
\end{lem}

\begin{proof}
	Let $\varphi \in \scrS(\Rn)$.
	We have
	\[
	a_t(x, hD) \varphi(x)
	= (2\pi h)^{-n} \iint e^{i(x-y) \cdot \xi /h} a(tx + (1-t)y,\xi;h) \varphi(y) \dif y \dif \xi.
	\]
	The integrability of $y$ is not a problem because $\varphi(y)$ is rapidly decay.
	For $\xi$, we should use integration by parts to gain enough decay on $\xi$.
	Notice that $\frac {1 - \xi \cdot hD_y} {\agl[\xi]^2} (e^{i(x-y) \cdot \xi /h}) = e^{i(x-y) \cdot \xi /h}$, we denote $L_1 = \frac {1 - \xi \cdot hD_y} {\agl[\xi]^2}$, then act $L_1^{n+1}$ on $e^{i(x-y) \cdot \xi /h}$ and use integration by parts, we will end up in a integrand of order $\agl[\xi]^{-n-1}$ on $\xi$ and rapidly decay on $y$, thus integrable.
	Hence we proved that $a_t(x, hD) \colon \scrS(\Rn) \to L^\infty(\Rn)$ for $a \in S_\delta(m)$.
	Adopt similar arguments on $x^\alpha \partial^\beta a_t(x, hD)$, we can obtain $x^\alpha \partial^\beta a_t(x, hD) \colon \scrS(\Rn) \to L^\infty(\Rn)$ for $\forall \alpha, \beta$.
	Therefore $a_t(x, hD) \colon \scrS(\Rn) \to \scrS(\Rn)$.
	And the continuity of the operator can also be seen from the arguments above.
	
	The second result holds due to duality arguments.
\end{proof}


\begin{lem} \label{lem:Qxind-PM2021}
	If symbol $a$ is independent of $\xi$, i.e.~$a(x,\xi) = a(x)$, then
	\[
	a_t(x,hD) u(x) = a(x) u(x), \quad \forall t \in [0,1].
	\]
\end{lem}

\begin{proof}
	It is enough to prove for $u \in \scrS$.
	When $t = 1$, we have
	\begin{align*}
	a_1(x,hD) u(x)
	& = (2\pi h)^{-n} \iint e^{i(x-y) \cdot \xi/h} a(x, \xi) u(y) \dif y \dif \xi \\
	& = (2\pi h)^{-n} \iint e^{i(x-y) \cdot \xi/h} a(x) u(y) \dif y \dif \xi \\
	& = a(x) u(x).
	\end{align*}
	We have
	\begin{align*}
	\partial_t \big( a_t(x,hD) u(x) \big)
	& = (2\pi h)^{-n} \iint e^{i(x-y) \cdot \xi/h} \partial_t \big( a(tx + (1-t)y) \big) u(y) \dif y \dif \xi \\
	& = (2\pi h)^{-n} \iint e^{i(x-y) \cdot \xi/h} (x-y) \cdot \nabla a(tx + (1-t)y) u(y) \dif y \dif \xi \\
	& \simeq (2\pi h)^{-n} \iint \nabla_\xi e^{i(x-y) \cdot \xi/h} \cdot \nabla a(tx + (1-t)y) u(y) \dif y \dif \xi \\
	& \simeq (2\pi h)^{-n} \iint e^{i(x-y) \cdot \xi/h} {\rm div}_\xi \big( \nabla a(tx + (1-t)y) u(y) \big) \dif y \dif \xi \\
	& = 0.
	\end{align*}
	We arrive at the conclusion.
\end{proof}

From Lemma \ref{lem:Qxind-PM2021} and \eqref{eq:QDe3-PM2021} we know that if $a(x,\xi)$ is either independent of $\xi$ or independent of $x$, the quantized operator $a_t(x,hD)$ will be independent of $t$.
Hence, for fixed $x^*$, $\xi^* \in \Rn$, and denote $l(x,\xi) := x^* \cdot x + \xi^* \cdot \xi$, then $l_t(x,hD)$ is independent of $t$, i.e.,
\begin{equation} \label{eq:lw-PM2021}
\boxed{ l_t(x,hD) = x^* \cdot x + \xi^* \cdot hD. }
\end{equation}

\section{Composition of the standard quantizations} \label{sec:CSP-PM2021}


For a non-degenerate, symmetric, real-valued $n \times n$ matrix $Q$, the quantization of the exponential of quadratic forms is defined as the standard quantization (cf.~\eqref{eq:QDe2-PM2021}),
\begin{align}
e^{\frac i {2h} \agl[Q hD, hD]} \varphi(x)
& = (2\pi h)^{-n} \iint_{\Rn \times \Rn} e^{i(x-y) \cdot \xi/h} e^{\frac i {2h} \agl[Q \xi, \xi]} \varphi(y) \dif y \dif \xi. \label{eq:eDQ-PM2021}
\end{align}
Readers may compare \eqref{eq:eDQ-PM2021} with \eqref{eq:elw-PM2021}.
The following lemma shows how to express $e^{\frac i {2h} \agl[Q hD, hD]}$.
The $e^{\frac i {2h} \agl[Q hD, hD]}$ can be expanded by using stationary phase lemmas.

\begin{lem} \label{lem:SeES-PM2021}
	Assume $Q$ is a non-degenerate, symmetric, real-valued $n \times n$ matrix.
	We have $e^{\frac i {2h} \agl[Q hD, hD]} \colon \scrS \to \scrS$ continuously.
	And when $0 \leq \delta \leq \frac 1 2$, we have that $e^{\frac i {2h} \agl[Q hD, hD]} \colon S_\delta(m) \to S_\delta(m)$, and the expression is
	\begin{equation} \label{eq:eQDD-PM2021}
	e^{\frac i {2h} \agl[Q hD, hD]} a(x) = \frac {e^{i \frac \pi 4 \sgn Q}} {|\det Q|^{1/2} (2\pi h)^{n/2}} \int_{\Rn} e^{\frac {-i} {2h} \agl[Q^{-1} y, y]} a(x+y) \dif y.
	\end{equation}
	The integral \eqref{eq:eQDD-PM2021} is defined in oscillatory sense.	
	Moreover, when $0 \leq \delta < \frac 1 2$, for $a \in S_\delta(m)$ we have the asymptotics
	\begin{equation} \label{eq:SeES-PM2021}
	\boxed{e^{\frac i {2h} \agl[Q hD, hD]} a = \sum_{j=0}^N \frac {(ih)^j} {j!} \Big( \frac {\agl[QD,D]} 2 \Big)^j a + h^{(1-2\delta)(N+1)} S_\delta(m).}
	\end{equation}
\end{lem}


\begin{proof}
	For a non-degenerate, symmetric, real-valued $n \times n$ matrix $Q$ we have (see \cite{masuma2020})
	\begin{equation} \label{eq:gvpveix2-PM2021}
	\calF \{ e^{\frac i 2 \agl[Qx,x]} \}(\xi)
	= \frac {e^{i \frac \pi 4 \sgn Q}} {|\det Q|^{1/2}} e^{-\frac i 2 \agl[Q^{-1} \xi,\xi]}.
	\end{equation}
	
	For any measurable function $a \in \scrS$, as long as the right-hand-side of \eqref{eq:eQDD-PM2021} is definable, by the definition \eqref{eq:eDQ-PM2021} we have
	\begin{align}
	e^{\frac i {2h} \agl[Q hD, hD]} a(x) 
	& = (2\pi h)^{-n} \iint e^{i(x-y) \cdot \xi/h} e^{\frac i {2h} \agl[Q \xi, \xi]} a(y) \dif y \dif \xi \nonumber \\
	& = (2\pi h)^{-n/2} \int \big[ (2\pi h)^{-n/2} \int e^{i(x-y) \cdot \xi/h} e^{\frac i {2h} \agl[Q \xi, \xi]} \dif \xi \big] a(y) \dif y \nonumber \\
	& = (2\pi h)^{-n/2} \int \big[ (2\pi)^{-n/2} \int e^{-i(y-x) / \sqrt h \cdot \xi} e^{\frac i 2 \agl[Q \xi, \xi]} \dif \xi \big] a(y) \dif y \nonumber \\
	& = (2\pi h)^{-n/2} \int \calF\{ e^{\frac i 2 \agl[Q \xi,\xi]} \}((y-x)/\sqrt{h}) \cdot a(y) \dif y \nonumber \\
	& = (2\pi h)^{-n/2} \int \frac {e^{i \frac \pi 4 \sgn Q}} {|\det Q|^{1/2}} e^{-\frac i {2h} \agl[Q^{-1} (y-x),y-x]} a(y) \dif y \quad \text{by~} \eqref{eq:gvpveix2-PM2021} \nonumber \\
	& = \frac {e^{i \frac \pi 4 \sgn Q}} {|\det Q|^{1/2} (2\pi h)^{n/2}} \int e^{\frac {i} {h} \agl[-Q^{-1} y,y]/2} a(x+y) \dif y. \label{eq:Qaxy-PM2021}
	\end{align}
	We arrive at \eqref{eq:eQDD-PM2021}.
	From \eqref{eq:Qaxy-PM2021} and \eqref{eq:IQ-PM2021} it is easy to see that $x^\alpha \partial^\beta e^{\frac i {2h} \agl[Q hD, hD]} \colon \scrS \to L^\infty$ for $\forall \alpha, \beta$, hence $e^{\frac i {2h} \agl[Q hD, hD]} \colon \scrS \to \scrS$ continuously.
	
	Now we use Proposition \ref{prop:2-PM2021}, to estimate \eqref{eq:Qaxy-PM2021} and confirm that $e^{\frac i {2h} \agl[Q hD, hD]}$ indeed maps $S_\delta(m)$ into itself.
	Denote the order of the symbol $a$ as $\tilde N$.
	Choose the $N$ in Proposition \ref{prop:2-PM2021} to be $N \geq \tilde N/2 - 1$.
	The constants $C_{N,n,\alpha}$ in Proposition \ref{prop:2-PM2021} satisfy $C_{N,n,\alpha} = h^{-\delta|\alpha|}$.
	From \eqref{eq:Qaxy-PM2021} and Proposition \ref{prop:2-PM2021} we have
	\begin{align} 
	& e^{\frac i {2h} \agl[Q hD, hD]} a(x) \nonumber \\
	= & \sum_{0 \leq j \leq N} \frac {h^j} {j!} \left( \frac {\agl[(-Q^{-1})^{-1}D, D]} {2i} \right)^j a(x) + \mathcal O \big( h^{N+1} \times \sum_{|\alpha| \leq n+2N+3} \sup_{y \in \Rn} \frac {|\partial^\alpha a(x + y; h)|} {\agl[y]^{n+4N+5-|\alpha|}} \big) \nonumber \\
	= & \sum_{0 \leq j \leq N} \frac {(ih)^{j}} {j!} \left( \frac {\agl[QD, D]} {2} \right)^j a(x) + \mathcal O \big( h^{N+1} \times \sum_{|\alpha| \leq n+2N+3} \sup_{y \in \Rn} \frac {h^{-\delta|\alpha|} \agl[y]^{\tilde N} m(x) } {\agl[y]^{2N+2}} \big) \nonumber \\
	= & \sum_{0 \leq j \leq N} \frac {(ih)^{j}} {j!} \left( \frac {\agl[QD, D]} {2} \right)^j a(x) + \mathcal O \big( h^{N+1-\delta(n+2N+3)} \sup_{y \in \Rn} \agl[y]^{\tilde N - 2N - 2} m(x) \big) \nonumber \\
	= & \sum_{0 \leq j \leq N} \frac {(ih)^{j}} {j!} \left( \frac {\agl[QD, D]} {2} \right)^j a(x) + h^{(1-2\delta)(N+1)} \mathcal O \big( h^{-\delta(n+1)} m(x) \big). \label{eq:I1hdh1-PM2021}
	\end{align}
	Now $e^{\frac i {2h} \agl[Q hD, hD]} \colon S_\delta(m) \to S_\delta(m)$ is justified by \eqref{eq:I1hdh1-PM2021} and similar arguments work on $\partial^\alpha(e^{\frac i {2h} \agl[Q hD, hD]} a)$.
	It can be checked that $h^j \agl[QD,D]^j a = h^{(1-2\delta)j} S_\delta(m)$.
	Here $f = h^{(1-2\delta)j} S_\delta(m)$ means there exists a symbol $g \in S_\delta(m)$ such that $f = h^{(1-2\delta)j} g$.
	Hence, these leading terms matched with the stipulation in Definition \ref{defn:Asy2-PM2021}.
	From \eqref{eq:I1hdh1-PM2021} it seems we didn't obtain the expansion because the remainder term may surpass some leading terms.
	However, when $\delta < \frac 1 2$, from \eqref{eq:I1hdh1-PM2021} we see that the order of the remainder term goes higher as $N$ goes larger (while when $\delta = \frac 1 2$ this doesn't happen), and when we set $N$ to be larger enough, these leading terms in front of the remainder term can exposed themselves from the remainder and will not be surpassed by the remainder.
	For example, if we want to expand \eqref{eq:I1hdh1-PM2021} up to $N'$, we first choose $N$ such that $(1-2\delta)(N+1) - \delta (2n + 1) \geq (1-2\delta)(N'+1)$, then \eqref{eq:I1hdh1-PM2021} can be continued as
	\begin{align}
	(*)
	& = \sum_{0 \leq j \leq N} \frac {(ih)^j} {j!} \left( \frac {\agl[QD, D]} 2 \right)^j a(x) + \mathcal O(h^{(1-2\delta)(N+1) - \delta (2n + 1)} m(x)) \nonumber \\
	& = \sum_{0 \leq j \leq N'} \frac {(ih)^j} {j!} \left( \frac {\agl[QD, D]} 2 \right)^j a(x) + \sum_{N'+1 \leq j \leq N} \frac {(ih)^j} {j!} \left( \frac {\agl[QD, D]} 2 \right)^j a(x) \nonumber \\
	& \quad + \mathcal O(h^{(1-2\delta)(N+1) - \delta (2n + 1)} m(x)) \nonumber \\
	& = \sum_{0 \leq j \leq N'} \frac {(ih)^j} {j!} \left( \frac {\agl[QD, D]} 2 \right)^j a(x) + \sum_{N'+1 \leq j \leq N} h^{(1-2\delta)j} S_\delta(m) + \mathcal O(h^{(1-2\delta)(N'+1)} m(x)) \nonumber \\
	& = \sum_{0 \leq j \leq N'} \frac {(ih)^j} {j!} \left( \frac {\agl[QD, D]} 2 \right)^j a(x) + \mathcal O_{S_\delta(m)}(h^{(1-2\delta)(N'+1)}). \label{eq:I1hO-PM2021}
	\end{align}
	Note that here we omitted the investigation of $|\partial^\alpha (e^{\frac i {2h} \agl[Q hD, hD]} a)|$, but the prove shall almost the same as above.
	We proved \eqref{eq:SeES-PM2021}.
\end{proof}

\begin{thm}[Composition of standard quantizations\index{composition of quantizations}] \label{thm:CoWQs-PM2021}
	Let $a \in S_\delta(m_1)$, $b \in S_\delta(m_2)$.
	Denote
	\[
	(a \# b)(x,hD) = a(x,hD) \circ b(x,hD),
	\]
	then $a \# b \in S_\delta(m_1 m_2)$, and
	\begin{equation} \label{eq:acobs1-PM2021}
	a \# b(x,\eta) = e^{\frac i {2h} \agl[Q hD_{(y,\xi)}, hD_{(y,\xi)}]} \big( a(x,\eta+\xi) b(x+y,\eta) \big) |_{y=0,\, \xi=0},
	\end{equation}
	where $Q = \begin{pmatrix}
	0 & I_{n \times n} \\ I_{n \times n} & 0
	\end{pmatrix}$.
	Moreover, when $h \to 0^+$ we have the semiclassical asymptotics,
	\begin{equation} \label{eq:abExp2-PM2021}
	\boxed{\begin{aligned}
		a \# b(x,\eta)
		& = \sum_{j=0}^N \frac {(ih)^j} {j!} (D_y \cdot D_\xi)^j \big( a(x,\eta+\xi) b(x+y,\eta) \big) |_{y=0,\xi=0} + h^{(1 - 2\delta)(N+1)} S_\delta(m_1 m_2) \\
		& = \sum_{|\alpha| \leq N} \frac {(-ih)^{|\alpha|}} {\alpha!} \partial_\eta^\alpha a(x,\eta) \partial_x^\alpha b(x,\eta) + h^{(1 - 2\delta)(N+1)} S_\delta(m_1 m_2).
		\end{aligned}}
	\end{equation}
\end{thm}

\begin{rem} \label{rem:CoWQs-PM2021}
	When either $a(x,\xi)$ or $b(x,\xi)$ is polynomial of $\xi$, the expansion \eqref{eq:abExp2-PM2021} will be finite, i.e.~when $N$ is large enough the remainders $h^{(1-2\delta) (N+1)} S_\delta(m_1 m_2)$ will be exactly zero, see also \cite[Remark 2.6.9]{mart02Anin}.
	This can be seen by directly working in the stationary phase lemma.
\end{rem}

\begin{proof}[Proof of Theorem \ref{thm:CoWQs-PM2021}]
	For a test function $\varphi \in \scrS(\Rn)$, we have
	\begin{align*}
	(a \# b)(x,hD)\varphi 
	& = (2\pi h)^{-2n} \int e^{i(x-z) \cdot \eta/h} \big( \int e^{i(x-y) \cdot (\xi - \eta)/h} a(x,\xi) b(y,\eta) \dif y \dif \xi \big) \varphi(z) \dif z \dif \eta \\
	& = (2\pi h)^{-2n} \int e^{i(x-z) \cdot \eta/h} \big( \int e^{-iy \cdot \xi/h} a(x,\xi + \eta) b(y+x,\eta) \dif y \dif \xi \big) \varphi(z) \dif z \dif \eta \\
	& = (2\pi h)^{-n} \int e^{i(x-z) \cdot \eta/h} \cdot [\cdots] \cdot \varphi(z) \dif z \dif \eta, \nonumber \\
	\text{where} \quad [\cdots] & = (2\pi h)^{-n} \int e^{\frac {-i} {2h} \agl[Q(y,\xi)^T,(y,\xi)^T]} a(x,\eta+\xi) b(x+y,\eta) \dif y \dif \xi,
	\end{align*}
	and $Q = \begin{pmatrix}
	0 & I \\ I & 0
	\end{pmatrix}$.
	Note that $Q^{-1} = Q$, $\sgn Q = 0$ and $\det Q = 1$ or $-1$.
	From \eqref{eq:eQDD-PM2021} and \eqref{eq:SeES-PM2021}, we see
	\begin{align*}
	[\cdots]
	& = e^{\frac i {2h} \agl[Q^{-1} hD_{(y,\xi)}, hD_{(y,\xi)}]} \big( a(x,\eta+\xi) b(x+y,\eta) \big) |_{y=0,\, \xi=0} \\
	& \sim \sum_j \frac {(ih)^j} {j!} \Big( \frac {\agl[Q D_{(y,\xi)},D_{(y,\xi)}]} 2 \Big)^j \big( a(x,\eta+\xi) b(x+y,\eta) \big) |_{y=0,\xi=0}, \quad \text{in} \ S_\delta(m_1 m_2) \\
	& \sim \sum_j \frac {(ih)^j} {j!} \big( D_y \cdot D_\xi \big)^j \big( a(x,\eta+\xi) b(x+y,\eta) \big) |_{y=0,\xi=0}, \quad \text{in} \ S_\delta(m_1 m_2) \\
	& = \sum_{|\alpha| \leq N} \frac {(-ih)^{|\alpha|}} {\alpha!} \partial_\eta^\alpha a(x,\eta) \partial_x^\alpha b(x,\eta) + \mathcal O_{S_\delta(m_1 m_2)}(h^{(1 - 2\delta)(N+1)}).
	\end{align*}
	We obtain \eqref{eq:acobs1-PM2021} and \eqref{eq:abExp2-PM2021}.
	Note that we have used \eqref{eq:npExp-PM2021}.
	The proof is complete. 
\end{proof}

Readers may compare Theorem \ref{thm:com-PM2021} with \eqref{eq:abExp2-PM2021}.
The asymptotics in Theorem \ref{thm:com-PM2021} is in terms of the decay of $|\xi|$, but the asymptotics in \eqref{eq:abExp2-PM2021} is about the order of $h$.
The first two leading terms in Theorem \ref{thm:com-PM2021} is $ab - i \nabla_\xi a \cdot \nabla_x b$ (no $h$), while that of \eqref{eq:abExp2-PM2021} is $ab - ih \nabla_\xi a \cdot \nabla_x b$.

\begin{cor} \label{cor:CoWQ-PM2021}
	The first two leading terms of $a \# b$ is
	\(
	\boxed{ab - ih \nabla_\xi a \cdot \nabla_x b.}
	\)
	Assume $a \in S_\delta(m_1)$ and $b \in S_\delta(m_2)$, then the symbol of the commutator of $a(x,hD)$ and $b(x,hD)$ is
	\begin{equation*}
	\boxed{\frac h i \{a, b\} - \frac {h^2} 2 {\rm tr} (\nabla_\xi^2 a \cdot \nabla_x^2 b - \nabla_x^2 a \cdot \nabla_\xi^2 b) + h^{3(1-2\delta)} S_\delta(m_1 m_2),}
	\end{equation*}
	where $\{a, b\}$ is the Poisson bracket of $a$ of $b$, and $\nabla_\xi^2 a \cdot \nabla_x^2 b$ is the product of two Hessian matrices, and ${\rm tr}$ is the trace.
\end{cor}

The proof is left as an exercise.
Finally, we also have symbolic calculus for the adjoint.

\begin{thm}[Adjoint of standard quantizations\index{adjoint of quantizations}] \label{thm:AdQs-PM2021}
	Let $a \in S_\delta(m)$.
	Denote
	\[
	(a(x,hD) u, v) = (u, a^*(x,hD) v),
	\]
	then $a^* \in S_\delta(m)$, and when $h \to 0^+$ we have the semiclassical asymptotics,
	\begin{equation} \label{eq:AdEx2-PM2021}
	\boxed{\begin{aligned}
		a^*(x,\xi)
		& = \sum_{|\alpha| \leq N} \frac {h^{|\alpha|}} {\alpha!} D_x^\alpha \partial_\xi^\alpha \bar a(x,\xi) + h^{(1 - 2\delta)(N+1)} S_\delta(m).
		\end{aligned}}
	\end{equation}
\end{thm}

We omit the proof.

\section{Composition of the Weyl quantizations} \label{sec:Qels-PM2021}

The composition of the Weyl quantizations are more peculiar than that of the standard ones, and we explain this in \S \ref{subsec:Qels-PM2021}.
Before that, we make some preparation first.

\subsection{Symplectic 2-form} \label{subsec:sym2f-PM2021}

We define the symplectic product.
\begin{defn}[Symplectic product\index{symplectic product}] \label{defn:SyPro-PM2021}
	The \emph{symplectic product} is defined as
	\[
	\sigma \colon \R^{2n} \times \R^{2n} \to \R, \quad \sigma((x,\xi), (y, \eta)) := \xi \cdot y - x \cdot \eta.
	\]
\end{defn}

\begin{rem} \label{rem:SyPro-PM2021}
	The underlying space $\R^{2n}$ in Definition \ref{defn:SyPro-PM2021} can be generalized to be a tangent bundle.
	When $\R^{2n}$ is replaced by a tangent bundle $TM$ (or $T^*M$) where $M$ is $n$-dimensional (hence $TM$ is locally homeomorphic to $\R^{2n}$), $\sigma$ can be generalized as a bilinear form on $T_p(TM) \times T_p(TM)$ in the following way.
	For any $p \in TM$ and $(u_x, u_\xi)$, $(v_x, v_\xi) \in T_p(TM)$, we define
	\[
	\sigma \colon T(TM) \times T(TM) \to \R, \quad \sigma |_p((u_x, u_\xi), (v_x, v_\xi)) := u_\xi \cdot v_x - u_x \cdot u_\xi
	\]
	Locally speaking, when imposed a local coordinates system $\{x^j\}$ on $M$ and the corresponding coordinates $\{\xi_j\}$ on the fiber, it can be checked that $\sigma = \df \xi_j \wedge \df x^j$ (Einstein summation convention invoked) and it is invariant w.r.t.~the coordinates systems.
	This $\sigma$ is a 2-form on the tangent bundle and is called the \emph{symplectic 2-form}\index{symplectic 2-form}.
\end{rem}

In what follows, we only work on $\R^{2n}$ rather than on general manifolds.
If without otherwise stated, we will use the following notations,
\begin{equation} \label{eq:zwzeta-PM2021}
z = (x, \xi)^T \in \R^{2n}, \ w = (y, \eta)^T \in \R^{2n}, \ \zeta = (x, \xi, y, \eta)^T = (z^T, w^T)^T \in \R^{4n}.
\end{equation}
Note that all of $z$, $w$ and $\zeta$ are vertical vectors.
Definition \ref{defn:SyPro-PM2021} is equivalent to
\begin{equation} \label{eq:sigJ-PM2021}
\sigma(z, w)
= z^T
\cdot 
\begin{pmatrix}
0 & -I \\
I & 0
\end{pmatrix}
\cdot w
= z^T \cdot \sigma \cdot w
= \agl[\sigma^T z, w]
\end{equation}
where $I$ is the identity $n \times n$ matrix and
\boxed{
\sigma = 
\begin{pmatrix}
0 & -I \\
I & 0
\end{pmatrix}.
}
Note that $\sigma$ is non-degenerate and anti-symmetric, i.e.~$\sigma^{-1} = \sigma^T = -\sigma$.

We note that \eqref{eq:sigJ-PM2021} is homogeneous of degree 2 of $\zeta$ (i.e.~$\sigma(h \zeta) = h^2 \sigma(\zeta)$), but not in a quadratic form of $\zeta$ under a symmetric matrix ($\sigma$ is not symmetric).
We can achieve this by
\begin{equation*} 
\sigma(\zeta) = \sigma(z,w)
= z^T \cdot \sigma \cdot w
= \frac 1 2 \zeta^T
\cdot 
\begin{pmatrix}
0 & \sigma \\
\sigma^T & 0
\end{pmatrix}
\cdot \zeta
= \frac 1 2 \zeta^T \cdot \Sigma \cdot \zeta,
\end{equation*}
where
\begin{equation} \label{eq:Sg-PM2021}
	\boxed{\Sigma = \begin{pmatrix}
		0 & \sigma \\
		\sigma^T & 0
		\end{pmatrix}}
\end{equation}
is a $4n \times 4n$ matrix.
Note that $\Sigma$ is non-degenerate and symmetric satisfying $\Sigma^{-1} = \Sigma^T = \Sigma$, $\det \Sigma = 1$ and $\sgn \Sigma = 0$.
In summary, we have
\begin{equation} \label{eq:sSizw-PM2021}
\boxed{\sigma(z,w) = z^T \cdot \sigma \cdot w = \agl[\sigma^T z, w] = \frac 1 2 \zeta^T \cdot \Sigma \cdot \zeta = \frac 1 2 \agl[\Sigma \zeta, \zeta].}
\end{equation}

\subsection{The composition} \label{subsec:Qels-PM2021}

If we mimic the proof of Theorem \ref{thm:CoWQs-PM2021}, we would have
\begin{align*}
& a^w(x,hD) \circ b^w(x,hD) \varphi \\
= & (2\pi h)^{-2n} \int e^{i(x-z) \cdot \eta/h} \big( \int e^{i(x-y) \cdot (\xi - \eta)/h} a(\frac {x+y} 2,\xi) b(\frac {y+z} 2,\eta) \dif y \dif \xi \big) \varphi(z) \dif z \dif \eta \\
= & (2\pi h)^{-2n} \int e^{i(x-z) \cdot \eta/h} \big( \int e^{-iy \cdot \xi/h} a(\frac y 2 + x,\xi + \eta) b(\frac y 2 + \frac {x+z} 2,\eta) \dif y \dif \xi \big) \varphi(z) \dif z \dif \eta \\
= & (2\pi h)^{-n} \int e^{i(x-z) \cdot \eta/h} \cdot c(\frac {x+z} 2, \eta) \cdot \varphi(z) \dif z \dif \eta
= c^w(x,hD) \varphi
\end{align*}
where the $c$ should satisfy
\[
c(\frac {x+z} 2, \eta) = (2\pi h)^{-n} \int e^{\frac {-i} {2h} \agl[Q(y,\xi),(y,\xi)]} a(\frac y 2 + {\color{red}x},\xi + \eta) b(\frac y 2 + \frac {x+z} 2,\eta) \dif y \dif \xi.
\]
However, this argument doesn't work, because there is an additional $x$ on the RHS.

Instead, from $a^w(x,hD) \circ b^w(x,hD) = c^w(x,hD)$ we can proceed as follows,
\begin{align*}
& a^w(x,hD) \circ b^w(x,hD) \varphi \\
= & (2\pi h)^{-2n} \int e^{i[(x-y) \cdot \xi + (y-z) \cdot \eta]/h} a(\frac {x+y} 2,\xi) b(\frac {y+z} 2,\eta) \varphi(z) \dif y \dif z \dif \xi \dif \eta \\
= & (2\pi h)^{-n} \int e^{i(x-z) \cdot \zeta/h} c(\frac {x+z} 2, \zeta) \varphi(z) \dif z \dif \zeta,
\end{align*}
which, due to the arbitrary of $\varphi$, suggests
\begin{equation*}
\int e^{i(x-z) \cdot \zeta/h} c(\frac {x+z} 2, \zeta) \dif \zeta
= (2\pi h)^{-n} \int e^{i[(x-y) \cdot \xi + (y-z) \cdot \eta]/h} a(\frac {x+y} 2,\xi) b(\frac {y+z} 2,\eta) \dif y \dif \xi \dif \eta.
\end{equation*}
Readers may note that the LHS is an inverse Fourier transform.
We make the following change of variable before we perform the Fourier transform:
\begin{equation*}
\left\{\begin{aligned}
\frac {x-z} 2 & = s \\
\frac {x+z} 2 & = t
\end{aligned}\right.
\quad \Rightarrow \quad
\left\{\begin{aligned}
x & = t + s \\
z & = t - s
\end{aligned}\right.
\end{equation*}
so
\begin{align*}
& \int e^{i2s \cdot \zeta/h} c(t, \zeta) \dif \zeta \\
= & (2\pi h)^{-n} \int e^{i[(t+s-y) \cdot \xi + (y-t+s) \cdot \eta]/h} a(\frac {t+s+y} 2,\xi) b(\frac {y+t-s} 2,\eta) \dif y \dif \xi \dif \eta,
\end{align*}
and
\begin{align}
& c(t, \zeta) \nonumber \\
= & (2\pi h)^{-2n} \int e^{-i2s \cdot \zeta/h} e^{i[(t+s-y) \cdot \xi + (y-t+s) \cdot \eta]/h} a(\frac {t+s+y} 2,\xi) b(\frac {y+t-s} 2,\eta) \dif y \dif \xi \dif \eta \dif (2s) \nonumber \\
= & 2^n (2\pi h)^{-2n} \int e^{i[(y-t) \cdot (\eta - \xi) + s \cdot (\xi + \eta - 2 \zeta)]/h} a(\frac {t+y+s} 2,\xi) b(\frac {t+y-s} 2,\eta) \dif y \dif s \dif \xi \dif \eta \nonumber \\
= & 2^n (2\pi h)^{-2n} \int e^{i[y \cdot (\eta - \xi) + s \cdot (\xi + \eta)]/h} a(\frac {y+s} 2 + t,\xi + \zeta) b(\frac {y-s} 2 + t,\eta + \zeta) \dif y \dif s \dif \xi \dif \eta \nonumber \\
= & 2^n (2\pi h)^{-2n} \int e^{i[(y + s) \cdot \eta - (y-s) \cdot \xi]/h} a(\frac {y+s} 2 + t,\xi + \zeta) b(\frac {y-s} 2 + t,\eta + \zeta) \dif y \dif s \dif \xi \dif \eta \nonumber \\
= & 2^n (2\pi h)^{-2n} \int e^{i(2y' \cdot \eta - 2s' \cdot \xi)/h} a(y' + t,\xi + \zeta) b(s' + t,\eta + \zeta) 2^n \dif y' \dif s' \dif \xi \dif \eta \nonumber \\
= & (\pi h)^{-2n} \int e^{i2(y \cdot \eta - s \cdot \xi)/h} a(y + t,\xi + \zeta) b(s + t,\eta + \zeta) \dif y \dif \xi \dif s \dif \eta \nonumber \\
= & (\pi h)^{-2n} \int e^{i h^{-1} \agl[-2\Sigma (y, \xi, s, \eta), (y, \xi, s, \eta)]/2} a(y + t,\xi + \zeta) b(s + t,\eta + \zeta) \dif y \dif \xi \dif s \dif \eta, \label{eq:cte-PM2021}
\end{align}
where the $4n \times 4n$ matrix $\Sigma$ is defined in \eqref{eq:Sg-PM2021}, and we used Exercise \ref{ex:SDe-PM2021}.

Recall that $\Sigma^{-1} = \Sigma$, $\det \Sigma = 1$ and $\sgn \Sigma = 0$.
Now we apply Proposition \ref{prop:2-PM2021} to \eqref{eq:cte-PM2021} and obtain
\begin{align*}
c(t,\zeta)
& \sim \frac {(\pi h)^{-2n}  (2\pi h)^{2n}} {|\det (2\Sigma)|^{1/2}} \sum_j \frac {h^j} {j!} \left( \frac {\agl[(-2\Sigma)^{-1}D_{(y, \xi, s, \eta)}, D_{(y, \xi, s, \eta)}]} {2i} \right)^j \big( a(y + t,\xi + \zeta) \\
& \qquad \times b(s + t,\eta + \zeta) \big) |_{y=s=\xi=\eta= 0} \\
& = \sum_j \frac {(h/(2i))^j} {j!} \Big( \agl[-\frac 1 2 \Sigma D_{(y, \xi, s, \eta)}, D_{(y, \xi, s, \eta)}] \Big)^j \big( a(y,\xi) b(s,\eta) \big) |_{y=s=t,\, \xi=\eta= \zeta} \\
& = \sum_j \frac {(ih/2)^j} {j!} \big( D_s \cdot D_\xi + \nabla_y \cdot \nabla_\eta \big)^j \big( a(y,\xi) b(s,\eta) \big) |_{y=s=t,\, \xi=\eta= \zeta},
\end{align*}
where we used Exercise \ref{ex:SDe-PM2021}.
Here for simplicity we omitted the analysis of the remainder terms, and for the detailed analysis of the remainder, readers may refer to \cite{zw2012semi}.
Noticing that
\begin{align*}
(D_s \cdot D_\xi + \nabla_y \cdot \nabla_\eta)^j
& = \sum_{0 \leq k \leq j} \binom{j}{k} (D_s \cdot D_\xi)^{k} (\nabla_y \cdot \nabla_\eta)^{j-k} \\
& = \sum_{0 \leq k \leq j} \binom{j}{k} \sum_{|\alpha| = k} \frac {k!} {\alpha!} D_s^\alpha D_\xi^\alpha \sum_{|\beta| = j-k} \frac {(j-k)!} {\beta!} \nabla_y^\beta \nabla_\eta^\beta \quad (\text{by~} \eqref{eq:npExp-PM2021}) \\
& = \sum_{0 \leq k \leq j} \sum_{|\alpha| = k} \sum_{|\beta| = j-k} \frac {j!} {\alpha! \beta!} D_s^\alpha D_\xi^\alpha \partial_y^\beta \partial_\eta^\beta \\
& = \sum_{|\alpha| + |\beta| = j} \frac {j!} {\alpha! \beta!} D_s^\alpha D_\xi^\alpha \partial_y^\beta \partial_\eta^\beta,
\end{align*}
we can continue
\begin{align*}
c(t,\zeta)
& \sim \sum_{j} \sum_{|\alpha| + |\beta| = j} \frac {(ih/2)^j} {j!} \frac {j!} {\alpha! \beta!} D_s^\alpha D_\xi^\alpha \partial_y^\beta \partial_\eta^\beta \big( a(y,\xi) b(s,\eta) \big) |_{y=s=t,\, \xi=\eta= \zeta} \\
& = \sum_{\alpha, \beta} \frac {(ih/2)^{|\alpha| + |\beta|}} {\alpha! \beta!} \big[ D_x^\beta \partial_\xi^\alpha \big( a(x,\xi) \big)
D_x^\alpha \partial_\xi^\beta \big( b(x,\xi) \big) \big] \big|_{x=t,\, \xi = \zeta}.
\end{align*}

We have just proved the following result:
\begin{thm}[Composition of quantizations of semiclassical symbols\index{composition of quantizations}] \label{thm:CoWw-PM2021}
	Assume that $a \in S_\delta(m_1)$, $b \in S_\delta(m_2)$.
	Denote
	\[
	(a \#^w b)^w(x,hD) = a^w(x,hD) \circ b^w(x,hD),
	\]
	then $a \#^w b \in S_\delta(m_1 m_2)$ and
	\begin{equation} \label{eq:acobs2-PM2021}
	a \#^w b(x,\eta) = e^{i hA(D)} \big( a(x,\xi) b(y,\eta) \big) |_{y = x,\, \xi = \eta},
	\end{equation}
	where $A(D) = \frac 1 2 \sigma((D_x, D_\xi), (D_y, D_\eta)) = \frac 1 2 (D_y \cdot D_\xi - D_x \cdot D_\eta)$ and the $\sigma$ is defined in \eqref{eq:zwzeta-PM2021}.
	Moreover, when $h \to 0^+$ we have the semiclassical asymptotics,
	\begin{equation} \label{eq:abExp2w-PM2021}
	\boxed{\begin{aligned}
		a \#^w b(x,\eta)
		& = \sum_{j = 0}^N \frac {(ih/2)^j} {j!} (D_y \cdot D_\xi + \nabla_x \cdot \nabla_\eta)^j \big( a(x,\xi) b(y,\eta) \big) |_{y = x,\, \xi = \eta} \\
		& \quad + h^{(1-2\delta) (N+1)} S_\delta(m_1 m_2) \\
		& = \sum_{|\alpha| + |\beta| \leq N} \frac {(ih/2)^{|\alpha| + |\beta|}} {\alpha! \beta!} \partial_x^\beta D_\eta^\alpha a(x,\eta) D_x^\alpha \partial_\eta^\beta b(x,\eta) \\
		& \quad + h^{(1-2\delta) (N+1)} S_\delta(m_1 m_2).
		\end{aligned}}
	\end{equation}
\end{thm}

\begin{rem} \label{rem:CoWw-PM2021}
	Similar to Remark \ref{rem:CoWQs-PM2021}, the expansion \eqref{eq:abExp2w-PM2021} will be finite when either $a(x,\xi)$ or $b(x,\xi)$ is polynomial of $\xi$.
\end{rem}

Readers may refer to \cite[\S 4.11]{zw2012semi} for an another proof of Theorem \ref{thm:CoWw-PM2021}.

\begin{cor} \label{cor:CoWw-PM2021}
	The first two leading terms of $a \#^w b$ is
	\(
	\boxed{ab - ih \{a, b\}/2.}
	\)
	Assume $a \in S_\delta(m_1)$ and $b \in S_\delta(m_2)$, then the commutator of $a^w(x,hD)$ and $b^w(x,hD)$ is
	\[
	\boxed{[a^w(x,hD), b^w(x,hD)] = \frac h i \{a, b\}^w(x,hD) + h^{3(1-2\delta)} {\rm Op}_h(S_\delta(m_1 m_2)),}
	\]
	where $\{a, b\}$ is the Poisson bracket of $a$ of $b$.
\end{cor}

The remainder in the commutator expression looks out of expectation; it is of order $h^{3(1-2\delta)}$ rather that $h^{2(1-2\delta)}$.
This is because the second order leading term is in fact zero.

\begin{proof}
	From \eqref{eq:abExp2w-PM2021} we have
	\begin{align*}
	a \#^w b
	& = ab + \frac {ih} 2 D_\xi a \cdot D_x b + \frac {ih} 2 \nabla_x a \cdot \nabla_\xi b \\
	& \quad + \frac {(ih/2)^2} {2!} [(\nabla_x \cdot \nabla_\eta)^2 + (\nabla_\xi \cdot \nabla_y)^2 - 2 (\nabla_x \cdot \nabla_\eta) (\nabla_\xi \cdot \nabla_y)] (ab) \\
	& \quad + h^{3(1-2\delta)} S_\delta(m_1 m_2) \\
	& = ab - ih \{a, b\}/2 \\
	& \quad + \frac {(ih/2)^2} {2!} [\mathop{\rm tr} (\nabla_x^2 a \cdot \nabla_\eta^2 b) + \mathop{\rm tr} (\nabla_\xi^2 a \cdot \nabla_y^2 b) - 2 \mathop{\rm tr} (\nabla_{(x, \xi)}^2 a \cdot \nabla_{(y, \eta)}^2 b)] |_{y = x,\, \xi = \eta} \\
	& \quad + h^{3(1-2\delta)} S_\delta(m_1 m_2) \\
	& = ab - ih \{a, b\}/2 + \frac {(ih/2)^2} {2!} [\mathop{\rm tr} (\nabla_x^2 a \cdot \nabla_\eta^2 b) + \mathop{\rm tr} (\nabla_x^2 b \cdot \nabla_\eta^2 a) - 2 \mathop{\rm tr} (\nabla_{(x, \eta)}^2 a \cdot \nabla_{(x, \eta)}^2 b)] \\
	& \quad + h^{3(1-2\delta)} S_\delta(m_1 m_2),
	\end{align*}
	hence
	\[
	a \#^w b - b \#^w a
	= -ih\{a, b\} + h^{3(1-2\delta)} S_\delta(m_1 m_2).
	\]
	The proof is complete.
\end{proof}

\subsection{Specialties of Weyl quantization} \label{subsec:SWQ-PM2021}

\begin{lem} \label{lem:WeS-PM2021}
	For $u$, $v \in \scrS(\Rn)$, we have
	\[
	(a^w(x, hD)u, v) = (u, \bar a^w(x, hD) v).
	\]
\end{lem}

The proof is left as an exercise.

The Weyl quantization is the correct generalization of a solution operator of an ODE.
It is straightforward to check that $v(x,t) = e^{t f(x)} u(x)$ is the solution of an ODE
\begin{equation*}
\left\{\begin{aligned}
\partial_t v(x,t) & = f(x) v(x,t), \quad t \in \R, \\
v(x,0) & = u(x).
\end{aligned} \right.
\end{equation*}
Recall the linear form $l(x,\xi) = x^* \cdot x + \xi^* \cdot \xi$.
Now we would like to generalize the aforementioned idea by replacing $f(x)$ with an operator $\frac i h l(x,hD)$ and define $e^{\frac {it} h l(x,hD)} u$ as the unique solution of the corresponding ODE.
But in order to avoid notational confusion between ``$e^{\frac {it} h l(x,hD)} u$'' and ``$(e^{\frac {it} h l}) (x,hD) u$'' defined in \eqref{eq:QDe2-PM2021}, we deprecate the use of $e^{\frac {it} h l(x,hD)} u$.
We will see from the following result that the correct generalization will be the Weyl quantization $(e^{\frac {it} h l})^w (x,hD) u$ instead of the standard quantization $(e^{\frac {it} h l}) (x,hD) u$.

\begin{lem} \label{lem:Weylequiv-PM2021}
	Let $l(x,\xi) = x^* \cdot x + \xi^* \cdot \xi$ for fixed $x^*$, $\xi^* \in \Rn$.
	For every $u \in \scrS$, the Weyl quantization $(e^{\frac {it} h l})^w (x,hD) u$ is the unique solution of the ODE
	\begin{equation} \label{eq:ODEv-PM2021}
	\left\{\begin{aligned}
	\partial_t v(x,t) & = \frac i h l(x,hD) v(x,t), \quad t \in \R, \\
	v(x,0) & = u(x).
	\end{aligned} \right.
	\end{equation}
	\begin{equation} \label{eq:elw-PM2021}
	\boxed{e^{\frac {it} h l(x,hD)} = (e^{\frac {it} h l})^w (x,hD).}
	\end{equation}
	Specifically, we have
	\begin{equation} \label{eq:ewEx-PM2021}
	\boxed{ (e^{\frac {it} h l})^w (x,hD) u
		= e^{\frac i h [(x^* \cdot x)t + (x^* \cdot \xi^*)t^2/2]} u(x + \xi^* t). }
	\end{equation}
	And we have the composition relation
	\begin{equation} \label{eq:elm-PM2021}
	\boxed{e^{\frac {it} h l(x,hD)} e^{\frac {it} h m(x,hD)} = e^{\frac {it^2} {2h} \sigma(l,m)} e^{\frac {it} h (l+m)(x,hD)},}
	\end{equation}
	where the $\sigma$ is given in Definition \ref{defn:SyPro-PM2021}.
\end{lem}

\begin{proof}
	First, we solve \eqref{eq:ODEv-PM2021}.
	By this ODE we have
	\(
	\partial_t v(x,t) = \frac i h (x^* \cdot x + \xi^* \cdot hD) v(x,t),
	\)
	which gives a transport equation
	\(
	(\partial_t - \xi^* \cdot \nabla) v(x,t) = \frac i h x^* \cdot x v(x,t).
	\)
	Let $\gamma \colon t \in \R \mapsto (x - \xi^* t,t) \in \R^{n+1}$ be a curve, then we can obtain
	\begin{equation*}
	\frac {\df} {\df t} \big( v(\gamma(t)) \big)
	= (\partial_t - \xi^* \cdot \nabla) v(\gamma(t))
	= \frac i h x^* \cdot (x - 
	\xi^* t) v(\gamma(t)).
	\end{equation*}
	This is a one-dimensional ODE and the solution is straightforward,
	\[
	v(\gamma(t)) = e^{\frac i h [(x^* \cdot x)t - (x^* \cdot \xi^*)t^2/2]} v(\gamma(0)),
	\]
	which is equivalent to
	\[
	v(x - \xi^* t,t) = e^{\frac i h [(x^* \cdot x)t - (x^* \cdot \xi^*)t^2/2]} v(x,0).
	\]
	By replacing $x$ with $x + \xi^* t$ and substituting the boundary condition $v(x,0) = u(x)$ into the solution above, we obtain
	\begin{align}
	v(x,t)
	& = v((x + \xi^* t) - \xi^* t,t) \nonumber \\
	& = e^{\frac i h [x^* \cdot (x + \xi^* t)t - (x^* \cdot \xi^*)t^2/2]} u(x + \xi^* t) \nonumber \\
	& = e^{\frac i h [(x^* \cdot x)t + (x^* \cdot \xi^*)t^2/2]} u(x + \xi^* t) \nonumber \\
	\Rightarrow
	e^{\frac {it} h l(x,hD)} u & = e^{\frac i h [(x^* \cdot x)t + (x^* \cdot \xi^*)t^2/2]} u(x + \xi^* t). \label{eq:ODEvE-PM2021}
	\end{align}

	Second, we compute $(e^{\frac {it} h l})^w (x,hD) u$.
	We have
	\begin{align*}
	(e^{\frac {it} h l})^w (x,hD) u
	& = (2\pi h)^{-n} \iint e^{i(x-y) \cdot \xi/h} e^{\frac {it} h [x^* \cdot(\frac {x+y} 2) + \xi^* \cdot \xi]} u(y) \dif y \dif \xi \nonumber \\
	& = e^{\frac {it} h (x^* \cdot \frac x 2)} (2\pi h)^{-n} \iint e^{i (x + \xi^*t-y) \cdot \xi/h} \cdot e^{\frac {it} h (x^* \cdot \frac y 2)} u(y) \dif y \dif \xi \nonumber \\
	& = e^{\frac {it} h (x^* \cdot \frac x 2)} \int \delta(x + \xi^*t-y) e^{\frac {it} h (x^* \cdot \frac y 2)} u(y) \dif y \quad \big( (2\pi h)^{-n} \int e^{ia \cdot \xi/h} \dif \xi = \delta(a) \big) \nonumber \\
	& = e^{\frac {it} h (x^* \cdot \frac x 2)} e^{\frac {it} h x^* \cdot \frac {(x + \xi^*t)} 2} u(x + \xi^*t)
	= e^{\frac i h [(x^* \cdot x)t + (x^* \cdot \xi^*)t^2/2]} u(x + \xi^* t),
	\end{align*}
	which is \eqref{eq:ewEx-PM2021}.
	From \eqref{eq:ODEvE-PM2021} and \eqref{eq:ewEx-PM2021} we arrive at the first equality in the theorem.
	
	For the composition relation, let $l(x,\xi) = x_1^* \cdot x + \xi_1^* \cdot \xi$ and $m(x,\xi) = x_2^* \cdot x + \xi_2^* \cdot \xi$, then from \eqref{eq:ODEvE-PM2021} we have
	\begin{equation*}
	e^{\frac {it} h (l+m)(x,hD)} u
	= e^{\frac i h [(x_1^* + x_2^*) \cdot xt + (x_1^* + x_2^*) \cdot (\xi_1^* + \xi_2^*)t^2/2]} u(x + (\xi_1^* + \xi_2^*) t),
	\end{equation*}
	and
	\begin{align*}
	e^{\frac {it} h l(x,hD)} e^{\frac {it} h m(x,hD)} u
	& = (e^{\frac {it} h l})^w (x,hD) \circ e^{\frac i h [(x_2^* \cdot x)t + (x_2^* \cdot \xi_2^*)t^2/2]} u(x + \xi_2^* t) \\
	& = e^{\frac i h [(x_1^* \cdot x)t + (x_1^* \cdot \xi_1^*)t^2/2]} e^{\frac i h [(x_2^* \cdot (x + \xi_1^* t))t + (x_2^* \cdot \xi_2^*)t^2/2]} u(x + \xi_1^* t + \xi_2^* t) \\
	& = e^{\frac i h [(x_1^* + x_2^*) \cdot x t + (x_1^* \cdot \xi_1^* + 2 x_2^* \cdot \xi_1^* + x_2^* \cdot \xi_2^*)t^2/2]} u(x + (\xi_1^* + \xi_2^*) t) \\
	& = e^{\frac i h [(x_1^* + x_2^*) \cdot x t + (x_1^* + x_2^*) (x_2^* + \xi_2^*)t^2/2]} u(x + (\xi_1^* + \xi_2^*) t) \cdot e^{\frac i h [(x_2^* \cdot \xi_1^* - x_1^* \cdot \xi_2^*)t^2/2]} \\
	& = e^{\frac {it^2} {2h} \sigma((x_1^*,\xi_1^*),(x_2^*,\xi_2^*))} e^{\frac {it} h (l+m)(x,hD)} u.
	\end{align*}
	Readers should note that the $\sigma$ here is the symplectic product defined in Definition \ref{defn:SyPro-PM2021}.
	
	The proof is complete.
\end{proof}

By using \eqref{eq:elw-PM2021}, we can represent the corresponding Weyl quantization of a symbol by its Fourier transform.

\begin{lem}[Fourier decomposition of $a^w$] \label{lem:ale-PM2021}
	For any $a(x,\xi) \in \scrS(\R^{2n})$, we have
	\begin{equation*}
	\boxed{a^w(x,hD) = (2\pi h)^{-n} \int_{\R^{2n}} \calF_h a(l) e^{\frac i h l(x,hD)} \dif l,}
	\end{equation*}
	where $l = (x^*,\xi^*) \in \R^{2n}$ and $l(x,hD)$ is defined as \eqref{eq:lw-PM2021}.
	This can directly generalize to the case where $a(x,\xi) \in \scrS'(\R^{2n})$ and then $\agl[ a^w(x,hD) u, v] = (2\pi h)^{-2n} \hat a( \agl[e^{\frac i h l(x,hD)}u, v] )$ for $\forall u,v \in \scrS(\Rn)$.
\end{lem}

\begin{rem}
	With the help of Lemma \ref{lem:ale-PM2021}, every Weyl quantization can be represented by means of operators of the form $e^{\frac i h l(x,hD)}$ where $l$ is a linear form.
	Therefore, quantizations of the form $e^{\frac i h l(x,hD)}$ plays an important role in semiclassical analysis.
\end{rem}

\begin{proof}[Proof of Lemma \ref{lem:ale-PM2021}]
	When $a \in \scrS$, we have
	\[
	a(x,\xi) = (2\pi h)^{-n} \int_{\R^{2n}} e^{\frac i h l(x,\xi)} \calF_h a(l) \dif l,
	\]
	thus by \eqref{eq:elw-PM2021} we arrive at the statement.
	The case where $a(x,\xi) \in \scrS'(\R^{2n})$ is left as an exercise.
\end{proof}

\section{Applications in Carleman estimates} \label{sec:ACar-PM2021}

One of the examples of Carleman estimates is of the following
\begin{equation} \label{eq:Car-PM2021}
\tau^3 \nrm{e^{\tau \phi} u}^2 \lesssim \nrm{e^{\tau \phi} Pu}^2.
\end{equation}
To prove it, we set $h = \tau^{-1}$,  $v(x) = e^{\phi(x)/h} u(x)$ and denote an operator $P_\phi$ as
\[
P_\phi \colon f \to  e^{\phi/h} h^2 P (e^{-\phi/h} f),
\]
then \eqref{eq:Car-PM2021} is equivalent to
\begin{equation} \label{eq:Carh-PM2021}
\nrm{P_\phi v}^2 \gtrsim h \nrm{v}^2.
\end{equation}
We assume
\begin{equation} \label{eq:Ppm-PM2021}
	\boxed{\sigma_{\text{scl}}(P_\phi) \in S(m)}
\end{equation}
for some order function $m$.
Here we use $\sigma_{\text{scl}}(A)$ to signify the semiclassical symbol of $A$.

Set $A = (P_\phi + P_\phi^*)/2$ and $B = (P_\phi - P_\phi^*)/(2i)$, and denote
\[
\sigma
= \sigma_{\text{scl}}(ih^{-1} [A,B]),
\]
then we can conclude
\begin{align}
& \nrm[L^2]{P_\phi v}^2 = (P_\phi v, P_\phi v) = \nrm{A v}^2 + \nrm{B v}^2 + (i[A,B]v,v) \geq h (\sigma(x, hD) v,v),  \label{eq:siC1-PM2021} \\
& \sigma(x, hD) \text{~is self-adjoint, i.e.~} \sigma(x, hD)^* = \sigma(x, hD). \label{eq:siC2-PM2021}
\end{align}
Using \eqref{eq:siC1-PM2021}, inequality \eqref{eq:Carh-PM2021} will be true if the following is true:
\begin{equation} \label{eq:Casg-PM2021}
(\sigma(x, hD) v,v) \gtrsim \nrm{v}^2,
\end{equation}
so \eqref{eq:Casg-PM2021} implies \eqref{eq:Car-PM2021}.
It's left to prove \eqref{eq:Casg-PM2021}.

To prove \eqref{eq:Casg-PM2021}, we compute
\begin{align}
\sigma
& = \sigma_{\text{scl}}(ih^{-1} [A,B])
= \sigma_{\text{scl}}(ih^{-1} [(P_\phi + P_\phi^*)/2,(P_\phi - P_\phi^*)/(2i)]) \nonumber \\
& = \frac 1 {2h} \sigma_{\text{scl}}([P_\phi^*, P_\phi])
= \frac 1 {2h} [\frac h i \{p_\phi^*, p_\phi\} + h^2 S(m^2)] \quad (\text{Corollary~} \ref{cor:CoWQ-PM2021}) \nonumber \\
& = \frac 1 {2i} \{\overline{p_\phi} + hS(m), p_\phi\} + h S(m^2) \quad (\text{Theorem~} \ref{thm:AdQs-PM2021}) \nonumber \\
& = \frac 1 {2i} \{\overline{p_\phi}, p_\phi\} + h S(m^2)
= \{\Re p_\phi, \Im p_\phi \} + h S(m^2). \label{eq:siRI-PM2021}
\end{align}
The remainder term $S(m^2)$ comes from the assumption \eqref{eq:Ppm-PM2021} and the fact $\sigma_{\text{scl}}(P_\phi^*) \in S(m)$.
If
\begin{equation} \label{eq:RIcm-PM2021}
\boxed{\{\Re p_\phi, \Im p_\phi \}(x,\xi) \geq C m(x,\xi)^2, \ (m(x,\xi) \geq 1)}
\end{equation}
holds, from \eqref{eq:siRI-PM2021} we know when $h$ is small enough we will have
\begin{equation} \label{eq:sicm-PM2021}
	|\sigma(x,\xi)| \gtrsim m(x,\xi)^2.
\end{equation}
By combining \eqref{eq:sicm-PM2021}, \eqref{eq:siC2-PM2021} and \cite[Theorem 4.19]{zw2012semi}, we can conclude \eqref{eq:Casg-PM2021}.
In summary, we have the following theorem.

\begin{thm}[Carleman estimates\index{Carleman estimates}]
Let $P$ be a {\rm S$\Psi$DO} and $\phi \in C^\infty(\Rn; \R)$ and denote $P_\phi := e^{\phi/h} \circ h^2 P \circ e^{-\phi/h} $ and $p_\phi := \sigma_{\rm{scl}}(P_\phi)$.
Assume
\begin{equation*}
	\left\{\begin{aligned}
	& \sigma_{\rm{scl}}(P_\phi) \in S(m) \text{~for some order function~} m(x,\xi) \geq 1 \\
	& \{\Re p_\phi, \Im p_\phi \}(x,\xi) \geq C m(x,\xi)^2,
	\end{aligned}\right.
\end{equation*}
then there exist positive constants $C$ and $\tau_0$ such that for $\forall \tau \geq \tau_0$, $\forall u \in \scrS(\Rn)$, there holds
\begin{equation*}
\boxed{\tau^3 \nrm{e^{\tau \phi} u}^2 \leq C \nrm{e^{\tau \phi} Pu}^2.}
\end{equation*}
\end{thm}


\section*{Exercise}

\begin{ex}
	Prove Corollary \ref{cor:CoWQ-PM2021}.
\end{ex}

\begin{ex} \label{ex:SDe-PM2021}
	Assume $y, \xi, s, \eta \in \Rn$, and define the $4n \times 4n$ matrix $\Sigma$ by \eqref{eq:Sg-PM2021}.
	Check
	\begin{align*}
	\frac 1 2 \agl[\Sigma (y, \xi, s, \eta), (y, \xi, s, \eta)] & = s \cdot \xi - y \cdot \eta, \\
	\frac 1 2 \agl[\Sigma D_{(y, \xi, s, \eta)}, D_{(y, \xi, s, \eta)}] & = D_s \cdot D_\xi + \nabla_y \cdot \nabla_\eta.
	\end{align*}
\end{ex}

\begin{ex}
	Prove Lemma \ref{lem:WeS-PM2021}.
\end{ex}

\chapter{The wavefront set} \label{ch:WF-PM2021}

In this chapter we follows closely \cite[\S 3]{chen2006pseudodifferential}.

\section{Basic facts} \label{sec:Basicfacts-PM2021}

Recall the Peetre's inequality (cf Lemma \ref{lem:Peetre-PM2021}):
\begin{equation} \label{eq:Peetre-PM2021}
	\boxed{\agl[a-b]^m \leq \agl[a]^m \cdot \agl[b]^{|m|}, \quad \forall a, b \in \Rn \text{~and~} \forall m \in \R.}
\end{equation}
We also need a angular separation inequality, which states that 
\begin{equation} \label{eq:AngSep-PM2021}
	\boxed{|a-b| \geq C(|a| + |b|), \quad \forall a \in V_1, \Forall b \in V_2,}
\end{equation}
provided that $V_1$ and $V_2$ are two cone in $\Rn$ separating each other by a positive angle, and the positive constant $C$ depends on this angle.
One example is that $V_1 = \{\rho(\cos\alpha,\sin\alpha) \in \R^2 \,;\, \rho \geq 0,\, 0 \leq \alpha \leq \pi/4\}$ and $V_2 = \{\rho(\cos\alpha,\sin\alpha)  \in \R^2 \,;\, \rho \geq 0,\, 3\pi/4 \leq \alpha \leq \pi\}$.
From \eqref{eq:AngSep-PM2021} we can easily derive
\begin{equation} \label{eq:AngSep2-PM2021}
	\boxed{\agl[a-b]^{-m_1 - m_2} \leq C\agl[a]^{-m_1} \cdot \agl[b]^{-m_2}, \quad \forall a \in V_1,\, \forall b \in V_2,\, \forall m_1, m_2 \geq 0.}
\end{equation}

\begin{proof}[Proof of \eqref{eq:AngSep2-PM2021}]
	From $|a-b| \gtrsim |a| + |b|$ we have $(1 + |a-b|)^{-1} \lesssim (1 + |a|)^{-1}$ and $(1 + |a-b|)^{-1} \lesssim (1 + |b|)^{-1}$, so $(1 + |a-b|)^{-m_1 - m_2} \lesssim (1 + |a|)^{-m_1} (1 + |b|)^{-m_2}$, which is equivalent to \eqref{eq:AngSep2-PM2021}.
\end{proof}

These inequalities are frequently used in microlocal analysis and sometimes play key role in the proofs of microlocal analysis.
We use the notation \underline{$T^* \Rn \backslash 0$} to stand for the cotangent bundle with the zero section excluded.
We deliberately exclude the zero section for some purpose, see Remark \ref{rem:exZS-PM2021}.
We introduce the notion of conic sets, the smooth direction and the wavefront set as follows.

\begin{defn}[Conic set\index{conic set}] \label{defn:coni-PM2021}
	A set $\Gamma \subset  T^* \Rn \backslash 0$ is called a \emph{conic set} if $\Gamma = \omega \times V$ for some $\omega \subset \Rn$ and some set $V \subset \Rn \backslash 0$, where the set $V$ is conic in $\Rn$, i.e.~if $\xi \in V$ then $t \xi \in V$ for all $t > 0$.
\end{defn}

\begin{defn}[$\Smo$\index{smooth direction, $\Smo$}] \label{defn:smo-PM2021}
	Let $m \in \R$ and $a \in S^m$, and $A$ is the {\rm $\Psi$DO} of $a$.
	Let $\Gamma \subset T^* \Rn \backslash 0$ be a open conic set.
	If for every integer $N$ there exists a constant $C_{\Gamma, N}$ such that
	\begin{equation} \label{eq:aSDef-PM2021}
		|a(x,\xi)| \leq C_{\Gamma, N} \agl[\xi]^{-N}, \ \forall (x,\xi) \in \Gamma,
		\quad \text{(or equivalently} \quad
		\boxed{a \in S^{-\infty} \ \text{in} \ \Gamma.} \text{)}
	\end{equation}
	holds, we say $\Gamma$ is a smooth direction set of $a$ (and of $A$).
	We write $\boxed{\Smo(A) := \bigcup \mathscr F}$ where $\mathscr F = \{\Gamma \,;\, \Gamma \text{~is a smooth direction set of~} A\}$.
\end{defn}

It can be checked that $\Smo(a)$ is always open in $T^* \Rn \backslash 0$.
We can also extend the $\Smo(a)$ to $a$ which is in $S^m(\R_x^n \times \R_\xi^N)$ ($n$ and $N$ need not to be the same).
The idea of the smooth direction is that, for any symbol $a \in S^m$, no matter what the value of $m$ is, there are chances that there exists some directions in $\xi$ such that $a$ decays at infinite speed in these direction.

\begin{lem} \label{lem:smo-PM2021}
	Assume $A,B \in \Psi^{+\infty}$, then
	\(
	\Smo(A) \cup \Smo(B) \subset \Smo(A \circ B).
	\)
\end{lem}

The proof is left as an exercise.

\begin{defn}[Wavefront set] \label{defn:wfset-PM2021}
	\index{wavefront set}
	Assume $\Omega \subset \Rn$ is a domain.
	For any given distribution $u \in \scrD'(\Omega)$, the \emph{wavefront set} $\boxed{\wf(u)} \subset T^* \Omega \backslash 0$ of $u$ is defined as a closed subset such that, for any $(x_0,\xi_0) \notin \wf(u)$,  there exists a neighborhood $\omega$ of $x_0$, a function $\varphi \in C_c^\infty(\Omega)$ with $\varphi(x_0) \neq 0$ and $\supp \varphi \subset \omega$, and a cone neighborhood $V$ of $\xi_0$ such that
	\begin{equation} \label{eq:wfset-PM2021}
		|\widehat{\varphi u}(\xi)| \leq C_{N,\omega,V} \agl[\xi]^{-N}, \quad \forall N \in \mathbb N,\, \forall \xi \in V
	\end{equation}
	holds for some positive constant $C_{N,\omega,V}$ depending on $N$, $\omega$ and specially on $V$.
\end{defn}

\begin{exmp}
	Compute the wavefront set of $u(x_1, x_2) := H(x_1)$, where $H$ is the Heaviside function.
	Fix a point $(\bar x_1, \bar x_2)$.
	When $\bar x_1 \neq 0$, we can always find a cutoff function $\varphi \in C_c^\infty(\R^2)$ such that $\varphi u \in C_c^\infty(\R^2)$, so $\widehat{\varphi u}$ will be rapidly decaying.
	This implies that
	\begin{equation} \label{eq:Hws1-PM2021}
	\{ (\bar x_1, \bar x_2; \xi_1, \xi_2) \,;\, \bar x_1 \neq 0 \} \cap \wf(u) = \emptyset,
	\end{equation}
	so $\wf(u)$ is made of points of the form $(0, x_2; \xi_1, \xi_2)$, thus in what follows we assume $\bar x_1 = 0$.
	
	Fix cutoff functions $\varphi_1, \varphi_2 \in C_c^\infty(\R^1)$, such that $\varphi_1$ is supported in the neighborhood of $0$ and $\varphi_2$ in the neighborhood of $\bar x_2$, and denote $\varphi(x_1, x_2) = \varphi_1(x_1) \varphi_2(x_2)$, then
	\begin{align}
		|\widehat{\varphi u}(\xi_1, \xi_2)|
		& \simeq |\int_0^\infty e^{-ix_1 \xi_1} \varphi_1(x_1) \dif x_1| \cdot |\int e^{-ix_2 \xi_2} \varphi_2(x_2) \dif x_2|. \label{eq:Hwf-PM2021}
	\end{align}
	When $\xi_2 \neq 0$, we can continue \eqref{eq:Hwf-PM2021} as
	\begin{equation} \label{eq:Hwf2-PM2021}
		|\widehat{\varphi u}(\xi_1, \xi_2)|
		\simeq \int|\varphi_1(x_1)| \dif x_1 \cdot \agl[\xi_2]^{-\infty}
		\lesssim \agl[\xi_2]^{-\infty}. 
	\end{equation}
	For any cone $V_C := \{ (\xi_1, \xi_2) \,;\, |\xi_1| \leq C |\xi_2|\}$ where $C > 0$, we have
	\(
	|\xi_2| \leq |x_1| + |\xi_2| \lesssim |\xi_2|,
	\)
	which implies
	\(
	\agl[(\xi_1, \xi_2)] \simeq \agl[\xi_2].
	\)
	Hence, \eqref{eq:Hwf2-PM2021} becomes
	\(
	|\widehat{\varphi u}(\xi_1, \xi_2)|
	\lesssim \agl[\xi]^{-\infty}. 
	\)
	Hence, for any constant $C > 0$, we have
	\begin{equation}  \label{eq:Hws2-PM2021}
	\{ (0, \bar x_2; \xi_1, \xi_2) \,;\, \bar x_2 \in \R,\, |\xi_1| \leq C |\xi_2| \} \cap \wf(u) = \emptyset.
	\end{equation}
	Combining \eqref{eq:Hws1-PM2021} and \eqref{eq:Hws2-PM2021}, we see that
	\begin{equation}  \label{eq:Hws3-PM2021}
	\wf(u) \subset \{ (0, \bar x_2; \xi_1, 0) \,;\, \bar x_2 \in \R\}.
	\end{equation}

	Finally, we show
	\begin{equation}  \label{eq:Hws4-PM2021}
		\wf(u) \supset \{ (0, \bar x_2; \xi_1, 0) \,;\, \bar x_2 \in \R\}.
	\end{equation}
	Fix $\bar x_2 \in \R$.
	For any $\varphi \in C_c^\infty(\R^2)$ supported in the neighborhood of $(0, \bar x_2)$, we have
	\begin{align}
		\widehat{\varphi u}(\xi_1, \xi_2)
		& \simeq \int e^{-ix_1 \xi_1} e^{-ix_2 \xi_2} \varphi(x_1, x_2) H(x_1) \dif x_1 \dif x_2 \nonumber \\
		& = \int_0^\infty e^{-ix_1 \xi_1} \tilde \varphi(x_1, \xi_2) \dif x_1
		= i\xi_1^{-1} \tilde \varphi(0, \xi_2) + \xi_1^{-1} \int_0^\infty e^{-ix_1 \xi_1} D_{x_1} \tilde \varphi(x_1, \xi_2) \dif x_1 \nonumber \\
		& = i\xi_1^{-1} \tilde \varphi(0, \xi_2) + \xi_1^{-1} \big( i\xi_1^{-1} D_{x_1} \tilde \varphi(0, \xi_2) + \xi_1^{-1} \int_0^\infty e^{-ix_1 \xi_1} D_{x_1}^2 \tilde \varphi(x_1, \xi_2) \dif x_1 \big) \nonumber \\
		& = i\xi_1^{-1} \tilde \varphi(0, \xi_2) + \mathcal O(|\xi_1|^{-2}), \label{eq:Hwf4-PM2021}
	\end{align}
	where $\tilde \varphi(x_1, \xi_2) = \int_{\R} e^{-ix_2 \xi_2} \varphi(x_1, x_2) \dif x_2$.
	We know $\tilde \varphi(0, \xi_2) = \int_{\R} e^{-ix_2 \xi_2} \varphi(0, x_2) \dif x_2$ is not compactly supported due to the uncertainty principle, so there is $\bar \xi_2 \neq 0$ such that $\tilde \varphi(0, \bar \xi_2) \neq 0$.
	For any cone $V'_C := \{ (\xi_1, \xi_2) \,;\, |\xi_2| \leq C |\xi_1|\}$ where $C > 0$, when $|\xi_1|$ is large enough we always have $(\xi_1, \bar \xi_2) \in V'_C$.
	Hence, \eqref{eq:Hwf4-PM2021} means that in any cone $V'_C$, we have
	\[
	|\widehat{\varphi u}(\xi_1, \bar \xi_2)|
	\simeq |\xi_1|^{-1} |\tilde \varphi(0, \bar \xi_2)| + \mathcal O(|\xi_1|^{-2})
	\simeq |\xi_1|^{-1},
	\]
	so $\varphi u$ is not rapidly decaying in any cone $V'_C$ which contains $\{ (\xi_1, 0) \}$ as their common part.
	By the definition of the wavefront set we can conclude
	\[
	 (0, \bar x_2; \xi_1, 0) \in \wf(u),
	\]
	which implies \eqref{eq:Hws4-PM2021}.
	Combining \eqref{eq:Hws3-PM2021} with \eqref{eq:Hws4-PM2021}, we obtain
	\[
	\underline{\wf(u) = \{ (0, x_2; \xi_1, 0) \,;\, x_2 \in \R,\, \xi_1 \neq 0 \}, \quad \text{where} \quad u(x_1, x_2) = H(x_1).}
	\]
\end{exmp}

\smallskip

It is easy to see from the definition that $(\wf u)^c$ is an open set, so every wavefront set is closed.
In fact, we can relax the restriction on the function $\varphi$ in the Definition \eqref{defn:wfset-PM2021} as follows.

\begin{lem} \label{lem:varphiRelax-PM2021}
	Using the same notation in Definition \ref{defn:wfset-PM2021}, if $(x_0,\xi_0) \notin \wf(u)$, then there exists a another neighborhood $\omega' \subset \omega$ of $x_0$, such that for any $\varphi \in C_c^\infty(\omega')$, which doesn't necessarily satisfy $\varphi(x_0) \neq 0$, the estimates \eqref{eq:wfset-PM2021} holds, with the constant $C_{N,\omega,V}$ now depends also on $\varphi$.
\end{lem}
\begin{proof}
	We call for the result (2) in Theorem \ref{eq:Sim3Pro-PM2021} below in advance. Assume that $(x_0,\xi_0) \notin \wf (u)$, then there exists a neighborhood $\omega$ of $x_0$, a function $\varphi_0 \in C_c^\infty(\Rn)$ with $\varphi_0(x_0) \neq 0$ and a cone neighborhood $V$ of $\xi_0$ such that \eqref{eq:wfset-PM2021} holds.
	Because $\varphi_0(x_0) \neq 0$ and $\varphi_0$ is continuous, there exists another neighborhood $\omega' \subset \omega$ of $x_0$ such that $|\varphi_0(x)| \geq |\varphi_0(x_0)/2| > 0$ for all $x \in \omega'$, and thus $1/\varphi_0(x)$ is well-defined in $\omega'$; 
	the denominator keeps a positive distance from 0 in $\omega'$. Now for any $\phi \in C_c^\infty(\omega')$, we know $\phi/\varphi_0 \in C_c^\infty(\omega')$, hence
	\begin{align*}
		|\widehat{\phi u} (\xi)|
		& = |(\phi/\varphi_0 \cdot \varphi_0 u)^\wedge (\xi)|
		\simeq |\int \widehat{\varphi_0 u}(\xi - \eta) \cdot \widehat{\phi/\varphi_0}(\eta) \dif \eta| \\
		& \leq \int |\widehat{\varphi_0 u}(\xi - \eta)| \cdot |\widehat{\phi/\varphi_0}(\eta)| \dif \eta
		\lesssim \int \agl[\xi - \eta]^{-N} \cdot \agl[\eta]^{-N-n-1} \dif \eta \\
		& \leq  \agl[\xi]^{-N} \int \agl[\eta]^{N} \cdot \agl[\eta]^{-N-n-1} \dif \eta
		\lesssim \agl[\xi]^{-N}, \quad \forall N \in \mathbb N.
	\end{align*}
	Note that we used Peetre's inequality \eqref{eq:Peetre-PM2021}. The proof is complete.
\end{proof}

The wavefront set possesses some simple facts
\cite{chen2006pseudodifferential}.
\begin{thm} \label{eq:Sim3Pro-PM2021}
	Assume that $u$, $v \in \scrD'(\Omega)$ and $a \in C_c^\infty(\Omega)$, then we have
	\begin{enumerate}
	\item $\wf (u + v) \subseteq \wf (u) \cup \wf (v)$;
	
	\item $\wf (au) \subseteq \wf (u)$;
	
	\item $\wf (D^\alpha u) \subseteq \wf (u)$.
	\end{enumerate}
\end{thm}
\begin{proof}
	For (1).
	Assume that $(x_0,\xi_0) \notin \wf (u) \cup \wf (v)$, then $(x_0,\xi_0) \in \big( \wf (u) \big)^c \cap \big( \wf (v) \big)^c$, so there exists neighborhoods $\omega_1$ and $\omega_2$ of $x_0$ and cone neighborhoods $V_1$ and $V_2$ of $\xi_0$ such that
	\begin{align*}
		|\widehat{\varphi_{x_0} u}(\xi)| & \leq C \agl[\xi]^{-N}, \quad \forall \xi \in V_1,\, \forall \varphi_{x_0} \in \scrD(\omega_1) \text{~with~} \varphi_{x_0}(x_0) \neq 0,\, \forall N \in \mathbb N, \\
		|\widehat{\varphi_{x_0} v}(\xi)| & \leq C \agl[\xi]^{-N}, \quad \forall \xi \in V_2,\, \forall \varphi_{x_0} \in \scrD(\omega_2) \text{~with~} \varphi_{x_0}(x_0) \neq 0,\, \forall N \in \mathbb N.
	\end{align*}
	Thus, we have
	\begin{equation*}
		|\widehat{\varphi_{x_0} w}(\xi)| \leq C \agl[\xi]^{-N}, \quad \forall \xi \in V_1 \cap V_2,\, \forall \varphi_{x_0} \in \scrD(\omega_1 \cap \omega_2) \text{~with~} \varphi_{x_0}(x_0) \neq 0,\, \forall N \in \mathbb N,
	\end{equation*}
	where $w = u$ or $v$, so $(x_0,\xi_0) \notin \wf (u + v)$. We can conclude (1).
	
	\smallskip
	
	For (2).
	Assume $(x_0,\xi_0) \notin \wf (u)$, then there exists a neighborhood $\omega$ of $x_0$, a function $\varphi \in \scrD(\Rn)$ with $\varphi(x_0) \neq 0$ and a cone neighborhood $V$ of $\xi_0$ such that for all $\xi \in V$,
	\begin{align*}
		|\widehat{\varphi a u} (\xi)|
		& = |(a \cdot \varphi u)^\wedge (\xi)|
		\simeq |\int \widehat{\varphi u}(\xi - \eta) \cdot \widehat{a}(\eta) \dif \eta| \\
		& \leq \int |\widehat{\varphi u}(\xi - \eta)| \cdot |\widehat{a}(\eta)| \dif \eta
		\lesssim \int \agl[\xi - \eta]^{-N} \cdot \agl[\eta]^{-N-n-1} \dif \eta \\
		& \leq  \agl[\xi]^{-N} \int \agl[\eta]^{N} \cdot \agl[\eta]^{-N-n-1} \dif \eta
		\lesssim \agl[\xi]^{-N}, \quad \forall N \in \mathbb N.
	\end{align*}
	Therefore $(x_0,\xi_0) \notin \wf (au)$.
	We can conclude (2).
	
	\smallskip
	
	For (3).
	Assume $(x_0,\xi_0) \notin \wf (u)$.
	For any $\varphi \in \scrD(\omega')$ where the $\omega'$ is as in Lemma \ref{lem:varphiRelax-PM2021}, we have
	\begin{align*}
		\widehat{\phi D^\alpha u} (\xi)
		& \simeq \int e^{-ix\cdot \xi} \phi(x) D^\alpha u(x) \dif x
		\simeq \int D^\alpha(e^{-ix\cdot \xi} \phi(x)) u(x) \dif x\\
		& = \int \sum_{|\beta| \leq |\alpha|} \binom{\alpha}{\beta} D^{\beta}(e^{-ix\cdot \xi}) D^{\alpha - \beta}\phi(x) u(x) \dif x \\
		& = \sum_{|\beta| \leq |\alpha|} \xi^\beta \binom{\alpha}{\beta} \int e^{-ix\cdot \xi} (D^{\alpha - \beta}\phi \cdot u)(x) \dif x \\
		& = \sum_{|\beta| \leq |\alpha|} \xi^\beta \binom{\alpha}{\beta} (D^{\alpha - \beta}\phi \cdot u)^\wedge(\xi).
	\end{align*}
	Thus, by Lemma \ref{lem:varphiRelax-PM2021},
	\begin{equation*}
		|\widehat{\phi D^\alpha u} (\xi)|
		\leq \sum_{|\beta| \leq |\alpha|} \xi^\beta \binom{\alpha}{\beta} |(D^{\alpha - \beta}\phi \cdot u)^\wedge(\xi)|
		\lesssim \sum_{|\beta| \leq |\alpha|} \binom{\alpha}{\beta} \agl[\xi]^{|\beta|} \agl[\xi]^{-N - |\alpha|}
		\lesssim \agl[\xi]^{-N},
	\end{equation*}
	for any $N \in \mathbb N$.
	Therefore $(x_0,\xi_0) \notin \wf (D^\alpha u)$.
	We can conclude (3).
	
	The proof is complete.
\end{proof}

\section{Wavefront set of product of distributions} \label{sec:wfPr-PM2021}

In this section we deal with some more sophisticated cases of the computations of the wavefront sets.

\subsection{Direct product} \label{subsec:Dipo-PM2021}

\index{direct product}
The first theorem is about the wavefront of the direct product $u \otimes v$.
For $u \colon \scrD(\Omega_x) \to \mathbb C$ and $v \colon \scrD(\Omega_y) \to \mathbb C$, we define the \emph{direct product} $u \otimes v$ of $u$ and $v$ as a distribution  on $\scrD(\Omega_x \times \Omega_y)$ that maps $\varphi(x,y) \in \scrD(\Omega_x \times \Omega_y)$ to $\agl[u,\langle v,\varphi(x,y) \rangle_y]_x$,
\begin{equation*}
	\agl[u \otimes v, \varphi(x,y)] := \agl[u,\langle v,\varphi(x,y) \rangle_y]_x.
\end{equation*}
\begin{thm} \label{thm:Compose-PM2021}
	For any given distributions $u \in \scrD'(\Omega_x)$ and $v \in \scrD'(\Omega_y)$, the wavefront set of the direct product $u \otimes v$ satisfies
	\begin{equation} \label{eq:Compose-PM2021}
		\boxed{\wf (u \otimes v) \subseteq \big( \wf (u) \times \wf (v) \big) \cup \big( \wf (u) \times {\rm supp}_0 v \big) \cup \big( {\rm supp}_0 u \times \wf (v) \big),}
	\end{equation}
	where ${\rm supp}_0 u := \{ (x,0) \,;\, x \in \supp u \}$, ${\rm supp}_0 v := \{ (y,0) \,;\, y \in \supp v \}$.
\end{thm}

\begin{proof}
	Assume that $(x_0,y_0;\xi_0,\eta_0)$ doesn't belong to the right-hand-side of \eqref{eq:Compose-PM2021}.
	
	For the case where $\xi_0 \neq 0$ and $\eta_0 \neq 0$, we know $(x_0;\xi_0) \notin \wf (u)$ and $(y_0;\eta_0) \notin \wf (v)$, so the Fourier transform $(\varphi_{(x_0,y_0)} u \otimes v)^\wedge(\xi_0,\eta_0)$ cannot have the decay of the order $\agl[(\xi,\eta)]^{-N}$ for any $N \in \mathbb N$. Therefore, $(x_0,y_0;\xi_0,\eta_0) \notin \wf (u \otimes v)$.
	
	For the case where $\xi_0 = 0$ and $\eta_0 \neq 0$, if $x_0 \notin \supp u$, obviously we can conclude $(x_0,y_0;\xi_0,\eta_0) \notin \wf (u \otimes v)$, so we suggest that $x_0 \in \supp u$, thus we must have $(y_0;\eta_0) \notin \wf (v)$. Choose $\varphi(x,y) = \varphi_1(x) \varphi_2(y)$ as the cutoff function where $\varphi_1 \in \scrD(\omega_1)$ and $\omega_1$ is some neighborhood of $x_0$. So does $\varphi_2$ accordingly. Thus we have
	\begin{equation*}
		(\varphi u \otimes v)^\wedge(\xi,\eta) = (\varphi_1 u)^\wedge(\xi) \cdot (\varphi_2 v)^\wedge(\eta).
	\end{equation*}
	We have that $(\varphi_2 v)^\wedge(\eta)$ is rapidly decaying and $(\varphi_1 u)^\wedge(\xi)$ grows in polynomial order of $\xi$ in a cone neighborhood of $(0,\eta)$. It's easy to check that, in such a cone neighborhood, we have $\agl[(\xi,\eta)] \lesssim \agl[\eta] \lesssim \agl[(\xi,\eta)]$. Therefore,
	\begin{align*}
		|(\varphi u \otimes v)^\wedge(\xi,\eta)|
		& = |(\varphi_1 u)^\wedge(\xi)| \cdot |(\varphi_2 v)^\wedge(\eta)|
		\lesssim \agl[\eta]^{-N+l} \cdot \agl[\xi]^{-l} \\
		& \lesssim \agl[(\xi,\eta)]^{-N+l} \cdot \agl[(\xi,\eta)]^{-l}
		= \agl[(\xi,\eta)]^{-N},
	\end{align*}
	for any $N \in \mathbb N$. 
	Therefore, $(x_0,y_0;\xi_0,\eta_0) \notin \wf (u \otimes v)$.
	
	The case where $\xi_0 \neq 0$ and $\eta_0 = 0$ is similar to the case where $\xi_0 = 0$ and $\eta_0 \neq 0$. 
	
	The proof is complete.
\end{proof}

\subsection{Product} \label{subsec:Pod-PM2021}

Next, we investigate the product of two distributions. 
In contrast to the product of functions, the product of two distributions\index{product of distributions} is not always well-defined. 
Under certain conditions, the product of two distributions can be defined, at least locally. 
We know that if $\varphi \in C_c^\infty(\Omega)$ and $u \in \scrD'(\Omega)$, 
we have $\varphi u \in \mathcal E'(\Omega)$ and thus the Fourier transform $\widehat{\varphi u}$ is well-defined and can be estimated of polynomial order at infinity. 
Thus we might have chance to define the product by using convolution,
\begin{equation} \label{eq:ProdDef-PM2021}
	(\varphi^2 uv)^\wedge(\xi) := (2\pi)^{-n/2} \int_{\Rn} (\varphi u)^\wedge(\xi - \eta) \cdot (\varphi v)^\wedge(\eta) \dif \eta,
\end{equation}
as long as the convolution \eqref{eq:ProdDef-PM2021} is integrable in the Lebesgue sense and grows under polynomial order in terms of $\agl[\xi]$ at infinity, which implies $\varphi^2 uv \in \scrE'(\Omega)$. 
This leads to the following result.

\begin{thm}[Product Theorem] \label{thm:Product-PM2021}
	For any given distributions $u$, $v \in \scrD'(\Omega)$, when
	\begin{equation} \label{eq:ProdCon-PM2021}
		\big( \wf (u) + \wf (v) \big) \cap O_x = \emptyset,
	\end{equation}
	where $\wf (u) + \wf (v) := \{(x,\xi_1 + \xi_2) \,;\, (x,\xi_1) \in \wf (u),\, (x,\xi_2) \in \wf (v) \}$, and $O_x := \{(x,0) \,;\, x \in \Omega\}$,
	the product ``$uv$'' can be well-defined in the sense of \eqref{eq:ProdDef-PM2021} and its wavefront set satisfies
	\begin{equation} \label{eq:Product-PM2021}
		\boxed{\wf (uv) \subseteq \big( \wf (u) + \wf (v) \big) \cup \wf (u) \cup \wf (v).}
	\end{equation}
\end{thm}

\begin{proof}
	We partially follow \cite[Proposition 11.2.3]{friedlander1998introduction}.
	The proof is divided into two parts:
	first, we show that under condition \eqref{eq:ProdCon-PM2021} the convolution \eqref{eq:ProdDef-PM2021} can be controlled at polynomial of $\xi$;
	second, we show the relation \eqref{eq:Product-PM2021}.

	{\bf Step 1.}
	For any open cone neighborhood $V_3$ of $\wf (u) + \wf (v)$, there exists open cone neighborhoods $V_1'$ and $V_2'$ of $\wf (u)$ and $\wf (v)$, respectively, such that $V_1' + V_2' \subset V_3$.
	Also, there must exists open cone neighborhoods $V_1$ and $V_2$ such that
	\begin{equation} \label{eq:ProdV-PM2021}
		\begin{cases}
			\wf (u) \subsetneqq V_1 \subsetneqq V_1'\\
			\wf (v) \subsetneqq V_2 \subsetneqq V_2'\\
			\wf (u) + \wf (v) \subsetneqq V_1 +  V_2 \subsetneqq V_1' + V_2' \subsetneqq V_3 \\
			\big( V_1 + V_2 \big) \cap O_x = \emptyset
		\end{cases}
	\end{equation}
	The $V_1'$ and $V_2'$ will be utilized in {\bf Step 2}. 
	
	Fix some $x_0 \in \Omega$, we can find some $\varphi \in \scrD(\Omega)$ with $\varphi(x_0) \neq 0$ and also $\varphi$ guarantees $\varphi u$ and $\varphi v$ that \eqref{eq:wfset-PM2021} hold.
	For any fixed $\xi_0 \in \Rn \backslash \{0\}$, the integral \eqref{eq:ProdDef-PM2021} can be divided into four parts,
	\begin{align}
		& \quad \int_{\Rn} (\varphi u)^\wedge(\xi_0 - \eta) \cdot (\varphi v)^\wedge(\eta) \dif \eta \nonumber\\
		& = \int_{\substack{\{\eta \,;\, (x_0,\xi_0 - \eta) \notin V_1\\(x_0,\eta) \notin V_2\}}} (\varphi u)^\wedge(\xi_0 - \eta) \cdot (\varphi v)^\wedge(\eta) \dif \eta
		+ \int_{\substack{\{\eta \,;\, (x_0,\xi_0 - \eta) \notin V_1\\(x_0,\eta) \in V_2\}}} (\varphi u)^\wedge(\xi_0 - \eta) \cdot (\varphi v)^\wedge(\eta) \dif \eta \nonumber\\
		& + \int_{\substack{\{\eta \,;\, (x_0,\xi_0 - \eta) \in V_1\\(x_0,\eta) \notin V_2\}}} (\varphi u)^\wedge(\xi_0 - \eta) \cdot (\varphi v)^\wedge(\eta) \dif \eta
		+ \int_{\substack{\{\eta \,;\, (x_0,\xi_0 - \eta) \in V_1\\(x_0,\eta) \in V_2\}}} (\varphi u)^\wedge(\xi_0 - \eta) \cdot (\varphi v)^\wedge(\eta) \dif \eta \nonumber\\
		& =: I_1 + I_2 + I_3 + I_4. \label{eq:ProInt0-PM2021}
	\end{align}
	The condition \eqref{eq:ProdCon-PM2021} will (only) be used to estimate $I_4$.
	
	According to Definition \ref{defn:wfset-PM2021} and Peetre's inequality, we can estimate $I_1$ as
	\begin{align}
		|I_1| 
		& \leq \int_{\substack{\{\eta \,;\, (x_0,\xi_0 - \eta) \notin V_1\\(x_0,\eta) \notin V_2\}}} |(\varphi u)^\wedge(\xi_0 - \eta)| \cdot |(\varphi v)^\wedge(\eta)| \dif \eta \nonumber\\
		& \lesssim \int_{\Rn} \agl[\xi_0 - \eta]^{-N} \cdot \agl[\eta]^{-N-n-1} \dif \eta
		\lesssim \agl[\xi_0]^{-N} \int_{\Rn} \agl[\eta]^{N} \cdot \agl[\eta]^{-N-n-1} \dif \eta \nonumber\\
		& \lesssim \agl[\xi_0]^{-N}, \quad \forall N \in \mathbb N. \label{eq:ProInt1-PM2021}
	\end{align}
	
	For $I_2$, we know that $\varphi v \in \scrE'(\Omega)$, so $|(\varphi v)^\wedge(\eta)|$ can be dominated by $\agl[\eta]^l$ for some $l \in \mathbb N$, thus
	\begin{align}
		|I_2| 
		& \leq \int_{\substack{\{\eta \,;\, (x_0,\xi_0 - \eta) \notin V_1\\(x_0,\eta) \in V_2\}}} |(\varphi u)^\wedge(\xi_0 - \eta)| \cdot |(\varphi v)^\wedge(\eta)| \dif \eta \nonumber\\
		& \lesssim \int_{\substack{\{\eta \,;\, (x_0,\xi_0 - \eta) \notin V_1\\(x_0,\eta) \in V_2\}}} \agl[\xi_0 - \eta]^{-l-n-1} \cdot \agl[\eta]^{l} \dif \eta \nonumber\\
		& \lesssim \agl[\xi_0]^{l+n+1} \int_{\Rn} \agl[\eta]^{-l-n-1} \cdot \agl[\eta]^{l} \dif \eta \quad (\text{Peetre's inequality}) \nonumber\\
		& \lesssim \agl[\xi_0]^{l+n+1}. \label{eq:ProInt2'-PM2021}
	\end{align}
	
	The estimation of $I_3$ is similar to that of $I_2$,
	\begin{align}
		|I_3| 
		& \leq \int_{\substack{\{\eta \,;\, (x_0,\xi_0 - \eta) \in V_1\\(x_0,\eta) \notin V_2\}}} |(\varphi u)^\wedge(\xi_0 - \eta)| \cdot |(\varphi v)^\wedge(\eta)| \dif \eta \nonumber\\
		& = \int_{\substack{\{\eta \,;\, (x_0,\gamma) \in V_1\\(x_0,\xi_0 - \gamma) \notin V_2\}}} |(\varphi u)^\wedge(\gamma)| \cdot |(\varphi v)^\wedge(\xi_0 - \gamma)| \dif \gamma \quad (\gamma = \xi_0 - \eta) \nonumber\\
		& \lesssim \int_{\substack{\{\eta \,;\, (x_0,\gamma) \in V_1\\(x_0,\xi_0 - \gamma) \notin V_2\}}} \agl[\gamma]^{l'} \cdot \agl[\xi_0 - \gamma]^{-l'-n-1} \dif \gamma \nonumber\\
		& \lesssim \agl[\xi_0]^{l'+n+1} \int_{\Rn} \agl[\gamma]^{l'} \cdot \agl[\gamma]^{-l'-n-1} \dif \gamma  \quad (\text{Peetre's inequality}) \nonumber\\
		& \lesssim \agl[\xi_0]^{l'+n+1}. \label{eq:ProInt3'-PM2021}
	\end{align}
	
	For $I_4$, we can show that the domain of integration $\{\eta \,;\, (x_0,\xi_0 - \eta) \in V_1,\, (x_0,\eta) \in V_2\}$ is bounded. 
	We temporarily use $\hat \eta$ to mean the direction of $\eta$, $\hat \eta = \eta/|\eta|$. 
	Therefore the direction of the vector $\xi_0 - \eta$ is parallel to $\xi_0/|\eta| - \hat \eta$, thus when $|\eta|$ is large enough, $(x_0,\xi_0 - \eta)$ will be in $-V_2 := \{(x,-\eta) \,;\, (x,\eta) \in V_2\}$. 
	We know $(x_0,\xi_0 - \eta) \in V_1$, so the set $\{ (x_0,\gamma) \,;\, (x_0,\gamma) \in V_1,\, (x_0,-\gamma) \in V_2 \}$ is not empty. 
	This contradict with $\big( V_1 + V_2 \big) \cap O_x = \emptyset$ in \eqref{eq:ProdV-PM2021}.
	Therefore, when $|\eta|$ is large enough, the conditions $(x_0,\xi_0 - \eta) \in V_1$ and $(x_0,\eta) \in V_2$ cannot be satisfies simultaneously, which implies the set $\{\eta \,;\, (x_0,\xi_0 - \eta) \in V_1,\, (x_0,\eta) \in V_2\}$ is bounded. Therefore,
	\begin{align}
		|I_4| 
		& \leq \int_{\substack{\{\eta \,;\, (x_0,\xi_0 - \eta) \notin V_1\\(x_0,\eta) \notin V_2\}}} |(\varphi u)^\wedge(\xi_0 - \eta)| \cdot |(\varphi v)^\wedge(\eta)| \dif \eta \nonumber\\
		& \lesssim \int_{\{\eta \,;\, |\eta| \text{~bounded}\}} \agl[\xi_0 - \eta]^{l'} \cdot \agl[\eta]^{l} \dif \eta \nonumber\\
		& \lesssim \agl[\xi_0]^{l'} \int_{\{\eta \,;\, |\eta| \text{~bounded}\}} \agl[\eta]^{|l'|} \cdot \agl[\eta]^{l} \dif \eta \quad (\text{Peetre's inequality}) \nonumber\\
		& \lesssim \agl[\xi_0]^{l'}. \label{eq:ProInt4'-PM2021}
	\end{align}
	
	From \eqref{eq:ProInt0-PM2021}-\eqref{eq:ProInt4'-PM2021}, we conclude that the convolution \eqref{eq:ProdDef-PM2021} is Lebesgue integrable and grows with polynomial order in terms of $\agl[\xi_0]$, thus	$\varphi^2 uv \in \scrE'(\Omega)$. Now $uv \in \scrD'(\Omega)$ is well-defined.

	{\bf Step 2.}
	Under condition \eqref{eq:ProdCon-PM2021}, we study the wavefront set of $uv$.
	Assume that 
	\begin{equation} \label{eq:Prodxi0-PM2021}
		(x_0,\xi_0) \notin V_3 \cup V_1' \cup V_2',
	\end{equation} 
	Again, the condition \eqref{eq:ProdCon-PM2021} will (only) be used to estimate $I_4$. Note the particular arrangements of the $V_1$, $V_1'$ and $V_2$, $V_2'$ in \eqref{eq:Prodxi0-PM2021} and \eqref{eq:ProInt0-PM2021}. We will utilize these arrangements combining with condition \eqref{eq:ProdV-PM2021} to estimates $I_2$ and $I_3$.
	
	We estimate $I_1$ the same way as in {\bf Step 1}, i.e.~as in \eqref{eq:ProInt1-PM2021}.
	
	For $I_2$, to get the rapid decay w.r.t.~$\xi_0$, we shall adapt different strategy. We know that $\varphi v \in \scrE'(\Omega)$, so $|(\varphi v)^\wedge(\eta)|$ can be dominated by $\agl[\eta]^l$ for some $l \in \mathbb N$. 
	Thanks to the condition \eqref{eq:Prodxi0-PM2021}, we know $(x_0,\xi_0) \notin V_2'$ and now $(x_0,\eta) \in V_2$. Because $V_2 \subsetneqq V_2'$, we know that $V_2$ and $V_2'$ are separated with a positive angle, so the inequality \eqref{eq:AngSep2-PM2021} can apply to $\agl[\xi_0 - \eta]$,
	\begin{align}
		|I_2| 
		& \leq \int_{\substack{\{\eta \,;\, (x_0,\xi_0 - \eta) \notin V_1\\(x_0,\eta) \in V_2\}}} |(\varphi u)^\wedge(\xi_0 - \eta)| \cdot |(\varphi v)^\wedge(\eta)| \dif \eta \nonumber\\
		& \lesssim \int_{\substack{\{\eta \,;\, (x_0,\xi_0 - \eta) \notin V_1\\(x_0,\eta) \in V_2\}}} \agl[\xi_0 - \eta]^{-N-l-n-1} \cdot \agl[\eta]^{l} \dif \eta \nonumber\\
		& \lesssim \agl[\xi_0]^{-N} \int_{\Rn} \agl[\eta]^{-l-n-1} \cdot \agl[\eta]^{l} \dif \eta \quad \big( \text{by~} \eqref{eq:AngSep2-PM2021} \big) \nonumber\\
		& \lesssim \agl[\xi_0]^{-N}, \quad \forall N \in \mathbb N. \label{eq:ProInt2-PM2021}
	\end{align}
	
	The estimation of $I_3$ is similar to \eqref{eq:ProInt2-PM2021},
	\begin{align}
		|I_3| 
		& \leq \int_{\substack{\{\eta \,;\, (x_0,\xi_0 - \eta) \in V_1\\(x_0,\eta) \notin V_2\}}} |(\varphi u)^\wedge(\xi_0 - \eta)| \cdot |(\varphi v)^\wedge(\eta)| \dif \eta \nonumber\\
		& = \int_{\substack{\{\eta \,;\, (x_0,\gamma) \in V_1\\(x_0,\xi_0 - \gamma) \notin V_2\}}} |(\varphi u)^\wedge(\gamma)| \cdot |(\varphi v)^\wedge(\xi_0 - \gamma)| \dif \gamma \quad (\gamma = \xi_0 - \eta) \nonumber\\
		& \lesssim \int_{\substack{\{\eta \,;\, (x_0,\gamma) \in V_1\\(x_0,\xi_0 - \gamma) \notin V_2\}}} \agl[\gamma]^{l'} \cdot \agl[\xi_0 - \gamma]^{-N-l'-n-1} \dif \gamma \nonumber\\
		& \lesssim \agl[\xi_0]^{-N} \int_{\Rn} \agl[\gamma]^{l'} \cdot \agl[\gamma]^{-l'-n-1} \dif \gamma \quad \big( \text{by~} \eqref{eq:AngSep2-PM2021} \big)\nonumber\\
		& \lesssim \agl[\xi_0]^{-N}, \quad \forall N \in \mathbb N. \label{eq:ProInt3-PM2021}
	\end{align}
	
	Now we work on $I_4$. From \eqref{eq:ProdCon-PM2021}, \eqref{eq:ProdV-PM2021} and \eqref{eq:Prodxi0-PM2021}, we know that $\xi_0 \notin V_1 + V_2$, thus the set $\{\eta \,;\, (x_0,\xi_0 - \eta) \in V_1,\, (x_0,\eta) \in V_2\}$ is empty. 
	Therefore $I_4 = 0$. Combining this fact with \eqref{eq:ProInt0-PM2021}, \eqref{eq:ProInt1-PM2021}, \eqref{eq:ProInt2-PM2021} and \eqref{eq:ProInt3-PM2021}, we arrive at
	\begin{equation*}
		|(\varphi^2 uv)^\wedge(\xi)| \leq C_N \agl[\xi_0]^{-N}, \quad \forall N \in \mathbb N,
	\end{equation*}
	for $(x_0,\xi_0) \notin  V_3 \cup V_1' \cup V_2'$. This implies $\wf (u+v) \subset V_3 \cup V_1' \cup V_2'$. The sets $V_3$, $V_1'$ and $V_2'$ can be close to $\wf (u) + \wf (v)$, $\wf (u)$ and $\wf (v)$, respectively, as close as possible, so we arrive at \eqref{eq:Product-PM2021}. 
	The proof is complete.
\end{proof}

\subsection{Convolution} \label{subsec:Cowf-PM2021}

We define
\begin{equation} \label{eq:wfxo-PM2021}
	\left\{\begin{aligned}
		\wf' (K) & := \{ (x,y; \xi, -\eta) \,;\, (x,y; \xi, \eta) \in \wf (K) \}, \\
		\wf_x (K) & := \{ (x; \xi) \,;\, \exists y \st (x,y; \xi, 0) \in \wf (K) \}, \\
		A \circ B & := \{ (x,\xi) \,;\, \exists (y,\eta) \in B \st (x,y;\xi,\eta) \in A \}, \\
		O_x & := \{(x,0) \,;\, x \in \Omega\}.
	\end{aligned}\right.
\end{equation}

We need the following lemma.

\begin{lem} \label{lem:iwff-PM2021}
	Assume $f \in \mathcal D'(\Omega \times \Omega)$, and there is a compact set $\mathcal K \subset \Omega$ such that $\supp f \subset \Omega \times \mathcal K$.
	Then
	\begin{equation} \label{eq:iwff-PM2021}
	\boxed{\wf \big( \int f(x,y) \dif y \big) = \wf_x (f).}
	\end{equation}
\end{lem}

\begin{proof}
	{\bf Step 1.} ($\supset$).
	Assume $(x_0, \xi_0) \notin \wf \big( \int f(x,y) \dif y \big)$, then there exists $\chi_{x_0} \in C_c^\infty(\Omega)$ such that
	\[
	\int e^{-ix_0 \cdot \xi_0} \chi_{x_0}(x) f(x,y) \dif y \dif x
	= \mathcal O(\agl[\xi_0]^{-\infty})
	= \mathcal O(\agl[(\xi_0, 0)]^{-\infty}).
	\]
	which gives
	\[
	\forall \bar y \in \mathcal K, \ \int e^{-i(x_0, \bar y) \cdot (\xi_0, 0)} \chi_{x_0}(x) \chi(y) f(x,y) \dif (x,y)
	= \mathcal O(\agl[(\xi_0, 0)]^{-\infty}),
	\]
	where $\chi \in C_c^\infty(\Omega)$ with $\chi \equiv 1$ on $\mathcal K$.
	This means $(x_0, \bar y; \xi_0, 0) \notin \wf(f)$ for $\forall \bar y \in \mathcal K$, so $(x_0, \xi_0) \notin \wf_x (f)$.
	Hence,
	\[
	\wf \big( \int f(x,y) \dif y \big) \supset \wf_x (f).
	\]
	
	{\bf Step 2.} ($\subset$).
	Assume $(x_0, \xi_0) \notin \wf_x(f)$, then for $\forall \bar y \in \Omega$ we have $(x_0, \bar y; \xi_0, 0) \notin \wf(f)$.
	Therefore, for $\forall \bar y \in \Omega$, there is a neighborhood of $\bar y$ such that
	\begin{equation} \label{eq:echy-PM2021}
	\int e^{-i(x_0, \bar y) \cdot (\xi_0, 0)} \chi_{x_0}(x) \chi(y) f(x,y) \dif (x,y)
	= \mathcal O(\agl[(\xi_0, 0)]^{-\infty}),
	\end{equation}
	for $\chi \equiv 1$ in that neighborhood.
	Because $\mathcal K$ is compact, so by using partition of unity technique, we can remove the term $\chi(y)$ in \eqref{eq:echy-PM2021}, and obtain
	\[
	\int e^{-ix_0 \cdot \xi_0} \chi_{x_0}(x) \big( \int f(x,y) \dif y \big) \dif x
	= \mathcal O(\agl[\xi_0]^{-\infty}),
	\]
	which gives
	\(
	(x_0, \xi_0) \notin \wf(\int f(x,y) \dif y).
	\)
	Hence,
	\[
	\wf \big( \int f(x,y) \dif y \big) \subset \wf_x (f).
	\]
	
	The proof is done.
\end{proof}

\begin{thm} \label{thm:KernelAct-PM2021}
	Assume $u \in \scrE'(\Omega)$, and $K \in \scrD'(\Omega \times \Omega)$. 
	When $\big( \wf' (K) \circ \wf (u) \big) \cap O_x = \emptyset$, the distribution
	\[
	w(x) := \agl[K(x,y),u(y)]_y
	\]
	is well-defined in the sense that
	\[
	\forall \varphi \in \scrD(\Omega), \quad w(\varphi) := \agl[K(x,y),u(y) \otimes \varphi(x)],
	\]
	and we have the following \underline{canonical relation}: \index{canonical relation}
	\begin{equation} \label{eq:KernelAct-PM2021}
		\boxed{\wf (w) \subseteq \big( \wf' (K) \circ \wf (u) \big) \cup \wf_x (K).}
	\end{equation}
\end{thm}

\begin{rem}
	Note that the $\supp u$ should be contained in $\Omega$, otherwise the $w$ may be ill-defined.
\end{rem}

\begin{proof}
	{\bf Step 1.} Turn into product.
	Denote $\tilde u(x,y) = 1(x) \otimes u(y)$ where $1(x)$ is the constant function.
	The wavefront set of the function $1(x)$ is empty, so by Theorem \ref{thm:Compose-PM2021} we have
	\begin{align}
		\wf (\tilde u)
		& \subset \big( \wf (1) \times \wf (u) \big) \cup \big( \wf (1) \times {\rm supp}_0 u \big) \cup \big( {\rm supp}_0 1 \times \wf (u) \big) \nonumber \\
		& = \emptyset \cup \emptyset \cup \big( {\rm supp}_0 1 \times \wf (u) \big) \nonumber \\
		& = \{(x, y; 0, \eta) \,;\, x \in \supp \Omega,\, (y,\eta) \in \wf(u) \}. \label{eq:t1u-PM2021}
	\end{align}
	
	The $w(x)$ can be written as
	\[
	w(x)
	= \agl[K(x,y),u(y)]_y
	= \int K(x,y) \cdot \tilde u(x, y) \dif y
	= \int K \tilde u(x,y) \dif y.
	\]
	where $K \tilde u$ stands for the product of $K$ and $\tilde u$.
	By Theorem \ref{thm:Product-PM2021}, to guarantee the product $K \tilde u$ is well-defined, we need to check if the prerequisite
	\begin{equation} \label{eq:KuO-PM2021}
		\big( \wf (K) + \wf (\tilde u) \big) \cap O_{x,y} = \emptyset
	\end{equation}
	is true.
	It can be shown that the condition $\big( \wf' (K) \circ \wf (u) \big) \cap O_x = \emptyset$ guarantees \eqref{eq:KuO-PM2021} (see Exercise \ref{ex:KOGu-PM2021}), so  $K \tilde u$ is well-defined.

	Because $u \in \scrE'(\Omega)$, we see that for $\forall x \in \Omega$, $\supp K \tilde u(x,\cdot)$ is uniformly compact, so by Lemma \ref{lem:iwff-PM2021} we have
	\(
	\wf \big( \int K \tilde u(x,y) \dif y \big) = \wf_x (K \tilde u),
	\)
	so,
	\begin{equation} \label{eq:wK1u-PM2021}
	\wf(w)
	= \wf \big( \int K \tilde u(x, y) \dif y \big) = \wf_x (K \tilde u)
	= \wf (K \tilde u) \circ O_y,
	\end{equation}
	where we used the fact that for general distribution $f \in \mathcal D'(\Omega \times \Omega)$,
	\[
	\wf_x (f) = \wf (f) \circ O_y.
	\]
	
	{\bf Step 2.} Use product Theorem.
	Combining \eqref{eq:wK1u-PM2021} with Theorem \ref{thm:Product-PM2021}, we can have
	\begin{align}
	\wf(w)
	& = \wf (K \tilde u) \circ O_y
	\subset \Big( \big( \wf (K) + \wf (\tilde u) \big) \cup \wf (K) \cup \wf (\tilde u) \Big) \circ O_y \nonumber \\
	& = M_1 \cup M_2 \cup M_3, \label{eq:wM13-PM2021}
	\end{align}
	where
	\begin{equation*}
	\left\{\begin{aligned}
	M_1 & := \big( \wf (K) + \wf (\tilde u) \big) \circ O_y, \\
	M_2 & := \wf (K) \circ O_y, \\
	M_3 & := \wf (\tilde u) \circ O_y.
	\end{aligned}\right.
	\end{equation*}
	The set $\wf (K) + \wf (\tilde u)$ can be expressed as
	\begin{align*}
	& \ \wf (K) + \wf (\tilde u) \\
	= & \ \{ (x,y; \xi, \eta) \,;\, \xi = \xi_1 + \xi_2,\, \eta = \eta_1 + \eta_2,\, (x,y; \xi_1, \eta_1) \in \wf(K),\, (x,y; \xi_2, \eta_2) \in \wf(\tilde u) \} \\
	= & \ \{ (x,y; \xi, \eta) \,;\, \eta = \eta_1 + \eta_2,\, (x,y; \xi, \eta_1) \in \wf(K),\, (y, \eta_2) \in \wf(u) \}. \quad \text{(by \eqref{eq:t1u-PM2021})}
	\end{align*}
	Thus,
	\begin{align}
	M_1
	& = \big( \wf (K) + \wf (\tilde u) \big) \circ O_y
	= \{ (x, \xi) \,;\, (x,y; \xi, -\eta) \in \wf(K),\, (y, \eta) \in \wf(u) \} \nonumber \\
	& = \{ (x, \xi) \,;\, (x,y; \xi, \eta) \in \wf'(K),\, (y, \eta) \in \wf(u) \}
	= \wf'(K) \circ \wf(u). \label{eq:M1Ku-PM2021}
	\end{align}
	By  \eqref{eq:t1u-PM2021} it can also be checked that
	\begin{equation} \label{eq:M23-PM2021}
	M_2 = \wf_x (K),
	\quad 
	M_3 = \emptyset.
	\end{equation}
	Combining \eqref{eq:M1Ku-PM2021}, \eqref{eq:M23-PM2021} with \eqref{eq:wM13-PM2021}, we obtain \eqref{eq:KernelAct-PM2021}.
	The proof is complete.
\end{proof}

\begin{rem} \label{rem:exZS-PM2021}
	In Theorem \ref{thm:KernelAct-PM2021}, if we know in advance that
	\[
	\wf(K) \subset (T^* \Omega_x \backslash 0) \times (T^* \Omega_y \backslash 0),
	\]
	then $\wf_x (K) = \emptyset$ and \eqref{eq:KernelAct-PM2021} can be reduced to
	\begin{equation} \label{eq:KeA0-PM2021}
	\boxed{\wf (w) \subseteq \wf' (K) \circ \wf (u).}
	\end{equation}
	The set $\wf' (K)$ is called the \emph{twist} \index{twist} of $\wf(K)$, and the operation ``$\wf' (K) \circ$'' is called \emph{canonical relation} \index{canonical relation} of the operator:
	\[
	u(y) \mapsto w(x) := \agl[K(x,y), u(y)]
	\]
	which takes $K$ as its kernel.
	These can be generalized to the theory of \emph{Fourier integral operators}\index{Fourier integral operators}.
\end{rem}

\section{The wavefront sets of Fourier integral operators} \label{sec:FIOwf-PM2021}

Recall the notion of phase function given in Definition \ref{defn:phaf-PM2021}.

\begin{thm} \label{thm:DisInt-PM2021}
	Assume $\varphi \in C^\infty(\Rn \times \R^N)$ is a phase function of order $1$, and $a \in S^m$ is a symbol.
	Define $A(x)$ as
	\begin{equation} \label{eq:DisIntA-PM2021}
		A(x) := \int e^{i\varphi(x,\theta)} a(x,\theta) \dif \theta,
	\end{equation}
	where the integral is understood as an oscillatory integral.
	Then $A$ induces a distribution (also denoted as $A$) $A \in \scrD'(\Omega)$ for any domain $\Omega \subset \Rn$, i.e.~$A \colon u \in \scrD(\Omega) \mapsto I_\varphi(au)$ by
	\begin{equation*}
		A(u) := I_\varphi(au) = \agl[A,u] = \int e^{i\varphi(x,\theta)} a(x,\theta) u(x) \dif x \dif \theta
	\end{equation*}
	in oscillatory integral sense.
	The wavefront set of $A$ satisfies
	\begin{equation} \label{eq:DisInt-PM2021}
		\boxed{\wf (A) \subset \{(x,\varphi_x(x,\theta)) \,;\, \varphi_\theta(x,\theta) = 0,\, (x,\theta) \notin \Smo(a) \}.}
	\end{equation}
\end{thm}

\begin{rem} \label{rem:Dit-PM2021}
	When the following conditions are satisfied, the inclusion ``$\subset$'' in \eqref{eq:DisInt-PM2021} can be improved to ``$=$'' (see contexts preceding \cite[Theorem 3.9]{ho202Xmicro}, \cite[Theorem 3.9]{chen1997fio}):
	\begin{enumerate}
		\item the phase function $\varphi$ is non-degenerate on $C_\phi := \{(x,\theta) \,;\, \varphi_\theta(x,\theta) = 0,\, (x,\theta) \notin \Smo(a) \}$, i.e.~the $N$-$(n+N)$ matrix $\df \phi_\theta$ is full rank on $C_\phi$, here
		\[
		\df \phi_\theta(x,\theta)
		= \begin{pmatrix}
		\phi_{\theta x}(x,\theta) & \phi_{\theta \theta}(x,\theta)
		\end{pmatrix};
		\]
		
		\item the map $(x,\theta) \mapsto (x,\varphi_x(x,\theta))$ is injective when restricted to $C_\phi$.
	\end{enumerate}
	Readers may distinguish the $A$ appeared in Theorem \ref{thm:DisInt-PM2021} with the operator $B$ defined as
	\[
	Bu(x) := \int e^{i\varphi(x,y,\theta)} a(x,y,\theta) u(y) \dif y \dif \theta.
	\]
	The $A$ is a distribution while the $B$ just defined is an operator, namely, $A$ maps a function to a scalar while $B$ maps a function to another function.
	
	However, $A$ is a generalization of $B$, because
	\begin{align*}
		\agl[Bu(x), v(x)]
		& = \int e^{i\varphi(x,y,\theta)} a(x,y,\theta) u(y) v(x) \dif x \dif y \dif \theta \\
		& = \agl[\int e^{i\varphi(x,y,\theta)} a(x,y,\theta) \dif \theta, (v \otimes u)(x, y)] \\
		& = \agl[K_B, v \otimes u].
	\end{align*}
	where $\tilde B$ is defined as
	\[
	K_B(x,y) := \int e^{i\varphi(x,y,\theta)} a(x,y,\theta) \dif \theta.
	\]
	Hence the operator $B$ can be turned into a form of \eqref{eq:DisIntA-PM2021}.
	
	Moreover, we have $Bu(x) = \agl[K_B(x,y), u(y)]$, so by combining Theorems \ref{thm:DisInt-PM2021} and \ref{thm:KernelAct-PM2021}, hopefully we can obtain $\wf(Bu)$.
\end{rem}

\begin{proof}[Short proof of Theorem \ref{thm:DisInt-PM2021}]
	This short proof is for summarizing the key idea of proving this theorem and thus the details may not be rigorously correct.
	After this short proof, we also present a formal proof of Theorem \ref{thm:DisInt-PM2021}.
	
	According to the Definition \ref{defn:wfset-PM2021} ,we fix a cutoff function $\phi$ with $\phi(x_0) \neq 0$ and compute
	\begin{align*}
		\widehat{\phi A}(\xi) 
		& = A(\phi e^{-ix\cdot \xi}) 
		= (2\pi)^{-n/2} \int e^{i(\varphi(x,\theta) - x\cdot \xi)} \phi(x) a(x,\theta) \dif x \dif \theta,
	\end{align*}
	and the basic idea is to use $N$ times (with $N$ large enough) the operator $L := \frac {(\varphi_x(x,\theta) - \xi) \cdot \nabla_x} {i|\varphi_x(x,\theta) - \xi|^2}$ acting on $e^{i(\varphi(x,\theta) - x\cdot \xi)}$ and the fact \eqref{eq:AngSep2-PM2021} to get the desired estimate.
	But in order to do so, one needs to first address some singularities in the oscillatory integral.
	We have
	\begin{align}
		\widehat{\phi A}(\xi) 
		& \simeq \int_{\R_x^n} e^{-ix\cdot \xi} \phi(x) \big( \int_{\R_\theta^n} e^{i\varphi(x,\theta)} a(x,\theta) \dif \theta \big) \dif x \nonumber\\
		& \simeq \int_{\R_x^n} e^{-ix\cdot \xi} \phi(x) \big( \int_{\R_\theta^n} (\frac {\varphi_\theta (x,\theta) \cdot \nabla_\theta} {|\varphi_\theta (x,\theta)|^2} )^N e^{i\varphi(x,\theta)} a(x,\theta) \dif \theta \big) \dif x \nonumber\\
		& \sim \int_{\R_x^n} e^{-ix\cdot \xi} \phi(x) \big( \int_{\R_\theta^n}  e^{i\varphi(x,\theta)} |\varphi_\theta (x,\theta)|^{-N}
		\partial_\theta^N a(x,\theta) \dif \theta \big) \dif x. \label{eq:phiAShort1-PM2021}
	\end{align}
	Then as $N$ be large enough, the $\partial_\theta^N a(x,\theta)$ will be integrable w.r.t.~$\theta$.
	But we notice that $|\varphi_\theta (x,\theta)|^{-1}$ has singularity at $\theta = 0$, so we first exclude the neighborhood of the origin of $\theta$ by using a cutoff function $\chi$ with $\chi(0) \neq 0$ as follows
	\begin{align*}
		\widehat{\phi A}(\xi) 
		& \simeq \int e^{i(\varphi(x,\theta) - x\cdot \xi)} \phi(x) \chi(\theta) a(x,\theta) \dif x \dif \theta + \int e^{i(\varphi(x,\theta) - x\cdot \xi)} \phi(x) (1 - \chi(\theta)) a(x,\theta) \dif x \dif \theta \\
		& = \int_{\R_x^n} e^{-ix\cdot \xi} \phi(x) \big( \int_{\R_\theta^n} e^{i\varphi(x,\theta)} \chi(\theta) a(x,\theta) \dif \theta \big) \dif x \\
		& \quad + \int_{\R_x^n} e^{-ix\cdot \xi} \phi(x) \big( \int_{\Gamma_x} e^{i\varphi(x,\theta)} (1 - \chi(\theta)) a(x,\theta) \dif \theta \big) \dif x \\
		& \quad + \int_{\theta \notin \Gamma_x} e^{i(\varphi(x,\theta) - x\cdot \xi)} \phi(x) (1 - \chi(\theta)) a(x,\theta) \dif x \dif \theta \\
		& =: I_1(\xi) + I_2(\xi) + I_3(\xi),
	\end{align*}	
	where
	\begin{equation}
		\Gamma_x := \{ \theta \,;\, a(x,\theta) = \mathcal O(|\theta|^{-\infty}) \}.
	\end{equation}
	The $I_1(\xi)$ and $I_2(\xi)$ are $O(|\xi|^{-\infty})$ as $|\xi| \to +\infty$, because these terms $\int_{\R_\theta^n} e^{i\varphi(x,\theta)} \chi(\theta) a(x,\theta) \dif \theta$ and $\int_{\Gamma_x} e^{i\varphi(x,\theta)} (1 - \chi(\theta)) a(x,\theta) \dif \theta$ are smooth in terms of $x$ (for the first term, it is because the actual integral domain 
	is compact, i.e.~is contained in $\supp \theta$;
	for the second term, it is because the integrand decays at infinity order).
	Then we can compute $I_3$ as follows,
	\begin{align*}
		I_3(\xi) 
		& \simeq \int_{\R_x^n} e^{-ix\cdot \xi} \phi(x) \big( \int_{\varphi_\theta(x,\theta) \neq 0 \text{~and~} \theta \notin \Gamma_x} (\frac {\varphi_\theta (x,\theta) \cdot \nabla_\theta} {|\varphi_\theta (x,\theta)|^2} )^N e^{i\varphi(x,\theta)} (1 - \chi(\theta)) a(x,\theta) \dif \theta \big) \dif x \\
		& \quad + \int_{\R_x^n} e^{-ix\cdot \xi} \phi(x) \big( \int_{\varphi_\theta(x,\theta) = 0 \text{~and~} \theta \notin \Gamma_x} e^{i\varphi(x,\theta)} (1 - \chi(\theta)) a(x,\theta) \dif \theta \big) \dif x \\
		& =: I_4(\xi) + I_5(\xi).
	\end{align*}
	
	{\bf Now here comes the key point: to obtain $\agl[\xi]^{-N}$, for $I_4$ we differentiate $e^{-ix\cdot \xi}$ w.r.t.~$x$, and for $I_5$ we differentiate $e^{i(\varphi - x\cdot \xi)}$  w.r.t.~$x$.}
	
	We can estimate $I_4$ by using the computation as in \eqref{eq:phiAShort1-PM2021},
	\begin{align*}
		I_4(\xi) 
		& \simeq \int_{\R_x^n} e^{-ix\cdot \xi} \phi(x) \big( \int_{\varphi_\theta(x,\theta) \neq 0 \text{~and~} \theta \notin \Gamma_x} (\frac {\varphi_\theta (x,\theta) \cdot \nabla_\theta} {|\varphi_\theta (x,\theta)|^2} )^N e^{i\varphi(x,\theta)} (1 - \chi(\theta)) a(x,\theta) \dif \theta \big) \dif x \\
		& \sim \int_{\R_x^n} e^{-ix\cdot \xi} \phi(x) \big( \int_{\varphi_\theta(x,\theta) \neq 0 \text{~and~} \theta \notin \Gamma_x}  e^{i\varphi(x,\theta)} |\varphi_\theta (x,\theta)|^{-N}
		\partial_\theta^N ((1-\chi) a) \dif \theta \big) \dif x,
	\end{align*}
	where the integer $N$ can be arbitrary.
	And hence we have $I_4(\xi) = \mathcal O(|\xi|^{-\infty})$ for the same reason as $I_1$.

	It is the $I_5$ which finally decides $\wf(A)$.
	For $\xi \neq \varphi_x(x,\theta)$, we can have
	$I_3$ as follows,
	\begin{align}
		I_5(\xi) 
		& = \int_{\substack{\varphi_\theta(x,\theta) = 0, \\ (x,\theta) \notin \Smo(a)}} e^{i(\varphi(x,\theta) - x\cdot \xi)} \phi(x) (1 - \chi(\theta)) a(x,\theta) \dif x \dif \theta \label{eq:I5s1-PM2021} \\
		& = \int_{\substack{\varphi_\theta(x,\theta) = 0, \\ (x,\theta) \notin \Smo(a)}} \big[ \big( \frac {(\varphi_x(x,\theta) - \xi) \cdot \nabla_x} {i|\varphi_x(x,\theta) - \xi|^2} \big)^{N_1 + N_2} e^{i(\varphi(x,\theta) - x\cdot \xi)} \big] \phi(x) (1 - \chi(\theta)) a(x,\theta) \dif x \dif \theta \label{eq:I5s2-PM2021}\\
		& \lesssim \int_{\substack{\varphi_\theta(x,\theta) = 0, \\ (x,\theta) \notin \Smo(a)}} \agl[\varphi_x(x,\theta) - \xi]^{-N_1-N_2} \phi(x) |a(x,\theta)|\dif x \dif \theta \nonumber\\
		& \lesssim \int_{\substack{\varphi_\theta(x,\theta) = 0, \\ (x,\theta) \notin \Smo(a)}} \agl[\xi]^{-N_1} \agl[\varphi_x(x,\theta)]^{-N_2} \phi(x) \agl[\theta]^{m} \dif x \dif \theta \nonumber\\
		& \simeq \agl[\xi]^{-N_1} \int_{\substack{\varphi_\theta(x,\theta) = 0, \\ (x,\theta) \notin \Smo(a)}} \phi(x) \agl[\theta]^{m-N_2} \dif x \dif \theta 
		\simeq \agl[\xi]^{-N_1}, \nonumber
	\end{align}
	To guarantee the derivation from \eqref{eq:I5s1-PM2021} to \eqref{eq:I5s2-PM2021}, we need $\xi \neq \varphi_x(x,\theta)$ for these $(x,\theta)$ which satisfy $\varphi_x(x,\theta) = 0$ and $(x,\theta) \notin \Smo(a)$.
	We finished the proof.
\end{proof}

\begin{proof}[Formal proof of Theorem \ref{thm:DisInt-PM2021}]
	We do some preparation first.
	Define $\mathcal A$ as the collection of subsets $\Omega$ in $\R_x^n \times (\R_\xi^n \backslash \{0\})$ where $(x,\xi) \in \Omega \Rightarrow (x,t\xi) \in \Omega$ for any $t > 0$, 
	and $\mathcal B$ as the collection of subsets in $\R_x^n \times \mathbb S_\xi^{n-1}$.
	Then there is a one-to-one correspondence between $\mathcal A$ and $\mathcal B$, and we denote the one-to-one mapping as $S$,
	\[
	S \colon \Omega \in \mathcal A \ \mapsto \ S\Omega = \{ (x,\eta) \,;\, \exists \xi \in \Rn \st \eta = \xi/|\xi| \text{~and~} (x,\xi) \in \Omega \} \in \mathcal B.
	\]
	Let $T \colon (x,\theta) \mapsto (x,\varphi_x(x,\theta))$. Note that $ST = TS$.
	For any positive integer $k$, denote 
	\begin{equation} \label{eq:Vk-PM2021}
		V_k := \{(x,\theta) \in \R_x^n \times \R_\xi^n \,;\, |\varphi_\theta(x,\theta)| \leq 1/k \}.
	\end{equation}
	It can be checked that 
	\begin{itemize}
		\item $\{V_k\}_k$ and $\{TV_k\}_k$ are decreasing in terms of $k$,
		
		\item $V_k$ and $TV_k$ are closed in $\R_x^n \times (\R_\xi^n \backslash \{0\})$,
		
		\item $V_k$, $TV_k \in \mathcal A$,
		
		\item $S V_k$ and $S TV_k$ ($= TSV_k$) are also closed in $\R_x^n \times \mathbb S_\xi^{n-1}$.
	\end{itemize}

	Now let's assume 
	\begin{equation} \label{eq:xxiVk-PM2021}
		(x_0,\xi_0) \notin T V_k,
	\end{equation}
	then $(x_0,\xi_0/|\xi_0|)$ is not in $S T V_k$, which is a closed set. 
	Therefore, there exists $\epsilon_1 > 0$ such that
	\begin{equation} \label{eq:xi0Dir-PM2021}
		\{(x,\xi/|\xi|) \,;\, |x - x_0| \leq \epsilon_1,\, |\xi/|\xi| - \xi_0/|\xi_0|| \leq \epsilon_1 \} \cap S T V_k = \emptyset.
	\end{equation}
	
	Because $\nabla_{(x,\theta)} \varphi(x,\theta)$ is always assumed to be nonzero, the number
	\[\inf_{|x - x_0| \leq \epsilon_1,\, \theta \in \mathbb S^{n-1}} \big( |\varphi_x(x,\theta)| + |\varphi_\theta(x,\theta)| \big)
	\]
	exists and is positive and we denote it as $\epsilon_2$,
	\begin{equation} \label{eq:epsInf-PM2021}
		\epsilon_2 := \inf_{|x - x_0| \leq \epsilon_1,\, \theta \in \mathbb S^{n-1}} \big( |\varphi_x(x,\theta)| + |\varphi_\theta(x,\theta)| \big) > 0.
	\end{equation} 
	Let $k_0$ be any positive integer such that
	\begin{equation} \label{eq:kEpsilon2-PM2021}
		k_0 > 2 \lceil 1 / \epsilon_2 \rceil.
	\end{equation}
	Now, for any $(x,\theta) \in S V_{k_0}$, we know $(x,\theta) \in V_{k_0}$, so \eqref{eq:Vk-PM2021} gives $|\varphi_\theta(x,\theta)| \leq 1/k_0 < \epsilon_2/2$, so from \eqref{eq:epsInf-PM2021} we can conclude that
	\begin{equation} \label{eq:gredx-PM2021}
		|\varphi_x(x,\theta)| > \epsilon_2/2 > 0 \quad \text{in}\quad W := \{ (x,\theta) \in S V_{k_0} \,;\, |x - x_0| \leq \epsilon_1 \}.
	\end{equation}
	
	Fix some $\phi \in C_c^\infty(B(x_0,\epsilon_1))$.
	And $\chi \in C_c^\infty(\Rn)$ is a cut-off function with support containing the origin.
	Now we estimate $\widehat{\phi A}(\xi)$.
	We have
	\begin{align}
		\widehat{\phi A}(\xi) 
		& = A(\phi e^{-ix\cdot \xi}) = (2\pi)^{-n/2} \int e^{i(\varphi(x,\theta) - x\cdot \xi)} \phi(x) a(x,\theta) \dif x \dif \theta \nonumber\\
		& \simeq \int e^{i(\varphi(x,\theta) - x\cdot \xi)} \phi(x) \chi(\theta) a(x,\theta) \dif x \dif \theta + \int e^{i(\varphi(x,\theta) - x\cdot \xi)} \phi(x) (1 - \chi(\theta)) a(x,\theta) \dif x \dif \theta \nonumber\\
		& = \int_{\R_x^n} e^{-ix\cdot \xi} \phi(x) \big( \int_{\R_\theta^n} e^{i\varphi(x,\theta)} \chi(\theta) a(x,\theta) \dif \theta \big) \dif x \nonumber \\
		& \quad + \int_{(x,\theta) \in \Smo'(a)} e^{i(\varphi(x,\theta) - x\cdot \xi)} \phi(x) (1 - \chi(\theta)) a(x,\theta) \dif \theta \dif x \nonumber\\
		& \quad + \int_{(x,\theta) \notin \Smo'(a)} e^{i(\varphi(x,\theta) - x\cdot \xi)} \phi(x) (1 - \chi(\theta)) a(x,\theta) \dif x \dif \theta \nonumber\\
		& =: I_1(\xi) + I_2(\xi) + I_3(\xi), \label{eq:AInt0-PM2021}
	\end{align}
	where $\Smo'(a)$ is an arbitrary subset of $\Smo(a)$ such that for every fixed $x$, the projection of the intersect $S((x,\R_\xi^n) \cap \Smo(a))$ is a compact subset of the sphere $\mathbb{S}^{n-1}$. The $(x,\R_\xi^n)$ means $\{ (x,\xi) \in \R_x^n \times \R_\xi^n \,;\, \xi \in \Rn \}$.
	The term $I_1$ is easy to estimate. 
	The $\int_{\R_\theta^n} e^{i\varphi(x,\theta)} \chi(\theta) a(x,\theta) \dif \theta$ in $I_1$ is $C^\infty$-smooth in terms of $x$, so by using integration by parts we can have
	\begin{equation} \label{eq:AIntI1Tmp1-PM2021}
		|I_1(\xi)| 
		= |\int_{\R_x^n} e^{-ix\cdot \xi} \mathcal C(x) \dif x|
		\leq C_\alpha \xi^{-\alpha}, \quad \forall \xi,\, \forall\, \text{multi-index~} \alpha.
	\end{equation}
	where $\mathcal C$ is some function in $C_c^\infty(\Rn)$. The estimation \eqref{eq:AIntI1Tmp1-PM2021} gives
	\begin{equation} \label{eq:AIntI1-PM2021}
		|I_1(\xi)|
		\leq C_N \agl[\xi]^{-N}, \quad \forall \xi,\, \forall N \in \mathbb N.
	\end{equation}
	And $I_2$ can be estimated as follows,
	\begin{align*}
		|I_2(\xi)|
		& = |\int_{\R_x^n} e^{-i x\cdot \xi} \phi(x) \big( \int_{\{\theta \,;\, (x,\theta) \in \Smo'(a)\}} e^{i\varphi(x,\theta)} (1 - \chi(\theta)) a(x,\theta) \dif \theta \big) \dif x| \\
		& = |\int_{\R_x^n} \big[ \big( i^{|\alpha|} \xi^{-\alpha} \partial_x^\alpha \big) e^{-i x\cdot \xi} \big] \phi(x) \big( \int_{\{\theta \,;\, (x,\theta) \in \Smo'(a)\}} e^{i\varphi(x,\theta)} (1 - \chi(\theta)) a(x,\theta) \dif \theta \big) \dif x| \\
		& = \xi^{-\alpha} |\int_{\R_x^n} e^{-i x\cdot \xi} \partial_x^\alpha \Big[ \phi(x) \big( \int_{\{\theta \,;\, (x,\theta) \in \Smo'(a)\}} e^{i\varphi(x,\theta)} (1 - \chi(\theta)) a(x,\theta) \dif \theta \big) \Big] \dif x|.
	\end{align*}
	It is easy to check that the term in $\big[ \cdots \big]$ are $C^\infty$-smooth and compactly supported w.r.t.~$x$, thus it is integrable.
	Therefore,
	\begin{equation} \label{eq:AIntx-PM2021}
		|I_2(\xi)|
		\leq C_N \agl[\xi]^{-N}, \quad \forall \xi,\, \forall N \in \mathbb N.
	\end{equation}

	Then we move on to $I_3$,
	\begin{align}
		I_3(\xi)
		& = \int_{ (\Smo'(a))^c} e^{i(\varphi(x,\theta) - x\cdot \xi)} \phi(x) (1 - \chi(\theta)) a(x,\theta) \dif x \dif \theta \nonumber\\
		& = \int_{(\Smo'(a))^c \cap (V_{k_0})^c} e^{i(\varphi(x,\theta) - x\cdot \xi)} \phi(x) (1 - \chi(\theta)) a(x,\theta) \dif x \dif \theta \nonumber\\
		& \quad + \int_{(\Smo'(a))^c \cap V_{k_0}} e^{i(\varphi(x,\theta) - x\cdot \xi)} \phi(x) (1 - \chi(\theta)) a(x,\theta) \dif x \dif \theta \nonumber\\
		& =: I_4(\xi) + I_5(\xi), \label{eq:AInt2-PM2021}
	\end{align}
	where $(\Smo'(a))^c$ signifies the complementary set of $\Smo'(a)$.
	Note that in $(V_{k_0})^c$, the $|\varphi_\theta(x,\theta)|$ is no less that $1/k_0$ (c.f.~\eqref{eq:Vk-PM2021}), thus no singularity will accrue when $|\varphi_\theta(x,\theta)|$ appears in the denominator.
	Hence, for $I_4$ we have
	\begin{align}
		I_4(\xi)
		& = \int_{(\Smo'(a))^c \cap (V_{k_0})^c} e^{i(\varphi(x,\theta) - x\cdot \xi)} \, \phi(x) (1 - \chi(\theta)) a(x,\theta) \dif x \dif \theta \nonumber\\
		& = \int \hspace*{32pt} e^{-ix\cdot \xi} \hspace*{32pt} \phi(x) \big( \int e^{i\varphi(x,\theta)} (1 - \chi(\theta)) a(x,\theta) \dif \theta \big) \dif x \nonumber\\
		& = \int \big[ ( i^{|\alpha|}\xi^{-\alpha} \partial_x^\alpha ) e^{-ix\cdot \xi} \big] \, \phi(x) \big( \int e^{i\varphi(x,\theta)} (1 - \chi(\theta)) a(x,\theta) \dif \theta \big) \dif x \nonumber\\
		& \simeq \xi^{-\alpha} \int e^{-ix\cdot \xi} \sum_{\beta \leq \alpha} \binom{\alpha}{\beta} \partial_x^{\alpha - \beta} \phi(x) \cdot \partial_x^{\beta} \big( \int e^{i\varphi(x,\theta)} (1 - \chi(\theta)) a(x,\theta) \dif \theta \big) \dif x \nonumber\\
		& \simeq \xi^{-\alpha} \int e^{-ix\cdot \xi} \sum_{\beta \leq \alpha} \binom{\alpha}{\beta} \partial_x^{\alpha - \beta} \phi(x) \nonumber \\
		& \quad \cdot \partial_x^{\beta} \big( \int (\frac {-i \varphi_\theta(x,\theta) \cdot \nabla_\theta} {|\varphi_\theta(x,\theta)|^2})^N (e^{i\varphi(x,\theta)}) (1 - \chi(\theta)) a(x,\theta) \dif \theta \big) \dif x \nonumber\\
		& \lesssim \xi^{-\alpha} \sum_{\beta \leq \alpha} \int e^{-ix\cdot \xi} C_{\alpha,\beta} \partial_x^{\alpha - \beta} \phi(x) \big( \int \agl[\theta]^{m+|\beta|-N} \dif \theta \big) \dif x \nonumber\\
		& \leq C_\alpha \xi^{-\alpha}, \quad \forall \xi,\, \forall\, \text{multi-index~} \alpha. \label{eq:AInt3-PM2021}
	\end{align}
	
	Now for the estimation of $I_5$, we need some constraints on the direction of $\xi$. 
	{\bf It is this term that determines $\wf(A)$}. 
	Because $(x_0,\xi_0) \notin T V_{k_0}$ (see \eqref{eq:xxiVk-PM2021}), according to \eqref{eq:xi0Dir-PM2021}, there is a cone $W \subset \R_{\xi}^n \backslash \{0\}$ such that
	\begin{equation} \label{eq:xiSep-PM2021}
		\xi_0 \in W
		\quad\text{and}\quad
		|\nabla_x \varphi(x,\theta) - \xi| \geq C(|\nabla_x \varphi(x,\theta)| + |\xi|), \forall \xi \in W.
	\end{equation}
	Define $L := \frac {-i (\partial_x \varphi(x,\theta) - \xi) \cdot \nabla_x} {|\nabla_x \varphi(x,\theta) - \xi|^2}$. 
	By using the fact that
	\[
	|\partial^\alpha L f(x)| \lesssim |\nabla_x \varphi(x,\theta) - \xi|^{-1} \sum_\beta |\partial^\beta f(x)|,
	\]
	we can have, for all $\xi \in W$,
	\begin{align}
		I_5(\xi)
		& = \int_{(\Smo'(a))^c \cap V_{k_0}} e^{i(\varphi(x,\theta) - x\cdot \xi)} \phi(x) (1 - \chi(\theta)) a(x,\theta) \dif x \dif \theta \nonumber\\
		& = \int_{(\Smo'(a))^c \cap V_{k_0}} L^N (e^{i(\varphi(x,\theta) - x\cdot \xi)}) \phi(x) (1 - \chi(\theta)) a(x,\theta) \dif x \dif \theta \nonumber\\
		& = \int_{(\Smo'(a))^c \cap V_{k_0}} e^{i(\varphi(x,\theta) - x\cdot \xi)} \,{}^t\!L^N \big( \phi(x) (1 - \chi(\theta)) a(x,\theta) \big) \dif x \dif \theta \nonumber\\
		& \lesssim \int_{\supp \phi \times \R_\theta^n} (1+|\varphi_x(x,\theta) - \xi|)^{-N} \agl[\theta]^{m} \dif x \dif \theta \nonumber\\
		( N_1 + N_2 = N )\quad
		& \lesssim \agl[\xi]^{-N_1} \int_{\supp \phi \times \R_\theta^n} \agl[|\theta| \cdot \varphi_x(x,\theta/|\theta|)]^{-N_2} \agl[\theta]^{m} \dif x \dif \theta \nonumber\\
		\big( \text{by~} \eqref{eq:gredx-PM2021} \big)\quad
		& \lesssim \agl[\xi]^{-N_1} \int_{\supp \phi \times \R_\theta^n} \agl[\theta]^{-N_2} \agl[\theta]^{m} \dif x \dif \theta \nonumber\\
		& \lesssim \agl[\xi]^{-N_1} \quad \forall N_1 \in \mathbb N.
		\label{eq:AInt4-PM2021}
	\end{align}
	
	Combining \eqref{eq:xxiVk-PM2021}, \eqref{eq:kEpsilon2-PM2021}, \eqref{eq:AInt0-PM2021},
	\eqref{eq:AIntI1-PM2021}, 
	\eqref{eq:AIntx-PM2021}, \eqref{eq:AInt2-PM2021}, \eqref{eq:AInt3-PM2021} and \eqref{eq:AInt4-PM2021}, we arrive at
	\begin{equation}
		|\widehat{\phi A}(\xi) | \lesssim \agl[\xi]^{-N}, \quad
		\xi \in W,~ \forall N \in \mathbb N,
	\end{equation}
	where the $W$ is a cone containing $\xi_0$, and $(x_0,\xi_0) \neq (x,\varphi_x(x,\theta))$ for these $(x,\theta)$ satisfying
	\[
	(x,\theta) \in (\Smo'(a))^c \cap V_{k_0}.
	\]
	Therefore, for any $k_0 > 2 \lceil 1 / \epsilon_2 \rceil$, there holds
	\[
	\Big( \{(x,\varphi_x(x,\theta)) \,;\, |\varphi_\theta(x,\theta)| \leq 1/k_0,\, (x,\theta) \in (\Smo'(a))^c \} \Big)^c \subset \Big( \wf (A) \Big)^c,
	\]
	thus
	\[
	\wf (A) \subset \{(x,\varphi_x(x,\theta)) \,;\, |\varphi_\theta(x,\theta)| \leq 1/k_0,\, (x,\theta) \notin \Smo'(a) \}.
	\]
	Finally, let $k_0$ goes to zero and choose $\Smo'(a)$ to be arbitrarily close to $\Smo(a)$, we arrive at the conclusion \eqref{eq:DisInt-PM2021}.	
\end{proof}

\section{Applications} \label{sec:wfAp-PM2021}

Now we are ready to apply those results.

\subsection{Microlocality of {\rm $\Psi$DOs}}
	
\begin{prop} \label{prop:PseuKer-PM2021}
	Assume $a \in S^{+\infty}(\R_x^n \times \R_y^n \times \R_\xi^n)$ is symbol and $K$ is the kernel of the corresponding {\rm $\Psi$DO} of $a$, then
	\begin{equation} \label{eq:PseuKer-PM2021}
		\boxed{\wf (K) \subset \{(x,x;\xi,-\xi) \,;\, (x,x,\xi) \notin \Smo(a) \}.}
	\end{equation}
\end{prop}
\begin{proof}
	Denote the corresponding {\rm $\Psi$DO} as $A$, then
	\begin{align*}
		\agl[K(x,y), u \otimes v(x,y)]
		& = \agl[K(x,y), u(x) \otimes v(y)]
		= \langle \langle K(x,y), u(x) \rangle_x, v(y) \rangle_y \\
		& = \agl[A u(y), v(y)]
		= (2\pi)^{-n} \int e^{i(x-y) \cdot \xi} a(x,y,\xi) u(x) v(y) \dif x \dif y \dif \xi \\
		& = (2\pi)^{-n} \int e^{i(x-y) \cdot \xi} a(x,y,\xi) u \otimes v(x,y) \dif x \dif y \dif \xi.
	\end{align*}
	Therefore, in the oscillatory integral sense,
	\[
	K(x,y) = (2\pi)^{-n} \int e^{i(x-y) \cdot \xi} a(x,y,\xi) \dif \xi.
	\]
	According to Theorem \ref{thm:DisInt-PM2021}, we have
	\[
	\wf (K)
	\subset \{(x,y;\varphi_{x,y}(x,y,\xi)) \,;\, \varphi_\xi(x,y,\xi) = 0,\, (x,y,\xi) \notin \Smo(a) \},
	\]
	where $\varphi(x,y,\xi) = (x-y) \cdot \xi$.
	Hence,
	\begin{align*}
	\wf (K)
	& \subset \{(x,y; \xi,-\xi) \,;\, x - y = 0,\, (x,y,\xi) \notin \Smo(a) \} \\
	& = \{(x, x;\xi, -\xi) \,;\, (x,x,\xi) \notin \Smo(a)\}.
	\end{align*}
	The proof is complete.
\end{proof}

\begin{thm} \label{thm:MLpr-PM2021}
	Assume that $A$ is a {\rm $\Psi$DO}, then for $u \in \scrE'$
	we have
	\begin{equation} \label{eq:MLpr-PM2021}
	\text{microlocality:} \quad \boxed{\wf (Au) \subseteq \wf(u) \backslash \Smo(A).}
	\end{equation}
	Moreover, if $A$ is elliptic in the sense of Definition \ref{defn:elp-PM2021}, then $\Smo(A) = \emptyset$ and
	\begin{equation} \label{eq:elwf-PM2021}
		\boxed{\wf (Au) = \wf(u).}
	\end{equation}
\end{thm}

\begin{proof}
	We have $Au(x) = \agl[K(x,y),u(y)]$ where $K$ is its kernel, so according to Theorem \ref{thm:KernelAct-PM2021} and Proposition \ref{prop:PseuKer-PM2021}, we can conclude
	\begin{align*}
		\wf (Au)
		& \subseteq \big( \wf' (K) \circ \wf (u) \big) \cup \wf_x (K) \\
		& = \big( \{(x,x;\xi,\xi) \,;\, (x,\xi) \notin \Smo(a) \} \circ \wf (u) \big) \cup \emptyset \\
		& = \wf (u) \backslash \Smo(a)
		= \wf (u) \backslash \Smo(A).
	\end{align*}
	
	If $A$ is elliptic, then according to the definition, we have
	\[
	|a(x,\xi)| \geq C\agl[\xi]^m, \quad \text{when~} x \in \Rn, \ |\xi| \geq R,
	\]
	for some constants $m \in \R$, $C > 0$ and $R > 0$, so it is obvious that $\Smo(A) = \emptyset$, thus by \eqref{eq:MLpr-PM2021},
	\[
	\wf (Au) \subseteq \wf u.
	\]
	Also, because $A$ is elliptic, then by Theorem \ref{thm:para1-PM2021} we know $A$ has a parametrix $B$ such that $R := BA - I \in \Psi^{-\infty}$, so
	\[
	\wf (u) = \wf (BAu - Ru)
	\subset \wf (BAu) \cup \wf (Ru)
	\subset \wf (Au).
	\]
	In total, $\wf (Au) = \wf u$.
	The proof is complete.
\end{proof}

\begin{lem} \label{lem:smoempty-PM2021}
	Assume $a$ is a symbol and $u \in \scrE'(\Omega)$. Denote the corresponding {\rm $\Psi$DO} of $a$ as $A$, then
	\begin{equation} \label{eq:smoempty-PM2021}
		\boxed{\wf (Au) \cap \Smo(a) = \emptyset.}
	\end{equation}
\end{lem}

\begin{proof}
	This is a straight forward outcome of \eqref{eq:MLpr-PM2021}.
\end{proof}

%
%
%
%
%

\subsection{Pull-back of distributions} \label{subsec:PBD-PM2021}

\begin{thm} \label{thm:PullBack-PM2021}
	Let $\Omega_1$ and $\Omega_2$ be two domain in $\Rn$, and $\psi \colon \Omega_1 \to \Omega_2$ is an diffeomorphism.
	Then for any $u \in \scrD'(\Omega_2)$, we have $\psi^* u \in \scrD'(\Omega_1)$, and
	\begin{equation} \label{eq:PullBack-PM2021}
		\boxed{\wf (\psi^* u) = \{(x, {}^t \psi'|_x \eta) \,;\, (\psi(x), \eta) \in \wf(u) \},}
	\end{equation}
	where $\psi'$ signifies the matrix whose ($i$-row, $j$-column) element is $\partial_{x_i} \psi^j$,
	and $({}^t \psi')^{-1}|_x$ is the inverse of transpose of the matrix $\psi'$ evaluated at $x$,
	and $({}^t \psi')^{-1}|_x \eta$ stands for the matrix multiplication of the matrix $({}^t \psi')^{-1}|_x$ and the vertical vector $\eta$.
\end{thm}

The mapping in \eqref{eq:PullBack-PM2021},
\[
(x, {}^t \psi'|_x \eta) \ \mapsto \ (\psi(x), \eta)
\]
is invariant on the cotangent bundle (see \cite[for Theorem 18.1.17]{maLSNH2020}).

\begin{proof}[First proof of Theorem \ref{thm:PullBack-PM2021}]
	It is obvious that $\psi^* u \in \scrD'(\Omega_1)$.
	
	For $(y_0,\xi_0)$, because $\psi$ is a diffeomorphism, we can find $(x_0, \eta_0)$ such that $x_0 = \psi^{-1}(y_0)$ and $\xi_0 = (\psi^{-1})'|_{y_0} \cdot \eta_0$.
	Assume that $(y_0,\xi_0) \notin \wf (u)$.
	For a smooth cutoff function $\varphi_{x_0}$ satisfying $\varphi_{x_0}(x_0) \neq 1$, we have
	\begin{align*}
		& (\varphi_{x_0} \psi^* u)^\wedge(\eta)
		\simeq \int e^{-ix \cdot \eta} \varphi_{x_0}(x) \psi^* u(x) \dif x
		= \int e^{-ix \cdot \eta_0} \varphi_{x_0}(x) u(\psi(x)) \dif x \\
		= & \int e^{-i\psi^{-1}(y) \cdot \eta} \varphi_{x_0}(\psi^{-1}(y)) u(y) \dif \psi^{-1}(y) \quad \big( y = \psi(x) \big) \\
		= & \int e^{-i {}^t\psi^{-1}(y) \eta} \cdot \varphi_{x_0}(\psi^{-1}(y)) |\frac {\partial \psi^{-1}} {\partial y} (y)| \phi^{-1}(y) \cdot (\phi u)(y) \dif y \\
		= & \int e^{-i{}^t\psi^{-1}(y) \eta} \cdot \tilde \varphi_{y_0}(y) \cdot (\phi u)(y) \dif y
		= \int e^{-i({}^t\psi^{-1}(y) \eta - y \cdot \xi)} \cdot \tilde \varphi_{y_0}(y) \cdot \widehat{\phi u}(\xi) \dif \xi \dif y \\
		= & \int_{|\hat \xi_0 - \hat\xi| > 1} e^{-i({}^t\psi^{-1}(y) \eta - y \cdot \xi)} \cdot \tilde \varphi_{y_0}(y) \cdot \widehat{\phi u}(\xi) \dif \xi \dif y \\
		& + \int_{|\hat \xi_0 - \hat\xi| \leq 1} e^{-i{}^t\psi^{-1}(y) \eta} e^{iy \cdot \xi} \cdot \tilde \varphi_{y_0}(y) \cdot \widehat{\phi u}(\xi) \dif \xi \dif y \\
		= &\!\!:  I_1 + I_2,
	\end{align*}
	where $\hat \xi := \xi / |\xi|$ and the same for $\hat \xi_0$.
	For convenience we have written $\psi^{-1}(y) \cdot \eta$ as ${}^t\psi^{-1}(y) \eta$, where ${}^t M$ signifies the transpose operation for any matrix $M$.
	By doing so it will be more straightforward when we make derivatives.
	Here $\tilde \varphi_{y_0}(y)$ is a generic function which is $C^\infty$-smooth and is compactly supported w.r.t.~$y$ and whose  precise definition may varies from line to line.
	
	For $I_1$, because $|\hat \xi_0 - \hat\xi| > 1$ and $(y_0, \xi_0) \notin \wf(u)$, we have $\big| ({}^t\psi^{-1})'|_y \eta - \xi \big| \geq \frac 1 2$ and $|\widehat{\phi u}(\xi)| \lesssim \agl[\xi]^{-N_0}$ for some integer $N_0$.
	The number $N_0$ comes from the fact that $\phi u$ is a compactly supported distribution so its Fourier transform has (at most) polynomial growth.
	Hence,
	\begin{align*}
		I_1
		& = \int_{|\hat \xi_0 - \hat\xi| > 1} \Big( \frac {1 + (({}^t\psi^{-1})'|_y \eta - \xi) \cdot \nabla_y} {\agl[({}^t\psi^{-1})'|_y \eta - \xi]^2} \Big)^N \big( e^{-i(\psi^{-1}(y) \cdot \eta - y \cdot \xi)} \big) \cdot \tilde \varphi_{y_0}(y) \cdot \widehat{\phi u}(\xi) \dif \xi \dif y \\
		& \simeq \int_{|\hat \xi_0 - \hat\xi| > 1} \agl[({}^t \psi^{-1})'|_y \eta - \xi]^{-N} e^{-i(\psi^{-1}(y) \cdot \eta - y \cdot \xi)} \cdot \tilde \varphi_{y_0}(y) \cdot \widehat{\phi u}(\xi) \dif \xi \dif y \\
		& \lesssim \int_{|\hat \xi_0 - \hat\xi| > 1} \agl[({}^t \psi^{-1})'|_y \eta]^{-N_1} \agl[\xi]^{-N_2} \cdot |\tilde \varphi_{y_0}(y)| \cdot |\widehat{\phi u}(\xi)| \dif \xi \dif y \quad (N = N_1 + N_2) \\
		& \lesssim \agl[\eta]^{-N_1} \int_{|\hat \xi_0 - \hat\xi| > 1} \agl[\xi]^{-N_2} \cdot (\int |\tilde \varphi_{y_0}(y)| \dif y) \cdot |\widehat{\phi u}(\xi)| \dif \xi \\
		& \lesssim \agl[\eta]^{-N_1} \int_{|\hat \xi_0 - \hat\xi| > 1} \agl[\xi]^{-N_2} \agl[\xi]^{N_0} \dif \xi
		\lesssim \agl[\eta]^{-N_1},
	\end{align*}
	provided that $N_2 - N_0 >$ the dimension of $\xi$.

	For $I_2$, we have
	\begin{align*}
		I_2
		& = \int_{|\hat \xi_0 - \hat\xi| \leq 1} \Big( \frac {1 + (({}^t \psi^{-1})'|_y \eta) \cdot \nabla_y} {\agl[({}^t \psi^{-1})'|_y \eta]^2} \Big)^N \big(e^{-i\psi^{-1}(y) \cdot \eta} \big) e^{iy \cdot \xi} \cdot \tilde \varphi_{y_0}(y) \cdot \widehat{\phi u}(\xi) \dif \xi \dif y \\
		& \simeq \int_{|\hat \xi_0 - \hat\xi| \leq 1} \agl[({}^t \psi^{-1})'|_y \eta]^{-N} e^{-i\psi^{-1}(y) \cdot \eta} \cdot \big( \nabla_y \big)^N \big( e^{iy \cdot \xi} \cdot \tilde \varphi_{y_0}(y) \big) \cdot \widehat{\phi u}(\xi) \dif \xi \dif y \\
		& \simeq \int_{|\hat \xi_0 - \hat\xi| \leq 1} \agl[\eta]^{-N} e^{-i\psi^{-1}(y) \cdot \eta} \cdot \agl[\xi]^N \tilde \varphi_{y_0}(y) \cdot \widehat{\phi u}(\xi) \dif \xi \dif y \\
		& \lesssim \agl[\eta]^{-N} \int_{|\hat \xi_0 - \hat\xi| \leq 1} \agl[\xi]^N \cdot (\int |\tilde \varphi_{y_0}(y)| \dif y ) \cdot |\widehat{\phi u}(\xi)| \dif \xi \\
		& \lesssim \agl[\eta]^{-N} \int_{|\hat \xi_0 - \hat\xi| \leq 1} \agl[\xi]^N \agl[\xi]^{-N - n - 1} \dif \xi
		\lesssim \agl[\eta]^{-N}.
	\end{align*}
	In total, we have $|(\varphi_{x_0} \psi^* u)^\wedge(\eta)| \lesssim \agl[\eta]^{-N}$ for any integer $N$ if $\hat \eta$ is in a small neighborhood of $\hat \eta_0$
	where $(\psi(x_0), ({}^t \psi^{-1})'|_{y_0} \eta_0) \notin \wf(u)$, namely,
	\[
	(\psi(x_0), ({}^t \psi^{-1})'|_{\psi(x_0)} \eta_0) \notin \wf(u)
	\ \Rightarrow \
	(x_0,\eta_0) \notin \wf(\psi^* u).
	\]
	Therefore,
	\[
	(x_0,\eta_0) \in \wf(\psi^* u)
	\ \Rightarrow \
	(\psi(x_0), ({}^t \psi^{-1})'|_{\psi(x_0)} \eta_0) \in \wf(u),
	\]
	so
	\[
	\wf(\psi^* u) \subset \{(x,\eta) \,;\, (\psi(x), ({}^t \psi^{-1})'|_{\psi(x)} \eta) \in \wf (u) \}.
	\]
	Because $\psi$ is invertible, we can obtain the opposite inclusion by looking at $\psi^{-1*}(\psi^* u)$.
	
	It can be shown that $(\psi^{-1})'|_{\psi(x)} = (\psi'(x))^{-1}$.
	Indeed, by differentiating $x = \psi^{-1}(\psi(x))$ w.r.t.~$x$ we obtain
	\(
	I = (\psi^{-1})'|_{\psi(x)} \cdot \psi'(x),
	\)
	so
	\(
	(\psi^{-1})'|_{\psi(x)} = (\psi'(x))^{-1},
	\)
	and by taking transpose we obtain
	\[
	({}^t \psi^{-1})'|_{\psi(x)} = ({}^t \psi'(x))^{-1},
	\]
	so
	\begin{align*}
		\wf (\psi^* u)
		& = \{(x, \eta) \,;\, (\psi(x), ({}^t \psi')^{-1}|_x \eta) \in \wf(u) \} \\
		& = \{(x, {}^t \psi'|_x \xi) \,;\, (\psi(x), \xi) \in \wf(u) \}.
	\end{align*}
	The proof is complete.
\end{proof}

There is also another proof for Theorem \ref{thm:PullBack-PM2021}.
As in Remark \ref{rem:Dit-PM2021}, we can use Theorems \ref{thm:DisInt-PM2021} and \ref{thm:KernelAct-PM2021} to obtain $\wf(\psi^* u)$.

\begin{proof}[second proof of Theorem \ref{thm:PullBack-PM2021}]
	The pull-back $\psi^*$ has a kernel: for $f \in C^\infty(\Omega_2)$ and $g \in C^\infty(\Omega_1)$, we have
	\begin{align*}
		\agl[\psi^* f, g]
		& = \int \psi^* f(x) g(x) \dif x
		= (2\pi)^{-n} \int e^{i(\psi(x) - y) \cdot \eta} f(y) g(x) \dif x \dif y \dif \eta \\
		& \simeq \agl[\int e^{i(\psi(x) - y) \cdot \eta} \dif \eta, (g \otimes f)(x, y)] \\
		& = \agl[K, g \otimes f], 
		\quad \text{where} \quad K(x,y) = \int e^{i({}^t \psi(x) - {}^t y) \eta} \dif \eta.
	\end{align*}
	and
	\(
	\psi^* u(x) = \agl[K(x,y), u(y)].
	\)
	By Theorem \ref{thm:DisInt-PM2021} we have
	\[
	\wf(K) \subset \{(x, \psi(x); {}^t \psi'|_x \eta, - \eta) \},
	\]
	and then by Theorem \ref{thm:KernelAct-PM2021} we have
	\begin{align*}
		\wf (\psi^* u)
		& \subseteq \big( \wf' (K) \circ \wf (u) \big) \cup \wf_x (K)
		= \{(x, {}^t \psi'|_x \eta) \,;\, (\psi(x), \eta) \in \wf(u) \}.
	\end{align*}
	The opposite inclusion can be obtained by looking at $\psi^{-1*}(\psi^* u)$.
	We obtain \eqref{eq:PullBack-PM2021}.
\end{proof}

\section*{Exercise}

\begin{ex} \label{ex:smo-PM2021}
	Prove Lemme \ref{lem:smo-PM2021}.
\end{ex}

\begin{ex} \label{ex:KOGu-PM2021}
	Prove that the condition $\big( \wf' (K) \circ \wf (u) \big) \cap O_x = \emptyset$ guarantees \eqref{eq:KuO-PM2021}.
\end{ex}

\begin{ex} \label{ex:Sde-PM2021}
	Show the details in the computations \eqref{eq:M1Ku-PM2021} and \eqref{eq:M23-PM2021}.
\end{ex}

\chapter{Propagation of the singularities} \label{ch:ProSing-PM2021}

\section{Microlocal parametrix} \label{sec:mlpara-PM2021}

To study the microlocal parametrix, we recall notion of conic sets and the smooth direction ``$\Smo$'' given in Definitions \ref{defn:coni-PM2021} \& \ref{defn:smo-PM2021}, and \underline{$T^* \Rn \backslash 0$} stands for the cotangent bundle with the zero section excluded.
Now we generalize Definitions \ref{defn:para-PM2021} \& \ref{defn:elp-PM2021}, microlocally, as follows.

\begin{defn}[Microlocal parametrix\index{microlocal parametrix}] \label{defn:mpa-PM2021}
	Assume $m \in \R$ and $T \in \Psi^m$.
	We call a {\rm $\Psi$DO} $S$ a \emph{left (resp.~right) microlocal parametrix} of $T$ if there exists a nonempty open conic set $\Gamma \subset T^* \Rn \backslash 0$ such that
	\[
	\Smo(ST - I) = \Gamma \quad \text{(resp.~} \Smo(TS - I) = \Gamma \text{)}.
	\]
	We call $S$ a \emph{microlocal parametrix} of $T$ if it is both a left and a right microlocal parametrix under the same set $\Gamma$.
\end{defn}

\begin{defn}[Microlocal ellipticity\index{microlocal ellipticity}] \label{defn:melp-PM2021}
	Assume $m \in \R$ and $a \in S^m$, and $A$ is the {\rm $\Psi$DO} of $a$.
	Let $\Gamma \subset T^* \Rn \backslash 0$ be a open conic set.
	We say $a$ (and $A$) is \emph{microlocally elliptic in $\Gamma$} if for some constants $C_\Gamma > 0$, $R > 0$,
	\[
	\boxed{|a(x,\xi)| \geq C_\Gamma \agl[\xi]^m, \quad \forall (x, \xi) \in \Gamma, \ |\xi| \geq R.}
	\]
	We write $\boxed{\Char A := (\bigcup \mathscr F)^c}$, where $\mathscr F = \{\Gamma \,;\, A \text{~is microlocally elliptic in~} \Gamma \}$, and the notation $\Omega^c$ stands for the complement of $\Omega$ in $T^* \Rn \backslash 0$.
\end{defn}

From Definition \ref{defn:melp-PM2021}, it is obvious that $\Char A$ is always closed.
The following claim is trivial.

\begin{lem} \label{lem:elem-PM2021}
$\Char A = \emptyset$ if and only if $A$ is elliptic in the sense of Definition \ref{defn:elp-PM2021}.
\end{lem}

\begin{lem} \label{lem:pin0-PM2021}
	Assume $P$ is the {\rm $\Psi$DO} with principal symbol $p_m(x,\xi)$ homogeneous in $\xi$, then $\Char P = p_m^{-1}(0)$, where $p_m^{-1}(0)$ signifies the set $\{ (x,\xi) \in T^* \Rn \backslash 0 \,;\, p_m(x,\xi) = 0 \}$.
\end{lem}

The proof is left as an exercise.
The $\Char A$ and $\Smo(A)$ is closely related.
Results in \S \ref{sec:para-PM2021} can be modified to a microlocal version.

\begin{thm}[Microlocal ellipticity $\Leftrightarrow$ Microlocal parametrix] \label{thm:mpa1-PM2021}
	Let $m \in \R$ and $A \in \Psi^m$.
	Assume $A$ is microlocally elliptic in $\Gamma$, where $\Gamma = (\Char A)^c$ is non-empty.
	Then $A$ has a microlocal parametrix $B$.
	Moreover, they satisfies
	\begin{equation} \label{eq:mpa1-PM2021}
	\boxed{(\Char A)^c \subset \Smo(I - B A).}
	\end{equation}
	
	Conversely, if $A$ has either a right or left microlocal parametrix, then $A$ is microlocally elliptic.
\end{thm}

\begin{proof}
	($\Rightarrow$)
	Fix $(x_0, \xi_0) \in \Gamma$.
	In the proof of Theorem \ref{thm:para1-PM2021}, we modify the function $\chi(\xi)$ to $\chi(x,\xi)$ which is given as
	\(
	\chi(x,\xi) := \chi(x - x_0) \chi ( \xi/|\xi| - \xi_0/ |\xi_0|).
	\)
	Define $b_j(x,\xi) := (1-\chi(x, \xi))/a(x,\xi) \cdot r_j(x,\xi)$ ($j \geq 0$) the same way, with $r_0 \equiv 1$, and follow the same steps as in the proof of Theorem \ref{thm:para1-PM2021} we can obtain $\forall N \in \N$,	
	\[
	\sigma(AB)
	= 1 - (1+r_1 + \cdots + r_N) \chi - r_{N+1} + S^{-N-1}
	= 1 + S^{-N-1} \ \text{in}
	\subset \Gamma,
	\]
	as in \eqref{eq:sAB-PM2021}, so
	\[
	\sigma(I - AB) \in S^{-\infty} \text{~in~}
	\subset \Gamma.
	\]
	Due to the arbitrary of $\chi$, we conclude
	\[
	\sigma(I - AB) \in S^{-\infty} \text{~in~} \Gamma.
	\]
	Hence, by Definition \ref{defn:smo-PM2021} we obtain
	\(
	\Gamma \subset \Smo(I - AB).
	\)

	($\Leftarrow$)
	Assume $B$ is the right parametrix of $A$, then there exists a nonempty conic open set $\Gamma \subset T^* \Rn \backslash 0$ such that
	\(
	\Smo(AB - I) \subset \Gamma,
	\)
	which means
	\(
	a \# b = 1 + S^{-\infty} \ \text{in} \ \Gamma.
	\)
	where $a$ and $b$ are symbols of $A$ and $B$, respectively.
	Similar to the proof of Theorem \ref{thm:para2-PM2021}, we can prove that
	\[
	|a(x,\xi)| \geq \agl[\xi]^{m}/2, \quad \text{when} \quad (x,\xi) \in \Gamma, \ \agl[\xi] \geq C/2.
	\]
	Therefore, $a$ is microlocally elliptic in $\Gamma$.
	The proof for the left-case is similar.
	
	The proof is complete.
\end{proof}

\begin{cor} \label{cor:wAuC-PM2021}
	For any $A \in \Psi^{+\infty}$ and any $u \in \scrS'$, there holds
	\[
	\boxed{\wf(Au) \subset \wf u \subset \wf(Au) \cup \Char A.}
	\]
\end{cor}

\begin{proof}
	The ``$\wf(Au) \subset \wf u$'' is from Theorem Theorem \ref{thm:MLpr-PM2021}.
	
	When $\Char A = T^* \Rn \backslash 0$, the claim is trivial.
	When $\Char A \subsetneqq T^* \Rn \backslash 0$, $(\Char A)^c$ is non-empty, so by Theorem \ref{thm:mpa1-PM2021}, there exists a microlocal parametrix $B$ of $A$, so we can have
	\begin{align*}
	\wf u
	& = \wf(BAu + (I - BA)u) \\
	& \subset \wf(BAu) \cup \wf((I - BA)u) \\
	& \subset \wf(BAu) \cup \big( \wf u \backslash \Smo(I - BA) \big) \qquad (\text{by~} \eqref{eq:MLpr-PM2021}) \\
	& = \wf(BAu) \cup \big( \wf u \cap (\Smo(I - BA))^c \big) \\
	& \subset \wf(BAu) \cup \big( \wf u \cap \Char A \big)  \qquad (\text{by Theorem~} \ref{thm:mpa1-PM2021}) \\
	& = \big( \wf(BAu) \cup \wf u \big) \cap \big( \wf(BAu) \cup \Char A \big) \\
	& \subset \wf u \cap \big( \wf(BAu) \cup \Char A \big), \qquad (\text{by~} \eqref{eq:MLpr-PM2021})
	\end{align*}
	which gives
	\[
	\wf u
	\subset \wf(BAu) \cup \Char A
	\subset \wf(Au) \cup \Char A.
	\]
	The proof is complete.
\end{proof}

\begin{rem}
	Combining Corollary \ref{cor:wAuC-PM2021} and Lemma \ref{lem:elem-PM2021}, we have
	\begin{equation*}
	\left\{\begin{aligned}
	A \text{~is microlocally elliptic:~}& \wf(Au) \subset \wf u \subset \wf(Au) \cup \Char A, \\
	A \text{~is elliptic:~}& \wf(Au) \subset \wf u \subset \wf(Au) \ \Leftrightarrow \ \wf(Au) = \wf u.
	\end{aligned}\right.
	\end{equation*}
	Hence, Corollary \ref{cor:wAuC-PM2021} can be viewed as a generalization of \eqref{eq:elwf-PM2021}.
\end{rem}


%

The following result is important.

\begin{thm} \label{thm:wuCh-PM2021}
	Assume $u \in \scrS'$, then
	\[
	\boxed{\wf(u) = \bigcap_{A \in \Psi^{+\infty},\, Au \in C^\infty} \Char A.}
	\]
\end{thm}

\begin{proof}
	By Corollary \ref{cor:wAuC-PM2021}, we have
	$\wf u \subset \wf(Au) \cup \Char A$, so
	\[
	\wf(u) \subset \bigcap_{Au \in C^\infty} \Char A.
	\]
	For the another direction, assume $(x_0, \xi_0) \notin \wf(u)$, then we shall construct a suitable {\rm $\Psi$DO} $A$ such that
	\begin{equation} \label{eq:AuCh-PM2021}
	A u \in C^\infty,
	\quad \text{and} \quad
	(x_0, \xi_0) \notin \Char(A),
	\end{equation}
	which gives $(x_0, \xi_0) \notin \bigcap_{Au \in C^\infty} \Char A$, and so
	\[
	(\wf(u))^c \subset (\bigcap_{Au \in C^\infty} \Char A)^c \ \Rightarrow \
	\bigcap_{Au \in C^\infty} \Char A \subset \wf(u),
	\]
	and the proof will be finished.
	
	It remains to construct such an operator $A$, and we present two ways to do it.

	{\bf Method 1.}
	Because $\wf(u)$ is closed and $(x_0, \xi_0) \notin \wf(u)$, so there exist bounded open neighborhoods $\omega$, $\omega'$ of $x_0$ and conic open neighborhoods $V$, $V'$ of $\xi_0$ such that
	\begin{equation*}
		\left\{\begin{aligned}
			& \omega \subsetneqq \omega', \quad V \subsetneqq V', \\
			& \omega' \times V' \cap \wf(u) = \emptyset.
		\end{aligned}\right.
	\end{equation*}
	We denote $\Gamma := \omega \times V$ and $\Gamma' := \omega' \times V'$, then $\Gamma \subset \Gamma'$ and $\Gamma' \cap \wf(u) = \emptyset$, so
	\begin{equation} \label{eq:wfG-PM2021}
		\wf(u) \subset \Gamma^c.
	\end{equation}
	
	Choose $a \in C^\infty(\R^{2n})$ such that
	\begin{equation} \label{eq:aG0-PM2021}
		\left\{\begin{aligned}
			& \supp a \subset \Gamma,
			\quad \text{and} \quad
			a(x_0, \xi_0) = 1, \\
			& a(x,\xi) = a(x,\xi/|\xi|) \ \text{when} \ |\xi| \geq 1.
		\end{aligned}\right.
	\end{equation}
	It can be shown that $a \in S^0$ (see Exercise \ref{ex:aG0-PM2021}) and $\underline{(x_0, \xi_0) \notin \Char(T_a)}$.
	Moreover, because $a \equiv 0$ in $\Gamma^c$ and $\Gamma^c$ is a conic set, we can conclude
	\begin{equation} \label{eq:GmA-PM2021}
		\Gamma^c \subset \Smo(T_a) \quad \Rightarrow \quad (\Smo(T_a))^c \subset \Gamma.
	\end{equation}
	Hence by Theorem \ref{thm:MLpr-PM2021} we have
	\begin{align*}
		\wf (T_a u)
		& \subset \wf(u) \backslash \Smo(T_a)
		= \wf(u) \cap (\Smo(T_a))^c
		\subset \wf(u) \cap \Gamma \quad (\text{by~} \eqref{eq:GmA-PM2021}) \\
		& = \wf(u) \backslash \Gamma^c = \emptyset, \quad (\text{by~} \eqref{eq:wfG-PM2021})
	\end{align*}
	which implies $\underline{T_a u \in C^\infty}$.
	Condition \eqref{eq:AuCh-PM2021} is satisfied.

	{\bf Method 2.}
	Because $(x_0, \xi_0) \notin \wf(u)$, there exists $\phi \in C_c^\infty(\Rn)$ such that $\widehat{\phi u}(\xi)$ is rapidly decaying when $\xi/|\xi|$ and $\xi_0/|\xi_0|$ are close enough, say, $\big| \xi/|\xi| - \xi_0/|\xi_0| \big| \leq \epsilon$ for certain $\epsilon > 0$.
	Hence we choose $\psi \in C^\infty(\mathbb S^{n-1})$ such that $\psi(\xi_0/|\xi_0|) = 1$ and $\psi(\eta) = 0$ when $\big| \eta - \xi_0/|\xi_0| \big| > \epsilon$ where $\eta \in \mathbb S^{n-1}$.
	Choose $\chi \in C_c^\infty(\Rn)$ such that $\chi(\xi) = 1$ when $|\xi| \geq 1/10$ and $\chi(\xi) = 0$ when $|\xi| \geq 1/5$.	
	Now we define a operator $A$ as follows
	\[
	A \varphi(x) := (2\pi)^{-n} \int_{\R^{2n}} e^{i(x-y) \cdot \xi} \phi(y) (1 - \chi(\xi)) \psi(\xi/|\xi|) \varphi(y) \dif y \dif \xi, \quad \varphi \in \scrS(\Rn).
	\]
	The purpose of the term ``$1-\chi(\xi)$'' is to cutoff the singularity near $\xi = 0$.
	By Theorem \ref{thm:ReV-PM2021} we see $A$ is a {\rm $\Psi$DO} of order $0$ with symbol
	\begin{equation} \label{eq:ayp-PM2021}
	a(x,\xi) = \phi(x) (1 - \chi(\xi)) \psi(\xi/|\xi|) + S^{-1},
	\end{equation}
	By \eqref{eq:ayp-PM2021} we can show $\underline{(x_0, \xi_0) \notin \Char A}$, see Exercise \ref{ex:aCh-PM2021}.
	We can extend $A$ from $\scrS$ to $\scrS'$, and we have
	\[
	\widehat{Au}(\xi)
	= (1 - \chi(\xi)) \psi(\xi/|\xi|) \widehat{\phi u}(\xi),
	\]
	so $\widehat{Au}$ is rapidly decaying, which means $\underline{Au \in C^\infty}$.
	Condition \eqref{eq:AuCh-PM2021} is satisfied.
	
	The proof is complete.
\end{proof}

\section{Bicharacteristics} \label{sec:bichar-PM2021}

To prove a main result, we first introduce the notion of bicharacteristics.
The Hamiltonian $H_p$ of $p$ is defined as:
\begin{equation}
	H_p := \nabla_\xi p \cdot \nabla_x - \nabla_x p \cdot \nabla_\xi.
\end{equation}
By Theorem \ref{thm:com-PM2021} we can see
\begin{align*}
\sigma([P,Q])
& = \nabla_\xi p_m \cdot \nabla_x q_m - \nabla_x p_m \cdot \nabla_\xi q_m + S^{m_1 + m_2 - 2} \\
& = H_{p_m} q_m + S^{m_1 + m_2 - 2},
\end{align*}
where $m_1$, $m_2$ are the order of $P$ and $Q$, and $p_m$, $q_m$ are principal symbols of $P$ and $Q$, respectively.

In what follows we use the notation $T^*\Rn \backslash 0 := \{(x,\xi) \in T^*\Rn \,;\, \xi \neq 0 \}$.
We introduce the notion of bicharacteristic.
For more details on the Hamiltonian flows, see \cite[\S 2]{salo2006sta}.

\begin{defn}[Null bicharacteristic] \label{defn:nbi-PM2021}
	Assume $p \in C^1(T^*\Rn \backslash 0; \R)$, and $I~(\ni 0)$ is an open connected subset of $\R$. 
	Let a curve $\gamma_{x_0,\xi_0} \colon s \in I \mapsto (x(s),\xi(s)) \in T^* \Rn \backslash 0$ satisfies
	\begin{equation} \label{eq:Has-PM2021}
	\left\{\begin{aligned}
	& \dot x(s) = \nabla_\xi p(x(s),\xi(s)), \ \dot \xi(s) =- \nabla_x p(x(s),\xi(s)), \\
	& (x(0),\xi(0)) = (x_0,\xi_0) \in T^* \Rn \backslash 0.
	\end{aligned}\right.
	\end{equation}
	We call $\boxed{\gamma_{x_0,\xi_0}(t)}$ a \emph{bicharacteristic} of $p$. \index{bicharacteristic}
	Furthermore, if $p(x_0,\xi_0) = 0$, then we shall have $p(x(s),\xi(s)) = 0$ for $\forall s$ and call $\gamma_{x_0,\xi_0}(t) := (x(t), \xi(t))$ a \emph{null bicharacteristic} \index{null bicharacteristic} of $p$.
\end{defn}

Note that in Definition \ref{defn:nbi-PM2021}, the function $p$ is assumed to be real-valued.
Without this assumption, we cannot guarantee $(x(s),\xi(s))$ are coordinates.

\begin{lem} \label{lem:haE-PM2021}
	For small enough $\epsilon > 0$, there exists a unique solution $\gamma \colon (-\epsilon, \epsilon) \mapsto (x(s), \xi(s)) \in T^* \Rn \backslash 0$ for the Hamiltonian equation \eqref{eq:Has-PM2021}.
	Moreover, assume either
	\begin{enumerate}
	\item $\nabla_{(x,\xi)} p$ is uniformly bounded in $\{ (x(s), \xi(s)/|\xi(s)|) \,;\, s \in (-\epsilon, \epsilon) \}$, $p$ is homogeneous of order $1$;
	
	\item or $\nabla_{x} p$ is uniformly bounded in $\{ (x(s), \xi(s)) \,;\, s \in (-\epsilon, \epsilon) \}$, $p$ is homogeneous of order $\mu \geq 1$, and $\nabla_\xi p$ is bounded in uniformly bounded in $\{ (x(s), \xi(s)/|\xi(s)|) \,;\, s \in (-\epsilon, \epsilon) \}$;
	\end{enumerate}
	then the domain of definition of $\gamma$ can be extended from $(-\epsilon, \epsilon)$ to $\R$.
\end{lem}

\begin{proof}
	{\bf Part 1:} local solution.
	We use the Banach fixed-point theorem to show the existence of local solution.
	For simplicity denote $\eta_0 := (x_0, \xi_0)$ and $\eta(s) := (x(s), \xi(s))$ and $F(\eta(s)) := (\nabla_\xi p(\eta(s)), -\nabla_x p(\eta(s)))$, and we define a mapping $\mathscr F$:
	\[
	\mathscr F \colon \eta \in C(I, T^* \Rn \backslash 0) \ \mapsto \ \eta_0 + \int_0^s F(\eta(\tau)) \dif \tau \in C(I, T^* \Rn \backslash 0).
	\]
	Fix $\epsilon \leq (2\nrm{\nabla F})^{-1}$,  and let $I = (-\epsilon, \epsilon)$.
	Then for any $\eta_1$, $\eta_2 \in C(I, T^* \Rn \backslash 0)$, we have
	\begin{align*}
	\nrm[C(I)]{\mathscr F \eta_1 - \mathscr F \eta_2}
	& = \nrm[C(I)]{\int_0^s [F(\eta_1(\tau)) - F(\eta_2(\tau))] \dif \tau} \\
	& \leq \epsilon \nrm[C(I)]{F(\eta_1) - F(\eta_2)}
	\leq \epsilon \nrm{\nabla F} \nrm[C(I)]{\eta_1 - \eta_2} \\
	& \leq \frac 1 2 \nrm[C(I)]{\eta_1 - \eta_2}.
	\end{align*}
	The Banach fixed-point theorem can be applied, and we can find a fixed point $\eta$ of $\mathscr{F}$ such that
	\[
	\eta(s) = \eta_0 + \int_0^s F(\eta(\tau)) \dif \tau, \ \forall s \in I
	\quad \Rightarrow \quad
	\eta \text{~satisfies~} \eqref{eq:Has-PM2021}.
	\]
	We proved the existence.
	
	For the uniqueness, assume $\eta_1$, $\eta_2$ solve \eqref{eq:Has-PM2021}.
	Because $\eta_1(0) = \eta_2(0)$, if there are not equal, their derivatives must be differ at a point, but this violates the first two equations in \eqref{eq:Has-PM2021}.
	The first part of the claim is proven.

	{\bf Part 2:} global solution (cf \cite[\S 2]{salo2006sta}).
	To obtain the global solution, we can extend the local solution from $(-\epsilon, \epsilon)$ to $[-\epsilon, \epsilon]$, and then just paste local solutions on $[-\epsilon, \epsilon]$, $[\epsilon - \epsilon', \epsilon + \epsilon']$, $[\epsilon + \epsilon' - \epsilon'', \epsilon + \epsilon' + \epsilon'']$, etc.
	Now we show the endpoints extensions can be done.
	Assume $\gamma$ is a local solution on $I = (-\epsilon, \epsilon)$ as given in {\bf Part 1}.
	
	Assume $\nabla_{(x,\xi)} p$ is uniformly bounded in $\{ (x(s), \xi(s)/|\xi(s)|) \,;\, s \in (-\epsilon, \epsilon) \}$, and $p$ is homogeneous of order $1$.
	From \eqref{eq:Has-PM2021} we can have
	\begin{equation} \label{eq:xin2-PM2021}
	\frac{\df}{\df s} \big( |\xi(s)|^2 \big)
	= 2 \xi(s) \cdot \dot \xi(s)
	= -2 \xi(s) \cdot \nabla_x p(x(s),\xi(s))
	= m(s) |\xi(s)|^2,
	\end{equation}
	where $m(s) := -2 \hat \xi(s) \cdot \nabla_x p(x(s), \hat \xi(s))$ with $\hat \xi(s) := \xi(s)/|\xi(s)|$.
	Here because  $\xi(0) = \xi_0 \neq 0$, and $\xi(s)$ is continuous on $s$, so we can choose the interval $I = (-\epsilon, \epsilon)$ to be small enough such that $\xi(s) \neq 0$ for $\forall s \in I$, and this can make $\hat \xi(s)$ always well-defined.
	Solve \eqref{eq:xin2-PM2021} we obtain
	\[
	|\xi(s)| = e^{-\int_0^s \hat \xi(\tau) \cdot \nabla_x p(x(\tau), \hat \xi(\tau))\dif \tau} |\xi(0)|,
	\]
	so
	\[
	e^{-\epsilon M_1} |\xi(0)|
	\leq |\xi(s)|
	\leq e^{\epsilon M_1} |\xi(s)|
	\]
	where $M_1 = \sup_{\{ (x(s), \xi(s)/|\xi(s)|) \,;\, s \in (-\epsilon, \epsilon) \}} |\nabla_x p|$, so $\{\xi(s) \,;\, s \in (-\epsilon, \epsilon)\}$ is contained in a bounded domain.
	Similarly, for $x(s)$ we have
	\[
	|\dot x(s)|
	= |\nabla_\xi p(x(s),\xi(s))|
	= |\nabla_\xi p(x(s),\hat \xi(s))|
	\leq \sup_{\R_x^n \times \mathbb S^{n-1}} |\nabla_\xi p|
	\]
	where $M_2 = \sup_{\{ (x(s), \xi(s)/|\xi(s)|) \,;\, s \in (-\epsilon, \epsilon) \}} |\nabla_\xi p|$.
	Note that we used the homogeneity of $p$ again.
	This gives
	\[
	|x(s) - x(0)| \leq s n^{1/2} M_2.
	\]
	
	Or, if $\nabla_{x} p$ is uniformly bounded in $\{ (x(s), \xi(s)) \,;\, s \in (-\epsilon, \epsilon) \}$, $p$ is homogeneous of order $\mu \geq 1$, and $\nabla_\xi p$ is uniformly bounded in $\{ (x(s), \xi(s)/|\xi(s)|) \,;\, s \in (-\epsilon, \epsilon) \}$, by $|\dot \xi(s)| = |\nabla_x p(x(s), \xi(s))|$ we can have
	\[
	|\xi(s)|
	= |\xi(0) + \int_0^s \dot \xi(\tau) \dif \tau|
	\leq |\xi(0)| + \int_0^s |\nabla_x p| \dif \tau
	\leq |\xi(0)| + s M
	\leq |\xi(0)| + \epsilon M_1'.
	\]
	where $M_1' = \sup_{\{ (x(s), \xi(s)) \,;\, s \in (-\epsilon, \epsilon) \}} |\nabla_x p|$.
	And similarly, for $x(s)$ we have
	\[
	|\dot x(s)|
	= |\nabla_\xi p(x(s),\xi(s))|
	= |\xi(s)|^{\mu - 1} |\nabla_\xi p(x(s),\hat \xi(s))|
	\leq (|\xi(0)| + \epsilon M_1') M_2',
	\]
	where $M_2' = \sup_{\{ (x(s), \xi(s)/|\xi(s)|) \,;\, s \in (-\epsilon, \epsilon) \}} |\nabla_\xi p|$.

	Therefore, in both two cases the $(x(s), \xi(s))$ lives in a bounded domain when $s \in (-\epsilon, \epsilon)$, thus due to the continuity of $x(s)$ and $\xi(s)$ we can extend the domain of definition of $\gamma$ from $(-\epsilon, \epsilon)$ to $[-\epsilon, \epsilon]$.
	
	After extension, we set new initial value $(x_0, \xi_0)$ to be $(x(\epsilon), \xi(\epsilon))$ and by {\bf Part 1} we can get a local solution on $(\epsilon - \epsilon', \epsilon + \epsilon')$ for some small enough $\epsilon'$.
	By doing this repeatedly, we can obtain a solution defined in $\R$.
	The proof is complete.
\end{proof}

\begin{lem} \label{lem:pHom-PM2021}
	Let a symbol $p$ be homogeneous, i.e.~$p(x,\lambda \xi) = \lambda p(x,\xi)$ for $\lambda \in \R_+$.
	Then we have $\boxed{\gamma_{x_0,\lambda \xi_0}(t) = (x(t), \lambda \xi(t))}$ for $\forall \lambda \in \R_+$.
\end{lem}

\begin{proof}
	Because $p$ is homogeneous, from \eqref{eq:Has-PM2021} we have
	\begin{equation*}
	\left\{\begin{aligned}
	& \dot x(s)
	= \nabla_\xi p(x(s),\xi(s))
	= \nabla_\xi p(x(s), \lambda \xi(s)), \\
	& \lambda \dot \xi(s)
	= -\lambda \nabla_x p(x(s),\xi(s))
	= -\nabla_x p(x(s), \lambda \xi(s)), \\
	& p(x_0,\lambda \xi_0) = \lambda p(x_0, \xi_0),
	\end{aligned}\right.
	\end{equation*}
	so $(x(t), \lambda \xi(t))$ is also a solution of \eqref{eq:Has-PM2021}, with $(x(0), \lambda \xi(0)) = (x_0, \lambda \xi_0)$.
	The proof is done.
\end{proof}

%
%

\begin{lem} \label{lem:tHq-PM2021}
	Let $T > 0$.
	Assume a real-valued symbol $p \in S^1$ is homogeneous of order 1,
	and $F \in C^\infty([0,T] \times (T^* \Rn \backslash 0))$ and $\phi \in C^\infty(T^* \Rn \backslash 0)$.
	Then there exists a unique solution $q \in C^\infty(\R \times (T^* \Rn \backslash 0))$ satisfying
	\begin{equation*}
	\left\{\begin{aligned}
	(\partial_t + H_p) q(t,x,\xi) & = F(t,x,\xi), \\
	q(0,x,\xi) & = \phi(x,\xi),
	\end{aligned}\right.
	\end{equation*}
	where $H_p$ is the Hamiltonian of $p$.
	The solution is given by
	\begin{equation*}
	\forall (x_0,\xi_0) \in T^* \Rn, \quad q(t, \gamma_{x_0,\xi_0}(t))
	= \phi(x_0,\xi_0) + \int_0^t F(\tau, \gamma_{x_0,\xi_0}(\tau)) \dif \tau.
	\end{equation*}
	More, when $F$ and $\phi$ are homogeneous (with $\xi$) of order $m \in \R$, then $q$ is also homogeneous (with $\xi$) of order $m \in \R$.
\end{lem}

\begin{proof}
	Let $\gamma_{x_0,\xi_0}(t) = (x(t), \xi(t))$ be the bicharacteristic of $p$ starting from $(x_0,\xi_0)$.
	The existence of $\gamma_{x_0,\xi_0}$ is guaranteed by Lemma \ref{lem:haE-PM2021}.
	Then we have
	\begin{align*}
	(\partial_t + H_p) q(t,x(t),\xi(t))
	& = (\partial_t + \nabla_\xi p \cdot \nabla_x - \nabla_x p \cdot \nabla_\xi) q(t,x(t),\xi(t)) \\
	& = (\partial_t + \dot x(t) \cdot \nabla_x + \dot \xi(t) \cdot \nabla_\xi) q(t, x(t),\xi(t)) \\
	& = \frac{\df} {\df t} \big( q(t, x(t), \xi(t)) \big),
	\end{align*}
	so
	\(
	\frac{\df} {\df t} \big( q(t, x(t), \xi(t)) \big)
	= F(t, x(t), \xi(t)),
	\)
	which gives
	\begin{align*}
	q(t, x(t), \xi(t))
	& = q(0, x(0),\xi(0)) + \int_0^t F(\tau, x(\tau),\xi(\tau)) \dif \tau \\
	& = \phi(x_0,\xi_0) + \int_0^t F(\tau, x(\tau),\xi(\tau)) \dif \tau.
	\end{align*}

	For the homogeneity, fix $(x,\xi) \in T^* \Rn \backslash 0$, we solve the Hamiltonian equation with initial point $(x,\xi)$ and we can obtain a bicharacteristic $\gamma_{x, \xi}$.
	Fix $t \in \R$, we set $(x_0, \xi_0) := \gamma_{x, \xi}(-t)$, so reversely we represent $(x, \xi)$ as $\gamma_{x_0, \xi_0}(t) = (x(t), \xi(t))$.
	Because $p$ is homogeneous of order 1, by Lemma \ref{lem:pHom-PM2021} we have $(x(t), \lambda \xi(t)) = \gamma_{x_0,\lambda \xi_0}(t)$, so
	\begin{align*}
	q(t, x, \lambda \xi)
	& = q(t, x(t), \lambda \xi(t))
	= q(t, \gamma_{x_0,\lambda \xi_0}(t))
	= \phi(x_0,\lambda \xi_0) + \int_0^t F(\tau, \gamma_{x_0,\lambda \xi_0}(\tau)) \dif \tau \\
	& = \phi(x_0,\lambda \xi_0) + \int_0^t F(\tau, (x(t), \lambda \xi(\tau))) \dif \tau \\
	& = \lambda^m [\phi(x_0, \xi_0) + \int_0^t F(\tau, x(\tau), \xi(\tau)) \dif \tau]
	= \lambda^m q(t, x(t), \xi(t)).
	\end{align*}
	The proof is done.
\end{proof}

\section{Propagation of singularities} \label{sec:ProSing-PM2021}

For other literature on the topic, \cite[\S 10]{jos99int} is a good reference for this section.
See \cite[A.1.3]{shu2001pse}, \cite[\S 8]{grigis94mic} for different proofs.
Now we are ready for the main result.

\begin{thm} \label{thm:ProSing-PM2021}
	\index{propagation of the singularities}
	Assume $m \in \R$ and $P \in \Psi^m$ is classical {\rm $\Psi$DO} of real principal type, and denote its principal symbol as $p_m(x,\xi)$.
	We assume either
	\begin{itemize}
		
	\item $u \in \scrS'(\Rn)$, or,
	
	\item $P$ is properly supported and $u \in \mathscr D'(\Rn)$.
	
	\end{itemize}
	Let $P u \in C^\infty$ and $p_m(x_0, \xi_0) = 0$.
	If $(x_0,\xi_0) \notin \wf(u)$, then $\gamma_{x_0, \xi_0} \cap \wf(u) = \emptyset$ where the $\gamma_{x_0, \xi_0}$ is a null bicharacteristic of $p_m$ defined in Definition \ref{defn:nbi-PM2021}.
	In other words, for a null bicharacteristic $\gamma$, it holds either $\gamma \subset \wf(u)$ or $\gamma \cap \wf(u) = \emptyset$.
\end{thm}

\begin{proof}
	{\bf Step 1:} change to $\Psi^1$.
	Choose an elliptic $T_a \in \Psi^{1-m}$ with $a(x,\xi) > 0$ and $a(x,\xi)$ be real-valued, then
	\[
	\wf (T_a Pu) = \wf(Pu) \quad \Rightarrow \quad
	T_a Pu \in C^\infty \ \text{if and only if} \ Pu \in C^\infty,
	\]
	namely, $T_a$ doesn't change the wavefront set.
	
	Also, we can show $T_a$ doesn't change null bicharacteristics of the principal symbols as follows.
	Assume $(x(s),\xi(s))$ solves \eqref{eq:Has-PM2021} with $p_m(x(0), \xi(0)) = 0$.
	\sq{Let's assume we can find a function $f(s)$ such that
	\[
	f(s) := \int_0^s a(x(f(r)), \xi(f(r))) \dif r.
	\]
	This is possible because it amounts to find a fix point of the transform $\mathcal I \circ a \circ \gamma \colon C^\infty(\R; \R) \to C^\infty(\R; \R)$ where $\mathcal I g(s) := \int_0^s g(r) \dif r$ and $\gamma(s) := (x(s), \xi(s))$.}
	
	After obtained such an $f$, we can see $f$ is a bijection because $f' = a > 0$.
	Denote
	\begin{equation*}
	\tilde x(s) := x(f(s)), \quad \tilde \xi(s) := \xi(f(s)).
	\end{equation*}
	If $(x(s),\xi(s))$ is defined on a interval $I$, then we say $(\tilde x(s), \tilde \xi(s))$ is defined on a interval $I' := f^{-1}(I)$,
	so $p_m(\tilde x(s), \tilde \xi(s)) = 0$ for $s \in I'$, and we can have
	\begin{align*}
	\dot{\tilde x}(s)
	& = f'(s) \dot x(f(s))
	= a(x(f(s)), \xi(f(s))) \nabla_\xi p_m(x(f(s)), \xi(f(s))) \\
	& = a(\tilde x(s), \tilde \xi(s)) \nabla_\xi p_m(\tilde x(s), \tilde \xi(s)) \\
	& = a(\tilde x(s), \tilde \xi(s)) \nabla_\xi p_m(\tilde x(s), \tilde \xi(s)) + p_m(\tilde x(s), \tilde \xi(s)) \nabla_\xi a(\tilde x(s), \tilde \xi(s)) \\
	& = \nabla_\xi (ap_m)(\tilde x(s), \tilde \xi(s)).
	\end{align*}
	Similarly, we have
	\[
	\dot{\tilde \xi}(s) = -\nabla_x (ap_m)(\tilde x(s), \tilde \xi(s)).
	\]
	These mean the null bicharacteristic $(x(s),\xi(s))$ of $p_m$, after a reparametrization, is also a null bicharacteristic of $ap_m$.
	Note that $ap_m$ is the principal symbol of $T_a P$.
	Hence, to prove the claim for $P \in \Psi^m$ is equivalent to prove the claim for $P \in \Psi^1$,
	so, in the rest of the proof we assume $P \in \Psi^1$ of real principal type.
	
	{\bf Step 2:} find a $t$-dependent $Q = Q(t,x,D)$ such that
	\begin{equation} \label{eq:Qut0-PM2021}
	Qu|_{t= 0} \in C^\infty.
	\end{equation}
	
	Our plan is to construct a sequence of $t$-dependent {\rm $\Psi$DOs} $Q_j = Q_j(t,x,D) \in \Psi^{-j}$ ($j \geq 0$) having classical symbol $q_j$, and set $Q \sim \sum_j Q_j$.
	Here $Q_j(t,x,D) \in \Psi^{-j}$ means its symbol $q_j(t,x,\xi)$ is in $S^{-j}([0,T] \times \R_x^n \times \R_\xi^n)$, i.e.,
	\[
	|\partial_t^{\alpha'} \partial_x^{\alpha''} \partial_\xi^{\beta} q_j(t,x,\xi)| \lesssim \agl[\xi]^{-j-|\beta|},
	\]
	see Definition \ref{defn:symbolx-PM2021}.
	
	Because $(x_0,\xi_0) \notin \wf(u)$, we have $(x_0,t\xi_0) \notin \wf(u)$ for $\forall t > 0$, and we can find a open conic neighborhood $\omega$ of $(x_0,\xi_0)$ such that $\omega \cap \wf(u) = \emptyset$.
	Choose a function $\chi(x,\xi) \in C^\infty$ satisfying
	\begin{equation} \label{eq:chho-PM2021}
	\left\{\begin{aligned}
	& \chi(x,\xi) \in C^\infty, \ \supp \chi \subset \omega, \ \chi(x,\lambda \xi) = \chi(x,\xi)~(\forall \lambda > 0), \\
	& \chi \equiv 1 \text{~in a sufficiently small open conic neighborhood~} \tilde \omega \text{~of~} (x_0,\xi_0).
	\end{aligned}\right.
	\end{equation}
	Set $q_0(0,x, \xi) := \chi(x,\xi)$, then $\omega^c \subset \Smo (Q_0|_{t=0})$ where $\omega^c$ signifies the complement of the set $\omega$ in $T^* \Rn$, so by Theorem \ref{thm:MLpr-PM2021},
	\begin{align*}
	\wf(Q_0 u |_{t=0})
	& = \wf((Q_0 |_{t=0}) u) 
	\subset \wf(u) \backslash \Smo (Q_0|_{t=0}) \\
	& = \wf(u) \cap \big( \Smo (Q_0|_{t=0}) \big)^c \\
	& \subset \wf(u) \cap \omega
	= \emptyset,
	\end{align*}
	so \underline{$Q_0 u |_{t=0} \in C^\infty$}.
	For $Q_j~(j \geq 1)$, we set their symbol at $t = 0$ as zero, i.e.,
	\begin{equation} \label{eq:iniQj-PM2021}
	q_0(0,x,\xi) := \chi(x,\xi), \quad
	q_j(0,\cdot,\cdot) := 0~(j \geq 1),
	\end{equation}
	then $Q_j |_{t = 0} \in \Psi^{-\infty}$, so \underline{$Q_j u |_{t=0} \in C^\infty$ for $j \geq 1$}.
	By Theorem \ref{thm:Asy-PM2021} we can find a $Q$ satisfying $Q \sim \sum_j Q_j$ (thus $Q$ is also $t$-dependent)\footnote{Note that such $Q$ is not unique.}.
	$Q$ is of order $0$.
	We can conclude \eqref{eq:Qut0-PM2021}.
	
	{\bf Step 3:} to make $Q$ satisfy
	\begin{equation} \label{eq:PQuC-PM2021}
	(D_t + P)(Q u) \in C^\infty. 
	\end{equation}
	To achieve \eqref{eq:PQuC-PM2021} is equivalent to achieve
	\begin{equation} \label{eq:DtPQ-PM2021}
	[D_t + P,Q] \in \Psi^{-\infty}
	\end{equation}
	because
	\begin{align*}
	(D_t + P)(Q u)
	& = [D_t + P,Q] u + Q(D_t + P)u \\
	& = [D_t + P,Q] u + QPu
	= [P,Q] u + C^\infty.
	\end{align*}
	Here we used $\wf(QPu) \subset \wf(Pu) = \emptyset$ so $QPu \in C^\infty$
	The fact $Q D_t u = 0$ is because $u$ is independent of $t$.
	Readers may note that in {\bf Step 2} we only determined $q_j$ on $\{t = 0\}$, while $q_j$ on $\{t > 0 \}$ hasn't been fixed yet.
	Here we design $q_j |_{t > 0}$ to achieve \eqref{eq:DtPQ-PM2021}.

	We use the notation $\sigma(A)$ to signify the symbol of $A$.
	Because $P$ is classical, we can expand $\sigma(P)$ as $\sum_k p_k$ for some homogeneous symbols $p_k \in S^{1-k}$.
	Recall {\bf Step 1}, we see the integral curve of $H_{p_m}$ is the same as $H_{p_1}$.
	Then by Theorem \ref{thm:com-PM2021} and Remark \ref{rem:com2-PM2021}, we have
	\begin{align}
	\sigma([D_t,Q])
	& \sim \sum_\alpha \frac {(-i)^{|\alpha|}} {\alpha!} (\partial_\tau^\alpha \tau) \partial_t^\alpha (\sum_{j \geq 0} q_j) - \sum_\alpha \frac {(-i)^{|\alpha|}} {\alpha!} (\partial_t^\alpha \tau) \partial_\tau^\alpha (\sum_{j \geq 0} q_j) \nonumber \\
	& = \tau \sum_{j \geq 0} q_j + (-i) \partial_t \sum_{j \geq 0} q_j -  \tau \sum_{j \geq 0} q_j
	= \frac 1 i \sum_{j \geq 0} \partial_t q_j, \label{eq:DtQ-PM2021}
	\end{align}
	and
	\begin{align*}
	\sigma([P,Q])
	& \sim \sum_\alpha \frac {(-i)^{|\alpha|}} {\alpha!} \partial_\xi^\alpha (\sum_{k \geq 0} p_k) \partial_x^\alpha (\sum_{j \geq 0} q_j) - \sum_\alpha \frac {(-i)^{|\alpha|}} {\alpha!} \partial_x^\alpha (\sum_{k \geq 0} p_k) \partial_\xi^\alpha (\sum_{j \geq 0} q_j) \nonumber \\
	& = \sum_{k \geq 0} \sum_{j \geq 0} \sum_\alpha \frac {(-i)^{|\alpha|}} {\alpha!} \big( (\partial_\xi^\alpha p_k) \partial_x^\alpha - (\partial_x^\alpha p_k) \partial_\xi^\alpha \big) q_j \nonumber \\
	& = \sum_{\ell \geq 0} \sum_{j+k+|\alpha| = \ell} \frac {(-i)^{|\alpha|}} {\alpha!} \big( (\partial_\xi^\alpha p_k) \partial_x^\alpha - (\partial_x^\alpha p_k) \partial_\xi^\alpha \big) q_j \nonumber \\
	& = \sum_{\color{red}\ell \geq 1} \sum_{\substack{ j+k+|\alpha| = \ell \\ {\color{red}|\alpha| \geq 1}}} L_{j,k,\alpha} q_j, \quad (\text{it can be checked that~} L_{j,k,\alpha} q_j \in S^{1-\ell}) \nonumber
	\end{align*}
	where the linear differential operator $L_{j,k,\alpha} := \frac {(-i)^{|\alpha|}} {\alpha!} (\partial_\xi^\alpha p_k) \partial_x^\alpha - (\partial_x^\alpha p_k) \partial_\xi^\alpha$.
	Note that $p_0$ is the principal symbol of $P$ so \underline{$p_0$ is real-valued}.
	Also note that the restriction $\ell \geq 1$ and $|\alpha| \geq 1$ come from the fact that when $|\alpha| = 0$, $L_{j,k,\alpha} = p_k - p_k = 0$.
	It can be checked
	\begin{equation} \label{eq:aLH-PM2021}
	|\alpha| = 1 \ \Rightarrow \ \sum_{|\alpha| = 1} L_{j,k,\alpha} q_j = \frac 1 i H_{p_k} q_j.
	\end{equation}
	We can further compute $\sigma([P,Q])$ as
	\begin{align}
	\sigma([P,Q])
	& \sim \frac 1 i H_{p_0} q_0 + \sum_{\ell \geq 2} \sum_{\substack{ j+k+|\alpha| = \ell \\ |\alpha| \geq 1}} L_{j,k,\alpha} q_j \nonumber \\
	& = \frac 1 i H_{p_0} q_0 + \sum_{\ell \geq 2} \big( \sum_{\substack{ j+k+|\alpha| = \ell \\ |\alpha| \geq 1,\, {\color{red}j = \ell-1}}} L_{j,k,\alpha} q_j + \sum_{\substack{ j+k+|\alpha| = \ell \\ |\alpha| \geq 1,\, {\color{red}j < \ell-1}}} L_{j,k,\alpha} q_j \big) \nonumber \\
	& = \frac 1 i H_{p_0} q_0 + \sum_{\ell \geq 2} \big( \sum_{|\alpha| = 1} L_{j=\ell-1,k=0,\alpha} q_j + \sum_{\substack{ j+k+|\alpha| = \ell \\ |\alpha| \geq 1,\, j < \ell-1}} L_{j,k,\alpha} q_j \big) \nonumber \\
	& = \frac 1 i H_{p_0} q_0 + \sum_{\ell \geq 2} \big( \frac 1 i H_{p_0} q_{\ell-1} + \sum_{\substack{ j+k+|\alpha| = \ell \\ |\alpha| \geq 1,\, j < \ell-1}} L_{j,k,\alpha} q_j \big) \qquad (\text{by~} \eqref{eq:aLH-PM2021}) \nonumber \\
	& = \frac 1 i H_{p_0} q_0 + \sum_{\ell \geq 1} \big( \frac 1 i H_{p_0} q_{\ell} + \sum_{\substack{ j+k+|\alpha| = \ell+1 \\ |\alpha| \geq 1,\, j < \ell}} L_{j,k,\alpha} q_j \big) \qquad (\ell \to \ell-1). \label{eq:PQ-PM2021}
	\end{align}
	Combining \eqref{eq:DtQ-PM2021} with \eqref{eq:PQ-PM2021}, we obtain
	\begin{align}
	\sigma([D_t + P,Q])
	& \sim \frac 1 i (\partial_t + H_{p_0}) q_0 + \sum_{\ell \geq 1} \big( \frac 1 i (\partial_t + H_{p_0}) q_{\ell} + \sum_{\substack{ j+k+|\alpha| = \ell+1 \\ |\alpha| \geq 1,\, j < \ell}} L_{j,k,\alpha} q_j \big). \label{eq:PDtQ-PM2021}
	\end{align}
	The requirement \eqref{eq:DtPQ-PM2021} thus amounts to require $\sigma([D_t + P,Q]) \in S^{-\infty}$, namely,
	\begin{equation} \label{eq:HQj-PM2021}
	\left\{\begin{aligned}
	(\partial_t + H_{p_0}) q_0 & = 0, \\
	\frac 1 i (\partial_t + H_{p_0}) q_{\ell} & = - \sum_{\substack{ j+k+|\alpha| = \ell+1 \\ |\alpha| \geq 1,\, j < \ell}} L_{j,k,\alpha} q_j, \quad \ell \geq 1.
	\end{aligned}\right.
	\end{equation}
	Combining \eqref{eq:HQj-PM2021} with initial condition \eqref{eq:iniQj-PM2021}, these $q_j~(j \geq 0)$ can be solved iteratively in $[0,T] \times T^* \Rn$ by using see Lemma \ref{lem:tHq-PM2021} (recall that $p_0$ is real-valued), and gives, $\forall (x,\xi) \in T^* \Rn \backslash 0$,
	\begin{equation} \label{eq:HQjs-PM2021}
	\left\{\begin{aligned}
	q_0(t, \gamma_{x,\xi}(t)) & = \chi(x,\xi), \\
	q_{\ell}(t, \gamma_{x,\xi}(t)) & = -i \int_0^t \sum_{\substack{ j+k+|\alpha| = \ell+1 \\ |\alpha| \geq 1,\, j < \ell}} L_{j,k,\alpha} q_j(\tau, \gamma_{x,\xi}(\tau)) \dif \tau, \quad \ell \geq 1.
	\end{aligned}\right.
	\end{equation}
	And they guarantee $\sigma([D_t + P,Q]) \in S^{-\infty}$, so \eqref{eq:DtPQ-PM2021} is achieved, thus \eqref{eq:PQuC-PM2021} is satisfied.
	
	By iteration we can show the RHS of \eqref{eq:HQjs-PM2021} is of order $-\ell$, so the second conclusion in Lemma \ref{lem:tHq-PM2021} implies $q_\ell$ is homogeneous of order $-\ell$, so they are all classical symbols.

	{\bf Step 4:} apply a hyperbolic PDE result.
	Combining \eqref{eq:Qut0-PM2021} and \eqref{eq:PQuC-PM2021}, we can conclude
	\begin{equation}
	\left\{\begin{aligned}
	(D_t + P)(Q u) & = F \ \text{in} \ \R_+ \times \Rn, \\
	Q u |_{t = 0} & = \varphi \ \text{on} \ \Rn,
	\end{aligned}\right.
	\end{equation}
	for some $F \in C^\infty(\R_+ \times \Rn)$ and $ \varphi \in C^\infty(\Rn)$.
	Now we use Lemma \ref{lem:DtPsm-PM2021} in advance to conclude $Q u \in C([0,T], C^\infty(\Rn))$.

	{\bf Step 5:} conclusion.
	From $Q u \in C([0,T], C^\infty(\Rn))$ we see $\underline{Q|_t u \in C^\infty}$ for each $t \in [0,T]$, where $Q|_t$ is an abbreviation of $Q(t,x,D)$.
	From \eqref{eq:HQjs-PM2021} we see
	\(
	q_0(t,\gamma_{x_0, \xi_0}(t)) = \chi(x_0, \xi_0).
	\)
	By Lemma \ref{lem:pHom-PM2021}, the homogeneity of $p_0$ gives $(x(t), \lambda \xi(t)) = \gamma_{x_0,\lambda \xi_0}(t)$, so
	\begin{align}
	\forall \lambda > 0, \ q_0(t,x(t), \lambda \xi(t))
	& = q_0(t,\gamma_{x_0,\lambda \xi_0}(t)) \qquad \text{(by Lemma \ref{lem:pHom-PM2021})} \label{eq:pqh-PM2021} \\
	& = \chi(x_0, \lambda \xi_0)
	= \chi(x_0, \xi_0) \qquad \text{(by \eqref{eq:HQjs-PM2021}, \eqref{eq:chho-PM2021})} \nonumber \\
	& \neq 0. \qquad \text{(by \eqref{eq:chho-PM2021})} \nonumber
	\end{align}
	This means $Q$ is elliptic at $(x(t), \xi(t)) = \gamma_{x_0, \xi_0}(t)$, i.e.~
	\(
	\gamma_{x_0,\xi_0}(t) \notin \Char (Q|_t).
	\)
	Therefore, by Corollary \ref{cor:wAuC-PM2021},
	\begin{align*}
	\wf(u)
	& \subset \wf(Q|_t u) \cup \Char (Q|_t)
	= \Char (Q|_t),
	\end{align*}
	so
	\[
	\gamma_{x_0, \xi_0}(t) \notin \wf(u), \text{~for any~} t \in [0,T].
	\]
	which means $\gamma_{x_0, \xi_0} \cap \wf(u) = \emptyset$.
	The proof is complete.
\end{proof}

\begin{rem} \label{rem:ProSing-PM2021}
	The condition that $P$ is of real principal type is used in the following ways:
	\begin{itemize}
	
	\item real-valued: in {\bf Step 3}, in order to use Lemma \ref{lem:tHq-PM2021}, $p_m$ has to be real-valued;
	Also, when $p_m$ is real-valued, then $R := iP + (iP)^*$ is of order $0$.
	This is used in {\bf Step 4} which calls for Lemma \ref{lem:DtPsm-PM2021};
	
	\item $|\nabla_\xi p_m(x,\xi)| \neq 0$ when $p(x,\xi) = 0$: \sq{related to the solvability of \eqref{eq:HQjs-PM2021}? Every $(x,\xi) \in T^* \Rn \backslash 0$ shall be reachable};
	
	\item homogeneity: the condition ``$p_m(x,\lambda \xi) = \lambda p_m(x,\xi)$'' is used at \eqref{eq:pqh-PM2021} to guarantee $(x(t), \lambda \xi(t)) = \gamma_{x_0,\lambda \xi_0}(t)$.
	
	\end{itemize}
\end{rem}

Theorem \ref{thm:ProSing-PM2021} can be interpreted by the following claim.

\begin{cor}
	Assume $m \in \R$ and $P \in \Psi^m$ is classical {\rm $\Psi$DO} of real principal type, and denote its symbol as $p(x,\xi)$.
	Assume $Pu$ is well-defined and $P u \in C^\infty$.
	Then $\wf(u)$ is made of null bicharacteristic curves $\gamma_{x,\xi}$ for some $(x,\xi) \in p_m^{-1}(0)$.
\end{cor}

\begin{proof}
We see that $P$ is a {\rm $\Psi$DO} with principal symbol $p_m(x,\xi)$ homogeneous in $\xi$, so we can apply Lemma \ref{lem:pin0-PM2021} to conclude $\Char P = p_m^{-1}(0)$.
Also, when $Pu \in C^\infty$, by Corollary \ref{cor:wAuC-PM2021} we have
$\wf(u) \subset \Char P$, so
\begin{equation} \label{eq:wum0-PM2021}
\wf(u) \subset p_m^{-1}(0).
\end{equation}

For any $(x,\xi) \in \wf(u)$, by \eqref{eq:wum0-PM2021} we know $(x,\xi) \in p_m^{-1}(0)$.
Denote as $\gamma_{x,\xi}$ the null bicharacteristic of $p_m$ passing through $(x,\xi)$, then $\gamma_{x,\xi} \subset p_m^{-1}(0)$ because the value of $p_m$ is constant in bicharacteristics.
According to Theorem \ref{thm:ProSing-PM2021}, we can conclude $\gamma_{x,\xi} \subset \wf(u)$.
In summary, for every $(x,\xi) \in \wf(u)$ we have $\gamma_{x,\xi} \subset \wf(u)$ and $\gamma_{x,\xi}$ is a null bicharacteristic, so $\wf(u)$ is made of null bicharacteristic curves.
\end{proof}

\section{Cauchy problems of hyperbolic PDEs} \label{sec:CaHy-PM2021}

\begin{lem} \label{lem:DtPs-PM2021}
	Assume $T > 0$ and $s \in \R$, $P \in \Psi^1$ has a real-valued principal symbol.
	Denote $L = D_t + P$.
	There exists a constant $\lambda_0 > 0$ such that for any
	\[
	u \in C^1([0,T], H^s(\Rn)) \cap C([0,T], H^{s+1}(\Rn)),
	\]
	we have
	\begin{equation} \label{eq:DtPs-PM2021}
	\sup_{t \in [0,T]} e^{-\lambda t} \nrm[H^s]{u(t,\cdot)}
	\leq \nrm[H^s]{u(0,\cdot)} + 2 \int_0^T e^{-\lambda t} \nrm[H^s]{Lu(t,\cdot)} \dif t.
	\end{equation}
\end{lem}

\begin{proof}
	Denote $Q = iP$ and $L' = iL = \partial_t + Q$.
	Then
	\begin{align*}
	\sigma(Q + Q^*)
	& = \sigma(iP + (iP)^*)
	= i \sigma(P - P^*)
	= i [\sigma(P) + S^0 - \overline{\sigma(P)} - S^0]
	\in S^0,
	\end{align*}
	because the principal symbol of $P$ is real-valued.
	We denote $R = Q + Q^*$, then $R \in \Psi^0$ and thus is bounded in $L^2$.
	
	We prove the case $s = 0$ first.
	Denote $f(t) := \nrm[L^2]{e^{-\lambda t} u(t,\cdot)}^2$, then
	\begin{align*}
	f'(t)
	& = 2e^{-2\lambda t} \Re (\partial_t u, u) - 2\lambda f(t)
	= 2e^{-2\lambda t} \Re ((L' - Q) u, u) - 2\lambda f(t) \\
	& = 2e^{-2\lambda t} \Re (L'u, u) + e^{-2\lambda t} (-R u, u) - 2\lambda f(t) \\
	& \leq 2e^{-2\lambda t} \nrm[L^2]{Lu(t,\cdot)} \nrm[L^2]{u(t,\cdot)} + e^{-2\lambda t} \nrm[L^2]{Ru(t,\cdot)} \nrm[L^2]{u(t,\cdot)} - 2\lambda f(t) \\
	& \leq 2e^{-2\lambda t} \nrm[L^2]{Lu(t,\cdot)} \nrm[L^2]{u(t,\cdot)} - (2\lambda - \nrm{R}) \nrm[L^2]{e^{-\lambda t} u(t,\cdot)}^2 \\
	& \leq 2e^{-2\lambda t} \nrm[L^2]{Lu(t,\cdot)} \nrm[L^2]{u(t,\cdot)}, \qquad (\text{when~} \lambda > \nrm{R}/2),
	\end{align*}
	where $\nrm{R}$ is the $L^2$ operator norm.
	Hence, for any $t \in [0,T]$,
	\begin{align*}
	e^{-2\lambda t} \nrm[L^2]{u(t,\cdot)}^2
	& \leq \nrm[L^2]{u(0,\cdot)}^2 + 2 \int_0^t e^{-2\lambda s} \nrm[L^2]{Lu(s,\cdot)} \nrm[L^2]{u(s,\cdot)} \dif s \\
	& \leq \nrm[L^2]{u(0,\cdot)}^2 + 2 \int_0^T e^{-2\lambda t} \nrm[L^2]{Lu(t,\cdot)} \nrm[L^2]{u(t,\cdot)} \dif t.
	\end{align*}
	By denoting $M := \sup_{t \in [0,T]} e^{-\lambda t} \nrm[H^s]{u(t,\cdot)}$, we can continue
	\begin{align*}
	M^2
	& \leq \nrm[L^2]{u(0,\cdot)}^2 + 2 \int_0^T e^{-2\lambda t} \nrm[L^2]{Lu(t,\cdot)} \nrm[L^2]{u(t,\cdot)} \dif t \\
	& \leq M \nrm[L^2]{u(0,\cdot)} + 2 \int_0^T e^{-\lambda t} \nrm[L^2]{Lu(t,\cdot)} M \dif t \\
	& \leq M (\nrm[L^2]{u(0,\cdot)} + 2 \int_0^T e^{-\lambda t} \nrm[L^2]{Lu(t,\cdot)} \dif t).
	\end{align*}
	We arrive at the conclusion for $s = 0$.
	
	For $s \neq 0$, we can do something similar as in {\bf Step 2} of the proof of Theorem \ref{thm:GarIne-PM2021}.
	This completes the proof.
\end{proof}

Based on the energy estimate in Lemma \ref{lem:DtPs-PM2021}, we can obtain the following result.

\begin{lem} \label{lem:DtPsm-PM2021}
Assume $T > 0$ and $s \in \R$, $P \in \Psi^1$ has a real-valued principal symbol.
Let $f \in L^1((0,T), H^s(\Rn))$ and $\phi \in H^s(\Rn)$.
Then there is a unique solution $u \in C([0,T], H^s(\Rn))$ of the PDE
\begin{equation} \label{eq:DtPsm-PM2021}
\left\{\begin{aligned}
(D_t + P)u & = f \ \text{in} \ (0,T) \times \Rn, \\
u |_{t = 0} & = \phi \ \text{on} \ \Rn,
\end{aligned}\right.
\end{equation}
\end{lem}

\begin{proof}
	{\bf Step 1:} variational formulation.
	Denote
	\begin{equation*}
	\left\{\begin{aligned}
	& X := \{ \varphi \in C^\infty([0,T] \times \Rn) \,;\, \varphi(T, \cdot) \equiv 0 \} \\
	& \ell(\varphi) := \int_0^T (f, \varphi) \dif t + \frac 1 i (\phi, \varphi).
	\end{aligned}\right.
	\end{equation*}
	We say $u \in \scrS'([0,T] \times \Rn)$ is a weak solution of \eqref{eq:DtPsm-PM2021} if $u$ satisfies
	\begin{equation} \label{eq:DtPv-PM2021}
	\int_0^T (u, (D_t + P^*) \varphi) \dif t = \ell(\varphi), \quad \forall \varphi \in X.
	\end{equation}
	To find a $u \in L^{\infty}((0,T), H^s)$ satisfying \eqref{eq:DtPv-PM2021}, we are to show $|\ell(\varphi)| \leq \nrm[L^1 H^{-s}]{(D_t + P^*) \varphi}$, and the call for the Hahn-Banach theorem.
	Here $\nrm[L^1 H^{-s}]{f}$ is a shorthand for $\int_0^T \nrm[H^{-s}]{f} \dif t$.

	{\bf Step 2:} energy estimate.
	Because $P \in \Psi^1$ has a real-valued principal symbol, we see $-P^* \in \Psi^1$ and $-P^*$ also has a real-valued principal symbol.
	Apply Lemma \ref{lem:DtPs-PM2021} to $D_t + (-P^*)$ and $\varphi(T - t,x)$ we obtain
	\[
	\sup_{t \in [0,T]} e^{-\lambda t} \nrm[H^{-s}]{\varphi(T-t,\cdot)}
	\leq \nrm[H^{-s}]{\varphi(T,\cdot)} + 2 \int_0^T e^{-\lambda t} \nrm[H^{-s}]{(D_t + (-P^*))(\varphi(T-t,\cdot))} \dif t,
	\]
	which gives
	\[
	\sup_{t \in [0,T]} e^{\lambda t} \nrm[H^{-s}]{\varphi(t,\cdot)}
	\leq 2 \int_0^T e^{\lambda t} \nrm[H^{-s}]{(D_t + P^*) \varphi(t,\cdot)} \dif t.
	\]
	so
	\begin{equation} \label{eq:Tms-PM2021}
	\forall s \in \R, \quad \sup_{t \in [0,T]} \nrm[H^{-s}]{\varphi(t,\cdot)}
	\leq 2 e^{\lambda T} \nrm[L^1H^{-s}]{(D_t + P^*) \varphi}.
	\end{equation}
	This means the map $\varphi \mapsto (D_t + P^*) \varphi$ is injective.
	\eqref{eq:Tms-PM2021} can be understood as a coercive condition.
	
	{\bf Step 3:} Hahn-Banach theorem.
	By using \eqref{eq:Tms-PM2021}, we can estimate $\ell$ as follows,
	\begin{align*}
	|\ell(\varphi)| 
	& \leq \int_0^T |(f, \varphi)| \dif t + |(\phi, \varphi)|
	\leq \int_0^T \nrm[H^s]{f} \nrm[H^{-s}]{\varphi} \dif t + \nrm[H^s]{\phi} \nrm[H^{-s}]{\varphi} \\
	& \leq \big( \int_0^T \nrm[H^s]{f} \dif t + \nrm[H^s]{\phi} \big) \sup_{t \in [0,T]} \nrm[H^{-s}]{\varphi(t,\cdot)} \\
	& \leq C \big( \int_0^T \nrm[H^s]{f} \dif t + \nrm[H^s]{\phi} \big) \nrm[L^1 H^{-s}]{(D_t + P^*) \varphi}.
	\end{align*}
	Therefore, the linear functional $\ell(\varphi)$ is also a linear functional for $(D_t + P^*) \varphi \in X$ under the norm $L^1((0,T), H^{-s})$.
	Because the dual space of $L^1((0,T), H^{-s})$ is $L^{\infty}((0,T), H^s)$, by the Hahn-Banach theorem, there exists a $u \in L^{\infty}((0,T), H^s)$ such that
	\[
	\ell(\varphi) = (u, (D_t + P^*) \varphi)_{t,x}, \quad \forall \varphi \in X,
	\]
	which is \eqref{eq:DtPv-PM2021}.
	This $u$ is a weak solution.
	
	{\bf Step 4:} weak to strong solution.
	Because $u$ is a distribution, on $(0,T)$ we have
	\[
	D_t u + P u = f.
	\]
	Because $u \in L^{\infty}((0,T), H^s)$, $P u \in L^\infty((0,T), H^{s-1})$.
	
	Let $f$, $\phi$ be Schwartz, then $f \in L^\infty([0,T], H^s)$, so $D_t u = f - Pu \in L^\infty((0,T), H^{s-1})$, which implies
	\[
	u \in C([0,T], H^{s-1}).
	\]
	Again, $f \in C([0,T], H^{s-2})$ and $P u \in C([0,T], H^{s-2})$, so $D_t u = f - Pu \in C((0,T), H^{s-2})$, which implies
	\[
	u \in C^1([0,T], H^{s-2}) \cap C([0,T], H^{s-1}) \ \text{with} \ u(0) = \phi.
	\]
	Due to the arbitrary of $s$, we can conclude
	\begin{equation} \label{eq:uC1-PM2021}
	u \in C^1([0,T], H^s) \cap C([0,T], H^{s+1}) \ \text{with} \ u(0) = \phi.
	\end{equation}
	Therefore, $(u, (D_t + P^*) \varphi)_{t,x}$ can be legally write as $((D_t + P)u, \varphi)_{t,x}$, which implies $u$ is a strong solution of \eqref{eq:DtPsm-PM2021}.
	
	{\bf Step 5:} density arguments for $f$, $\phi$.
	\eqref{eq:uC1-PM2021} is true when $f$ and $\phi$ are Schwartz.
	For general $f \in L^1((0,T), H^s(\Rn))$ and $\phi \in H^s(\Rn)$, due to the density, we can find $\{f_k\} \subset \scrS$ and $\{\phi_k\} \subset \scrS$ such that
	\[
	f_k \to f \ \text{in} \ L^1((0,T), H^s(\Rn)), \
	\phi_k \to \phi \ \text{in} \ H^s(\Rn),
	\quad \text{and} \quad 
	\]
	\begin{equation} \label{eq:DPuk-PM2021}
	(D_t + P) u_k = f_k, \
	u |_{t = 0} = \phi_k, \
	u_k \in C^1([0,T], H^s) \cap C([0,T], H^{s+1}).
	\end{equation}
	From \eqref{eq:DPuk-PM2021} and Lemma \ref{lem:DtPs-PM2021} we can obtain
	\begin{equation*}
		e^{-\lambda T} \nrm[C([0,T{]},H^s)]{u_k - u_{k'}}
		\leq \nrm[H^s]{\phi_{k} - \phi_{k'}} + 2 \nrm[L^1((0,T),H^s)]{(f_{k} - f_{k'})(t,\cdot)},
	\end{equation*}
	so $\{u_k\}$ is Cauchy in $C([0,T{]},H^s)$ and the limit $u \in C([0,T{]},H^s)$ is a desired solution.

	{\bf Step 6:} uniqueness.
	By the energy estimate \eqref{eq:DtPs-PM2021} it is easy to show the uniqueness of $u$.
	
	The proof is complete.
\end{proof}

\section*{Exercise}

\begin{ex}
	Proof Lemma \ref{lem:pin0-PM2021}.
\end{ex}

\begin{ex} \label{ex:aG0-PM2021}
	Prove the function $a$ constructed in \eqref{eq:aG0-PM2021} is in $S^0$.
\end{ex}

\begin{ex} \label{ex:aCh-PM2021}
	Show that $a$ defined in \eqref{eq:ayp-PM2021} gives $(x_0, \xi_0) \notin \Char T_a$.
	Hint: to borrow ideas from Lemma \ref{lem:GEv-PM2021}.
\end{ex}


{

}
\printindex

\end{document}